\documentclass[letterpaper, 11pt,  reqno]{amsart}

\usepackage[margin=1.1in,marginparwidth=1.5cm, marginparsep=0.5cm]{geometry}

\usepackage{amsmath,amssymb,amscd,amsthm,amsxtra, esint,comment}

\usepackage{kotex}
\usepackage{mathrsfs} 

\usepackage[implicit=true]{hyperref}

\usepackage{color} 
\usepackage{ulem}
\usepackage[makeroom]{cancel}

\allowdisplaybreaks[2]

\definecolor{gr}{rgb}   {0.,   0.69,   0.23 }
\definecolor{bl}{rgb}   {0.,   0.5,   1. }
\definecolor{mg}{rgb}   {0.85,  0.,    0.85}
\definecolor{yl}{rgb}   {0.8,  0.7,   0.}
\definecolor{or}{rgb}  {0.7,0.2,0.2}

\newtheorem{theorem}{Theorem} [section]

\newtheorem{lemma}[theorem]{Lemma}
\newtheorem{proposition}[theorem]{Proposition}
\newtheorem{remark}[theorem]{Remark}

\DeclareMathOperator*{\supp}{supp}

\DeclareMathOperator{\Id}{Id}

\newcommand{\I}{\hspace{0.5mm}\text{I}\hspace{0.5mm}}
\newcommand{\II}{\text{I \hspace{-2.8mm} I} }
\newcommand{\III}{\text{I \hspace{-2.9mm} I \hspace{-2.9mm} I}}

\newcommand{\noi}{\noindent}
\newcommand{\Z}{\mathbb{Z}}
\newcommand{\R}{\mathbb{R}}

\newcommand{\T}{\mathbb{T}}

\let\Re=\undefined\DeclareMathOperator*{\Re}{Re}
\let\Im=\undefined\DeclareMathOperator*{\Im}{Im}

\let\P= \undefined
\newcommand{\P}{\mathbf{P}}

\newcommand{\E}{\mathbb{E}}

\newcommand{\F}{\mathcal{F}}

\newcommand{\al}{\alpha}
\newcommand{\be}{\beta}
\newcommand{\dl}{\delta}

\newcommand{\nb}{\nabla}

\newcommand{\Dl}{\Delta}
\newcommand{\eps}{\varepsilon}

\newcommand{\g}{\gamma}

\newcommand{\ld}{\lambda}
\newcommand{\Ld}{\Lambda}
\newcommand{\s}{\sigma}

\newcommand{\ft}{\widehat}

\newcommand{\wt}{\widetilde}
\newcommand{\cj}{\overline}
\newcommand{\dx}{\partial_x}

\newcommand{\dd}{\partial}

\newcommand{\LRA}{\Longrightarrow}

\newcommand{\too}{\longrightarrow}

\newcommand{\ta}{\theta}

\renewcommand{\o}{\omega}

\newcommand{\les}{\lesssim}
\newcommand{\ges}{\gtrsim}

\newcommand{\jb}[1]
{\langle #1 \rangle}

\newcommand{\ind}{\mathbf 1}

\newcommand{\M}{\mathcal{M}}

\newcommand{\N}{\mathbb{N}}

\newcommand{\GG}{\mathcal{G}}

\newcommand{\J}{\mathcal{J}}

\newtheorem*{ackno}{Acknowledgements}

\numberwithin{equation}{section}
\numberwithin{theorem}{section}

\newcommand{\PP}{\mathbb{P}}

\newcommand{\C}{\mathbb{C}}

\DeclareMathOperator{\Law}{Law}

\newcommand{\dr}{\theta}

\newcommand{\Ha}{\mathbb{H}_a}

\usepackage{tikz}

\usetikzlibrary{shapes.misc}
\usetikzlibrary{shapes.symbols}
\usetikzlibrary{shapes.geometric}
\tikzset{
	dot/.style={circle,fill=black,draw=black,inner sep=1pt,minimum size=0.5mm},
	>=stealth,
	}
\tikzset{
	ddot/.style={circle,fill=white,draw=black,inner sep=2pt,minimum size=0.8mm},
	>=stealth,
	}

\tikzset{decision/.style={ 
        draw,
        diamond,
        aspect=1.5
    }}

\tikzset{dia2/.style
={diamond,fill=white,draw=black,inner sep=0pt,minimum size=1mm},
	>=stealth,
	}

\tikzset{dia/.style
={star,fill=black,draw=black,inner sep=0pt,minimum size=1mm},
	>=stealth,
	}

\colorlet{symbols}{black}
\colorlet{testcolor}{green!60!black}

\def\1{\mathbf{{1}}}

\usetikzlibrary{shapes.misc}
\usetikzlibrary{shapes.symbols}
\usetikzlibrary{snakes}
\usetikzlibrary{decorations}
\usetikzlibrary{decorations.markings}
\definecolor{dblue}{rgb}{0.1, 0.1, 0.9}

\tikzset{
	root/.style={circle,fill=testcolor,inner sep=0pt, minimum size=2mm},		
	dot/.style={circle,fill=black,draw=black, solid,inner sep=0pt,minimum size=0.75mm},
	bdot/.style={circle,fill=blue,draw=dblue, solid,inner sep=0pt,minimum size=0.75mm},
		}
\colorlet{symbols}{blue!90!black}

\makeatletter
\def\DeclareSymbol#1#2#3{\expandafter\gdef\csname MH@symb@#1\endcsname{\tikz[baseline=#2,scale=0.15]{#3}}%
\expandafter\gdef\csname MH@symb@#1s\endcsname{\scalebox{0.6}{\tikz[baseline=#2,scale=0.15]{#3}}}}
\def\<#1>{\csname MH@symb@#1\endcsname}
\makeatother

\DeclareSymbol{sunset}{-3}{\draw (0,0) circle [x radius = 1.5, y radius = 1]; \draw (-1.5,0) node[dot] {} -- (1.5,0) node[dot] {};}
\DeclareSymbol{tadpole}{-3}{\draw (0,0) circle [x radius = 0.75, y radius = 1]; \draw (0,-1) node[dot] {};}

\DeclareSymbol{1}{0}{\draw[white] (-.4,0) -- (.4,0); \draw (0,0)  -- (0,1.2) node[dot] {};}
\DeclareSymbol{2}{0}{\draw (-0.5,1.2) node[dot] {} -- (0,0) -- (0.5,1.2) node[dot] {};}
\DeclareSymbol{3}{0}{\draw (0,0) -- (0,1.2) node[dot] {}; \draw (-.7,1) node[dot] {} -- (0,0) -- (.7,1) node[dot] {};}
\DeclareSymbol{4}{0}{\draw (-0.36,1.2) node[dot] {} -- (0,0) -- (0.36,1.2) node[dot] {}; \draw (-1,1) node[dot] {} -- (0,0) -- (1,1) node[dot] {};}
\DeclareSymbol{30}{-3}{\draw (0,0) -- (0,-1); \draw (0,0) -- (0,1.2) node[dot] {}; \draw (-.7,1) node[dot] {} -- (0,0) -- (.7,1) node[dot] {};}
\DeclareSymbol{31}{-3}{\draw (0,0) -- (0,-1) -- (1,0) node[dot] {}; \draw (0,0) -- (0,1.2) node[dot] {}; \draw (-.7,1) node[dot] {} -- (0,0) -- (.7,1) node[dot] {};}
\DeclareSymbol{32}{-3}{\draw (0,0) -- (0,-1) -- (1,0) node[dot] {}; \draw (0,0) -- (0,-1) -- (-1,0) node[dot] {}; \draw (0,0) -- (0,1.2) node[dot] {}; \draw (-.7,1) node[dot] {} -- (0,0) -- (.7,1) node[dot] {};}
\DeclareSymbol{20}{-3}{\draw (0,0) -- (0,-1);\draw (-.7,1) node[dot] {} -- (0,0) -- (.7,1) node[dot] {};}
\DeclareSymbol{22}{-3}{\draw (0,0.3) -- (0,-1) -- (1,0) node[dot] {}; \draw (0,0.3) -- (0,-1) -- (-1,0) node[dot] {};\draw (-.7,1) node[dot] {} -- (0,0.3) -- (.7,1) node[dot] {};}
\DeclareSymbol{202}{-3}{\draw (-.6625,0) -- (0,-1) -- (.6625,0)  {}; \draw (-1.1,1) node[dot] {} -- (-.6625,0) -- (-0.365,1) node[dot] {}; \draw (0.365,1) node[dot] {} -- (.6625,0) -- (1.1,1) node[dot] {};}

\makeatletter
\def\DeclareSymbol#1#2#3{\expandafter\gdef\csname MH@symb@#1\endcsname{\tikz[baseline=#2,scale=0.15]{#3}}}
\def\<#1>{\csname MH@symb@#1\endcsname}
\makeatother

\DeclareSymbol{X}{-2.4}{\node[dot] {};}
\DeclareSymbol{1}{0}{\draw[white] (-.4,0) -- (.4,0); \draw (0,0)  -- (0,1.2) node[dot] {};}
\DeclareSymbol{2}{0}{\draw (-0.5,1.2) node[dot] {} -- (0,0) -- (0.5,1.2) node[dot] {};}

\DeclareSymbol{sunset}{-3}{\draw (0,0) circle [x radius = 1.5, y radius = 1]; \draw (-1.5,0) node[dot] {} -- (1.5,0) node[dot] {};}
\DeclareSymbol{tadpole}{-3}{\draw (0,0) circle [x radius = 0.75, y radius = 1]; \draw (0,-1) node[dot] {};}

\DeclareSymbol{T0}{-2.7}
 { \draw (0,0) node[ddot]{};}

\DeclareSymbol{T1}{0}
 {  
 \draw (0,0) node[dot]{} -- (0,4) node[ddot] {}; 
 \draw (-3,0) node[dot] {} -- (0,4)node[ddot] {} -- (3,0) node[dot] {};
\draw[line width=1pt] (0, 4)node[ddot, label=above:$r_1$]{} 
.. controls(3,6) and (5,6) ..
(8, 4)node[ddot, label=above:$r_2$]{}
(4, 12) node[label ] {$\underline{j = 1}$};
 }

\DeclareSymbol{T2}{0}
 {\draw (0,0) node[dot]{} -- (0,4) node[ddot] {}; 
 \draw (-3,0) node[dot] {} -- (0,4)node[ddot] {} -- (3,0) node[dot] {};
 \draw (0,-4) node[dot]{} -- (0,0) node[dot] {}; 
 \draw (-3,-4) node[dot] {} -- (0,0)node[dot] {} -- (3,-4) node[dot] {};
\draw[line width=1pt] (0, 4)node[ddot, label=above:$r_1$]{} 
.. controls(3,6) and (5,6) ..
(8, 4)node[ddot, label=above:$r_2$]{}
(4, 12) node[label ] {$\underline{j = 2}$};
 }

\DeclareSymbol{T3}{0}
 {\draw (0,0) node[dot]{} -- (0,4) node[ddot] {}; 
 \draw (-3,0) node[dot] {} -- (0,4)node[ddot] {} -- (3,0) node[dot] {};
 \draw (0,-4) node[dot]{} -- (0,0) node[dot] {}; 
 \draw (-3,-4) node[dot] {} -- (0,0)node[dot] {} -- (3,-4) node[dot] {};
\draw (10,0) node[dot]{} -- (10,4) node[ddot] {}; 
 \draw (7,0) node[dot] {} -- (10,4)node[ddot] {} -- (13,0) node[dot] {};
\draw[line width=1pt] (0, 4)node[ddot, label=above:$r_1$]{} 
.. controls(4,6) and (6,6) ..
(10, 4)node[ddot, label=above:$r_2$]{}
(5, 12) node[label ] {$\underline{j = 3}$};
 }

\tikzstyle{dot1} = [ draw=  gray!00, 
 rectangle, rounded corners, fill=gray!00,  inner sep=1pt, inner ysep=3pt]

\tikzstyle{dot2} = [ draw=  black, 
ellipse, fill=gray!00,  inner sep=1pt, inner ysep=3pt]

\tikzstyle{dot3} = [ draw=  gray!00, 
ellipse, fill=gray!00,  inner sep=1pt, inner ysep=3pt]

\DeclareSymbol{T4}{0}
 {\draw (0,0) node[dot1]{$\phantom{n_2^{(1)} = n^{(2)}}$} -- (0,15) node[dot1] {$\phantom{n^{(1)}}$}; 
 \draw (-13,0) node[dot1] {$n_1^{(1)}$} -- (0,15)node[ddot] {$\phantom{n^{(1)}}$} 
 -- (13,0) node[dot1] {$n_3^{(1)}$};
 \draw (0,-15) node[dot1]{$n_2^{(2)}$} -- (0,0) node[dot1] {$\phantom{n_2^{(1)} = n^{(2)}}$}; 
 \draw (-13,-15) node[dot1] {$n_1^{(2)}$} -- (0,0)node[dot1] {$n_2^{(1)} = n^{(2)}$} -- (13,-15) node[dot1] {$n_3^{(2)}$};
\draw (40,0) node[dot1]{{$n_2^{(3)}$}} -- (40,15) node[dot3]  {$\phantom{n^{(1)} = n^{(3)}}$} ; 
 \draw (27,0) node[dot1] {$n_1^{(3)}$} -- (40,15)node[dot3]  {$\phantom{n^{(1)} = n^{(3)}}$}  -- (53,0) node[dot1] {$n_3^{(3)}$};
\draw[line width=1pt] (0, 15)node[ddot, label=above:$r_1$]{$n^{(1)}$} 
.. controls(10,23) and (27,23) ..
(40, 15)node[dot2, label=above:$r_2$] {$n^{(1)} = n^{(3)}$} ;
 }

\makeatletter
\def\DeclareSymbol#1#2#3{\expandafter\gdef\csname MH@symb@#1\endcsname{\tikz[baseline=#2,scale=0.15]{#3}}}
\def\<#1>{\csname MH@symb@#1\endcsname}
\makeatother

\begin{document}
\baselineskip = 14pt

\title[Central limit theorem]
{Central limit theorem for the focusing $\Phi^4$-measure in the infinite volume limit}


\author[K.~Seong and P. Sosoe]
{Kihoon Seong and Philippe Sosoe}

\address{Kihoon Seong\\
Department of Mathematics\\
Cornell University\\ 
310 Malott Hall\\ 
Cornell University\\
Ithaca\\ New York 14853\\ 
USA }

\email{kihoonseong@cornell.edu}

\address{Philippe Sosoe\\
Department of Mathematics\\
Cornell University\\ 
310 Malott Hall\\ 
Cornell University\\
Ithaca\\ New York 14853\\ 
USA }

\email{ps934@cornell.edu}

\subjclass[2020]{60H30, 35Q55, 81T08 }

\keywords{Gibbs measure; white noise;
soliton manifold; fluctuations;
 nonlinear Schr\"odinger equation}

\begin{abstract}
We study the fluctuations of the focusing $\Phi^4$-measure on the one-dimensional torus in the infinite volume limit. This measure is an invariant Gibbs measure for the nonlinear Schrödinger equation. It had previously been shown by B.~Rider that the measure is strongly concentrated around a family of minimizers of the Hamiltonian associated with the measure. These exhibit increasingly sharp spatial concentration, resulting in a trivial limit to first order. We study the fluctuations around this soliton manifold. We show that the scaled field under the Gibbs measure converges to white noise in the limit, identifying the next order fluctuations predicted by B.~Rider.
\end{abstract}



\maketitle


\tableofcontents

\setlength{\parindent}{0mm}
\setlength{\parskip}{6pt}


\section{Introduction}


\subsection{Implications of the main result}

In this paper, we study the fluctuations of the focusing $\Phi^4$-measure in the infinite volume limit. This measure has the following formal expression
\begin{align}
d\rho_L(\phi)=Z_L^{-1} e^{\frac{1}{4}\int_{\T_{L}}|\phi(x)|^4 dx-\frac 12\int_{\T_L}|\dx \phi(x)|^2 dx} \prod_{x\in \T_L} d\phi(x)
\label{Gibbs1}.
\end{align}

\noi 
Here, $Z_L$ is the partition function, $\T_L=\R / 2L\Z$ is a dilated torus of side length $2L>0$, and $\prod_{x\in\T_L}d\phi(x)$ represents the (non-existent) Lebesgue measure on fields $\phi: \T_L \to \C$. 
The Gibbs measure \eqref{Gibbs1} can be understood as an invariant measure for the following Hamiltonian PDE
\begin{align}
i \partial_t u + \dx^2 u +  |u|^2 u=0,
\label{NLS}
\end{align}

\noi 
known as the nonlinear Schrödinger equation (NLS) on $\T_L$, with the Hamiltonian
\begin{align}
H_L(\phi)=\frac 12 \int_{\T_L} |\dx \phi|^2 dx-\frac 14 \int_{\T_L} |\phi|^4 dx.
\label{Ham0}
\end{align}

\noi 
The nonlinear term in the Hamiltonian 
$H(\phi)$ is said to be focusing because it appears with a negative sign in \eqref{Ham0}. As a consequence, $H(\phi)$ is unbounded from below: it tends to 
$-\infty$ along certain directions in the phase space.  Due to this focusing nature, a rigorous construction of the measure \eqref{Gibbs1} 
\begin{align*}
d\rho_L(\phi)=Z_{L}^{-1}e^{-H_L(\phi)}\prod_{x \in \T_L} d\phi(x)
\end{align*}

\noi 
requires a mass cutoff, as in \eqref{Gibbs2}. The study of this Gibbs measure was initiated in seminal works of Lebowitz, Rose, and Speer \cite{LRS}, McKean and Vaninsky \cite{McKVa}, and continued by Bourgain \cite{BO94}, who confirmed that the Gibbs measure, suitably interpreted, is invariant under the flow of the equation \eqref{NLS}, as well as Brydges and Slade \cite{BS}, who studied the two-dimensional case. The original inspiration for \cite{LRS} was an analogy with equilibrium statistical mechanics for Hamiltonian systems. A major later development in this area was the construction, achieved by Burq and Tzvetkov \cite{BT1, BT2}, of global solutions to nonlinear wave equations in regimes where the equation is deterministically ill-posed, using random initial data and Bourgain’s methods. This approach has attracted significant attention in recent years. See Subsection \ref{SUBSEC:motiv}. 

Most of the works mentioned above use invariant Gibbs measures as a tool for proving global wellposedness for various classes of dispersive equations. The intrinsic properties of the invariant measures themselves have received more limited attention, especially in non-compact situations. This is because on short scales, which are most relevant to the wellposedness theory, sample paths have a Brownian structure. Understanding the structure of these measures as the volume tends to infinity is more delicate. 

In the defocusing case where one replaces the negative sign in the interaction \eqref{Ham0} by a positive one, one expects convergence to an analogous measure on $\mathbb{R}$ \cite{BL}. We note that the details of this convergence and the behavior of the corresponding invariant solutions have not been spelled out in all cases relevant to the analysis of PDEs. See \cite{BO20} and  \cite{FKV} for some examples in the case of the NLS, KdV and mKdV equations.

In the focusing case, McKean \cite{McK1} initially raised the question of the nature of the infinite volume limit. Rider \cite{Rider} realized that the concentration of measure due to the presence of solitons leads to triviality of the limit. This motivated the later investigations by Tolomeo–Weber \cite{TW} for Gibbs measures associated to the continuum NLS, and those of Rider \cite{Rider1},  Chatterjee–Kirkpatrick \cite{ChaKirk}, Chatterjee \cite{Chatt}, and Dey–Kirkpatrick–Krishnan \cite{DKK} for discrete models. The present work further advances this line of research.

Our interest is in the fluctuation behavior in the large scale limit in one dimension. For two and higher dimensions, a fundamental obstruction arises. As shown by Brydges and Slade \cite{BS}, the construction of the focusing Gibbs measure is not possible in the continuum, that is, there exists no meaningful probability distribution when  $d \ge 2$ corresponding to the two dimensional analog of \eqref{Ham0}. See also \cite{OSeoT}. This dimensional constraint naturally motivates a detailed analysis of the one-dimensional case, where the continuum measure is well-defined. Existing results for lattice models, further discussed in Section \ref{SEC:prers}, suggest that similar fluctuation behavior occurs in those models when they are in the appropriate phase. 

The main result of this article has the following implications:

\begin{itemize}
\item[(i)]  We resolve a conjecture of Rider \cite{Rider} by proving that the fluctuation behavior of the focusing Gibbs measure is characterized by white noise in the infinite volume limit.

\medskip 

\item[(ii)] 
By analyzing fluctuations of the Gibbs measure around the soliton manifold (a family of minimizers of the Hamiltonian \eqref{Ham0} on $\R$), our main theorem provides evidence for a generic statistical behavior of solutions of the equation \eqref{NLS}
\begin{align*}
u(t) \approx \text{soliton}+L^{-1}\cdot\textup{white noise}.
\end{align*}

\noi
This is consistent with the expectation that global solutions to NLS \eqref{NLS} asymptotically decompose into a localized (soliton-like) component and a small radiation term.

\medskip 

\item[(iii)]

We highlight a contrast in fluctuation behavior between the focusing Gibbs measure \eqref{Gibbs1} and the standard $\Phi^4$ quantum field theory, whose fluctuations are described by the Gaussian free field rather than white noise. In particular, the scaling required to observe Gaussian fluctuations in the focusing case differs from that in the $\Phi^4$ model.


\medskip

\item[(iv)] Our method is robust and can be extended to study fluctuation behavior for other continuum focusing Gibbs measures in the infinite volume limit, such as those associated with the KdV \cite{BO94} and Zakharov models \cite{BO94Zak}, where Gibbs measure constructions are also possible. 

\end{itemize}

As explained earlier, the study of the infinite volume limit behavior of the focusing Gibbs measure was initiated by McKean \cite{McK1} and later developed by Rider \cite{Rider} and Tolomeo–Weber \cite{TW}. These works primarily focus on describing the concentration behavior of the focusing Gibbs measure on $\T_L$ as $L\to \infty$. Here we take the next step, analyzing the fluctuations of the Gibbs measure in the infinite volume limit. In particular, the nature of the fluctuations is somewhat unexpected since in many quantum field theoretic and statistical field models, fluctuations are typically described by a Gaussian free field (or the underlying base field), rather than the rougher white noise; see \cite{ER1, ER2} and Subsection~\ref{SUBSEC:motiv}.

The proofs are based on the integration formula in Lemma \ref{LEM:chan}, inspired by Ellis–Rosen \cite{ER2}, and a careful study of the resulting conditional  Gibbs measure (with the $L^2$-cutoff explained below \eqref{Gibbs2}), large deviation theory, and the geometry of the soliton manifold.

As already stated, the construction of the focusing (continuum) Gibbs measure is not possible when $d \ge 2$, and thus our methods do not directly apply to extend our fluctuation result to higher dimensions. It would be interesting to investigate whether similar fluctuation behavior in the infinite volume limit can be established for discretized models on higher-dimensional lattices.



\subsection{Main result}\label{SUBSEC:result}

Before stating the main theorem, we first review the construction of the Gibbs measure \eqref{Gibbs1}.

Compared to the well-studied $\Phi^4$ model in quantum field theory, which corresponds to the defocusing case with a sign-definite Hamiltonian satisfying $H(\phi) \ge 0$, the focusing Gibbs measure \eqref{Gibbs1}  cannot be normalized as a probability measure. This is because the Hamiltonian \eqref{Ham0} is unbounded from below, and thus the density $e^{\frac 14 \int_{\T_L} |\phi|^4 dx }$ is not integrable with respect to the periodic Wiener measure $e^{-\frac 12 \int_{\T_L} |\dx \phi|^2 dx   }\prod d\phi(x)  $. 
In \cite{LRS}, Lebowitz, Rose, and Speer proposed introducing an additional $L^2$-cutoff to recover integrability as follows\footnote{
Here $Z_L$ denotes different normalizing constants that may differ from one line to line.}:
\begin{align}
d\cj \rho_L(\phi)=Z_L^{-1} e^{-H_L(\phi)} \ind_{ \{ M_L(\phi) \le LD \} } \prod_{x\in \T_L} d\phi(x),
\label{Gibbs2}
\end{align}

\noi
where $M_L(\phi)$ denotes the mass functional 
\begin{align}
M_L(\phi)=\int_{-L}^{L} |\phi|^2 dx.
\label{L2mass}
\end{align}


\noi 
The Gibbs measure \eqref{Gibbs2} is called a mixed ensemble, as it is canonical in the energy $H_L(\phi)$ and microcanonical in particle number $M_L(\phi)$, which is informally understood as a conditional distribution.   This $L^2$-mass truncation is particularly appropriate for the study of the statistical mechanics of NLS \eqref{NLS}, since the mass $M_L(\phi)$ is preserved by the flow of NLS \eqref{NLS}. In the thermodynamic limit $L\to \infty$, the $L^2$-mass $M_L(\phi)=\| \phi \|_{L^2 (\T_L)}^2$ and the volume $L$ grow proportionally, while the ``particle density" $D>0$  remains fixed.

Notice that the massless Gaussian field $e^{-\frac 12 \int_{\T_L}|\dx \phi|^2 dx }\prod_{x\in \T_L} d\phi(x)$, corresponding to the kinetic energy part of the Hamiltonian \eqref{Ham0}, is defined on the space of mean-zero fields to avoid the issue at the zeroth frequency. However, since the NLS \eqref{NLS} does not preserve the spatial mean under its flow, the massive Gaussian field 
\begin{align*}
e^{-\frac 12 \int_{\T_{L} }|\dx \phi|^2 dx -\frac 12 \int_{\T_L}|\phi|^2 dx  }\prod_{x \in \T_L} d\phi(x)
\end{align*}

\noi 
provides a more natural base measure for the Gibbs measure 
\begin{align}
d\rho_L(\phi) &= Z_L^{-1} e^{-H_L(\phi) - M_L(\phi)} \ind_{ \{ M_L(\phi) \le LD \} } \prod_{x \in \T_L} d\phi(x)
\label{Gibbs22}
\end{align}

\noi
because, in view of the conservation of the $L^2$-mass under the flow of NLS \eqref{NLS}, the corresponding Gibbs measure is invariant under the dynamics.

The main result of this paper is a central limit theorem for the Gibbs measure $ \rho_L$ \eqref{Gibbs22} in the infinite volume limit $L \to \infty$. More precisely, the field $\phi$ sampled from the Gibbs measure, appropriately scaled, exhibits Gaussian fluctuations around a soliton manifold  
\[\{e^{i\dr} Q(\cdot -x_0) \}_{x_0\in \R, \dr \in [0,2\pi]},\]
a family of minimizers of the variational problem
\begin{equation}\label{eqn: variational}
H(Q_{x_0,\dr})=\inf_{\substack{\phi\in H^1(\mathbb{R}),\\ \| \phi \|_{L^2(\R)}^2\le D}} H(\phi),
\end{equation}

\noi 
where   $Q_{x_0,\dr}:= e^{i\dr} Q(\cdot-x_0) $   and 
\begin{align}
H(\phi)=\frac 12 \int_{\R} |\dx \phi|^2 dx-\frac 14 \int_{\R} |\phi|^4 dx.
\label{HAmR}
\end{align}

\noi 
We begin by stating the main theorem of this paper, which describes the asymptotic fluctuations of the Gibbs measure $\rho_L$ in \eqref{Gibbs22}. In the following statement, $\mu>0$ is a fixed but arbitrary parameter and the weighted Sobolev spaces $H_{\mu}^s(\R)$ are as in Definition \ref{wesob} below.

\begin{theorem}\label{THM:1}
Let $s<-\frac 12$, $D>0$, and let $F$ be a bounded and continuous function on $H^{s}_\mu(\R)$. Then, we have
\begin{align}
\int F\big(L(\phi-\pi_L(\phi))\big) \rho_L(d\phi)  \too \int F(\phi) \nu(d\phi)
\label{THM1}
\end{align}

\noi
as $L \to \infty$. Here, $\pi_L$ is the projection map in $H^1(\T_L)$ onto the soliton manifold $\M_L$ 
\begin{align}
\M_L:=\big\{ e^{i \dr} \eta(\cdot-x_0) LQ(L(\cdot-x_0)): \dr \in [0,2\pi], \; x_0 \in \T_L  \big\}, 
\label{solm0}
\end{align}

\noi 
defined in \eqref{project}, where $\eta$ is the cutoff function\footnote{Here, the cutoff function is necessary to ensure that the solitons preserve the periodic boundary condition and that the manifold remains smooth.} in \eqref{cutoffeta},  and $\nu$ is the white noise measure on $H^s_\mu(\R)$ 
\begin{align*}
\nu(d\phi)=Z^{-1} e^{-\frac{1}{2} \int_{\R} |\phi|^2 dx } \prod_{x\in \R} d\phi(x).
\end{align*}

\noi
Hence, the scaled field under the Gibbs measure $\rho_L$ converges in distribution to white noise in the infinite volume limit.

\end{theorem}


Theorem \ref{THM:1} shows Gaussian fluctuations of the focusing $\Phi^4$-measure around the soliton manifold in the infinite volume limit $L\to \infty$. This fluctuation behavior was predicted by B. Rider in \cite{Rider}. Before discussing the fluctuations, we briefly review the first-order behavior of the focusing $\Phi^4$ measure, specifically the concentration in the infinite volume limit $L \to \infty$.

In \cite{Rider}, Rider proved a large deviation estimate identifying the asymptotic behavior of the free energy as $L\to \infty$ (Proposition \ref{PROP:con}). He found that, as $L \to \infty$, the leading-order paths under the Gibbs measure $\rho_L$ concentrate near a single soliton $e^{i\dr}LQ(L(\cdot - x_0))$ with height $L$ and width $\frac 1L$. The translation invariance of the measure $\rho_L$ implies that the position $x_0$ of this peak is uniformly distributed over the circle. These two facts combined result in the measure converging weakly to a delta function on the zero path: $\rho_L \to \dl_0$ as $L\to \infty$. Rider's result was refined by Tolomeo and Weber in \cite{TW}. For further details on other studies of the infinite volume limit of the Gibbs measure $\rho_L$ in \eqref{Gibbs22} and a comparison with our main theorem, see Remarks \ref{REM:RIDER}, \ref{REM:CHATT}, \ref{REM:TOLWEB}, and \ref{REM:DKK}.




Here we address the question of the fluctuations away from the mean behavior. Given the collapse, with the typical path tending to 0, one looks for a scaling $\lambda_L \to \infty$ to ensure a nontrivial limit law for the scaled field $\lambda_L(\phi - e^{i\dr}LQ(L(\cdot - x_0)))$ under the Gibbs measure $\rho_L$. We show that with the scaling $\lambda_L = L$, under the ensemble $\rho_L$, the field exhibits a general behavior
\begin{align}
\phi \approx e^{i \dr} LQ(L(\cdot-x_0)))+L^{-1}\cdot\textup{white noise}
\label{eqn:des1}
\end{align}

\noi
as $L\to \infty$, where $\dr \in [0, 2\pi]$ and $x_0\in \T_L$ are independent and uniformly distributed random variables. The surprising fact here is that the fluctuation is of a different nature than the initial base field, which has the regularity of Brownian motion (recall that white noise has regularity $-\frac 12-$ in the classical Sobolev or H\"older scales, while the Brownian path has regularity $\frac 12-$). Compare this, for example, to fluctuations of the $\Phi^4$ model in quantum field theory, where the fluctuations behave like the base field, namely the Gaussian free field. For more details on these fluctuations in the low temperature limit, see \cite{ER2, ER1, GST24} and Subsection \ref{SUBSEC:motiv}. We also note that for the $\Phi^4$ model in quantum field theory, the infinite-volume limit is qualitatively influenced by the temperature parameter $\be>0$.  In contrast, theorem \ref{THM:1} is true regardless of the temperature scale for the temperature-dependent ensemble $e^{-\be H(\phi)} \prod_{x} d\phi(x)$. In other words, the system does not exhibit a phase transition.

We emphasize that the choice of scaling $L$  in \eqref{THM1} is crucial for obtaining the white noise fluctuation around the soliton manifold.  Notably, this scaling differs from the classical square root scaling that appears in central limit theorem behavior. This reflects the distinct fluctuation nature of the focusing Gibbs measure, in contrast to the standard $\Phi^4$ quantum field theory, where the Gaussian fluctuation (given by the Gaussian free field) arises under square root scaling. See Subsections~\ref{SUBSEC:strucpf} and \ref{SUBSEC:motiv} for details on the choice of scaling $L$ and further discussion.

\begin{remark}\rm 
Theorem \ref{THM:1} also applies when the Gibbs measure in \eqref{Gibbs22}, with the massive Gaussian field as the base field, is replaced by the one in \eqref{Gibbs2}, with the massless Gaussian field as the base field.
\end{remark}


\subsection{Related results}\label{SEC:prers}

\begin{remark}\rm \label{REM:RIDER}

In \cite{Rider1}, Rider studied a non-trivial limit for the scaled field $\sqrt{L}\phi$ in a discretized approximation of the Gibbs measure in \eqref{Gibbs2}. He proved white noise fluctuations of this field, but only for sufficiently large densities $D \gg 1$. In particular, Rider stated in \cite{Rider1}, ``Now such a computation in the infinite dimensional diffusion ensemble is beyond us".  Our proof of Theorem \ref{THM:1} directly handles the ensemble in the continuum, with no restriction on the magnitude of the density $D$ in \eqref{Gibbs2}. 


\end{remark}

\begin{remark}\rm \label{REM:CHATT}

In \cite{Chatt}, Chatterjee studied the long-term behavior of the discretized version of PDE \eqref{NLS} using microcanonical invariant measures. He showed that, in a joint infinite volume and continuum limit, a typical function in the ensemble decomposes into a ``visible" part, which is close to a single soliton in the $L^\infty$ distance, and an ``invisible" or a radiating part that is small in the $ L^\infty$ norm. 

Our result in Theorem \ref{THM:1} and the representation \eqref{eqn:des1} can be interpreted as providing a precise description of the radiation part in \cite{Chatt} as $L^{-1} \cdot $ white noise. For more details on the long-term behavior of PDE \eqref{NLS} in terms of solitons and radiation, and a comparison between deterministic and probabilistic initial data, see Subsection \ref{SUBSEC:motiv}.

\end{remark}

\begin{remark}\rm \label{REM:TOLWEB}
In \cite{TW}, Tolomeo and Weber studied the infinite volume limit of the Gibbs measure, emphasizing the role of the size of nonlinear interaction. Specifically, they studied the following Gibbs measure with a massive Gaussian free field: for any $\be,\g,\al>0$
\begin{equation}\label{eqn: TW-measure}
d\rho_{L,\g}(\phi)=Z_{L}^{-1}e^{\frac \be{L^\g} \int_{\T_L} |\phi|^4  dx}  \ind_{ \{M_L(\phi) \le LD \}} d\mu_{L,\al},
\end{equation}

\noi
where $\mu_{L,\al}$ denotes the massive free field with covariance
$(-\dx^2+\al)^{-1}$ on $\T_L$. Here, $\frac{\be}{L^\g}$ represents the strength of nonlinear interactions,  and $\al>0$ is a mass term. Tolomeo and Weber consider more general power nonlinearities $|\phi|^p$, $2<p<6$. For clarity we only discuss the case $p=4$ which we also treat in this paper.  The Gibbs measure $\rho_{L,\g}$ exhibits different behaviors in the limit $L\to \infty$, depending on the size of $\frac{\be}{L^\g}$ in front of the nonlinear interaction.

When $\g<1$ (strong nonlinearity), the Gibbs measure $\rho_{L,\g}$  converges weakly to the delta measure on the zero path, that is, $\rho_{L,\g} \to \dl_0$ as $L\to \infty$.

When $\g>1$ (weak nonlinearity), the Gibbs measure $\rho_{L,\g}$ converges weakly to an Ornstein-Uhlenbeck measure $\mu_{\text{OU}}$ on $\R$ as $L\to \infty$. Thus, in this regime, the nonlinearity no longer plays a role in the limit. We note that this seemingly intuitive fact is by no means trivial to prove.\footnote{Tolomeo and Weber prove this result under the assumption that the mass density $D$ in the constraint is larger than $1/2\sqrt{\alpha}$, the average mass density of the Ornstein-Uhlenbeck process. By the equivalence of ensembles heuristic, we expect that the convergence can be extended beyond this case, provided one adjusts the mass of the limiting OU process to reflect the constraint.}

When $\g=1$ (critical case), the limiting behavior is determined by the size of the parameter $\be$. Specifically, there exist $\beta_0\ll 1 \ll \be_1$ such that for $\be < \be_0$, $\rho_{L,1}$ converges to $\mu_{\text{OU}}$ weakly under the further density condition $D>\frac 1{2\sqrt{\al}}$, while  for $\beta> \beta_1$, no limit point of $\rho_{L,1}$ is equal to $\mu_{\text{OU}}$. 

The threshold $\gamma=1$ can be guessed heuristically by considering the effect of inserting a rescaled, smooth (``soliton") profile of the form
\[\varphi_L(\cdot):=L^a\varphi(L^b \cdot)\]
in the energy 
\begin{align*}
H_{\beta}(\phi)&=\int_{\T_L} |\partial_x \phi|^2 dx-\frac{\be}{L^\g} \int_{\T_L} |\phi|^4  dx+\frac{\alpha}{2}\int_{\T_L} |\phi|^2 dx\\
&=H_{0,\beta}(\phi)+\frac{\alpha}{2}\int_{\T_L} |\phi|^2 dx 
\end{align*}

\noi 
corresponding to \eqref{eqn: TW-measure}.
The mass constraint $\{M_L(\phi)\le LD\}$ leads to $a=\frac{1}{2}(b+1)$. Balancing the kinetic and potential terms requires $b=1-\g$ and so\footnote{The two $H_{0,\beta}$ on the left and right hand sides are defined on different intervals due to the change of variables. However, we ignore this distinction here and focus only on their size.  }
\begin{align*}
H_{0,\beta}(\varphi_L)= L^{3-2\g}H_{0,\beta}(\varphi)\sim L^{3-2\g},
\end{align*}

\noi
in which case the energy scales as
\begin{align*}
H_{\beta}(\varphi_L)= L^{3-2\gamma}H_{0,\beta}(\varphi)+O(L)\sim L^{3-2\g}+O(L).
\end{align*}


\noi 
Only when $\g=1$ does the energy scale linearly in volume, making it comparable to that of the original Ornstein–Uhlenbeck process, whose Hamiltonian with respect to a Brownian loop includes the additional $L^2$ term. In this critical regime, $\beta$ becomes significant, and depending on its magnitude, two distinct behaviors emerge: one featuring soliton formation and the other not. See the next remark \ref{REM:DKK} for a discussion of this critical case in the discrete setting. A detailed analysis of the critical case $\g=1$ for the one-dimensional continuum model is forthcoming in the work of Tolomeo and Weber \cite{TW-upcoming}.




In Theorem \ref{THM:1} we study the next-order fluctuation behavior under a sufficiently strong nonlinear effect ($\g=0$ and $\be>0$) for the leading order behavior to be driven by soliton formation.  Our results show that the measure decomposes into
$L^{-1}\cdot$ white noise plus a soliton with uniform center and modulation parameters, as described in \eqref{eqn:des1}.
In contrast, in the weakly nonlinear case ($\g>1$ or $\g=1$ with $\be<\be_0$), the Ornstein-Uhlenbeck noise $\mu_{\text{OU}}$ dominates, and no nonlinear effects appear in the limit.
We expect our methods can also identify the asymptotic second-order behavior of the measure in the case $0<\g<1$. It would be interesting to investigate the fluctuations in the case $\g=1$ with $\be>\be_1$, in the presence of the soliton.

\begin{remark}\rm \label{REM:DKK}
In \cite{DKK}, Dey, Kirkpatrick and Krishnan derive the higher dimensional analog of Tolomeo and Weber's result for the critical scaling $\gamma=1$, in the setting of NLS  on the lattice.   The Hamiltonian for their model, defined  on $\Lambda_n=\{0,\ldots,n-1\}^d$, which represents a discrete unit torus in dimension $d \ge 3$, is given by
\begin{equation}\label{eqn: discrete-H}
H_{\nu, N}(\psi) := \sum_{\substack{x,x'\in \Lambda_n\\x\sim x'}}|\psi_x-\psi_{x'}|^2 - \left( \frac{\nu}{N} \right)^{(p-1)/2} \cdot \frac{2 }{p+1}\sum_{x\in \Lambda_n}|\psi_x|^{p+1},
\end{equation}

\noi 
where $N=n^d$ denotes the total number of sites in the discrete unit torus $\Ld_n$ and $\nu$
represents the strength of the nonlinearity. Under the mass restriction
\[M(\psi):=\sum_{x\in \Ld_n} |\psi_x|^2 \le N,\]
the two terms in the Hamiltonian \eqref{eqn: discrete-H} are bounded by $N$. They become proportional if $\psi$ has size $N^{1/2}$ on a set of size $O(1)$, corresponding to the soliton phase. Otherwize, the first term, representing a kinetic energy, dominates. 

Dey-Kirkpatrick-Krishnan consider the Gibbs measure on $\mathbb{C}^{\Lambda_n}$ defined by the density
\[\frac{1}{Z_N(\theta, \nu)}\frac{}{}e^{-\theta H_{\nu,N}(\psi)}\mathbf{1}_{ \{M(\psi)\le N \}}\,\mathrm{d}\psi,\]

\noi
where $\dr$ denotes the inverse temperature. They identify a phase transition, similar to that in Tolomeo and Weber's result above (Remark \ref{REM:TOLWEB}) depending on the parameters $\theta$ and $\nu$. In one regime corresponding to small $\nu$, referred to as the dispersive phase, the limiting free energy density
\[\lim_{N\rightarrow \infty} \frac{1}{N}\log Z_N(\theta,\nu)\] 

\noi 
coincides with that of a Gaussian field. In the other regime, associated with large $\nu$, the ``solitonic" phase, the free energy receives a nontrivial contribution from the nonlinear term. 
The results in \cite{DKK} do not explicitly identify the limiting processes in either phase but they strongly suggest that in the dispersive phase, the limit is a free field, whereas in the solitonic phase, the field concentrates a macroscopic amount of mass on $O(1)$ sites (perhaps on a single site). 

At the beginning of the paper, we mentioned that the construction of the focusing Gibbs measure \eqref{Gibbs2} is not possible as a meaningful probability measure in two and higher dimensions. However, by discretizing the Gibbs measure, which is clearly well-defined on a higher dimensional lattice, it becomes natural to ask whether the fluctuation behavior around the soliton manifold, as described in Theorem \ref{THM:1}, arises in the regime where soliton formation occurs.

\end{remark}



\end{remark}

\subsection{Structure of the proof}\label{sec:structure_proof}\label{SUBSEC:strucpf}

In this subsection, we outline the structure of the proof, which is divided into the following five steps.



\textbf{Step~1~(Proposition \ref{PROP:REDUC}):} By the results of Rider \cite{Rider} and Tolomeo-Weber \cite{TW}, most of the probability mass of the Gibbs measure $\rho_L$ is concentrated in the region $\{\text{dist}(\phi, \M_L)\le \dl L^\frac 12\} $, where\footnote{For simplicity, we ignore the cutoff function $\eta$ introduced in \eqref{solm0}  } $\M_L=\{e^{i\dr} LQ(L(\cdot-x_0))  \}_{\dr, x_0}$,  with distance measured in the  $L^2$ metric. See also Proposition \ref{PROP:con}. 

In Proposition \ref{PROP:REDUC}, we show that the main contribution to the measure $\rho_L$ comes from configurations where
$M_L(\phi) \approx LD$, with $M_L$ denoting   the $L^2$ mass, as defined in \eqref{mass}. Combining this with the rescaling $\phi(x) \mapsto \phi_L(x)=L\phi(Lx)$,  we restrict our analysis to the region
\begin{align}
Z_{L}^{-1}\int_{  \{   \text{dist}(\phi, \M  ) \le \dl    \} } F(L(\phi_L-\pi_L(\phi_L )))  e^{-L^3H_{L^2}(\phi)} \ind_{  \{ M_{L^2}(\phi) \approx D \} } \prod_{x\in \T_{L^2}} d\phi(x),
\label{Gibbs3}
\end{align}

\noi
where  $H_L(\phi_L)=L^3H_{L^2}(\phi)$ and  $\mathcal{M}=\{ e^{i\dr} Q(\cdot-x_0) \}_{\dr,x_0} $.
In the neighborhood $\{\text{dist}(\phi, \M  ) \le \dl \}  $, we introduce a new coordinate system, $\phi=e^{i\dr}Q(\cdot-x_0)+h=Q_{x_0,\dr}+h $, which we refer to as the orthogonal coordinate system $(x_0,\dr, h)$. Here, the axes  $x_0 \in \T_{L^2}$, $\dr \in [0,2\pi]$ are tangential to the manifold $\M$, while $h \in V^L_{x_0,\dr}$, defined in \eqref{normalsp},  is normal to the manifold and represents a small perturbation $\|h \|_{L^2(\T_{L^2})} \le \dl$.

\noi 
In these coordinates, we expand the Hamiltonian\footnote{In this discussion, for simplicity, we assume that the Hamiltonian $H$ and $L^2$-mass $M$ are defined on $\R$. In fact, as $L\to \infty$,  $H_L$ and $M_L$  on $\T_{L}$, as defined in \eqref{Ham0} and \eqref{L2mass}, approximate  $H$ and $M$ on $\R$. } 
\begin{align}
H(Q_{x_0,\dr}+h)=H(Q_{x_0,\dr})+\jb{\nb H(Q_{x_0,\dr}),h }+\frac 12 \jb{\nb^2 H(Q_{x_0,\dr})h,h }+\mathcal{E}(Q_{x_0,\dr},h ),
\label{EXPHa}
\end{align}

\noi 
where $\mathcal{E}(Q_{x_0,\dr},h )$ represents a higher-order error term. Concerning this expansion, we make the following key observations:

\begin{enumerate}
\item (First variation): The Lagrange multiplier method for the variational problem in \eqref{eqn: variational} implies that
\begin{align*}
\jb{\nb H(Q_{x_0,\dr}),h }= \Ld \jb{\nb M (Q_{x_0,\dr}),h }=\Ld\jb{Q_{x_0,\dr},h},
\end{align*}

\noi 
where $\Ld>0$. See \eqref{Lagmul}.

\smallskip

\item (Cutoff constraint):  The $L^2$ cutoff restriction $\|Q_{x_0,\dr}+h \|_{L^2}^2 \approx D $ in \eqref{Gibbs3} implies
\begin{align*}
\jb{Q_{x_0,\dr}, h } \approx -\frac 12 \|h \|_{L^2}^2,
\end{align*}

\noi
given that $\|Q_{x_0,\dr} \|_{L^2} =D$ from Proposition \ref{PROP:Min}.

\smallskip
\item (Second Variation): The second variation is given by 
\begin{align*}
\nb^{2}H(Q_{x_0,\dr})=-\dx^2-3Q_{x_0,\dr}^2.
\end{align*}

\end{enumerate}

\noi 
Combining (1), (2), and (3), the expansion \ref{EXPHa} shows that $h$ is approximately distributed as Gaussian with covariance operator $A=(-\dx^2-3Q_{x_0,\dr}^2+\Ld)^{-1}$, the inverse of a Schr\"odinger operator. 
Using the new coordinate system $(x_0,\dr,h)$ and the expansion \eqref{EXPHa}, we find
\begin{align}
\eqref{Gibbs3}\approx &Z_{L}^{-1} \int_{x_0 \in \T_{L^2} } \int_{\dr \in [0,2\pi]} \int_{ \| h\|_{L^2} \le \dl } F(L^2 h(L\cdot)) e^{ \mathcal{E}(Q_{x_0,\dr}, h)  } e^{-\frac {L^3}2 \jb{ A^{-1} h,h  }} \notag \\
&\hphantom{XXXXXXXXXXXXXXX} \cdot \ind_{ \{  \jb{Q_{x_0,\dr},h }  \approx -\frac 12 \|h \|_{L^2}^2  \}  }\prod_{x\in \T_{L^2}} dh(x)\, d\s(x_0,\dr),
\label{average}
\end{align}

\noi
where $d\s$ is a surface measure on the soliton manifold. See also Lemma \ref{LEM:chan}.

\noi 
Here, $Q_{x_0,\dr}$ is localized at $x_0$ with an exponentially decaying tail. Consequently,  $\nb^2 H(Q_{x_0,\dr})=-\dx^2-3Q_{x_0,\dr}^2  \approx -\dx^2 $ away from $x_0$, an operator for which the Green's function does not exhibit correlation decay. On the other hand, the two-point function (kernel) associated to  $A$ decays exponentially due to the positive mass term $\Ld>0$.

\textbf{Step 2 (Lemma \ref{lem: replace}, Proposition \ref{PROP:main}):} 
Applying the rescaling $h \mapsto L^{-\frac 32}h$ and controlling the error term $\mathcal{E}(Q_{x_0,\dr}, L^{-\frac 32}h)$, we rewrite the expression as follows:
\begin{align}
&Z_{x_0,\dr}^{-1}\int F(L^\frac 12 h(L\cdot))  e^{-\frac 12 \jb{A^{-1}h,h } } \ind_{ \{  \jb{Q_{x_0,\dr},h }  \approx -\frac 12 L^{-\frac 32} \|h \|_{L^2}^2  \}  } \prod_{x \in \T_{L^2} } dh(x)\notag \\
&\approx \E_{\nu_{Q,x_0}^\perp }\bigg[F(L^\frac 12 h(L\cdot) ) \Big|  \jb{Q_{x_0,\dr}, h }\approx -\frac 1{2L^\frac 32} \| h\|_{L^2}^2 \bigg], 
\label{COND}
\end{align}

\noi
where $Z_{x_0,\dr}$ is the corresponding partition function, and $\nu^\perp_{Q,x_0}$ is the Gaussian measure with covariance  $A=(-\dx^2-3Q_{x_0,\dr}^2+\Ld)^{-1}$, projected onto the normal space. See Lemma \ref{LEM:SCHOP}.

The expectation in \eqref{COND} is conditioned on an event describing a nonlinear constraint on $h$, making it difficult to handle, even for a Gaussian field. This conditioning represents an atypical event and therefore cannot be ignored. See the explanation in \eqref{COND2}. In Lemma \ref{lem: replace}  and Proposition \ref{PROP:main}, we find an appropriate conditional density $f_{h|\jb{Q,h} }$ such that
\begin{align}
\eqref{COND} \approx  \E_{\nu^\perp_{Q,x_0} }\Big[ F(L^\frac 12 h(L\cdot))  f_{h|\jb{Q,h} }(h) \Big].
\label{COND1}
\end{align}

\noi
Regarding the conditioning in \eqref{COND}, we make the following observation. Applying the law of large numbers for the Ornstein-Uhlenbeck-type measure $\nu^\perp_{Q,x_0}$, we obtain $\| h\|_{L^2(\T_{L^2})}^2 \approx L^2$ from Proposition \ref{PROP: gaussian-conc}, which leads to the approximation
\begin{align}
\jb{Q_{x_0,\dr}, h }\approx -L^{\frac 12}.
\label{COND2}
\end{align}

\noi 
Thus, one expects that under the measure $\nu^\perp_{Q,x_0}$, the field $h$ is forced to align in the opposite direction to the (positive) soliton $Q(\cdot-x_0)$. In other words, $h$ is pushed downward, and in particular $h<0$ in a neighborhood of $x_0$.

\textbf{Step 3 (Proposition \ref{PROP:glim}):} In the third step, in Subsection \ref{SUBSEC:Gaussianlimit} we  analyze a characteristic function\footnote{When considering the characteristic function, we need to separate $\jb{\Re h,g_L}$ and $\jb{\Im h,g_L}$. However, for simplicity of notation, we assume $h$ here as if it were real-valued in this sketch.} $F(L^\frac 12 h(L\cdot))=e^{i\jb{h,g_L}}$ under the conditional expectation \eqref{COND1}, where $g$ is a test function with $g_L=L^{-\frac 12 }g(L^{-1}\cdot)$.

The main strategy is to prove that  $\jb{h,g_L}$ and the conditional density function $f_{h|\jb{Q,h} }$ are weakly correlated. This allows us to exploit the almost independent structure as follows
\begin{align*}
\frac{ \E_{\nu^\perp_{Q, x_0} }\Big[ e^{i \jb{h, g_L} } f_{h|\jb{Q,h} }(h) \Big]   }{   \E_{\nu^\perp_{Q,x_0} }\Big[ f_{h|\jb{Q,h} }(h) \Big] }
&\approx \frac{ \E_{\nu^\perp_{Q,x_0} }\Big[ e^{i \jb{h, g_L} } \Big] \E_{\nu^\perp_{Q,x_0} } \Big[f_{h|\jb{Q,h} }(h) (1+o_L(1))\Big]   }{   \E_{\nu^\perp_{Q,x_0} }\Big[ f_{h|\jb{Q,h}}(h) \Big] }\\
& \approx  \E_{\nu^\perp_{Q,x_0} }\Big[ e^{i \jb{h, g_L} } \Big] 
\end{align*}

\noi 
as $L\to \infty$. Hence, we can reduce the analysis of the conditional expectation to that of the expectation $\E_{\nu^\perp_{Q}}\Big[ e^{i\langle h^\perp, g_L\rangle} \Big] $.

\textbf{Step 4 (Lemma \ref{LEM:varest}):}
Given step 3 and the Gaussian nature of the field, we are left with the task of studying the limit of the variance $\E_{\nu_{Q,x_0}^\perp}\big[|\jb{h,g_L}|^2\big]=\jb{G_{Q,x_0}g_L, g_L }$ as $L\to \infty$ where $G_{Q,x_0}(x,y)=(-\dx^2-3Q_{x_0,\dr}^2+1 )^{-1}(x,y)$ is the Green's function, projected onto the normal space. 
The analysis in Lemma \ref{LEM:varest} shows that if the test function $g$ is supported away from the tangential component $x_0 \in \T_{L}$, that is, $\text{dist}(\supp g,x_0)>0$, then 
\begin{align*}
\jb{G_{Q,x_0}g_L, g_L}  \approx \jb{G_{ \text{OU}  }g_L, g_L}  \too \| g\|_{L^2}^2
\end{align*}


\noi
as $L\to \infty$, where $G_{\text{OU}}(x,y)=(-\dx^2+1)^{-1}(x,y)$. This approximation is valid because when $G_{Q,x_0}$ is tested against the scaled test function $g_L$, the ground state term $-3Q_{x_0,\dr}^2$, localized at $x_0$,
in the covariance $A=(-\dx^2-3Q_{x_0,\dr}^2+1)^{-1}$ has a negligible effect, as there is almost no interaction between the ground state and the scaled test function $g_L$.  Furthermore, the scaling $g_L(x)=L^{-\frac 12}g(L^{-1}x)$, derived from $L^\frac 12 h(L\cdot)$, ensures that 
the Ornstein-Uhlenbeck operator  $\jb{G_{\text{OU}}g_L,g_L}$  converges to white noise as $L\to \infty$.

\textbf{Step 5 (Proposition \ref{PROP:CHAa}):} The mean-zero white noise limit in previous step holds only for test functions $g$ whose support is separated from the location parameter $x_0 \in \T_{L}$:  $\text{dist}(\supp g,x_0)>0$. In \eqref{average} we take an average over both tangential directions $x_0$ and $\dr$, so the region where the mean-zero white noise limit potentially fails becomes relatively small as the size of the circle grows:
\begin{align*}
\frac{\big| x_0\in \T_{L}: \text{dist}(\supp g, x_0)=0 \big|}{\big|  \T_{L}    \big|}=\frac{|\supp g|}{L} \too 0
\end{align*}

\noi
as $L\to \infty$ since $g$ has a compact support. This observation is used in Proposition \ref{PROP:CHAa} to complete the proof of our main result.

\smallskip

By combining the five major steps presented above with the tightness of the scaled Gibbs measure in Section \ref{SEC:Tight}, we present the proof of Theorem \ref{THM:1} in Section \ref{SEC:MAIN}.

\subsection{Comments on the literature}\label{SUBSEC:motiv}

In this subsection, we present several motivations for our main theorem \ref{THM:1} and related problems.

\subsubsection{Infinite volume limit of Gibbs measures}
The construction of continuum Gibbs measures, first in finite volumes and later in infinite volume, was a major achievement in the constructive quantum field theory program. See, for example, the works \cite{FO76, MW1, GHb, BG3, GS0} on the construction of infinite volume $\Phi^4$ and sine-Gordon measures.

In contrast to the Gibbs measures in quantum field theory, Hamiltonians in focusing Gibbs measures are unbounded from below, as in \eqref{Ham0}. This unboundedness drives the field to exhibit significant local concentration as $L$ increases, ultimately leading to a collapse at $L=\infty$. This collapse can be seen as the result of intense competition between the quartic interaction and the Gaussian free field conditioned on a particular event as in \eqref{Gibbs2}. As explained in Subsection \ref{SUBSEC:result}, Rider, in \cite{Rider}, proved that the infinite volume limit of the focusing $\Phi^4$ measure on $\T_L$ is not only unique but also trivial, in the sense that the limiting measure $\delta_0$ concentrates all its mass on the zero path. We expect that a similar triviality phenomenon occurs in the large torus limit for other focusing models \cite{BO94Zak, BO97a, TzBO, OOT1, Seong, Seo1}. Methods similar to those in \cite{Rider,TW,SSosoe} and the present paper can likely be applied to obtain the fluctuations of other focusing Gibbs measures in the infinite volume limit.


\subsubsection{Gaussian Fluctuations in Euclidean $\Phi^4$ QFT}


In comparison with the focusing Hamiltonian in \eqref{Ham0}, the Hamiltonians in the quantum field theory model exhibit a coercive structure, meaning that they are bounded from below\footnote{For simplicity, in the following discussion, we consider Hamiltonians before applying renormalizations. In other words, we disregard the ultraviolet problem.}, which is referred to as defocusing. In this case, the fluctuation behavior is determined by the second variation of the Hamiltonian in a straightforward way. More precisely, in \cite{ER1, ER2, GST24}, it was shown that defocusing $\Phi^4_1$ and $\Phi^4_2$ measures exhibit Gaussian fluctuations in the low-temperature limit
\begin{align}
\lim_{\eps \to 0}Z^{-1}_\eps\int F\big(\sqrt{\eps}^{-1}(\phi-\pi(\phi) )   \big)\exp\Big\{-\frac 1\eps H(\phi)\Big\} \prod_{x} d\phi(x)=\sum_{w\in \M}  \dr_w \int   F(v)   \mu_w(dv),
\label{GFFQFT}
\end{align}

\noi
where $\eps$ is a positive (temperature) parameter, $\pi$ is the projection onto $\M$, $\{w\}_{w\in \M}$ is the family of minimizers of the Hamiltonian, $\dr_w$ represents the weight of each base measure $\mu_w$, and $\mu_w$ is the Gaussian free field characterized by the covariance operator $\big(\nb^{2} H(w) \big)^{-1}=(-\Dl+\nb^{2}V(w))^{-1}$. Here $V$ represents the interaction part in the $\Phi^4$ Hamiltonian.

Compared to Theorem \ref{THM:1}, the fluctuation behavior in \eqref{GFFQFT}, given by the Gaussian free field $\mu_w$, is different from that of white noise. This indicates that the defocusing $\Phi^4$ measures exhibit a fluctuation behavior of the same nature as the underlying base field, namely, the Gaussian free field. In particular, the fluctuations in \eqref{GFFQFT} are derived by taking the low-temperature limit with the square root scaling $\sqrt{\eps}$, which differs from the scaling by $L$ in Theorem \ref{THM:1}. Moreover, in Theorem \ref{THM:1}, the result holds regardless of whether we are in the high- or low-temperature regime, as it does not rely on an explicit parameter in the Hamiltonian. 


\subsubsection{Higher dimensional focusing $\Phi^4$ measure}
Since the pioneering work of Lebowitz, Rose, and Speer \cite{LRS} on constructing the one-dimensional focusing $\Phi^4$-measure in \eqref{Gibbs2}, Brydges and Slade \cite{BS} continued this line of study in two dimensions. They proved that the partition function of the two-dimensional focusing $\Phi^4$-measure blows up, even after renormalization, indicating that this measure cannot be constructed as a probability measure. Moreover, in \cite{OSeoT}, Oh, the first author, and Tolomeo demonstrated that under the two-dimensional focusing $\Phi^4$-measure, the typical configuration exhibits a concentrated blow-up profile resembling a soliton, providing an alternative proof of the non-construction of this measure. In addition, in \cite{GOTT24} Greco, Oh, Tao, and Tolomeo studied the weakly interacting focusing $\Phi^4_2$-measure, where the coupling constant $\ld_N$ depends on the ultraviolet cutoff $N$ and approaches zero ($\ld_N \to 0$) as the regularization is removed $(N\to \infty)$. In particular, they proved a phase transition by showing that the truncated $\Phi^4_2$-measure either converges or fails to converge, depending on the size of the coupling constant $\ld_N$. Given the non-normalizability of the two-dimensional focusing  $\Phi^4$-measure for fixed coupling constant $\ld_N=\ld$, we expect that a similar approach would yield the non-construction of the three-dimensional focusing $\Phi^4$-measure.


\subsubsection{Invariance of the Gibbs ensemble under Hamiltonian PDEs}

The construction of the Gibbs ensemble for Hamiltonian PDEs, including the focusing cubic Schrödinger equation \eqref{NLS}, was first studied by Lebowitz, Rose, and Speer \cite{LRS}. Subsequently, the invariance of Gibbs measures under the flow of Hamiltonian PDEs was proved in the seminal works of Bourgain \cite{BO94, BO96} and McKean \cite{McKean}. In recent years, significant progress has also been made on invariant Gibbs measures for nonlinear dispersive Hamiltonian PDEs. See, for example, the works \cite{BT1, BT2, DNY2, DNY3, BDNY, BR2, OOT1, OOT2}.  We point out that the articles cited above primarily address compact domains, such as the periodic box $\T^d$. However, our interest in this paper lies in the infinite-volume limit of Gibbs measures for Hamiltonian PDEs, as considered in \cite{BO20, Bring, OTWZ22, FKV}. In this context, since the Gibbs measures for classical equations such as the Schr\"odinger and wave equations, are translation-invariant, random distributions drawn from these measures exhibit no spatial decay on $\R^d$. This lack of decay poses significant challenges in establishing the well-posedness of the Hamiltonian PDEs under the Gibbs ensemble.





\subsubsection{Long time behavior of the dynamics}
The long-term behavior of solutions to the focusing Schr\"odinger equation \eqref{NLS} has been extensively studied. Finite energy solutions to the PDE \eqref{NLS} are expected to asymptotically decompose into a finite number of solitons, plus a small radiation part that disperses over time. This long-term behavior has been primarily studied with smooth, localized initial data decaying at infinity.

Since the Gibbs measure $\rho_L$ in \eqref{Gibbs22} is invariant under the flow of the PDE \eqref{NLS}, high probability events under $\rho_L$ reflect the long-time behavior of the dynamics of \eqref{NLS}. In this context, Theorem \ref{THM:1} has an interesting implication for the long-term behavior of solutions to the PDE \eqref{NLS}. Since the field under the ensemble $\rho_L$ behaves like a single soliton $e^{i\dr}LQ(L(\cdot-x_0))$ plus white noise, at any fixed time $t$, solutions to \eqref{NLS} approximately have this form (as a function of $x$) with high probability. It would be interesting to know if this distributional statement can be supplemented with a corresponding dynamical statement relating the solutions at different times.

This behavior, characterized by a single soliton plus white noise,  is different from that of generic localized initial data in the energy space, for which one expects a decomposition into a sum of multisolitons and a small, highly oscillatory remainder bearing some similarities to a linear wave. As noted above, the Gibbs measure $\rho_L$ is translation-invariant, so a random field drawn from $\rho_L$ lacks spatial decay in addition to not having finite energy.

\section{Notations and preliminary lemmas}\label{SEC:NOT}

\subsection{Notations}
We use $ A \les B $ to denote an estimate of the form $ A \leq CB $
for some $C> 0$.
We also write  $ A \approx B $ to denote $ A \les B $ and $ B \les A $ and use $ A \ll B $ 
when we have $A \leq c B$ for some small $c > 0$.
We may include subscripts to show dependence on external parameters; for example, $A\les_{p} B$ means $A\le C(p) B$, where the constant $C(p)$ depends on a parameter $p$. In addition, we use $a-$ 
and $a+$ to denote $a- \eps$ and $a+ \eps$, respectively for arbitrarily small $\eps > 0$. If this notation appears in an estimate, 
then an implicit constant can depend on $\eps> 0$.

Let $A_1,\dots, A_k$ be measurable sets. In the following, we use the notation 
\begin{align*}
\E \Big[F(\phi), A_1,\dots A_k \Big]=\E\bigg[F(\phi)\prod_{j=1}^k \ind_{A_j} \bigg],
\end{align*}

\noi
where $\E$ denotes the expectation with respect to the probability distribution of $\phi$ under consideration.

Throughout the paper, we study complex-valued functions as elements of real Hilbert and Banach spaces. In particular, the inner product denotes the real inner product
\begin{align}
\langle u,v \rangle_{L^2(\T_L)}:=\Re \int_{\T_{L}} u\bar{v}\,\mathrm{d}x.
\label{INER}
\end{align}

\subsection{Weighted Besov spaces}

In this subsection, we introduce weighted Besov spaces, beginning with the Littlewood-Paley projector, which is specifically adapted for use in these spaces.

Recall the following definition \cite[Definition 1]{MW}. Given $\ta_0 \ge 1$,  the Gevrey class 
 $\GG^{\ta_0} \subset C^\infty(\R^d; \R)$ 
 of order $\ta_0$
consists of functions $f$ satisfying
the following;
given any compact set $K \subset \R^d$, 
there exists $C_K > 0$ such that 
$\sup_{x \in K}|\dd^\al f(x)| \le (C_K)^{|\al|+1} (\al!)^{\ta_0}$
for any multi-index $\al$.
We use   $\GG^{\ta_0}_c$ to denote
the set of compactly supported functions in  $\GG^{\ta_0}$.

Let  $\Z_{\ge 0} := \N \cup\{0\}$. We introduce the 
Littlewood-Paley projector
$\Dl_k$, $k \in \Z_{\ge 0}$,  defined by Gevrey class multipliers.
As seen in \cite{MW},  this choice is suited for weighted spaces with a stretched exponential weight. Let us take 
\begin{align}
\chi_0, \chi_1 \in \GG^{\ta_0}_c(\R^d; [0, 1])
\label{GV}
\end{align}

\noi 
with 
\begin{align*}
\supp \chi_0 \subset \big\{|\xi|\le \tfrac{4}{3}\big\}
\qquad \text{and}\qquad 
\supp \chi_1 \subset \big\{\tfrac 34 \le |\xi|\le \tfrac{8}{3}\big\}
\end{align*}

\noi
such that 
\begin{align*}
\sum_{k = 0}^\infty \chi_k(\xi) \equiv 1
\end{align*}

\noi
on $\R^d$, where $\chi_k(\xi) = \chi_1(2^{1-k} \xi)$ for $k \ge 2$.
We define the Littlewood-Paley projector $\Dl_k$ by
\begin{align}
\Dl_k f = \F^{-1} (\chi_k \ft f)=\varphi_k*f,
\label{LPR}
\end{align}

\noi
where the Fourier transform of a function is 
\begin{align*}
\ft f(\xi) =  \F(f) (\xi)
=  \int_{\R^d} f(x) e^{- 2\pi i \xi \cdot x}\, \mathrm{d}x
\end{align*}

\noi
with the inverse Fourier transform given by 
$\F^{-1}(f) (\xi ) = \ft f(-\xi)$, and $\varphi_k$ is defined by $\varphi_k=\F^{-1}(\chi_k)$ for $k\ge 0$. Specifically, for $k\ge 1$, we have $\varphi_k(x)=2^{k}\varphi(2^k x)$, where $\varphi(x)=2^{-1}\varphi_1(2^{-1} x)$.

Before defining the weighted Besov spaces, we begin by recalling the definition of the stretched exponential weight introduced in \cite{MW}.
Let $\ta_0$ be as in \eqref{GV}, used in the definition of the Littlewood-Paley projector $\Dl_k$ in \eqref{LPR}. For $0 < \dl < \frac 1{\ta_0}$
and $\mu > 0$, 
we define the weight $w_\mu(x)$
by setting
\begin{align}
w_\mu(x) = e^{-\mu \jb{x}^\dl}.
\label{WEITH}
\end{align}

\noi
In the remaining part of the paper, we fix $0 < \dl < \frac{1}{\ta_0}$  as in \eqref{WEITH} and we frequently omit the dependence on $\dl$ in various estimates.

Let $1 \le r <  \infty$ and $\mu > 0$.
We then define the weighted Lebesgue space $ L^p_\mu(\R^d)$ by the norm
\begin{align*}
\|f  \|_{L^r_\mu}   =\bigg(\int_{\R^d} |f(x)|^r w_\mu (x) dx\bigg)^\frac 1r.
\end{align*}

Given $s\in \R$, $1\le r ,q  \le \infty$, and $\mu>0$,  we subsequently define the weighted Besov spaces $B^{r,\mu}_{r,q}(\R^d)$ as the completion of $C^{\infty}_c(\R^d)$ under the norm
\begin{align*}
\| f \|_{B^{s, \mu }_{r,q} }=  \Big\|  2^{s k}  \| \Dl_k f\|_{L^r_\mu (\R^d)}    \Big\|_{\ell^q_k(\Z_{\ge 0} )} .
\end{align*}

\noi 
When $\mu=0$, the weighted Besov spaces $B_{r,q}^{s,\mu}(\R^d)$ correspond to the classical Besov space $B^{s}_{r,q}(\R^d)$. In particular, for $p=q=2$, we define 
\begin{align}
H^{s}_\mu(\R^d)=B^{s,\mu}_{2,2}(\R^d).
\label{wesob}
\end{align}

We now present properties of the weighted Besov spaces introduced in \cite{MW}.

\begin{lemma}\label{LEM:compact}
Let $s,s'\in \R$ with $s>s'$, $r,q \in [1,\infty]$ with $r\neq \infty$, and $\mu'>\mu>0$. Then, the embedding $B^{s,\mu}_{r,q}  \hookrightarrow B^{s', \mu'}_{r,1}$ is compact.
\end{lemma}

\begin{lemma}\label{LEM:Ber}
For every $\mu_0$, and $q,r \in [1,\infty]$ with $r\ge q$, there exists $C<\infty$ such that for every $\mu \le \mu_0$ and $\ld \ge 1$,
\begin{align*}
\supp \ft f \subset \{ |\xi|\le \ld\} \qquad \LRA \qquad 
\| f  \|_{L^r_\mu} \le C \ld^{d(1/q-1/r)  } \| f\|_{L^q_{\mu q/r} }.
\end{align*}

\end{lemma}

\subsection{Concentration of measure around a family of minimizers}

In this subsection, we present the family of minimizers for the Hamiltonian \eqref{HAmR}, known as the soliton manifold, and  the concentration of the Gibbs measure $\rho_L$ \eqref{Gibbs22} around the manifold.

Let $D>0$. We first recall the minimization problem
\begin{align}
I(D)=\inf_{ \|\phi \|_{L^2(\R)}^2 \le D } H(\phi),
\label{MINPR}
\end{align}

\noi
where the Hamiltonian $H$ is given by
\begin{align}
H(\phi)=\frac{1}{2}\int_{\R} |\dx \phi|^2 dx -\frac 14\int_{\R} |\phi|^4 dx.
\label{Ham1}
\end{align}

\noi
The following proposition shows that the minimizers of the variational problem form a two-dimensional manifold, known as the soliton manifold.

\begin{proposition}\label{PROP:Min}
For every $D>0$, there exists $Q=Q_D>0$ in $H^1(\R)$ such that $\|Q_D \|_{L^2(\R^2)}^2=D$ and 
\begin{align*}
I(D)=\frac{1}{2} \int_{\R} |\dx Q_D|^2-\frac{1}{4} \int_{\R} |Q_D|^4 dx.
\end{align*}

\noi
Moreover, if $\phi \in H^1(\R)$ satisfies $\|\phi \|_{L^2(\R)}^2\le D$ and
\begin{align*}
I(D)=\frac{1}{2} \int_{\R} |\dx \phi|^2-\frac{1}{4} \int_{\R} |\phi|^4 dx,
\end{align*}

\noi
then there exists $x_0\in \R$ and $\dr \in [0,2\pi]$ such that
\begin{align*}
\phi(x)=e^{i\dr}Q_D(x-x_0).
\end{align*}

\end{proposition}

For the proof of Proposition \ref{PROP:Min}, see Theorem 1.1 in \cite{Frank} or Lemma 2.4 in \cite{TW}.

\begin{remark}\rm 
The minimization problem $I(D)$ in \eqref{MINPR}  yields a negative value, that is, $I(D)<0$. 

For any fixed $D>0$, take any function $\psi \in H^1(\R)$ with $\| \psi \|_{L^2(\R)} \le D$. Define $\psi_\s(x):=\s^{-\frac 12} \psi (\s^{-1}x)$. Then, we have 
\begin{align*}
I(D)& \le \frac 12 \int_{\R} |\dx \psi_\s |^2 dx-\frac 14 \int_{\R} |\psi_\s |^4  dx \\
&\le \frac{\s^{-2}}{2} \int_{\R} |\dx \psi|^2-\frac {\s^{-1}}4\int_{\R} |\psi|^4 dx. 
\end{align*}

\noi
By choosing $\s \gg 1$ sufficiently large, we obtain
\begin{align}
I(D)<0
\label{Ineg}
\end{align}

\noi
for every $D>0$. Hence, we obtain the result.

\end{remark}

Note that the minimizer $Q=Q_D$  for the minimization problem 
\begin{align*}
I(D)=\inf_{ \|\phi \|_{L^2(\R)}^2 \le D } H(\phi)
\end{align*}

\noi
is characterized by the Euler-Lagrange equation
\begin{align}
\partial_{x}^2 Q+Q^3-\Ld Q=0,
\label{Lagmul}
\end{align}

\noi 
where $\Lambda>0$ is a Lagrange multiplier. See Remark \ref{REM:LAG} for the positive sign.

\begin{remark}\rm\label{REM:LAG} 
Suppose that $Q=Q_D$ is a minimizer for the variational problem \eqref{MINPR}.
Using the Lagrange multiplier method, we find that there exists $\ld \in \R$ such that
\begin{align}
-\dx^2 Q- |Q|^2Q=\ld Q.
\label{LMeqn0}
\end{align}

\noi
Multiplying \eqref{LMeqn0} by $ \cj \dx Q$, we have 
\begin{align*}
\dx\Big(-\frac 12 |\dx Q|^2 -\frac 14 |Q|^4-\frac \ld2 |Q|^2 \Big)=0.
\end{align*}

\noi
From Proposition \ref{PROP:Min}, the minimizer $Q$ is in $H^1(\R)$. This implies that
\begin{align*}
-\frac 12 |\dx Q|^2 -\frac 14 |Q|^4=\frac \ld2 |Q|^2. 
\end{align*}

\noi
Therefore, we conclude that $\ld<0$. Therefore, $\Ld=-\ld$ in \eqref{Lagmul} is a positive number.

\end{remark}

Throughout the paper, the Lagrange multiplier $\Ld$ \eqref{Lagmul} plays an important role by introducing a mass term in the Gaussian measures, which leads to correlation decay. See Subsection \ref{SUBSEC:strucpf}, especially \textbf{Step~1}. Note that $\Ld=\Ld(D)$ depends only on the fixed density parameter $D>0$ in \eqref{Gibbs2}, as explained in the remark below.

\begin{remark} \rm 
Note that the minimizer $Q_1$  for the minimization problem when $D=1$
\begin{align*}
I(1)=\inf_{ \|\phi \|_{L^2(\R)}^2 \le 1 } H(\phi)
\end{align*}

\noi
is characterized by the Euler-Lagrange equation
\begin{align*}
\partial_{x}^2 Q_1+Q_1^3-\Ld_1 Q_1=0,
\end{align*}

\noi 
where $\Lambda_1>0$ is a Lagrange multiplier. By taking the rescaling $Q_D(x)=DQ_1(Dx)$ for $D>0$, we have 
\begin{align}
\partial_{x}^2 Q_D+Q_D^3-\Ld_1 D^2 Q_D=0.
\label{Soleqn}
\end{align}

\noi
Here, we define 
\begin{align}
\Ld:=\Ld_1 D^2.
\label{Lagmul}
\end{align}

\noi 
Then, \eqref{Soleqn} is the Euler-Lagrange equation and $\Ld$ is the Lagrange multiplier for the problem
\begin{align*}
I(D)=\inf_{ \|\phi \|_{L^2(\R)}^2 \le D } H(\phi).
\end{align*}

\noi
This can be easily verified from $H(Q_D)=D^3 H(Q_1)=D^3 I(1)=I(D)$ and $\| Q_D\|_{L^2(\R)}^2=D$.

\end{remark}

The following proposition shows the limit of the free energy $\log Z_L$ as $L\to \infty$ and, correspondingly, the concentration of the measure  $\rho_L$  around the family of minimizers.

\begin{proposition}\label{PROP:con}
Let $D>0$. Then, we have the free energy limit
\begin{align*}
\lim_{L\to\infty } \frac{\log Z_L}{L^3}=-\inf_{\substack{ \|\phi \|_{L^2(\R)}^2 \le D  } } H(\phi),
\end{align*}

\noi
where  $Z_L$ is the partition function in \eqref{Gibbs22}, the Hamiltonian $H$ is as in \eqref{Ham1}, and the infimum is taken over all $\phi \in H^1(\R)$ satisfying  $\|\phi \|_{L^2(\R)}^2\le D$. Moreover, for any $\dl>0$, there exists $c=c(\dl)>0$ such that 
\begin{align*}
 \rho_L\Bigg(\bigg\{  \inf_{x_0 \in \T_{L}, \dr \in [0,2\pi] } \| \phi-e^{i\dr}\eta(x-x_0)LQ(L(\cdot-x_0)) \|_{L^2(\T_L)} \ge  \dl L^{\frac 12}    \bigg\}\Bigg)\les e^{-cL^3},
\end{align*}

\noi
where $\rho_L$ is the Gibbs measure in \eqref{Gibbs2}, $\eta$ is the 
cutoff function in \eqref{cutoffeta}, and $Q=Q_D$ is the minimizer arising from the variational problem with the parameter $D$ in Proposition \ref{PROP:Min}.
\end{proposition}

For the proof of Proposition \ref{PROP:con}, see \cite{Rider, TW}.

\begin{remark}\rm
Using the cutoff function $\eta$, where $\supp \eta \subset [-\frac 14, \frac 14]$, we can interpret $\eta(\cdot-x_0)LQ(L(\cdot-x_0))$ as a function on the torus $\T_L$. Note that in \cite{TW}, Proposition \ref{PROP:Min} was proved without the cutoff function $\eta$. Since $Q(\cdot-x_0)$ is localized at $x_0$ and has an exponentially decaying tail, as $L\to \infty$, $LQ(L(\cdot-x_0))$ increasingly approximates $\eta(\cdot-x_0)LQ(L(\cdot-x_0))$. Hence, we can adapt the argument to address the case with the cutoff function $\eta$.

\end{remark}

\subsection{Orthogonal coordinate system }\label{SEC:Ortho}

In this subsection, we introduce the orthogonal coordinate system in a neighborhood of the soliton manifold $\{e^{i\dr}Q(\cdot-x_0) \}_{\dr, x_0} $ and a corresponding change of variable formula using the coordinate system.

Note that although the ground state $Q$ in Proposition \eqref{PROP:Min} is initially defined on $\R$, we now interpret it as a function on $[-L^2, L^2)$, identified with the torus $\T_{L^2} \cong [-L^2, L^2)$.
According to Proposition \ref{PROP:Min}, given $x_0\in \T_{L^2}$ and $\dr \in [0,2\pi]$, we consider the translated ground state  on $\T_{L^2}$
\begin{align*}
Q_{x_0,\dr}(x):=e^{i\dr}Q(x-x_0),
\end{align*}

\noi
where $x-x_0$ is interpreted mod $2L^2$, taking values in $[-L^2, L^2)$. 
Furthermore, based on Proposition \ref{PROP:con}, we also consider the rescaled and translated ground state on $\T_L\cong [-L, L)$
\begin{align*}
Q_{x_0,\dr,L}(x):=e^{i\dr }L Q(L(x-x_0)),
\end{align*}

\noi
where $x_0 \in \T_L$.  When restricted to the torus, the rescaled and shifted ground states lie in $H^1$, but fail to belong to $H^2$ (or any $H^k$ for any higher $k$). 
This arises because $Q_{x_0,\dr}'$ and  $Q'_{x_0,\dr,L}$, defined on $[-L^2, L^2)$ and $[-L, L)$, respectively,  do not satisfy the periodic boundary condition when truncated at $x = \pm L^2$ and $x=\pm L$. 
To address this issue, we introduce an even cutoff function $\eta \in C^\infty(\R; [0, 1])$ such that
\begin{equation} \label{cutoffeta}
\supp(\eta) \subset \big[-\tfrac14, \tfrac 14\big]
 \qquad \text{and}\qquad  \eta \equiv 1 \text{ on } \big[-\tfrac18, \tfrac 18\big],
\end{equation}

\noi
and consider 
\begin{align}
Q^{\eta_L}_{x_0,\dr}(x)=e^{i\dr}Q^{\eta_L}_{x_0}(x)=e^{i\dr}\eta(L^{-1}(x-x_0))Q(x-x_0)
\label{QetaL}
\end{align}

\noi 
for $x_0 \in \T_{L^2}$ and $\dr \in [0,2\pi]$, which is smooth at $x = \pm L^2$. In the following, we interpret $Q^{\eta_L}_{x_0,\dr}$ and $Q^{\eta_L}_{x_0}$
as functions on the torus $\T_{L^2}$. Additionally, by defining
\begin{align}
Q^{\eta}_{x_0,\dr,L}(x)=e^{i\dr}\eta(x-x_0)LQ(L(x-x_0)),
\label{QQ1}
\end{align}

\noi
where $x_0 \in \T_{L}$ and $\dr \in [0,2 \pi]$, we also interpret $ Q^{\eta}_{x_0,\dr,L}$ as a function on the torus $\T_L$.

\begin{remark}\rm\label{REM:SOL}
Note that on $\T_{L^2}$, $e^{i\dr}Q^{\eta_L}_{x_0}$ is not the minimizer of the Hamiltonian \eqref{Ham1} defined on $\T_{L^2}$.
However,  from the definition \eqref{QetaL}, as $L \to \infty$, 
$e^{i\dr}Q^{\eta_L}_{x_0}$ on $\T_{L^2}=[-L^2,L^2)$ 
provides a more and more precise approximation of the ground state $e^{i\dr}Q(x-x_0)$ on $\R$. For this reason, $e^{i\dr}Q^{\eta_L}_{x_0}$ is called the approximate soliton, mainly because of the presence of the cutoff function $\eta_L$.
\end{remark}

Given $L\ge 1$, we define the approximate soliton manifold
\begin{align}
\bar \M_L:=\big\{ e^{i\dr}\eta(L^{-1}(\cdot-x_0))Q(\cdot-x_0): \dr \in [0,2\pi], \; x_0 \in \T_{L^2}  \big\} \subset H^1(\T_{L^2}).
\label{solm0}
\end{align}

\noi
Then, $\bar \M_L$ is a smooth manifold of dimension $2$, embedded into $H^1(\T_{L^2})$. In particular, we remark that $\bar  \M_L$ is the compact manifold  for any fixed $L \ge 1$.

In the next section, we show that for a rescaled Gibbs measure under the transformation $\phi \mapsto \phi_L = L\phi(L \cdot)$, most of the probability mass is concentrated around the soliton manifold $\bar \M_L$.
Hence, the question reduces to studying the Gibbs measure
on an  $\eps$-neighborhood $U_\eps^L$ of the soliton manifold $\bar \M_L$
\begin{align*}
U_\eps^L=\Big\{\phi \in L^2(\T_{L^2}): \|\phi- e^{i\dr}Q^{\eta_L}_{x_0} \|_{L^2(\T_{L^2}) }<\eps \quad \text{for some $x_0 \in \T_{L^2}$ and $\dr \in \R$}  \Big\}, 
\end{align*}

\noi 
where $e^{i\dr}Q^{\eta_L}_{x_0}$ is the approximate soliton defined in \eqref{QetaL}. The key point is that for sufficiently small $\eps > 0$,
within the $\eps$-neighborhood $U^L_\eps$ of the soliton manifold $\bar \M_L$, we can introduce an orthogonal coordinate system in $U^L_\eps$ with the tangential directions
$x_0 \in \T_{L^2}$, $\dr \in [0,2\pi]$, and the normal direction $h \in V^L_{x_0,\dr}$, as defined in \eqref{normalsp}.

\begin{lemma}\label{LEM:coord}
Let $L\ge 1$. Given small $\eps_1>0$, there exists $\eps=\eps(\eps_1)>0$ such that
\begin{align*}
U_\eps^L\subset \Big\{ \phi \in L^2(\T_{L^2}): \; \phi=e^{i\dr}Q^{\eta_L}_{x_0}+h \quad &\text{for some $x_0\in \T_{L^2}$, $\dr\in \R$,}\\
&\text{  and $h\in V^{L^2}_{x_0,\dr}$ with $\|h \|_{L^2(\T_{L^2})}<\eps_1$} \Big\},
\end{align*}

\noi
where, at each point $e^{i\dr}Q^{\eta_L}_{x_0}$ on the manifold $\bar \M_L$, $V^{L^2}_{x_0,\dr}$ represents the normal space to the soliton manifold
\begin{align}
V_{x_0,\dr}^{L^2}=\big\{\phi \in L^2(\T_{L^2}): \;  &\jb{\phi, (1-\dx^2) \dd_{x_0}( e^{i\dr}Q^{\eta_L}_{x_0} )  }=0 \notag \\  
&\jb{\phi, (1-\dx^2) \dd_{\dr}( e^{i\dr}Q^{\eta_L}_{x_0})  }=0 \big\}.
\label{normalsp}
\end{align}

\noi
Here, the inner product on $L^2(\T_{L^2})$, defined in \eqref{INER}, means the real inner product.

\end{lemma}

The proof of Lemma \ref{LEM:coord} very closely follows that of \cite[Proposition 6.4]{OST1}, so we omit it in the in the interest of brevity.

\begin{remark} \rm

By the definition \eqref{normalsp}, the space $V_{x_0,\dr}^{L^2}$ represents a real subspace of codimension $2$ in $L^2(\T_{L^2})$, orthogonal (with the weight $1-\dd_x^2$) to the tangent vectors  $\dd_{x_0}( e^{i\dr}Q^{\eta_L}_{x_0} ) $ and  $\dd_{\dr}( e^{i\dr}Q^{\eta_L}_{x_0})$ of the soliton manifold $\bar \M^L$.  We note that, given the limited regularity of $\phi \in L^2(\T_{L^2})$, where the underlying Gaussian measure is supported, orthogonality in \eqref{normalsp} is measured using the $L^2$-inner product with the weight $1-\dx^2$, ensuring that the inner products are well-defined for $\phi \in L^2(\T_{L^2})$. 
Furthermore, it can be easily verified that two tangent vectors $\dd_{x_0}( e^{i\dr}Q^{\eta_L}_{x_0} ) $ and  $\dd_{\dr}( e^{i\dr}Q^{\eta_L}_{x_0})$ are orthogonal in $L^2(\T_{L^2})$ and $H^1(\T_{L^2})$.


\end{remark}

Thanks to Lemma \ref{LEM:coord}, we can establish a new coordinate system
in terms of the translation $x_0 \in \T_{L^2}$, the rotation $\dr \in [0,2\pi]$, and the component $h \in L^2(\T_{L^2})$ orthogonal to the soliton manifold $\bar \M_L$. This setup allows us to introduce a change of variables for the Gaussian functional integral
\begin{align*}
\int_{U_\eps^L } F(\phi) \mu_{L^2}(d\phi)
\end{align*}

\noi 
on the $\eps$-neighborhood $U^L_\eps$ of the soliton manifold $\bar \M_L$, where $\mu_{L^2}$ is the Gaussian measure on $\T_{L^2}$, having the mass term $m>0$,
\begin{align}
\mu_{L^2}(d\phi)&=Z_{L }^{-1} e^{-\frac {1}2  \int_{\T_{L^2} } |\dx \phi|^2 dx -\frac {m}{2} \int_{\T_{L^2}} |\phi|^2 dx  }  \prod_{x \in \T_{L^2}} d\phi(x) \notag \\
&=Z_{L}^{-1} e^{-\frac 12 \jb{A^{-1} \phi, \phi }  } \prod_{x \in \T_{L^2}} d\phi(x) 
\label{GFF},
\end{align}

\noi 
where $A=(-\dx^2+m)^{-1}$ is the covariance operator.

\begin{remark}\rm \label{REM:mass0}
When the underlying Gaussian measure has mass $m>0$ as in \eqref{GFF}, we modify the definition of the normal space to the soliton manifold \eqref{normalsp} by using the weight $m-\dx^2$, instead of $1-\dx^2$. 
Note that the Gaussian measure $\mu_L$ with mass $m$ has the Cameron–Martin space $H_m^1$ associated with the operator  $-\dx^2+m$. If there is no confusion, we may write in the following that $\mu_L$ has the Cameron–Martin space $H^1(\T_L)$, corresponding to the case $m=1$, since we will change the mass term in several places. 
\end{remark}

We are now ready to introduce the change of variable formula.

\begin{lemma}\label{LEM:chan}
Let $L\ge 1$ and $F$ be a bounded, continuous function on $H^s(\T_{L^2})$ with $s<\frac 12$. Then, we have
\begin{align}
\int_{U_\eps^L}F(\phi) \mu_{L^2}(d\phi)&=\iiint_{\wt U_\eps^L} F\big(e^{i\dr}(Q^{\eta_L}_{x_0}+h) \big)e^{-\frac 12 \jb{A^{-1} Q^{\eta_L}_{x_0}, Q^{\eta_L}_{x_0} }_{L^2(\T_{L^2})} -\jb{ A^{-1 }Q^{\eta_L}_{x_0}, \Re h }_{L^2(\T_{L^2})} } \notag  \\
&\hphantom{XXXX} \cdot \det\big( \Id-W_{x_0,\dr,h}\big)  \mu^\perp_{x_0, L^2,m}(dh)d\s_L(x_0,\dr),
\label{cha0}
\end{align}

\noi
where $A=(-\dx^2+m)^{-1}$ and 
\begin{align*}
\wt U^L_\eps=\Big\{ (x_0,\dr,h)\in \T_{L^2}\times [0,2\pi]\times V^{L^2}_{x_0,0}: \|h \|_{L^2(\T_{L^2})}< \eps  \Big\}
\end{align*}


\noi
and $W_{x_0, \dr, h}$ is the Weingarten map on the tangent space of $\bar \M^L$ at $e^{i\dr}Q^{\eta_L}_{x_0}$. See \eqref{eqn: WM} for an explicit expression in our case. 
Here, the measure $\mu_{x_0,L^2,m }^{\perp}$ is a Gaussian measure defined on the normal space $V^{L^2}_{x_0,0}$, with $V_{x_0,0}^{L^2} \cap   H^1(\T_{L^2} ) \subset  H^1(\T_{L^2} )$ as its Cameron-Martin space, and admits the formal expression
\begin{align}
\mu_{x_0,L^2,m}^{\perp}(dh)=Z_{L^2}^{-1} e^{-\frac {1}2  \int_{\T_{L^2} } |\dx h|^2 dx -\frac m2 \int_{\T_{L^2}} |h|^2 dx  }  \prod_{x \in \T_{L^2}} dh(x).
\label{GFFNOR}
\end{align} 

\noi
In addition,  $d\s_L $ is a surface measure on the soliton manifold $\bar \M^L$ in \eqref{solm0}, parametrized by $x_0 \in \T_{L^2}$ and $\dr \in [0, 2\pi]$ as in \eqref{surff}.
\end{lemma}

\begin{proof}
The formula \eqref{cha0} follows from  \cite[Lemma 3]{ER2} and Lemma \ref{LEM:coord}.
\end{proof}

Note that the geometry of the soliton manifold $\bar \M_L$ is reflected in the surface measure $d\s_L$ and the Weingarten map $W_{x_0,\dr,h}$. 

The orthonormal vectors $t_k = t_k(x_0, \dr)$, for $k = 1, 2$, are obtained by applying the Gram-Schmidt orthonormalization procedure in $H^1(\T_{L^2})$ to the tangent vectors $\big\{ \partial_{x_0}(e^{i\ta}Q_{ x_0}^{\eta_L}), \partial_{\dr}(e^{i\ta}Q_{ x_0, \dr}^{\eta_L})\big\}$. Then, the surface measure $d\s_L(x_0,\dr)$ is given by 
\begin{align}
d \s_L (x_0,\dr)=|\g_L(x_0,\dr)| d x_0   d \dr_0,
\label{surff} 
\end{align}

\noi
where
\begin{align*}
\g_L(x_0,\dr)=\det 
\begin{pmatrix}
\langle \partial_{x_0}(e^{i\ta}Q_{ x_0}^{\rho_L}),t_1\rangle_{H^1(\T_{L^2})} 
& \langle \partial_{\ta}e^{i\ta} Q_{ x_0}^{\rho_L},t_1\rangle_{H^1(\T_{L^2})}\\
\langle \partial_{x_0}(e^{i\ta}Q_{ x_0}^{\rho_L}),t_2\rangle_{H^1(\T_{L^2})} 
& \langle \partial_{\ta}e^{i\ta} Q_{ x_0}^{\rho_L},t_2\rangle_{H^1(\T_{L^2})}
\end{pmatrix}.
\end{align*}

\noi 
In particular, the tangent vectors $t_k$, $k=1,2$ are chosen so that 
$t_k( x_0, \ta)
= \tau_{x_0}t_k( 0, \ta)$, where $\tau_{x_0}$ is the translation defined by $\tau_{x_0}\phi(x)=\phi(x-x_0)$. Combined with the invariance under multiplication by a unitary complex number, we obtain
\begin{align}
|\g_L( x_0,\ta)|=|\g_L( 0, 0)|
\label{gasurf}
\end{align}

\noi
for any $x_0 \in \T_{L^2}$ and $\dr\in [0,2\pi ]$.

The Weingarten map $W_{x_0,\dr,h}$ encodes the curvature of the surface $\bar \M_L$ by capturing how the normal vector changes direction as we move along different tangent directions on the surface. 
More precisely, the Weingarten map $W_{x_0,\dr,h}=-dN_{x_0,\dr}(h)$ at a point $e^{i\dr}Q^{\eta_L}_{x_0} \in \bar \M_L$, defined via the differential of the Gauss map $N_{x_0,\dr}$ at $e^{i\dr}Q^{\eta_L}_{x_0} \in \bar \M_L$, is the linear map 
\begin{align*}
W_{x_0,\dr,h}: T_{x_0,\dr} \bar \M_L \to T_{x_0,\dr} \bar \M_L, 
\end{align*}

\noi 
where $T_{x_0,\dr} \bar \M_L:=\text{span}\big\{ \partial_{x_0}(e^{i\ta}Q_{ x_0}^{\eta_L}), \partial_{\dr}(e^{i\ta}Q_{ x_0, \dr}^{\eta_L})\big\}$ is the tangent space of $\bar M_L$ at $e^{i\dr}Q_{x_0}^{\eta_L}$. Specifically, the Weingarten map $W_{x_0,\theta,h}$ in the basis $\{\partial_{x_0} (e^{i\theta}Q_{x_0}^{\eta_L}), \partial_\theta (e^{i\theta} Q_{x_0}^{\eta_L})\}$ is given by
\begin{equation}\label{eqn: WM}
\left(\begin{array}{cc}
\langle -\partial_\theta N_{x_0,\theta}(h), \partial_{x_0} e^{i\theta} Q_{x_0,\theta}^{\rho_L}\rangle_{H^1(\T_{L^2})}& \langle- \partial_\theta N_{x_0,\theta}(h), \partial_{\theta} e^{i\theta} Q_{x_0,\theta}^{\rho_L}\rangle_{H^1(\T_{L^2})}\\
\langle- \partial_{x_0} N_{x_0,\theta}(h), \partial_{x_0} e^{i\theta} Q_{x_0,\theta}^{\rho_L}\rangle_{H^1(\T_{L^2})}& \langle- \partial_{x_0} N_{x_0,\theta}(h), \partial_{\theta} e^{i\theta} Q_{x_0,\theta}^{\rho_L}\rangle_{H^1(\T_{L^2})}
\end{array}\right),
\end{equation}

\noi 
 where 
\begin{equation*}
L^2(\T_{L^2}) \ni h\mapsto N_{x_0,\theta}(h)
\end{equation*}

\noi
is a parametrization of the normal space to $\bar \M_L$ at a point $e^{i\dr}Q^{\eta_L}_{x_0}$.  In particular, the $2\times 2$ determinant 
\begin{align}
\det\big( \Id-W_{x_0,\dr,h}\big)=1+O(\| h\|_{L^\infty}^2)
\label{Det0}
\end{align}

\noi 
is a quadratic function of $h$.


Before concluding this subsection, we define a projection operator onto a compact manifold.  Let $M$ be a compact manifold in a Hilbert space $\mathcal{H}$.
If $\dl>0$ is sufficiently small, we can assign to any $\phi \in \mathcal{H}$ with $\text{dist}(\phi,M) \le \dl$ a unique closest point $\pi(\phi)$ in $\M$. This follows from the $\eps$-neighborhood theorem \cite[p.69]{GP}. If  $\text{dist}(\phi,M) \ge \dl$, then we set $\pi(\phi)=0$.  Note that 
\begin{align*}
\bar \M_L:=\big\{ e^{i\dr}Q^{\eta_L}_{x_0}: \dr \in [0,2\pi], \; x_0 \in \T_{L^2}  \big\},
\end{align*}

\noi 
where $Q^{\eta_L}_{x_0}(\cdot)=\eta(L^{-1}(\cdot-x_0))Q(\cdot-x_0)$ as in \eqref{QetaL}, is a two dimensional compact manifold.  So, if the field $\phi$ on $\T_{L^2}$ satisfies 
\begin{align}
\text{dist}(\phi, \bar \M_L)=\inf_{x_0\in \T_{L^2},\dr\in [0,2\pi]}\|\phi(\cdot)-e^{i\dr}\eta(L^{-1}(\cdot-x_0))Q(\cdot-x_0)\|_{L^2(\T_{L^2})}<\dl
\label{rescpro0}
\end{align}

\noi
for sufficiently small $\dl>0$,  we can assign a unique pair $(x_0, \dr)\in \T_{L^2} \times [0,2\pi]$ such that 
\begin{align*}
\phi=e^{i\dr}\eta(L^{-1}(\cdot-x_0))Q(\cdot-x_0)+h,
\end{align*}

\noi 
where the coordinate $h$ satisfies $\|h \|_{L^2(\T_{L^2} )}<\dl$. Therefore, for the rescaled field  $\phi_L=L\phi(L\cdot)$ defined on $\T_L$, we can uniquely assign $x_0 \in \T_{L}$ and $\dr \in [0,2\pi]$ such that 
\begin{align}
\phi_L=L\phi(L\cdot)=e^{i\dr} \eta(\cdot-x_0)LQ(L(\cdot-x_0))+Lh(L\cdot),
\label{rescpro}
\end{align}

\noi
satisfying $\|Lh(L\cdot)\|_{L^2(\T_L)}<\dl L^\frac 12$, that is,
\begin{align*}
\text{dist}(\phi_L,\M_L)=\inf_{x_0\in \T_L,\dr\in [0,2\pi] }\|L\phi(L\cdot)-e^{i\dr }\eta(\cdot-x_0)LQ(L(\cdot-x_0))\|_{L^2(\T_L)}<\dl L^{\frac 12}.
\end{align*}

\noi
Here $\M_L$ denotes the rescaled soliton manifold
\begin{align*}
\M_L:=\big\{ e^{i\dr} \eta(\cdot-x_0)LQ(L(\cdot-x_0)): x_0\in \T_L, \dr \in [0,2\pi] \big\}.
\end{align*}

\noi
Based on \eqref{rescpro}, we are now ready to define the projection operator $\pi_L$, which appears in Theorem~\ref{THM:1}, onto the soliton manifold $\M_L$. If a field $\psi$ on $\T_L$ is given by $\psi=L\phi(L\cdot )$, where $\phi$ on $\T_{L^2}$
satisfies \eqref{rescpro0}, then there exist unique $x_0 \in \T_L$ and $\dr \in [0,2\pi]$ such that 
\begin{equation}\label{project}
\pi_L(\psi):=
e^{i\dr} \eta(\cdot-x_0)LQ(L(\cdot-x_0)), 
\end{equation}

\noi
where $ \text{dist}(\psi,\M_L) \le \dl L^\frac 12 $. 
Otherwise, we define $\pi_L(\psi)=0$. For the use of the projection operator $\pi_L$ in context, see \eqref{project1}.

\begin{remark}\rm 
In the massless case $m=0$,  the change of variable formula in Lemma \ref{LEM:chan} still holds, provided that the underlying Gaussian measures
$\mu_L$ in \eqref{GFF} and $\mu^\perp_{x_0,L}$ in \eqref{GFFNOR}  are defined on the space of mean zero fields. In this setting, the definition of the normal space and the notion of orthogonality are defined in terms of the operator $-\dx^2$.
\end{remark}

\section{Reduction to a neighborhood of the soliton manifold}\label{SEC:INSIDE}

In this section, we study a representation of the scaled Gibbs measure $(T_L)_{\#}\rho_L$ on a neighborhood of the soliton manifold, where $T_L(\phi)=L(\phi-\pi_L(\phi))$, as defined in \eqref{THM:1}.

\begin{proposition}\label{PROP:REDUC}
Let $F$ be a bounded, continuous function. Then, we have 
\begin{align}
&\int F\big(L(\phi-\pi_L(\phi)) \big) \rho_L(d\phi) \notag \\
&=\cj Z_{L}^{-1}\iiint\limits_{\mathcal{S}_{Q,x_0}\cap \mathcal{K}^L_\dl } F(L^\frac 12h(L\cdot))e^{L^3\mathcal{E}(Q^{\eta_L}_{x_0}, L^{-\frac 32}h)}
e^{\Ld_L L^{\frac{3}{2}}\langle Q^{\eta_L}_{x_0}, \Re h\rangle} e^{\frac 12 \jb{Q^{\eta_L }_{x_0}, 3|\Re h|^2+|\Im h|^2 } } \notag \\
&\hphantom{XXXXXXXXXXXXX}\cdot \mathrm{Det}_L(h) \, \mu^\perp_{x_0,L^2}(\mathrm{d}h) \,\mathrm{d} \dr \, \mathrm{d}x_0+O(e^{-cL^{0+}})
\label{REDUCT}
\end{align}

\noi
for some constant $c>0$. Here, the partition function $\cj Z_{L}$ represents the integral $\iiint$ over $S_{Q,x_0}\cap \mathcal{K}^L_\dl$ without $F$. The measure $\mu^\perp_{x_0, L^2}$ is the free field  as defined in \eqref{GFFNOR}, corresponding to the mass term $m=\frac 1{L^2}$ \textup{(}see \eqref{massl2}\textup{)}, 
and $\Ld_L=\Ld-\frac 1{L^2}$, where $\Ld$ is the Lagrange multiplier from \eqref{Lagmul}. The terms  $L^3\mathcal{E}(Q^{\eta_L}_{x_0}, L^{-\frac 32}h)$ and $\mathrm{Det}_L(h)$ are given by
\begin{align}
L^3\mathcal{E}(Q^{\eta_L}_{x_0}, L^{-\frac 32}h)&=\frac 1{L^\frac 32} \int_{\T_{L^2} }  Q^{\eta_L}_{x_0}  |h|^2 \Re h \, \mathrm{d}x  +\frac {1}{L^3} \int_{\T_{L^2}} |h|^4 \, \mathrm{d}x  \label{Eerror}\\
\mathrm{Det}_L(h)&=\det\big( \Id-W_{x_0,\dr,L^{-\frac 32}h}\big).  
\label{Det}
\end{align}

Here, the integral is taken over 
\begin{align}
\mathcal{S}_{Q,x_0}&=\big\{h\in L^2(\mathbb{T}_{L^2}):-\frac{1}{2}\|h\|_{L^2(\T_{L^2})}^2-L^{0+} < L^{\frac{3}{2}}\langle Q^{\eta_L}_{x_0},\Re h\rangle < -\frac{1}{2}\|h\|^2_{L^2(\T_{L^2})}\big\} \notag \\
\mathcal{K}^L_\dl&=\big\{ (x_0,\dr,h)\in \T_{L^2}\times [0,2\pi]\times V^{L^2}_{x_0,0}: \|h\|_{L^2(\T_{L^2}) }<\dl L^{\frac 32}  \big\}.
\label{SQK}
\end{align}

\end{proposition}

\begin{proof}

We first separate the main contribution and the error term. Note that 
\begin{align}
&\int F\big(L(\phi-\pi_L(\phi)) \big) \rho_L(d\phi) \notag \\
&=\int_{\{  \text{dist}(\phi,\M_L)    \ge \dl L^{\frac 12} \} } F\big(L(\phi-\pi_L(\phi))\big) \rho_L(d\phi) \notag \\
&\hphantom{X} + \int_{\{  \text{dist}(\phi,\M_L)   < \dl L^{\frac 12} \} } F\big(L(\phi-\pi_L(\phi))\big) \ind_{ \{ M_L(\phi) \le L(D-\eps)  \}  }   \rho_L(d\phi) \notag \\
&\hphantom{X}+ \int_{\{  \text{dist}(\phi,\M_L)   < \dl L^{\frac 12} \} } F\big(L(\phi-\pi_L(\phi))\big) \ind_{ \{ L(D-\eps)  \le M_L(\phi) \le LD  \}  }   \rho_L(d\phi) \notag \\
&=\I_1+\I_2+\I_3,
\label{SOLRED0}
\end{align}

\noi
where 
\begin{align*}
&\big\{ \text{dist}(\phi,\M_L)   < \dl L^{\frac 12} \big\}\\
&=\bigg\{\phi\in H^{\frac 12-}(\T_L): \inf_{x_0\in \T_L, \dr\in[0,2\pi] } \| \phi-e^{i\dr}\eta(x-x_0)LQ(L(\cdot-x_0))  \|_{L^2(\T_L)}<\dl L^{\frac 12} \bigg\} .
\end{align*}

\noi 
Here, we verify below that $\I_1$ and $\I_2$  are error terms of order
$O(e^{-cL^3})$ as $L\to \infty$ for some constant $c>0$. Therefore, most of the probability mass is concentrated near the soliton manifold, and
$\M_L(\phi)  \approx LD$,  which corresponds to the $\I_3$ term
\begin{align*}
\I_3=\int_{\{  \text{dist}(\phi,\M_L)   < \dl L^{\frac 12} \} } F\big(L(\phi-\pi_L(\phi))\big) \ind_{ \{ L(D-\eps)  \le M_L(\phi) \le LD  \}  }  \rho_L(d\phi).
\end{align*}

\noi
In particular, we check below that $\eps$  can be at most $\eps=O(L^{-3+})$. See \eqref{alloweps}.

\noi 
Regarding the term $\I_1$ in \eqref{SOLRED0}, we have  from Proposition \ref{PROP:con} that
\begin{align}
|\I_1| \les e^{-c(\dl)L^3}
\label{I1decay}
\end{align}

\noi
as $L\to \infty$.

\noi 
Next, we show that $\I_2$ in \eqref{SOLRED0} is an error term as $L\to \infty$. Note that 
\begin{align}
\I_2&=\int_{\{  \text{dist}(\phi,\M_L)   < \dl L^{\frac 12} \} } F\big(L(\phi-\pi_L(\phi))\big) \ind_{ \{ M_L(\phi) \le L(D-\eps)  \}  }  \rho_L(d\phi) \notag \\
&=\int_{\{  \text{dist}(\phi,\M_L)   < \dl L^{\frac 12} \} } F\big(L(\phi-\pi_L(\phi))\big) e^{\frac 14 \int_{\T_L}|\phi|^4 dx } \ind_{ \{ M_L(\phi) \le L( D-\eps)  \} }  \frac{ \mu_L(d\phi)  } {Z_L^D},
\label{III}
\end{align}

\noi
where $\mu_L$ is the  free field, as defined in \eqref{GFF}, and 
\begin{align}
Z_L^D= \int e^{\frac 14 \int_{\T_L}|\phi|^4 dx }    \ind_{ \{ M_L(\phi) \le LD  \} } \mu_L(d\phi).
\end{align}

\noi 
From Proposition \ref{PROP:con}, we have  
\begin{align*}
\lim_{L\to\infty } \frac{\log Z_L^D}{L^3}=-\inf_{\substack{ \|\phi \|_{L^2(\R)}^2 \le D  } } H(\phi),
\end{align*}

\noi
where the infimum is taken over all $\phi \in H^1(\R)$ satisfying  $\|\phi \|_{L^2(\R)}^2=D$. Therefore, we obtain
\begin{align}
\lim_{L\to\infty } \frac{\log Z_L^D}{L^3}&=\sup_{\| \phi \|_{L^2}^2 \le D} \bigg\{ \frac 14\int_{\R} |\phi|^4 dx- \frac 12 \int_{\R} |\dx \phi|^2 dx \bigg\}\notag \\
&=\sup_{\ld \le D}  \sup_{\| \phi \|_{L^2}^2=\ld } \bigg\{ \frac 14\int_{\R} |\phi|^4 dx- \frac 12 \int_{\R} |\dx \phi|^2 dx \bigg\}\notag \\
&= \sup_{\ld \le D} \ld^3  \sup_{\| \phi \|_{L^2}^2=1 } \bigg\{ \frac 14\int_{\R} |\phi|^4 dx- \frac 12 \int_{\R} |\dx \phi|^2 dx \bigg\} \notag \\
&=-D^3I(1),
\label{free1}
\end{align}

\noi
where in the last line we used $I(1)<0$ from \eqref{Ineg}, and from the second to third line,  we applied the scaling argument $\phi_\ld(x)=\ld\phi(\ld x)$. From \eqref{free1}, we have that as $L\to \infty$  
\begin{align}
\log |\I_2| \le L^3 \bigg( \frac{ \log Z_L^{D-\eps} }{L^3}-\frac{ \log Z_L^D}{L^3} \bigg)+\log(1+\|F\|_{L^\infty}) \les c I(1) L^3\big( D^3-(D-\eps)^3 \big)+c
\label{logII}
\end{align}

\noi
for some constant $c>0$. Since $I(1)<0$  from \eqref{Ineg}, this implies that
\begin{align}
|\I_2|\les e^{- c_1 L^3} \to 0
\label{III0}
\end{align}

\noi
as $L\to\infty$, where $ c_1>0$ is a positive constant. In particular, in \eqref{logII}, since $(D^3-(D-\eps)^3)=O(\eps D^2)$ as $\eps \to 0$, $\eps$ is allowed to be at most $\eps=O(L^{-3+})$ to obtain
\begin{align}
|\I_2| \les  e^{ -c_1L^3(D^3-(D-\eps)^3 )}=e^{-c_1L^{0+}} \too 0
\label{alloweps}
\end{align}

\noi
as $L\to \infty$. Therefore, from \eqref{I1decay} and \eqref{alloweps}, we conclude that it suffices to consider the term $\I_3$ in \eqref{SOLRED0}, where most of the probability mass is concentrated.

Let $u$ be a Gaussian field on $\T_L$ 
\begin{align*}
u(x; \o)=\frac{1}{\sqrt{L}} \sum_{n \in \Z} \frac{g_n(\o)}{\sqrt{1+4\pi^2 |\tfrac{n}{L}|^2}  } e^{2\pi i \frac{n}{L}x }
\end{align*}

\noi
whose law is given by $\mu_L$ as in \eqref{GFF}, corresponding to the mass term $m=1$. Similarly, let $v$ be a Gaussian field on $\T_{L^2}$ 
\begin{align*}
v(x; \o)=\frac{1}{L} \sum_{n \in \Z} \frac{g_n(\o)}{  \sqrt{\tfrac{1}{L^2}+4\pi^2 |\tfrac{n}{L^2}|^2}    } e^{2\pi i \frac{n}{L^2}x }
\end{align*}

\noi
whose distribution is described by $\mu_{L^2}$ as in \eqref{GFF}, corresponding to the mass term $m=\frac 1{L^2}$. Then, we have the following relation $v(x)=\sqrt{L}u(L^{-1}x)$ on $\T_{L^2}$. Thanks to the relation, we have the change of variable 
\begin{align}
u(x)=L\wt u(Lx) \quad \text{on $\T_L$} 
\label{a1}
\end{align}

\noi
where
\begin{align}
\wt u=L^{-\frac 32}v.
\label{a2}
\end{align}

\noi
Then, the distribution of $\wt u$ can be formally expressed as follows
\begin{align*}
\nu_L(d\phi)=Z_{L}^{-1} e^{-\frac {L^3}2  \big(\int_{\T_{L^2} } |\dx \phi|^2 dx +\frac{1}{L^2}\int_{\T_{L^2} } |\phi|^2 dx \big) }  \prod_{x \in \T_{L^2}} d\phi(x),
\end{align*} 

\noi
where the covariance is $L^{-3}(-\dx^2+\frac 1{L^2})^{-1}$. By using the change of variable \eqref{a1} and \eqref{a2}, we have
\begin{align} 
\I_3&=\int_{\{  \text{dist}(\phi,\M_L)   <  \dl L^{\frac 12}  \}}   F\big(L(\phi-\pi_L(\phi))\big) e^{\frac 14 \int_{\T_L}|\phi|^4 dx } \ind_{ \{ L( D-\eps) \le M_L(\phi) \le LD \} } \frac{ \mu_L(d\phi)  } {Z_L^D} \notag \\
&=\E_{\mu_L}\bigg[ F(L(\phi-\pi_L(\phi)) )  e^{\frac{1}{4}\int_{-L}  ^{L} |\phi|^4 dx } \ind_{ \{ L( D-\eps) \le M_L(\phi) \le LD \} }   \ind_{ \{ \text{dist}(\phi,\M_L)<  \dl L^{\frac 12}    \}  } \bigg] \bigg/Z_{L}^D  \notag \\
&=\E_{\nu_L}\bigg[ F(L (\phi_L -\pi_L(\phi_L) ) e^{\frac{L^3}{4}\int_{-L^2}^{L^2} |\phi|^4 dx }   \ind_{ \{ D-\eps \le M_{L^2}(\phi)    \le D   \}   }  \ind_{ \{ \text{dist}(\phi_L,\M_L)< \dl L^{\frac 12}  \}  }  \bigg] \bigg/Z_{L}^D,
\label{INS0}
\end{align}

\noi
where we denote $\phi_L:=L\phi(L\cdot)$. Here, the condition $  \{ \text{dist}(\phi_L,\M_L)<\dl L^{\frac 12} \}   $ implies that 
\begin{align*}
\inf_{x_0\in \T_L,\dr\in [0,2\pi] }\|L\phi(L\cdot)-e^{i\dr }\eta(x-x_0)LQ(L(x-x_0))\|_{L^2(\T_L)}<\dl L^{\frac 12}.
\end{align*}

\noi
Then, the change of variables gives
\begin{align*}
\inf_{x_0\in \T_{L^2},\dr\in [0,2\pi]}\|\phi(\cdot)-e^{i\dr}\eta(L^{-1}(\cdot-x_0))Q(\cdot-x_0)\|_{L^2(\T_{L^2})}<\dl.
\end{align*}

\noi
Then, Lemma  \ref{LEM:coord} implies that if $\dl>0$ is sufficiently small, there exists $x_0\in \T_{L^2}, \dr\in [0,2\pi]$, and $h \in V_{x_0,\dr}^{L^2}$ such that
\begin{align*}
\phi=e^{i\dr}\eta(L^{-1}(\cdot-x_0))Q(\cdot-x_0)+h   =e^{i\dr}Q^{\eta_L}_{x_0}+h,
\end{align*}

\noi
where the normal coordinate $h$ satisfies $\|h \|_{L^2(\T_{L^2} )}<\dl$.
In other words, we can use the new coordinate system $(x_0,\dr,h)$.
Then, by the definition of the projection operator $\pi_L$,  given in \eqref{project},  then there exist unique $x_0 \in \T_L$ and $\dr \in [0,2\pi]$ such that 
 \begin{align}
\pi_L(\phi_L)=e^{i\dr} \eta(\cdot-x_0)LQ(L(\cdot-x_0)), 
\label{project1}
\end{align}

\noi
where $\phi_L=L\phi(L\cdot)=e^{i\dr} \eta(\cdot-x_0)LQ(L(\cdot-x_0))+Lh(L\cdot)$.

The change of variable formula, as given in Lemma \ref{LEM:chan}, implies that the expectation in \eqref{INS0} can be expressed as\footnote{Here, we apply Lemma \ref{LEM:chan} to the Gaussian measure $\nu_L$ with covariance $L^{-3}(-\dx^2+\tfrac 1{L^2})^{-1}$. }
\begin{align}
&\int_{\{ \text{dist}(\phi_L,\M_L)< \dl L^{\frac 12}  \}}
F(L (\phi_L -\pi_L(\phi_L) ) e^{\frac{L^3}{4}\int_{-L^2}^{L^2} |\phi|^4 dx }   \ind_{ \{ D-\eps \le M_{L^2}(\phi)    \le D   \}   } \nu_L(d\phi) \notag \\
&=\iiint_{\wt U^L_\dl} F(L\cdot Lh(L\cdot) ) \exp\bigg\{\frac{L^3}{4}\int_{\T_{L^2}} |e^{i\dr}(Q^{\eta_L}_{x_0}+h)|^4 dx -\frac {L^3}2
\jb{A_L^{-1}Q^{\eta_L}_{x_0},  Q^{\eta_L}_{x_0}}_{L^2(\T_{L^2})}  \notag \\
&\hphantom{XXXXXXXXXXX} -L^3  \jb{ A_L^{-1}  Q^{\eta_L}_{x_0} ,\Re h}_{L^2(\T_{L^2}) } \bigg\}\cdot 
\ind_{ \{ D-\eps \le \| e^{i\dr}(Q^{\eta_L}_{x_0}+h) \|_{L^2(\T_{L^2})}^2 \le D  \} }  \notag  \\
&\hphantom{XXXXXXXXXXX}\times \det\big( \Id-W_{x_0,h}\big)  \mu^{\perp, L^3}_{ x_0, L^2 }(dh) d\s(x_0, \dr),
\label{cha1}
\end{align}

\noi
where $A_L=(-\dx^2+\tfrac 1{L^2})^{-1}$, and the domain of the integration on the right hand side of \eqref{cha1} is
\begin{align}
\wt U^L_\dl=\Big\{ (x_0,\dr,h)\in \T_{L^2}\times [0,2\pi]\times V^{L^2}_{x_0,0}: \|Lh(L\cdot)\|_{L^2(\T_L) }<\dl L^{\frac 12}  \Big\}
\label{dlsize1}
\end{align}

\noi
and 
\begin{align}
\mu_{x_0,L^2}^{\perp, L^3}(dh)=Z_{L}^{-1} e^{-\frac {L^3}2  \big(\int_{\T_{L^2} } |\dx h|^2 dx  + \frac 1{L^2} \int_{\T_{L^2}} |h|^2 dx  \big)}  \prod_{x \in \T_{L^2}} dh(x) 
\label{GaussL3}
\end{align} 

\noi
denotes  the Gaussian measure with $V_{x_0,0}^{L^2} \cap H^1(\T^2_{L^2})\subset H^1(\T^2_{L^2})$ as its Cameron-Martin space and the covariance operator $L^{-3}(-\dx^2+\tfrac{1}{L^2})^{-1}$. 
In fact, when applying Lemma \ref{LEM:chan} in \eqref{cha1}, it is necessary for the integrand to be both bounded and continuous.
However, the integrand
\begin{align*}
e^{\frac{L^3}{4}\int_{-L^2}^{L^2} |\phi|^4 dx }   \ind_{ \{ D-\eps \le M_{L^2}(\phi)    \le D   \}   }
\end{align*}

\noi 
is neither bounded nor continuous. Here, the discontinuity arises from the sharp cutoff $\ind_{ \{ D-\eps \le M_{L^2}(\phi)    \le D   \}   }$. However, since the set of discontinuities has $\nu_L$-measure zero, we can start by applying a smooth cutoff and then take the limit. We can also replace the integrand with a bounded one, provided the bounds are uniform, relying on a simple approximation argument that we omit.

We expand the quartic term in \eqref{cha1} as follows:
\begin{align}
\frac{L^3}{4} \int_{\T_{L^2}} |Q^{\eta_L}_{x_0}+h |^4 \,\mathrm{d}x &= \frac {L^3}4 \int_{\T_{L^2}} |Q^{\eta_L}_{x_0}|^4\, \mathrm{d}x+L^3\int_{\T_{L^2} } (Q^{\eta_L}_{x_0})^3 \Re h \, \mathrm{d}x \notag \\
&\hphantom{X}+L^3\int_{\T_{L^2}} |Q^{\eta_L}_{x_0}|^2\Big(\frac 12 |h|^2 +|\Re h|^2\Big)\, \mathrm{d}x+L^3 \int_{\T_{L^2} }  Q^{\eta_L}_{x_0}  |h|^2 \Re h \, \mathrm{d}x  \notag \\
&\hphantom{X}+\frac {L^3}4 \int_{\T_{L^2}} |h|^4 \, \mathrm{d}x.
\label{expquar}
\end{align}

\noi
By combining \eqref{expquar}, $e^{-\frac {L^3}2
\jb{A_L^{-1}Q^{\eta_L}_{x_0},  Q^{\eta_L}_{x_0}}_{L^2(\T_{L^2})} }$, and $e^{-L^3  \jb{ A_L^{-1}  Q^{\eta_L}_{x_0} ,\Re h}_{L^2(\T_{L^2}) } }$ from \eqref{cha0}, where $A_L=(-\dx^2+\tfrac 1{L^2})^{-1}$, we obtain
\begin{align}
&-\frac{L^3}{2} \int_{\T_{L^2}} |\dx   Q^{\eta_L}_{x_0} |^2\, \mathrm{d}x+\frac{L^3}{4} \int_{\T_{L^2}} |Q^{\eta_L}_{x_0} |^4 \, \mathrm{d}x -\frac L2 \int_{\T_{L^2}} |Q^{\eta_L}_{x_0} |^2 dx \notag \\
&=-L^{3}H_{L^2}( Q^{\eta_L}_{x_0})(1+O(\tfrac{1}{L^2})), 
\label{exp00}
\end{align}

\noi
and
\begin{align}
&L^3 \langle\dx^2 Q^{\eta_L}_{x_0},  \Re h \rangle_{L^2(\T_{L^2})}+L^3\langle  (Q^{\eta_L}_{x_0} )^3, \Re h \rangle_{L^2(\T_{L^2})}-L\langle  Q^{\eta_L}_{x_0}, \Re h \rangle_{L^2(\T_{L^2})} \notag \\
&= \Lambda L^3\langle  Q^{\eta_L}_{x_0}, \Re h \rangle_{L^2(\T_{L^2})} \cdot (1-\tfrac{1}{\Ld L^2})+O(e^{-cL}),
\label{eexp00} 
\end{align}

\noi
where $\Ld>0$ is the Lagrange multiplier in \eqref{Lagmul} and $c>0$ is a constant. 
In \eqref{eexp00}, the error term  $O(e^{-cL})$  arises from the presence of the cutoff function $\eta$. Note that 
\begin{align}
\dx^2 Q^{\eta_L}_{x_0}&=L^{-2}(\dx^2 \eta)(L^{-1}(\cdot-x_0) )Q(\cdot-x_0)+2L^{-1}(\dx \eta)(L^{-1}(\cdot-x_0))\dx Q(\cdot-x_0) \notag \\
&\hphantom{X}+\eta(L^{-1}(\cdot-x_0) ) (\dx^2 Q)(\cdot-x_0).
\label{firstterm0}
\end{align}

\noi
The cutoff function $\eta$, as defined in \eqref{cutoffeta}, satisfies
$\eta=1 $ on $\big[-\tfrac 18, \tfrac 18\big]$, 
Consequently, both $(\dx^2 \eta)(L^{-1}(\cdot-x_0) )$ and $(\dx \eta)(L^{-1}(\cdot-x_0))$  vanish for $|x-x_0|\le \frac L8$.  On the other hand, $Q(\cdot-x_0)$ is localized at $x_0$ with an exponential tail.
Therefore, for $|x-x_0| \ge \frac L8$, we have $Q(\cdot-x_0)=O(e^{-cL})$.  Combined with $\| h\|_{L^2(\T_{L^2})} \le \dl $, as given \eqref{dlsize1}, we can conclude that the first and second terms in \eqref{firstterm0}, when multiplied by $L^3$ and paired with $\jb{\cdot, \Re h}$, are absorbed into the error term $O(e^{-cL})$ in \eqref{eexp00}. 
The last term in \eqref{firstterm0} and the term $(Q^{\eta_L}_{x_0})^3$ in \eqref{eexp00} take the form
\begin{align}
&\eta(L^{-1}(\cdot-x_0) ) (\dx^2 Q)(\cdot-x_0)+\eta(L^{-1}(\cdot-x_0) )^3 Q(\cdot-x_0)^3 \notag \\
&=\eta(L^{-1}(\cdot-x_0) ) \big( (\dx^2 Q)(\cdot-x_0)+ Q(\cdot-x_0)^3 \big) \notag \\
&\hphantom{X}+Q(\cdot-x_0)^3\big(\eta(L^{-1}(\cdot-x_0) )^3-\eta(L^{-1}(\cdot-x_0) )      \big) \notag \\
&=\Ld Q^{\eta_L}_{x_0}+Q(\cdot-x_0)^3\big(\eta(L^{-1}(\cdot-x_0) )^3-\eta(L^{-1}(\cdot-x_0) )      \big), 
\label{LM0}
\end{align}

\noi
where we used the structure of the equation for the ground state  $Q=Q_D$ in \eqref{Lagmul}. As before, since $\eta=1 $ on $\big[-\tfrac 18, \tfrac 18\big]$, 
$\eta(L^{-1}(\cdot-x_0) )^3-\eta(L^{-1}(\cdot-x_0) )=0$ on $|x-x_0|\le \frac L8$. For $|x-x_0| \ge \frac L8$, we have $Q(\cdot-x_0)=O(e^{-cL})$. Hence, combined with $\| h\|_{L^2(\T_{L^2})} \le \dl $, as given \eqref{dlsize1}, the second term in \eqref{LM0}, when multiplied by $L^3$ and paired with $\jb{\cdot, \Re h}$, are absorbed into the error term $O(e^{-cL})$ in \eqref{eexp00}.

The fourth and fifth terms on the right-hand side of equation \eqref{expquar} can be interpreted as higher-order error terms
\begin{align}
L^3\mathcal{E}(Q^{\eta_L}_{x_0},h)=L^3 \int_{\T_{L^2} }  Q^{\eta_L}_{x_0}  |h|^2 \Re h \, \mathrm{d}x +\frac {L^3}4 \int_{\T_{L^2}} |h|^4\, \mathrm{d}x.
\label{exp01}
\end{align}

\noi
It follows from \eqref{cha1}, \eqref{expquar}, \eqref{exp00}, \eqref{eexp00}, and \eqref{exp01} that
\begin{align}
&\int_{\{ \text{dist}(\phi_L,\M_L)< \dl L^{\frac 12}  \}}
F(L (\phi_L -\pi_L(\phi_L) ) e^{\frac{L^3}{4}\int_{-L^2}^{L^2} |\phi|^4 dx }   \ind_{ \{ D-\eps \le M_{L^2}(\phi)    \le D   \}   } \nu_L(d\phi)
\notag \\
&=(1+O(e^{-cL}))\iiint_{\wt U^L_\eps } F(L^2 h(L\cdot ) ) \exp\Big\{-L^3 H_{L^2}(Q^{\eta_L}_{x_0} )(1+O(\tfrac{1}{L^2}))  +L^3\mathcal{E}(Q^{\eta_L}_{x_0}, h)  \Big\}  \notag \\
&\hphantom{XXXXXXX}\cdot  \exp\bigg\{ \Ld L^3 \langle Q^{\eta_L}_{x_0}, \Re h \rangle (1+O(\tfrac{1}{L^2})) +\frac {L^3}2 \langle (Q^{\eta_L}_{x_0})^2h,h \rangle+ L^3 \langle (Q^{\eta_L}_{x_0})^2 \Re h, \Re h  \rangle   \bigg\} \notag  \\
&\hphantom{XXXXXXX} \cdot \ind_{ \big\{ D-\eps \le \| e^{i\dr}(Q^{\eta_L}_{x_0}+h) \|_{L^2(\T_{L^2})}^2 \le D  \big\} } \notag \\
&\hphantom{XXXXXXX}\cdot  \det\big( \Id-W_{x_0,\dr,h}\big) \mu^{\perp, L^3}_{x_0,L^2}(dh) d\s(x_0, \dr),
\label{INS01}
\end{align}

\noi
where the term $ 1+O(e^{-cL}) $ arises from the error term in \eqref{eexp00} and $e^{x}=1+O(x)$ when $|x| \ll 1$. Here, the inner product means the real inner product as defined in \eqref{INER}. We observe that if $u$ represents a Gaussian random field with $\Law(u)=\mu_{x_0,L^2}^{\perp}$ on the normal space $V^{L}_{x_0,0}$
\begin{align}
\mu_{x_0,L^2}^{\perp}(dh)=Z_{L}^{-1} e^{-\frac {1}2  \int_{\T_{L^2} } |\dx h|^2 dx -\frac 1{2L^2} \int_{\T_{L^2}} |h|^2 dx  }  \prod_{x \in \T_{L^2}} dh(x),
\label{massl2}
\end{align}

\noi 
applying the linear transformation $u \mapsto L^{-\frac 32} u$, $L^{-\frac 32} u$ becomes a Gaussian random field with $\Law(L^{-\frac 32} u )=\mu_{x_0,L^2}^{\perp, L^3}$ in \eqref{GaussL3}. Therefore, we obtain 
\begin{align}
&\eqref{INS01}\notag \\
&=(1+O(e^{-cL}))\iiint_{ \mathcal{K}^L_\dl } F(L^{\frac 12} h(L\cdot ) ) \exp\Big\{-L^3 H_{L^2}(Q^{\eta_L}_{x_0} )(1+O(\tfrac{1}{L^2}))  +L^3\mathcal{E}(Q^{\eta_L}_{x_0}, L^{-\frac 32}h)  \Big\}  \notag \\
&\hphantom{XXXXXXXXX}\cdot  \exp\bigg\{ \Lambda L^{\frac 32} \langle Q^{\eta_L}_{x_0},  \Re h \rangle (1+O(\tfrac{1}{L^2}))  +\frac {1}2 \langle (Q^{\eta_L}_{x_0})^2h,h \rangle+  \langle (Q^{\eta_L}_{x_0})^2 \Re h, \Re h  \rangle   \bigg\} \notag  \\
&\hphantom{XXXXXXXXX} \cdot \ind_{ \{ D-\eps \le \| e^{i\dr}(Q^{\eta_L}_{x_0}+L^{-\frac 32}h) \|_{L^2(\T_{L^2})}^2 \le D  \} } \notag  \\
&\hphantom{XXXXXXXXX}\cdot \det\big( \Id-W_{x_0,\dr,L^{-\frac 32}h}\big) \mu^{\perp}_{x_0,L^2}(dh) d\s(x_0, \dr), 
\label{expL}
\end{align}

\noi
where
\begin{align}
\mathcal{K}^L_\dl&=\Big\{ (x_0,\dr,h)\in \T_{L^2} \times [0,2\pi]\times V^{L^2}_{x_0,0}: \|h\|_{L^2(\T_{L^2}) }<\dl L^{\frac 32}  \Big\}.
\label{expL1}
\end{align}

\noi
Note that the condition  $D-\eps \le \| e^{i\dr}(Q^{\eta_L}_{x_0}+L^{-\frac 32}h) \|_{L^2(\T_{L^2})}^2 \le D  $ is equivalent to 
\begin{align}
L^3e^{-cL}-\eps L^3 -\|h\|_{L^2(\T_L^2)}^2\le 2L^{\frac 32} \langle Q^{\eta_L}_{x_0}, \Re h \rangle \le  L^3e^{-cL}-\|h\|_{L^2(\T_L^2)}^2
\label{mass}
\end{align}

\noi
for some $c>0$ since $\|Q^{\eta_L}_{x_0} \|_{L^2(\T_{L^2})}^2=D-e^{-cL}$ as $L\to \infty$. Here, we used the fact that $Q(\cdot-x_0)$ has an exponentially decaying tail.  From \eqref{alloweps}, we see that $\eps$ can be at most $\eps=O(L^{-3+})$. Hence, the functional integral in \eqref{expL} is taken over the regions $\mathcal{K}^L_\dl$ and $S_{Q,x_0}$ as defined in \eqref{SQK}.

By using the translation invariance \eqref{gasurf} of the surface measure $|\g(x_0,\dr)|\, \mathrm{d}\dr\, \mathrm{d}x_0$ in \eqref{surff}, that is, the independence of all the quantities from $x_0$ and $\dr$,
and noting that the term $-L^3 H_{L^2}(Q^{\eta_L}_{x_0})(1+O(\tfrac{1}{L^2}))$ depends only on the $\dot{H}^1$, $L^4$, and $L^2$ size,  we can take out the terms from the integrals over the two tangential directions $x_0$ and $\dr$ to obtain 
\begin{align}
C_L&=e^{-L^3 H_{L^2}(Q^{\eta_L}_{0})\big(1+O(\tfrac{1}{L^2})\big)  }|\g_L(0,0)| \notag .
\end{align}

\noi
Then, this constant cancels out from the partition function $Z_L^D$, which contains the same factor, as shown in the arguments below. 
In summary, by combining \eqref{INS0}, \eqref{INS01}, and \eqref{expL}, we obtain 
\begin{align}
\I_3&=C_L(1+O(e^{-cL}))\notag \\
&\hphantom{X}\times \iiint\limits_{\mathcal{S}_{Q,x_0}\cap \mathcal{K}^L_\dl } F(L^\frac 12h(L\cdot))e^{L^3\mathcal{E}(Q^{\eta_L}_{x_0}, L^{-\frac 32}h)}
e^{\Ld_L L^{\frac{3}{2}}\langle Q^{\eta_L}_{x_0}, \Re h\rangle  } e^{\frac 12 \jb{Q^{\eta_L }_{x_0}, 3|\Re h|^2+|\Im h|^2 } } \notag \\
&\hphantom{XXXXXXXXXXXXXXX}\cdot \mathrm{Det}_L(h) \, \mu^\perp_{x_0,L^2}(\mathrm{d}h) \,\mathrm{d} \dr \, \mathrm{d}x_0 \bigg/ Z_L^D,
\label{numeradeno} 
\end{align}

\noi
where $\Ld_L$ follows from \eqref{eexp00}
\begin{align*}
\Ld_L=\Ld(1-\tfrac{1}{\Ld L^2})=\Ld-\frac 1{L^2}.
\end{align*}

We now study the partition function $Z^D$. As in \eqref{SOLRED0}, we decompose the partition function into the main contribution and error terms as follows 
\begin{align}
Z_L^D=Z_L^D[A_1]+Z_L^D[A_2]+Z_L^D[A_3],
\label{partidecom}
\end{align}

\noi
where
\begin{align*}
Z_L^D[A_i]:=\int_{A_i} e^{\frac 14\int_{\T_{L}} |\phi|^4 \, \mathrm{d}x} \ind_{ \{ M_L(\phi) \le LD \} }\mu_L(d\phi)
\end{align*}

\noi
and
\begin{align*}
A_1:&=\{  \text{dist}(\phi,\M_L)    \ge \dl L^{\frac 12} \} \\
A_2:&=\{  \text{dist}(\phi,\M_L)   < \dl L^{\frac 12} \}\cap  \{ M_L(\phi)\le L(D-\eps) \} \\
A_3:&=\{  \text{dist}(\phi,\M_L)   < \dl L^{\frac 12} \}\cap \{ L(D-\eps)  \le M_L(\phi)\le LD \}.
\end{align*}

From \eqref{partidecom}, Proposition \ref{PROP:con}, \eqref{III}, and \eqref{III0}, we have
\begin{align*}
1=O(e^{-cL^3})+Z_L^D[A_3]/ Z_L^D.
\end{align*}

\noi
This implies 
\begin{align}
Z_L^D[A_3]=Z_L^D(1+O(e^{-cL^3})).
\label{MATT3}
\end{align}

\noi
Therefore, from \eqref{partidecom} and  \eqref{MATT3}, we can rewrite 
\begin{align}
Z_L^D&=Z_L^D[A_1]+Z_L^D[A_2]+Z_L^D[A_3] \notag \\
&=Z_L^D[A_3]\bigg(1+  \frac{Z_L^D[A_1]}{Z_L^D(1+O(e^{-cL^3} ))  }+\frac{Z_L^D[A_1]}{Z_L^D(1+O(e^{-cL^3} ))  }  \bigg) \notag \\
&=Z_L^D[A_3](1+O(e^{-cL^3})),
\label{partidecom1}
\end{align}

\noi
where in the last line, we used  Proposition \ref{PROP:con}, \eqref{III}, and \eqref{III0}.

Regarding the term $Z_L^D[A_3]$,  we can follow the argument of obtaining the numerator in \eqref{numeradeno} to get
\begin{align}
Z_L^D[A_3]&=C_L(1+O(e^{-cL}))\iiint\limits_{\mathcal{S}_{Q,x_0}\cap \mathcal{K}^L_\dl } e^{L^3\mathcal{E}(Q^{\eta_L}_{x_0}, L^{-\frac 32}h)}
e^{\Ld_L L^{\frac{3}{2}}\langle Q^{\eta_L}_{x_0}, \Re h\rangle} e^{\frac 12 \jb{Q^{\eta_L }_{x_0}, 3|\Re h|^2+|\Im h|^2 } } \notag \\
&\hphantom{XXXXXXXXXXXXXXXXXXXX}\cdot \mathrm{Det}_L(h) \, \mu^\perp_{x_0,L^2}(\mathrm{d}h) \,\mathrm{d} \dr \, \mathrm{d}x_0 \notag \\
&=C_L(1+O(e^{-cL})) \cj Z_{L}.
\label{ZLDA3}
\end{align}

\noi
Combining \eqref{numeradeno}, \eqref{partidecom1}, and \eqref{ZLDA3} yields 
\begin{align}
\I_3&=\cj Z_L^{-1}\iiint\limits_{\mathcal{S}_{Q,x_0}\cap \mathcal{K}^L_\dl } F(L^\frac 12h(L\cdot))e^{L^3\mathcal{E}(Q^{\eta_L}_{x_0}, L^{-\frac 32}h)}
e^{\Ld_L L^{\frac{3}{2}}\langle Q^{\eta_L}_{x_0}, \Re h\rangle} e^{\frac 12 \jb{Q^{\eta_L }_{x_0}, 3|\Re h|^2+|\Im h|^2 } }\notag \\
&\hphantom{XXXXXXXXXXXXX}\cdot \mathrm{Det}_L(h) \, \mu^\perp_{x_0,L^2}(\mathrm{d}h) \,\mathrm{d} \dr \, \mathrm{d}x_0 
\label{I3final}
\end{align}

\noi
as $L\to \infty$. In summary, the desired result follows from \eqref{SOLRED0}, \eqref{I1decay}, \eqref{alloweps}, and \eqref{I3final}.


\end{proof}

\begin{remark}\rm \label{REM:Lx_0}
In the tangential component $x_0$-integral over $\T_{L^2}$, given in \eqref{REDUCT},  we can take the change of variable $x_0 \mapsto Lx_0$ to use 
\begin{align}
\eta(L^{-1}(\cdot-Lx_0) )Q(\cdot-L x_0)
\label{appground0}
\end{align}

\noi
instead of $Q^{\eta_L}_{x_0}=\eta(L^{-1}(\cdot-x_0) ) Q(\cdot-x_0)$ in \eqref{REDUCT}. Consequently, if necessary, we may assume that the approximate ground state $Q^{\eta_L}_{x_0}$ is localized at $Lx_0$, as expressed in \eqref{appground0}. In this case, the tangential component $x_0$-integral in \eqref{REDUCT} is taken over $\T_{L}$ not $\T_{L^2}$. Note that when performing the change of variable $x_0 \to Lx_0$, the Jacobian determinant also arises from the partition function $\cj Z_L$. Therefore, we can disregard the common factors. 
\end{remark}



\section{Conditional Ornstein Uhlenbeck measure}\label{SEC:conOU}

According to Proposition \ref{PROP:REDUC}, the measure is conditioned on the nonlinear event $S_{Q,x_0}$ defined in \eqref{SQK}
\begin{align}
S_{Q, x_0}=\Big\{   \jb{Q^{\eta_L}_{x_0}, \Re h }\approx -\frac {1}{2L^{\frac 32}}\| h\|_{L^2}^2  \Big\},
\label{RECOND}
\end{align}

\noi
where all norms are taken over the torus $\T_{L^2}$,  unless otherwise stated.

\noi 
Our task in this section is to study the conditional expectation
\begin{align}
\E_{\s_{x_0,\dr}^\perp }\bigg[F(h) \Big|  \jb{Q^{\eta_L}_{x_0}, \Re h }\approx -\frac 1{2L^\frac 32} \| h\|_{L^2}^2 \bigg]=\frac{\int_{S_{Q,x_0}} F(h)\, \s^{\perp}_{x_0,\dr}(dh)}{\int_{S_{Q,x_0}} \s^{\perp}_{x_0,\dr}(dh)  },
\label{sigmamea00}
\end{align}

\noi
where $\s^\perp_{x_0,\dr}$ comes from Proposition \ref{PROP:REDUC}
\begin{align}
d\s_{x_0,\dr}^\perp(h)= e^{L^3\mathcal{E}(Q^{\eta_L}_{x_0}, L^{-\frac 32}h)}e^{\Ld_L L^{\frac{3}{2}}\langle Q^{\eta_L}_{x_0}, \Re h\rangle} e^{\frac 12 \jb{Q^{\eta_L }_{x_0}, 3|\Re h|^2+|\Im h|^2 } }\mathrm{Det}_L(h) \ind_{\mathcal{K}^L_\dl }(h)\, \frac{\mu^\perp_{x_0,L^2}(\mathrm{d}h)}{Z_{x_0,L^2}}.
\label{sigmamea}
\end{align}

\noi
Here $Z_{x_0, L^2}$ denotes the corresponding partition function, and
\begin{align}
\mu_{x_0,L^2}^{\perp}(dh)=Z_{x_0, L}^{-1} e^{-\frac {1}2  \int_{\T_{L^2} } |\dx h|^2 dx -\frac 1{2L^2} \int_{\T_{L^2}} |h|^2 dx  }  \prod_{x \in \T_{L^2}} dh(x).
\label{JW01}
\end{align}

\noi 
In this section, we show that the conditional expecation \eqref{sigmamea00} admits a conditional density (see Lemma \ref{lem: replace}).

Notice that as $L \to \infty$, the mass term $\frac 1{L^2}$ in the Gaussian measure $\mu_{x_0,L^2}^{\perp}$ \eqref{JW01} vanishes, leading to the disappearance of correlation decay in the infinite volume limit.
However, by combining the term $\Ld_L L^{\frac{3}{2}}\langle Q^{\eta_L}_{x_0}, \Re h\rangle$ in \eqref{sigmamea} and the nonlinear conditioning $S_{Q,x_0}$ in \eqref{RECOND},  we introduce an Ornstein Uhlenbeck type measure $\mu^\perp_{x_0, \Ld}$   in what follows
\begin{align}
\mathrm{d}\mu^{\perp }_{x_0,\Ld} = e^{-\frac{\Ld_L}{2}\|h\|^2_{L^2(\T_{L^2})}}\,\frac{\mathrm{d}\mu^\perp_{x_0,L^2}}{\tilde{Z}_{x_0,L^2}},
\label{OU00}
\end{align}

\noi 
where $\tilde{Z}_{x_0, L^2}$ is the appropriate normalization and $\mu_{x_0, L^2}^{\perp}$ is the Gaussian measure in \eqref{JW01}.
Hence, the Ornstein Uhlenbeck measure \eqref{OU00} has covariance operator $(-\dx^2+\Ld_L)^{-1}=(-\dx^2+\Ld-\tfrac 1{L^2})^{-1}$ projected on the normal space $V^{L}_{x_0,0}$, where $\Ld$ is the Lagrange multiplier in \eqref{Lagmul}.

\begin{remark}\rm\label{REM:mass}
As $L\to \infty$, we have $\Ld_L=\Ld-\frac 1{L^2}\approx \Ld$, which is a small perturbation of $\Ld>0$. Thereore, in the limit $L\to \infty$,  this perturbation $\frac 1{L^2}$ does not affect key features of the Gaussian process, such as the decay of correlations. In what follows, we may thus assume that the Gaussian measure
\begin{align}
\mathrm{d}\mu^{\perp }_{x_0,\Ld} =\tilde{Z}_{x_0,L^2}^{-1} e^{-\frac {1}2  \int_{\T_{L^2} } |\dx h|^2 dx -\frac {\Ld}{2} \int_{\T_{L^2}} |h|^2 dx  }  \prod_{x \in \T_{L^2}} dh(x).
\label{OU1}
\end{align}

\noi
has covariance operator $(-\dx^2+\Ld)^{-1}$.

As a result of the replacement in \eqref{OU00}, we also study the error term in \eqref{sigmamea}
\begin{align*}
e^{\Ld_L L^{\frac{3}{2}}\langle Q^{\eta_L}_{x_0}, \Re h\rangle} e^{\frac{\Ld_L}{2}\|h\|^2_{L^2}}=e^{\Ld_L (L^{\frac{3}{2}}\langle Q^{\eta_L}_{x_0}, \Re h\rangle+\frac 12 \|h\|_{L^2}^2  )     }
\end{align*}

\noi
using the conditioning \eqref{RECOND}. Here, the perturbative term $\frac 1{L^2}$ in  $\Ld_L=\Ld-\frac 1{L^2}$ does not affect the analysis that follows, and thus we may assume $\Ld_L=\Ld$. 

\end{remark}

Correspondingly, combined with Proposition \ref{PROP:REDUC}, \eqref{sigmamea}, \eqref{OU00}, and Remark \ref{REM:mass},  we introduce the notation 
\begin{align}
\mathcal{I}(F)=\int\limits_{S_{Q,x_0}\cap \mathcal{K}^L_\dl}  F(h)&e^{L^3\mathcal{E}(Q^{\eta_L}_{x_0}, L^{-\frac{3}{2}}h)} e^{\Ld L^{\frac 32} \jb{ Q^{\eta_L}_{x_0}, \Re h }+\frac \Ld2\|h \|_{L^2}^2 } \notag  \\
\cdot & e^{\frac 12\jb{ (Q^{\eta_L}_{x_0})^2, 3|\Re h|^2+|\Im h|^2 }  }\mathrm{Det}_L(h) \,\mathrm{d}\mu_{x_0,\Ld}^\perp(h).
\label{Idef}
\end{align}

\noi 
Then, we can write the conditional expectation conditioned on the set $S_{Q,x_0}$ in \eqref{SQK} as follows
\begin{align}
\E_{\s_{x_0,\dr}^\perp }\bigg[F(h) \Big|  \jb{Q^{\eta_L}_{x_0}, \Re h }\approx -\frac 1{2L^\frac 32} \| h\|_{L^2}^2 \bigg]=\frac{\mathcal{I}(F) }{\mathcal{I}(1) }.
\label{condssig}
\end{align}

\noi
Note that, using the notation introduced in \eqref{Idef} along with \eqref{INS01}, \eqref{expL}, and \eqref{OU1}, we can write
\begin{align}
&\int_{ \{ \text{dist}(\phi,\M_L)<\dl L^{\frac 12} \}  } F\big(L(\phi_L-\pi_L(\phi_L)) \big) e^{\frac 14 \int_{\T_L}|\phi|^4 \, \mathrm{d}x } \ind_{ \{ L(D-\eps) \le M_{L}(\phi)    \le LD   \}   }   \mu_L(d\phi) \notag \\
&=C_L\int_{\T_L^{2} }\int_{[0,2\pi] } \mathcal{I}(F) \, \mathrm{d}\dr\, \mathrm{d}x_0,
\label{CHA00}
\end{align}

\noi 
where $\eps$ is chosen as $\eps=O(L^{-3+})$ from \eqref{alloweps}, and
\begin{align}
C_L&=e^{-L^3 H_{L^2}(Q^{\eta_L}_{0} )(1+O(\tfrac{1}{L^2} )) }|\g_L(0,0)|. \notag
\end{align}

\noi
Here, we used the translation invariance \eqref{gasurf} of the surface measure $|\g(x_0,\dr)|\, \mathrm{d}\dr\, \mathrm{d}x_0$ in \eqref{surff}, that is, the independence of all the quantities from $x_0$ and $\dr$.

\noi 
In the following subsections, we derive a conditional density function associated with the conditional expectation given in \eqref{condssig}. see Lemma \ref{lem: replace}.

\begin{remark}\rm 
Under the Ornstein-Uhlenbeck-type measure $\mu^\perp_{x_0, \Ld}$ in \eqref{OU1}, thanks to the law of large numbers for the measure, we have $\| h\|_{L^2(\T_{L^2})}^2 \approx L^2$ with high probability. Hence, we can simplify the nonlinear conditioning as follows 
\begin{align*}
\eqref{condssig} \approx \E_{\s_{x_0,\dr}^\perp }\bigg[F(h) \Big|  \jb{Q^{\eta_L}_{x_0}, \Re h }\approx -L^{\frac 12} \bigg].
\end{align*}

\noi 
even if the effect remains a large negative number $-L^\frac 12$. 

\end{remark}



\subsection{Orthogonal decomposition of the field}

To express the conditional expectation in \eqref{condssig}
in terms of conditional density, we separate the mean-zero Gaussian random variable component $\jb{ Q^{\eta_L}_{x_0}, \Re h }$  from the rest of the field $h$.

For the computations below, it will be useful to have a concrete representation of $h$. We begin with a white noise
\begin{equation}\label{eqn: WN-def}
w_0=\sum_n  g_n u_n.
\end{equation}

\noi 
Here, $\{g_n\}$ is a sequence of i.i.d. complex Gaussian random variables with mean zero and variance $1$, and $\{u_n\}$ is an orthonormal basis of $L^2(\T_{L^2};\mathbb{C})$. Then 
\begin{align*}
\mathbb{E}\Big[\big\|\big(\sqrt{1-\dx^2}\big)^{-\frac{1}{2}-\frac{\eps}{2}}w_0\big\|_{L^2(\T_{L^2})}^2 \Big]=\sum_{n=1}^\infty \big\|\big(\sqrt{1-\dx^2}\big)^{-\frac{1}{2}-\frac{\eps}{2}}u_n\big\|_{L^2(\T_{L^2}) }^2.
\end{align*}

\noi 
The quantity is the Hilbert-Schmidt norm of the operator $\big(\sqrt{1-\dx^2}\big)^{-\frac{1}{2}-\frac{\eps}{2}}$ on $\T_{L^2}$, which is independent of the orthonormal basis, and thus can be computed on the basis $\frac{1}{\sqrt{2}L}e^{2\pi i \frac{n}{L^2}x}$:
\[\big\|\big(\sqrt{1-\dx^2}\big)^{-\frac{1}{2}-\frac{\eps}{2}} \big\|_{\mathrm{HS}}^2\le C\sum_{n\in \mathbb{Z}}\frac{1}{(1+|\frac{n}{L^2}|^2)^{\frac 12+\frac \eps2}}\cdot \frac 1{L^2}<\infty,\]

\noi 
uniformly in $L\ge 1$, where we used the Riemann sum approximation. Thus \eqref{eqn: WN-def} converges in $H^{-\frac{1}{2}-\eps}(\T_{L^2})$.

Define the (unbounded) operators $B_1$ and $B_2$ on real subspace $\Re L^2(\T_{L^2})$ and the imaginary subspace $\Im L^2(\T_{L^{2}})$, respectively \begin{equation}\label{eqn:defB}
\begin{split}
B_1&=-\dx^2+\Ld\\
B_2&=-\dx^2+\Ld,
\end{split}
\end{equation}

\noi
where $\Ld>0$ is the Lagrange multiplier in \eqref{Lagmul}. Note that the base Gaussian measure $\mu_{x_0,\Lambda}$ \eqref{OU1} is defined on the normal space $V^{L^2}_{x_0,0}$, where the normal space is understood with the mass term $\Lambda$, as explained in Remark~\ref{REM:mass0}.
Accordingly, we need to project the covariance operators $B_1$ and $B_2$ onto this space.
Define $\mathbf{P}_{V^{L^2}_{x_0,0}}$ to be the orthogonal projector in $H^1_\Ld$ onto $V_{x_0,0}^{L^2}$, where $H^1_\Ld$ denotes the Sobolev space associated with the operator $-\dx^2+\Ld$. Then
\begin{equation}\label{eqn: P-def}
\begin{split}
\mathbf{P}_{V^{L^2}_{x_0,0}} &= \mathrm{id}-\frac{\langle (\Ld-\partial_{x}^2) \partial_{x_0}Q_{x_0}^{\eta_L}, \cdot\,\rangle_{L^2}}{\|\partial_{x_0} Q_{x_0}^{\eta_L}\|^2_{ H^1_\Ld} 
}\partial_{x_0} Q_{x_0}^{\eta_L}-\frac{\langle i(\Ld-\partial_{x}^2) Q_{x_0}^{\eta_L}, \cdot\,\rangle_{L^2}}{\| Q_{x_0}^{\eta_L}\|^2_{H^1_\Ld} } iQ_{x_0}^{\eta_L}\\
&= \mathrm{id}- \langle Q^{\eta_L}_{1, x_0},\,\cdot\rangle Q^{\eta_L}_{1, x_0}-\langle  Q^{\eta_L}_{2, x_0},\, \cdot\rangle  Q^{\eta_L}_{2, x_0},
\end{split}
\end{equation}

\noi 
where 
\begin{align}
Q^{\eta_L}_{1, x_0}&:=\frac{ \partial_{x_0}Q_{x_0}^{\eta_L}}{\|\partial_{x_0} Q_{x_0}^{\eta_L}\|_{     H^1_\Ld }} \label{TQ1}\\
Q^{\eta_L}_{2, x_0}&:= \frac{i Q_{x_0}^{\eta_L}}{\| Q_{x_0}^{\eta_L}\|_{  H^1_\Ld}}. \label{TQ2}
\end{align}

\noi 
Finally, we define
\begin{equation}\label{eqn: C1-def}
C_i:= \mathbf{P}_{V^{L^2}_{x_0,0}} B_i^{-1}\mathbf{P}_{V^{L^2}_{x_0,0}}, \quad i=1,2.
\end{equation}

\noi 
Note that in defining the orthogonal projection \eqref{eqn: P-def}, the insufficient regularity of the field  $h$ under the measure $\mu^\perp_{x_0, \Ld}$ requires measuring orthogonality with respect to the weighted $L^2$-inner product (with weight $-\Ld+\dx^2$) instead of the $H^1_\Ld$-inner product, ensuring that the projectors in \eqref{eqn: P-def} are well-defined for $h\in V^{L^2}_{x_0,0}\subset L^2(\T_{L^2})$. In the following lemma, we show that $C_1$ and  $C_2$ in \eqref{eqn: C1-def}  are the covariance operators for the Gaussian processes $\Re h$ and $\Im h, $respectively, on the normal space.

\begin{proposition}\label{PROP:OUproj}
Define the field
\begin{align}
h&:=\mathbf{P}_{V^{L^2}_{x_0,0}} B^{-\frac{1}{2}}w_0 \notag \\
&=\mathbf{P}_{V^{L^2}_{x_0,0}} B_1^{-\frac{1}{2}}\Re w_0+i\mathbf{P}_{V^{L^2}_{x_0,0}} B_2^{-\frac{1}{2}} \Im w_0
\label{gprocess}
\end{align}

\noi 
lies in $H^{\frac{1}{2}-\epsilon}(\T_{L^2})$ for each $\epsilon>0$. Moreover, the $h$ is a Gaussian process with covariance operator
\begin{equation}\label{eqn: cov-def}
\left(\begin{array}{cc}
C_1& \\
& C_2
\end{array}\right)=\mathbf{P}_{V^{L^2}_{x_0,0}}\left(\begin{array}{cc}
B_1^{-1}& \\
& B_2^{-1}
\end{array}\right)\mathbf{P}_{V^{L^2}_{x_0,0}}
\end{equation}
on $L^2(\mathbb{T}_{L^2})=\Re L^2(\T_{L^2})\oplus i\Im L^2(\T_{L^2})$.
\begin{proof}
For $f, g\in L^2(\T_{L^2};\mathbb{R})$, we have
\begin{align*}
\mathbb{E}[\langle \Re h, f \rangle \langle \Re h, g\rangle]&=\sum_{n} \langle \mathbf{P}_{V^{L^2}_{x_0,0}} B_1^{-\frac{1}{2}} u_n,f\rangle\langle \mathbf{P}_{V^{L^2}_{x_0,0}} B_1^{-\frac{1}{2}} u_n,g\rangle\\
&=\langle B_1^{-\frac{1}{2}} \mathbf{P}_{V^{L^2}_{x_0,0}} f, B_1^{-\frac{1}{2}}\mathbf{P}_{V^{L^2}_{x_0,0}} g\rangle\\
&=\langle f, \mathbf{P}_{V^{L^2}_{x_0,0}} B_1^{-1}\mathbf{P}_{V^{L^2}_{x_0,0}} g\rangle\\
&=\langle f,C_1 g\rangle.
\end{align*}
Similar computations show
\begin{align*}
\mathbb{E}[\langle \Im h, f \rangle \langle \Im h, g\rangle]&=\langle f, C_2 g\rangle,\\
\mathbb{E}[\langle \Im h, f \rangle \langle \Re h, g\rangle]&=0.
\end{align*}
This proves the claim regarding the covariance.
\end{proof}
\end{proposition}

Recall that if $X$ and $Y$ are jointly Gaussian, then we can decompose $X$ into two independent Gaussians as follows
\begin{align*}
X=\frac{\E[XY]}{\E[Y^2] } Y +X^\perp,
\end{align*}

\noi
where $X^\perp:=X-\frac{\E[XY]}{\E[Y^2] } Y$. Since $X^\perp$ and $\frac{\E[XY]}{\E[Y^2] } Y$ are independent, we have that 
for any measurable $F$
\begin{align*}
\E[F(X)]&=\E\Big[F\big(X^\perp + \frac{\E[XY]}{\E[Y^2] } Y \big)\Big]\\
&=\int_{\R} \E\Big[F\big(X^\perp + \frac{\E[XY]}{\E[Y^2] } \al \big)\Big] f_Y(\al) d\al,
\end{align*}

\noi
where $f_Y$ is the marginal PDF of $Y$. This implies that
\begin{align*}
\Law(X|Y=\al)=\Law\Big(X^\perp + \frac{\E[XY]}{\E[Y^2] } \al  \Big).
\end{align*}

\noi

In the following lemma, we isolate the mean-zero Gaussian random variable component $\jb{ Q^{\eta_L}_{x_0}, \Re h }$ from the rest of the field $h$.

\begin{lemma}\label{LEM:decomh}
Under the Ornstein Uhlenbeck measure $\mu^\perp_{x_0, \Ld}$ in \eqref{OU1}, the Gaussian process $h$ in \eqref{gprocess} admits the orthogonal decomposition
\begin{align}
&h(x)=\gamma(x) \langle Q^{\eta_L}_{x_0}, \Re h\rangle +h^\perp(x), \label{eqn: decomp}
\end{align}

\noi
where $\jb{Q^{\eta_L}_{x_0}, \Re h }$ and $h^\perp$ are independent Gaussian random fields, and $\|\g \|_{L^p}=O(1)$ for any $1\le p \le \infty$, as defined in \eqref{ga0}. In particular, we have
\begin{align}
\jb{Q^{\eta_L}_{x_0}, \Re  h^\perp }=0.
\label{Qhperp}
\end{align}

\noi 
Furthermore, from \eqref{eqn: decomp},  we obtain the following decomposition of the measure
\begin{align}
\mathrm{d}\mu_{x_0, \Ld}^\perp(h)&=\frac{1}{\sqrt{2\pi }\s }e^{-\frac{t^2}{2\s^2}}\, \mathrm{d}t \, \mathrm{d}\nu_{x_0}^\perp(h^\perp),\label{measuredecom0}
\end{align}

\noi
where $\s^2=\E_{\mu^\perp_{x_0,\Ld}}\big[ |\jb{Q^{\eta_L}_{x_0}, \Re h } |^2 \big]=O(1)$, uniformly in $x_0\in \T_{L^2}$ and $\dr \in [0,2\pi]$.  Here, $ \nu^\perp_{x_0}$ is an Ornstein-Uhlenbeck measure with covariance $(-\dx^2+\Ld)^{-1}$ projected on the Cameron-Martin space $\cj V^{L^2}_{x_0,0} \cap H^1_\Ld(\T_{L^2})$, where
\begin{align}
V^{L^2}_{x_0,0}=\cj V^{L^2}_{x_0,0}\oplus \textup{span}\{e\}
\label{Vbar}
\end{align}

\noi
with $\| e\|_{H^1(\T_{L^2})}=1$, and $e$ is orthogonal in $H^1(\T_{L^2})$ to
\begin{align}
\cj V^{L^2}_{x_0,0}=\big\{w \in V^{L^2}_{x_0, 0}: \jb{Q^{\eta_L}_{x_0}, w }=0  \big\}.
\label{ORTHOGO1}
\end{align}

\end{lemma}

\begin{proof}
Recall the covariance operator $C_1$ for the process $\Re h$ introduced in \eqref{eqn: C1-def} and define 
\begin{align}
\tilde{Q}^{\eta_L}_{x_0}(x):=\frac{C_1 Q^{\eta_L}_{x_0}(x)}{\|C_1 Q^{\eta_L}_{x_0}\|_{L^2}}.
\label{wtQeta}
\end{align}

\noi
We  complete $\tilde{Q}^{\eta_L}_{x_0}$ with a collection $\{f_n\}_{n\ge 1}$ so that
\[\{\tilde{Q}^{\eta_L}_{x_0},f_n: n\ge 1\}\]

\noi 
forms an orthonormal basis of $V_{x_0, \dr}^{L^2}$ in \eqref{normalsp}.
Since $\Re h$ and $\Im h$ are independent under the measure $\mu^{\perp}_{x_0, \Ld}$, and 
the orthogonality of $\wt Q^{\eta_L}_{x_0} $ and $f_n$ in $L^2$, we have 
\begin{align*}
\mathbb{E}_{\mu^{\perp}_{x_0, \Ld}}[\langle Q^{\eta_L}_{x_0},\Re h\rangle \langle f_n,\Im h\rangle]&=0,\\
\mathbb{E}_{\mu^{\perp}_{x_0, \Ld}}\big[\langle Q^{\eta_L}_{x_0},\Re h\rangle \langle f_n,\Re h\rangle \big]&=\langle C_1 Q^{\eta_L}_{x_0},f_n\rangle =0.
\end{align*}

\noi
Therefore, $\jb{Q^{\eta_L}_{x_0}, \Re h}$ and $\jb{f_n, h}$ are independent Gaussian random variables.

Since $\{\tilde{Q}^{\eta_L}_{x_0},f_n: n\ge 1\}$ is an orthonormal basis, we can write 
\begin{equation}\label{eqn: hhtilde}
\begin{split}
h(x)&=\langle \tilde{Q}^{\eta_L}_{x_0},h\rangle \tilde{Q}^{\eta_L}_{x_0}(x)+\sum_n \langle f_n, h \rangle f_n \\
&= \langle \tilde{Q}^{\eta_L}_{x_0},\Re h\rangle \tilde{Q}^{\eta_L}_{x_0}(x) + i\langle \tilde{Q}^{\eta_L}_{x_0},\Im h\rangle \tilde{Q}^{\eta_L}_{x_0}(x)+ \tilde{h},
\end{split}
\end{equation}

\noi
where $\tilde{h}:=\sum_n \langle f_n, h \rangle f_n$.

As a next step, we decompose the Gaussian random variable $\langle \tilde{Q}^{\eta_L}_{x_0},\Re h\rangle$ as follows
\begin{equation}
\label{eqn: Re-decomp}
\begin{split}
\langle \tilde{Q}^{\eta_L}_{x_0},\Re h\rangle&= \frac{\mathbb{E}_{\mu^\perp_{x_0,\Ld}}\big[\langle \tilde{Q}^{\eta_L}_{x_0},\Re h\rangle\langle Q^{\eta_L}_{x_0},\Re h\rangle\big]}{\mathbb{E}_{\mu^\perp_{x_0, \Ld}}\big[|\langle Q^{\eta_L}_{x_0},\Re h\rangle |^2\big]}\langle Q^{\eta_L}_{x_0},\Re h\rangle+g\\
&= a\langle Q^{\eta_L}_{x_0},\Re h\rangle + g,
\end{split}
\end{equation}

\noi 
where $g$ and $\jb{Q^{\eta_L}_{x_0}, \Re h } $ are independent Gaussian random variables, and from \eqref{wtQeta}, 
\begin{align}
a:=\frac{\|C_1 Q_{x_0}^{\eta_L}\|_{L^2}}{\|C_1^{1/2}Q_{x_0}^{\eta_L}\|_{L^2}^2 }.
\label{aprojec}
\end{align}

\noi 
Combining \eqref{eqn: hhtilde}, \eqref{eqn: Re-decomp}, and \eqref{wtQeta}, we have
\begin{equation}
\begin{split}
h(x)&= \big(a\langle Q^{\eta_L}_{x_0},\Re h\rangle +g + i\langle \tilde{Q}^{\eta_L}_{x_0},\Im h\rangle\big) \tilde{Q}^{\eta_L}_{x_0}(x) +\tilde{h}\\
&=\gamma(x) \langle Q^{\eta_L}_{x_0}, \Re h\rangle +h^\perp(x),
\end{split}
\label{hdec}
\end{equation}

\noi 
where
\begin{align}
\gamma(x):&=a\tilde{Q}^{\eta_L}_{x_0}(x)=\frac{C_1 Q^{\eta_L}_{x_0}(x)}{\|C_1^{1/2}Q_{x_0}^{\eta_L}\|_{L^2}^2 } \label{ga0}\\
h^\perp(x):&=\big(g+i\langle \tilde{Q}^{\eta_L}_{x_0},\Im h\rangle\big)\tilde{Q}^{\rho_L}_{x_0}(x)+\tilde{h}(x).
\label{eqn: hperp-g}
\end{align}

\noi
Here, $\langle Q^{\eta_L}_{x_0}, \Re h\rangle$ and $h^\perp(x)$ are independent since $\langle Q^{\eta_L}_{x_0}, \Re h\rangle$, $g$, $\langle Q^{\eta_L}_{x_0}, \Im h\rangle$, and $\wt h$ are independent Gaussian random fields. In particular, from \eqref{hdec} and \eqref{ga0} we have  
\begin{align*}
\jb{Q^{\eta_L}_{x_0}, \Re h^\perp }=\jb{Q^{\eta_L}_{x_0}, \Re h }-\jb{Q^{\eta_L}_{x_0},\g }\jb{Q^{\eta_L}_{x_0}, \Re h }=0,
\end{align*}

\noi
which implies \eqref{Qhperp}.

Since $Q^{\eta_L}_{x_0}$ is a  Schwartz function, from \eqref{ga0} we have $\|\g \|_{L^p}=O(1)$ for any $1\le p \le  \infty$, uniformly in $x_0\in \T_{L^2}, \dr \in [0, 2\pi]$ and
\begin{align*}
\s^2=\E_{\mu^\perp_{x_0,\Ld}}\big[ |\jb{Q^{\eta_L}_{x_0}, \Re h } |^2 \big]=\jb{C_1Q^{\eta_L}_{x_0}, Q^{\eta_L}_{x_0} }=O(1),
\end{align*}

\noi
uniformly in $x_{0}\in \T_{L^2}$ and $\dr \in [0,2\pi]$. This completes the proof of Lemma \ref{LEM:decomh}.

\end{proof}

According to Lemma \ref{LEM:decomh}, we introduce the new coordinate $(h^\perp , t)$ with $h(x)=h^\perp(x)+t\g(x)$, where $t$ is the symbol for the Gaussian random variable $\jb{Q^{\eta_L}_{x_0}, \Re h }$ that appears in the conditioning in \eqref{RECOND}. We represent original quantities using the new coordinate system in the following lemmas.

\begin{lemma}\label{LEM:Error0}
Under the orthogonal decomposition $h(x)=h^\perp(x)+\langle Q^{\eta_L}_{x_0}, \Re h\rangle \gamma(x)$ in \eqref{eqn: decomp}, we have that on the set $S_{Q,x_0}\cap \mathcal{K}^L_\dl$ in Proposition \ref{PROP:REDUC}
\begin{align}
\|h \|_{L^2(\T_{L^2})} &\le (1+\dl_1) \|h^\perp \|_{L^2(\T_{L^2})} \label{hperpL20}\\
|\jb{Q^{\eta_L}_{x_0},  \Re h }|& \le \dl_2  \|h^\perp \|_{L^2(\T_{L^2})} \label{hperpL21}\\
\| h\|_{L^\infty(\T_{L^2}) } &\le \|h^\perp \|_{L^\infty(\T_{L^2})}+\dl_3 \| h^\perp\|_{L^2(\T_{L^2}) } \label{hperpLinf}\\
\|h^\perp \|_{L^2(\T_{L^2})} &\le \dl_4 L^{\frac 32} \label{hperpL2}
\end{align}

\noi
for arbitrary small constants $\dl_1, \dl_2, \dl_3, \dl_4>0$.

\begin{proof}
From the condition $S_{Q,x_0}\cap \mathcal{K}^L_\dl$, we can obtain
\begin{align}
|\jb{Q^{\eta_L}_{x_0}, \Re h }| \le \frac{2}{L^\frac 32} \|h \|_{L^2}^2 \le 2 \dl \| h\|_{L^2}.
\label{SJW1}
\end{align}

\noi
This, together with Young's inequality, implies that 
\begin{align}
|h(x)|^2 &\le (1+\eps)|h^\perp(x)|^2+(1+\eps^{-1})|\jb{Q^{\eta_L}_{x_0}, \Re h }|^2 \g^2(x) \notag \\
&\le (1+\eps)|h^\perp(x)|^2+2 \eps^{-1} \dl^2 \|h \|_{L^2}^2 \g^2(x)
\label{SJW3}
\end{align}

\noi
for some small constant $\eps>0$.
By integrating and choosing $0<\dl \ll \eps$ sufficiently small, we have
\begin{align}
\int_{\T_{L^2}} |h|^2 \, \mathrm{d}x \le \frac{1+\eps}{1-\cj \dl } \int_{\T_{L^2}} |h^\perp|^2 \, \mathrm{d}x
\label{SJW2}
\end{align}

\noi
for some small constant $\cj \dl>0$. This implies \eqref{hperpL20}. Furthermore, from \eqref{SJW1} and \eqref{SJW2}, we obtain \eqref{hperpL21}. 

From $h(x)=h^\perp(x)+\langle Q^{\eta_L}_{x_0}, \Re h\rangle \gamma(x)$, $\|\g \|_{L^\infty}=O(1)$, and \eqref{hperpL21}, we have
\begin{align*}
\| h\|_{L^\infty} \le \| h^\perp \|_{L^\infty}+ \| \g\|_{L^\infty} |\langle Q^{\eta_L}_{x_0}, \Re h\rangle | \le \| h^\perp\|_{L^\infty}+\dl \|h^\perp \|_{L^2}
\end{align*}

\noi
for some small $\dl>0$. This implies \eqref{hperpLinf}.

From $\|h \|_{L^2}\le \dl L^\frac 32$ on the set $\mathcal{K}^L_\dl $, $\| \g\|_{L^\infty}=O(1)$, and \eqref{hperpL21}, we have 
\begin{align*}
\| h^\perp\|_{L^2} \le \|h \|_{L^2}+ \| \g\|_{L^\infty} |\langle Q^{\eta_L}_{x_0}, \Re h\rangle | \le  \dl L^\frac 32+ \zeta \| h^\perp\|_{L^2}
\end{align*}

\noi
for some small $\zeta>0$. By choosing $\zeta>0$ sufficiently small, we obtain \eqref{hperpL2}.

\end{proof}

\end{lemma}

\begin{lemma}\label{LEM:Error1}
Under the orthogonal decomposition $h(x)=h^\perp(x)+\langle Q, \Re h\rangle \gamma(x)$ in \eqref{eqn: decomp}, we have that on the set $S_{Q,x_0}\cap \mathcal{K}^L_\dl$ in Proposition \ref{PROP:REDUC}
\begin{align}
\int_{\T_{L^2}} (Q^{\eta_L}_{x_0})^2|\Re h|^2 \, \mathrm{d}x &\le (1+\eps ) \int_{\T_{L^2}} (Q^{\eta_L}_{x_0})^2 | \Re h^\perp|^2 \, \mathrm{d}x+\dl\int_{\T_{L^2}}  | h^\perp|^2 \, \mathrm{d}x \label{NPL1}\\
L^3\mathcal{E}(Q^{\eta_L}_{x_0}, L^{-\frac 32}h) &\le \frac 1{L^\frac 32} \int_{\T_{L^2} }  Q^{\eta_L}_{x_0}  |h^\perp|^3 \, \mathrm{d}x  +\frac {1}{L^3} \int_{\T_{L^2}} |h^\perp |^4 \, \mathrm{d}x +\dl \int_{\T_{L^2}} |h^\perp|^2 \, \mathrm{d}x
\label{NPL2}
\end{align}

\noi
for arbitrary small constants $0<\dl \ll \eps$, where 
$L^3\mathcal{E}(Q^{\eta_L}_{x_0}, L^{-\frac 32}h)$ is defined in \eqref{Eerror}.


\end{lemma}

\begin{proof}

By integrating  \eqref{SJW3} and using \eqref{hperpL20}, we have 
\begin{align*}
\int_{\T_{L^2}} (Q^{\eta_L}_{x_0})^2|\Re h|^2 \, \mathrm{d}x &\le (1+\eps )\int_{\T_{L^2}} (Q^{\eta_L}_{x_0})^2 |\Re h^\perp|^2\, \mathrm{d}x+2\eps^{-1}\dl^2 \|Q^{\eta_L}_{x_)} \|_{L^\infty} \|\g \|_{L^2}^2 \int_{\T_{L^2}} |h|^2\, \mathrm{d}x\\
&\le  (1+\eps )\int_{\T_{L^2}} (Q^{\eta_L}_{x_0})^2 |\Re h^\perp|^2\, \mathrm{d}x+\zeta \int_{\T_{L^2}} |h^\perp|^2 \, \mathrm{d}x
\end{align*}

\noi
by choosing $\dl>0$ sufficiently small, where $\zeta>0$ is a small constant. This shows \eqref{NPL1}.

\noi
We now prove \eqref{NPL2}. From the orthogonal decomposition of the field $h(x)=\g(x)\jb{Q^{\eta_L}_{x_0}, \Re h }+h^\perp(x)$ in \eqref{eqn: decomp} and $L^{3}\mathcal{E}(Q^{\eta_L}_{x_0}, L^{-\frac 32}h):=\mathcal{E}_1(h)+\mathcal{E}_2(h) $ in \eqref{Eerror}, we expand
\begin{align*}
\mathcal{E}_1(h)&:=\frac 1{L^{\frac 32}} \int_{\T_{L^2} }  Q^{\eta_L}_{x_0}  |h|^2 \Re h dx \notag \\
&= \sum_{\substack{u_i\in \{h^\perp, \gamma(x)\langle Q^{\eta_L}_{x_0}, \Re h\rangle\}\\i=1,2,3} } \frac{1}{L^{\frac{3}{2}}}\int_{\T_{L^2}} Q_{x_0}^{\eta_L}u_1\bar{u}_2\Re u_3\,\mathrm{d}x, \\
\mathcal{E}_2(h):&=\frac{1}{L^3} \int_{\T_{L^2}}|h(x)|^4\, \mathrm{d}x \notag \\
&= \sum_{\substack{u_i\in \{h^\perp, \gamma(x)\langle Q^{\eta_L}_{x_0}, \Re h\rangle\}\\i=1,2,3,4}} \frac{1}{L^3}\int_{\T_{L^2}} u_1\bar{u}_2 u_3\bar{u}_4\,\mathrm{d}x. 
\end{align*}

\noi
By using Young's inequality, \eqref{hperpL21}, and \eqref{hperpL2}, we have
\begin{align*}
|\mathcal{E}_1(h)|&\le \frac 1{L^\frac 32} \int_{\T_{L^2}} Q^{\eta_L}_{x_0}|h^\perp|^3, \mathrm{d}x +\frac{| \jb{Q^{\eta_L}_{x_0}, \Re h }|^3 }{L^\frac 32} \| \g\|_{L^3}^3\\
&\le L^{-\frac{3}{2}}\int_{\T_{L^2}} Q^{\eta_L}_{x_0}|h^\perp|^3\,\mathrm{d}x  + \dl \|h^\perp \|_{L^2}^2, \\
|\mathcal{E}_2(h)|&\le  C_1 \frac 1{L^3} \int_{\T_{L^2}} |h^\perp|^4 \, \mathrm{d}x+\frac{| \jb{Q^{\eta_L}_{x_0}, \Re h }|^4 }{L^3} \| \g\|_{L^4}^4\\
&\le  C_1 \frac 1{L^3} \int_{\T_{L^2}} |h^\perp|^4 \, \mathrm{d}x+\dl \| h^\perp\|_{L^2}^2
\end{align*}

\noi
for arbitrary small $\dl>0$. This shows \eqref{NPL2}.

\end{proof}

\subsection{Gaussian measures associated with Schrödinger operators}

In this subsection, we introduce Gaussian measures associated with the Schrödinger operators in \eqref{SCHOP}, using the orthogonal decomposition of the field $h$ as established in Lemma \ref{LEM:decomh}.

From the expression of the conditional expectation \eqref{condssig} involving  $\mathcal{I}(F)$  defined in \eqref{Idef},  we focus in this subsection on studying the quadratic part, making the Gaussian measure as follows
\begin{align}
e^{\frac 32 \int_{\T_{L^2}   } (Q^{\eta_L}_{x_0} )^2 |\Re  h|^2\, \mathrm{d}x+ \frac 12 \int_{\T_{L^2}   } (Q^{\eta_L}_{x_0} )^2 |\Im  h|^2\, \mathrm{d}x}\, \mathrm{d}  \mu^\perp_{x_0,\Ld} (h),
\label{GAUSS0}
\end{align}

\noi
where $\mu^\perp_{x_0, \Ld}$ is the Gaussian measure, as defined in \eqref{OU1}, with covariance $(-\dx^2+\Ld)^{-1}$ projected on the normal space $V^{L^2}_{x_0,0} \cap H^1_\Ld(\T_{L^2})$  as the Cameron-Martin space. The key issue with \eqref{GAUSS0} is that the covariance operators for the real and imaginary parts, given by
\begin{align*}
&-\dx^2-3(Q^{\eta_L}_{x_0})^2+\Ld \notag \\
&-\dx^2-(Q^{\eta_L}_{x_0})^2+\Ld,
\end{align*}

\noi
are not positive on $\Re L^2(\T_{L^2})$ and $\Im L^2(\T_{L^2})$, respectively, as $L\to \infty$.  This prevents the Gaussian measure in \eqref{GAUSS0} from being properly defined as written. See \eqref{Degeneracy}.


\noi
Therefore, instead of focusing on the coordinate $h$ under $\mu^\perp_{x_0,\Ld}$, we use the alternative coordinate $ h^\perp$  under the measure $\nu_{x_0}^\perp$, obtained from  Lemma \ref{LEM:decomh} (particularly \eqref{eqn: decomp} and \eqref{measuredecom0}). Expanding \eqref{GAUSS0} in terms of the orthogonal decomposition $h(x)=\g(x) \jb{Q^{\eta_L}_{x_0}, \Re h }+h^\perp(x)$ from \eqref{eqn: decomp}, we write the Gaussian part   as
\begin{align}
e^{\frac 32 \int_{\T_{L^2}   } (Q^{\eta_L}_{x_0} )^2 |\Re  h^\perp|^2\, \mathrm{d}x+ \frac 12 \int_{\T_{L^2}   } (Q^{\eta_L}_{x_0} )^2 |\Im  h^\perp|^2\, \mathrm{d}x}\, \mathrm{d}  \nu^\perp_{x_0}( h^\perp),
\label{SCH0}
\end{align}

\noi
where we used $\Im h=\Im h^\perp$, and the weight, using the symbol $t=\jb{Q^{\eta_L}_{x_0}, \Re h }$, as
\begin{align}
e^{\frac {3}{2}t^2 \int_{\T_{L^2} } (Q^{\eta_L}_{x_0}\g)^2\,\mathrm{d}x+3t\int_{\T_{L^2}} (Q^{\eta_L}_{x_0})^2 \g \Re h^\perp \, \mathrm{d}x}. 
\label{WEIGHT}
\end{align}

\noi
In \eqref{SCH0}, $ \nu^\perp_{x_0}$ is the Gaussian measure, as defined in \eqref{measuredecom0}, with covariance $(-\dx^2+\Ld)^{-1}$ projected on the normal space $\cj V^{L^2}_{x_0,0} \cap H^1_\Ld(\T_{L^2})$  as the Cameron-Martin space. The weight \eqref{WEIGHT} will be controlled separately in the subsequent sections.

\noi 
The expression \eqref{SCH0} suggests defining Gaussian measures on $\Re L^2(\T_{L^2})$ and $\Im L^2(\T_{L^2})$ as follows
\begin{align}
e^{-\frac 12 \jb{ B_1^Q \Re  h^\perp, \Re  h^\perp } } \prod_{x\in \T_{L^2}} \mathrm{d} \Re   h^\perp(x) \label{GRE1}\\
e^{-\frac 12 \jb{ B_2^Q  \Im  h^\perp, \Im  h^\perp }} \prod_{x\in \T_{L^2}} \mathrm{d} \Im  h^\perp(x), \label{GIM1}
\end{align}

\noi
where each measure has the covariance operator $B_i^{-1}$
\begin{align}
B_1^{Q}&=-\dx^2-3(Q^{\eta_L}_{x_0})^2+\Ld \notag \\
B_2^{Q}&=-\dx^2-(Q^{\eta_L}_{x_0})^2+\Ld,
\label{SCHOP}
\end{align}

\noi
where $\Ld>0 $ is the Lagrange multiplier  in \eqref{Lagmul}. 
This also corresponds to the second variation $\nb^2 H_{L^2}(Q^{\eta_L}_{x_0})$ of the Hamiltonian in \eqref{Ham0}.  As $L \to \infty$, these operators, known as Schr\"odinger operators, play a central role in the stability theory of ground states.

In particular, as $L\to \infty$, that is, when the approximate soliton $Q^{\eta_L}_{x_0}$ in \eqref{QetaL}  approaches the exact ground state (see Remark \ref{REM:SOL}), it is well known that $B_1^Q$ and $B_2^Q$ defined on $\mathbb{R}$ have the following spectral properties
\begin{align}
B_1^Q&=0  \quad \text{on} \quad \text{span}\,\{ \partial_{x_0}Q_{x_0} \} \notag \\
B_1^Q&<0  \quad \text{on} \quad \text{span}\,\{Q_{x_0} \} \notag \\
B_2^Q&=0  \quad \text{on} \quad \text{span}\,\{Q_{x_0} \},
\label{Degeneracy}
\end{align}

\noi
where $Q_{x_0}=Q(\cdot-x_0)$. See \cite[Lemma 2.1, Lemma 2.2]{NAK}.
Furthermore, when $Q^{\eta_L}_{x_0}$ is the exact ground state, the operator $B_2^Q $ on $\R$ is nonnegative: $B_2^Q \ge 0$, while $B_1^Q$ on $\R$ has exactly one negative eigenvalue and so $B_1^Q \ge 0$ on span$\{ Q_{x_0} \}^\perp$. Therefore, we cannot define the Gaussian measures \eqref{GRE1} and \eqref{GIM1} on $\Re L^2(\T_{L^2})$ and $\Re L^2(\T_{L^2})$ as $L\to \infty$. This implies that to define the Gaussian measures associated with the Schrödinger operators $B_1^Q$ and $B_2^Q$ on $\R$, we need to remove the directions $\text{span}\{Q_{x_0}, \partial_{x_0} Q_{x_0} \}$ and $\text{span}\{Q_{x_0} \}$, respectively, which lead to degeneracy in the covariance operators. By doing so, we ensure that
\begin{align}
B_1^Q>0 \quad \text{on} \quad &\text{span}\{Q_{x_0}, \partial_{x_0} Q_{x_0} \}^\perp \label{DEGN0}\\
B_2^Q>0 \quad \text{on} \quad &\text{span}\{Q_{x_0} \}^\perp.
\label{DEGN}
\end{align}

\noi
In this way, we can define Gaussian measures on the space while excluding certain directions.


\noi 
Recall that $\text{span}\big\{ \partial_{x_0}(e^{i\ta}Q_{ x_0}^{\eta_L})|_{\dr=0}, \partial_{\dr}(e^{i\ta}Q_{ x_0, \dr}^{\eta_L})|_{\dr=0}\big\}=T_{x_0,0} \bar \M_L$ is the tangent space of the soliton manifold $\bar M_L$ \eqref{solm0} at $Q_{x_0}^{\eta_L}$. 
From Proposition \ref{PROP:REDUC}, since  we are working on the normal space $V^{L^2}_{x_0,0}$ in \eqref{normalsp} 
\begin{align}
V^{L^2}_{x_0,0}=\big\{ w\in L^2(\T_{L^2}):   \jb{w, (\Ld-\dx^2)\partial_{x_0}Q^{\eta_L}_{x_0}  }=0, \;  \jb{w, i(\Ld-\dx^2)Q^{\eta_L}_{x_0}  }=0    \big\},
\label{ORTHOGO}
\end{align}

\noi 
which excludes the tangential directions $T_{x_0,0} \bar \M_L$, it follows that $ h^\perp \in V^{L^2}_{x_0,0} $. Here, the inner product, as defined in \eqref{INER}, means the real inner product. 
In particular, 
due to \eqref{ORTHOGO} and \eqref{ORTHOGO1}, we have 
\begin{align}
\Re  h^\perp \in \cj V_{x_0,0}^{L^2, \text{Re}}:=\big\{ w \in \Re L^2(\T_{L^2}) :  \jb{w, (\Ld-\dx^2)\partial_{x_0}Q^{\eta_L}_{x_0}  }=0,   \; \jb{Q^{\eta_L}_{x_0}, w }=0  \big\}.
\label{VRE}
\end{align}

\noi
Note that   $Q^{\eta_L}_{x_0}, \partial_{x_0} Q^{\eta_L}_{x_0} \notin \cj V_{x_0,0}^{L^2, \text{Re}}$ since $\| Q^{\eta_L}_{x_0}\|_{L^2}, \|\dx \partial_{x_0} Q^{\eta_L}_{x_0}  \|_{L^2} \neq 0$. Furthermore, 
for any $c_1,c_2 \in \R\setminus\{0\}$, we have 
\begin{align}
c_1Q^{\eta_L}_{x_0}+c_2 \partial_{x_0} Q^{\eta_L}_{x_0} \notin  \cj V_{x_0,0}^{L^2, \text{Re}},
\label{VRE1}
\end{align}

\noi 
due to the orthogonality of   $Q^{\eta_L}_{x_0}$ and $\partial_{x_0} Q^{\eta_L}_{x_0}$ in $\dot{H}^k(\T_{L^2})$ for any $k \in \Z_{ \ge 0}$. This orthogonality arises from the fact that $Q^{\eta_L}_{x_0}$ is an even function, while $\partial_{x_0}Q^{\eta_L}_{x_0} $ is odd, based on parity considerations of these functions centered at $x=x_0$. Therefore, regarding $\Re  h^\perp$, we can remove all directions in \eqref{DEGN0} that cause degeneracy as $L\to \infty$. On the other hand, regarding the imaginary part, we have 
\begin{align}
\Im  h^\perp \in \cj V_{x_0,0}^{L^2, \text{Im}}:=\big\{ w \in \Im L^2(\T_{L^2}) :   \jb{w,(\Ld-\dx^2)Q^{\eta_L}_{x_0} }=0  \big\}.
\label{VIM}
\end{align}

\noi
Note that  
\begin{align}
Q^{\eta_L}_{x_0} \notin \cj V_{x_0,0}^{L^2, \text{Im}}
\label{VIM1}
\end{align}

\noi 
since $\| \dx Q^{\eta_L}_{x_0}\|_{L^2} \neq 0$.
Therefore, regarding $\Im  h^\perp$, we remove the direction in \eqref{DEGN} that causes degeneracy as $L\to \infty$.

\begin{lemma}\label{LEM:positive}
There exits $\zeta>0$ such that for any sufficiently large $L\ge 1$, we have 
\begin{align*}
\jb{B_1^Qw,w}&\ge \zeta \|w \|_{L^2(\T_{L^2})}^2 \quad \text{on} \quad \cj V_{x_0,0}^{L^2, \textup{Re}}\\
\jb{B_2^Qw,w}&\ge \zeta \|w \|_{L^2(\T_{L^2})}^2 \quad \text{on} \quad \cj V_{x_0,0}^{L^2, \textup{Im}},
\end{align*}

\noi
where $B_1^Q$ and $B_2^Q$ are Schr\"odinger operators, defined in \eqref{SCHOP}. Here, the spaces $\cj V_{x_0,0}^{L^2, \textup{Re}}$ and $\cj V_{x_0,0}^{L^2, \textup{Im}}$ are defined in \eqref{VRE} and \eqref{VIM}, respectively.
\end{lemma}

\begin{proof}
For the details of the proof, we can follow \cite[Proposition 6.15]{OST1}.
Here, we briefly outline the idea.

As discussed in \eqref{DEGN0} and \eqref{DEGN}, it follows from \cite[Lemma 2.1, Lemma 2.2]{NAK} that 
\eqref{DEGN0} and \eqref{DEGN} represent the directions causing degeneracy for the operators $B_1^Q$ and $B_2^Q$ as $L\to \infty$.
From the definitions of the spaces $\cj V_{x_0,0}^{L^2, \textup{Re}}$ and $\cj V_{x_0,0}^{L^2, \textup{Im}}$, along with \eqref{VRE1} and \eqref{VIM1}, we eliminate these directions, ensuring that there exists a constant $\zeta>0$ such that
\begin{align}
\jb{ B_{i}^Q w, w } \ge \zeta \| w \|_{L^2}^2
\label{positive0}
\end{align}

\noi
for $i=1,2$. That is, $B_1^Q$ and $B_2^Q$ are positive operators on
$\cj V_{x_0,0}^{L^2, \textup{Re}}$ and $\cj V_{x_0,0}^{L^2, \textup{Im}}$, respectively. 
Note that $\zeta$ appearing in \eqref{positive0}  does not depend on $L$. If $\zeta$ depends on $L$, especially $\zeta=\zeta_L\to0$ as $L\to \infty$, then a contradiction arises because the operators  $B_i^Q $ on $\R$ for $i=1,2$ are positive on $\cj V_{x_0,0}^{ \textup{Re}}$ and $\cj V_{x_0,0}^{ \textup{Im}}$, where $\cj V_{x_0,0}^{ \textup{Re}}$ and $\cj V_{x_0,0}^{ \textup{Im}}$ are interpreted in the case $\T_{\infty}=\R$.
\end{proof}

To define Gaussian measures corresponding to the Schrödinger operators
$B_1^Q$ and $B_2^Q$ in \eqref{SCHOP}, we first introduce the following projection operators.
Define $\mathbf{P}_{\cj V^{L^2, \text{Re}}_{x_0,0}}$ to be the orthogonal projector onto $\cj V_{x_0,0}^{L^2, \text{Re}}$ defined in \eqref{VRE} as follows 
\begin{equation}\label{eqn: project}
\begin{split}
\mathbf{P}_{\cj V^{L^2, \text{Re}}_{x_0,0}} &= \Id-\frac{\langle (\Ld- \partial_{x}^2) \partial_{x_0}Q_{x_0}^{\eta_L}, \cdot\,\rangle_{L^2}}{\|\partial_{x_0} Q_{x_0}^{\eta_L}\|^2_{    H^1} 
  }\partial_{x_0} Q_{x_0}^{\eta_L}-\frac{\langle Q^{\eta_L}_{x_0}, \cdot\,\rangle_{L^2}}{\| Q_{x_0}^{\eta_L}\|^2_{ L^2}} Q_{x_0}^{\eta_L}    \\
&= \Id- \langle Q^{\eta_L}_{1, x_0},\,\cdot\rangle_{L^2} Q^{\eta_L}_{1, x_0}-\langle  Q^{\eta_L}_{2, x_0},\, \cdot\rangle_{L^2}  Q^{\eta_L}_{2, x_0},
\end{split}
\end{equation}

\noi 
where
\begin{align}
Q^{\eta_L}_{1, x_0}&:=\frac{ \partial_{x_0}Q_{x_0}^{\eta_L}}{\|\partial_{x_0} Q_{x_0}^{\eta_L}\|^2_{ H^1_\Ld}  }  \label{TQF1}\\
Q^{\eta_L}_{2, x_0}&:= \frac{ Q_{x_0}^{\eta_L}}{\| Q_{x_0}^{\eta_L}\|_{L^2}^2 }. \label{TQF3}
\end{align}

\noi 
Define $\mathbf{P}_{\cj V^{L^2, \text{Im}}_{x_0,0}}$ to be the orthogonal projector onto $\cj V_{x_0,0}^{L^2, \text{Im}}$ defined in \eqref{VIM} as follows 
\begin{equation}
\begin{split}
\mathbf{P}_{\cj V^{L^2, \text{Im}}_{x_0,0}} &= \Id-\frac{\langle (\Ld-\partial_{x}^2) Q_{x_0}^{\eta_L}, \cdot\,\rangle_{L^2}}{\| Q_{x_0}^{\eta_L}\|^2_{ H^1_\Ld}  } Q_{x_0}^{\eta_L}= \Id- \langle Q^{\eta_L}_{3, x_0},\,\cdot\rangle_{L^2} Q^{\eta_L}_{3, x_0}, 
\end{split}
\end{equation}

\noi 
where
\begin{align}
Q^{\eta_L}_{3, x_0}&:= \frac{ Q_{x_0}^{\eta_L}}{\| Q_{x_0}^{\eta_L}\|^2_{ H^1_\Ld}  } \label{TQF2}.
\end{align}

\noi 
Finally, we define
\begin{align}
C_{1,x_0}^{Q}:&= \mathbf{P}_{\cj V^{L^2, \text{Re}}_{x_0,0}} (B_1^Q)^{-1}\mathbf{P}_{\cj V^{L^2, \text{Re}}_{x_0,0}} \notag \\
C_{2,x_0}^{Q}:&= \mathbf{P}_{\cj V^{L^2, \text{Im}}_{x_0,0}} (B_2^Q)^{-1}\mathbf{P}_{\cj V^{L^2, \text{Im}}_{x_0,0}}.
\label{CovC}
\end{align}

\noi 
In the following lemma  $C_{1,x_0}^{Q}$ and  $C_{2,x_0}^{Q}$ play a role 
as the covariance operators for the Gaussian processes $\Re  h^\perp$ and $\Im  h^\perp $, respectively.

\begin{lemma}\label{LEM:SCHOP}
For sufficiently large $L\ge 1$ and any $x_0 \in \T_{L^2}$, we can define an Ornstein Uhlenbeck type measure $ \nu^\perp_{Q, x_0}$ as follows
\begin{align}
\nu^\perp_{Q,x_0}= \nu_{C_{1,x_0}^Q}^{\perp} \otimes  \nu_{C_{2,x_0}^Q}^{\perp}, 
\label{GFFSch}
\end{align}

\noi
where 
\begin{align*}
d\nu_{C_{1,x_0}^Q}^{\perp}(\Re  h^\perp )&=Z_{L}^{-1} e^{-\frac 12 \jb{(C_{1,x_0}^Q)^{-1} \Re  h^\perp, \Re  h^\perp} } \prod_{x\in \T_{L^2}}\mathrm{d} \Re  h^\perp(x)\\
d\nu_{C_{2,x_0}^Q}^{\perp}(\Im  h^\perp)&=Z_{L}^{-1} e^{-\frac 12 \jb{(C_{2,x_0}^Q)^{-1} \Im  h^\perp, \Im  h^\perp} } \prod_{x\in \T_{L^2}}\mathrm{d} \Im  h^\perp(x).
\end{align*}

\noi
Here, $\nu_{C_{1,x_0}^Q}^{\perp}$, $\nu_{C_{2,x_0}^Q}^{\perp}$ have 
covariance operators $(B_1^Q)^{-1}$ and $(B_2^Q)^{-1}$, projected on $\cj V_{x_0,0}^{L^2, \textup{Re}}$ and $\cj V_{x_0,0}^{L^2, \textup{Im}}$, defined in \eqref{VRE} and \eqref{VIM}.
\end{lemma}

\begin{proof}
From Lemma \ref{LEM:positive}, $B_1^Q$ and $B_2^Q$ are positive operators on $\cj V_{x_0,0}^{L^2, \textup{Re}}$ and $\cj V_{x_0,0}^{L^2, \textup{Im}}$, respectively, for sufficiently large $L \ge 1$. Therefore, these operators are invertible on the spaces $\cj V_{x_0,0}^{L^2, \textup{Re}}$ and $\cj V_{x_0,0}^{L^2, \textup{Im}}$. 
Furthermore, by the elementary theory of Schr\"odinger operators, $(B_1^Q)^{-1}$ and $(B_2^Q)^{-1}$ are of trace class since $(Q^{\eta_L}_{x_0})^2$ is a Schwartz function. This implies that we can define the Gaussian measures $\nu_{C_{1,x_0}^Q}^{\perp}$ and  $\nu_{C_{2,x_0}^Q}^{\perp}$.

\end{proof}


\subsection{Conditional density function}

In this subsection, we express the conditional measure in \eqref{condssig} in terms of a conditional density.

We begin by recalling that if $X$ and $Y$ are jointly continuous random variables, the conditional probability $\PP\big\{ X\in A \;|\;  Y=s  \big\}$ can be expressed as
\begin{align}
\PP\big\{ X\in A \;|\;  s \le Y \le s+\dl  \big\}=\frac{ \int_A \int_{s\le y \le s+\dl } f_{X,Y}(x,y)\, \mathrm{d}y \, \mathrm{d}x   }{\int_\R \int_{s\le y \le s+\dl  } f_{X,Y}(x,y)\, \mathrm{d}y \, \mathrm{d}x   } \too \int_A f_{X|Y}(x|s)\,\mathrm{d}x
\label{CONDden}
\end{align}

\noi
as $\dl \to 0$, where $f_{X|Y}(x|s)=\frac{f_{X,Y}(x,s)}{\int_\R f_{X,Y}(x,s) \, \mathrm{d}x   }$. In the following discussion, $s$ may depend on $x$, that is, $s=s(x)$. The expression \eqref{CONDden}  shows that by integrating out $y$-component, $\dl^{-1}\int_{s(x) \le y \le s(x)+\dl }f_{X,Y}(x,y)\, \mathrm{d}y \approx f_{X,Y}(x,s(x))$, we can obtain the conditional density function.


In the following, we focus on the integral $\mathcal{I}(F)$ in \eqref{condssig}  without $\mathcal{K}^L_\dl$ and $\mathrm{Det}_L(h)$ which will be controlled separately later
\begin{align}
\int_{S_{Q,x_0}}  F(h)&e^{L^3\mathcal{E}(Q^{\eta_L}_{x_0}, L^{-\frac{3}{2}}h)} e^{\Ld L^{\frac 32} \jb{Q^{\eta_L}_{x_0}, \Re h }+\frac \Ld2\|h \|_{L^2}^2 }e^{\frac 12\jb{ (Q^{\eta_L}_{x_0})^2, 3|\Re h|^2+|\Im h|^2 }  } \,\mathrm{d}\mu_{x_0,\Ld}^\perp(h),
\label{R1}
\end{align}

\noi
where the conditioning $S_{Q,x_0}$, defined in \eqref{SQK}, is given by
\begin{align}
S_{Q,x_0}&=\big\{ -\frac 12 \| h\|_{L^2(\T_{L^2})}^2-L^{0+}<L^\frac 32 \jb{Q^{\eta_L}_{x_0}, \Re h  }<-\frac 12 \|h \|_{L^2(\T_{L^2})}^2    \big\}.
\label{SQK1}
\end{align}

\noi
By the law of large numbers for the Ornstein–Uhlenbeck-type measure 
$\mu^\perp_{Q,x_0}$ (and also $\nu^\perp_{Q,x_0}$ defined in Lemma \ref{LEM:SCHOP}), we have $\| h\|_{L^2(\T_{L^2})}^2 \approx L^2$. See Propsotion \ref{PROP: gaussian-conc}. Hence, the conditioning in \eqref{SQK1} can be expressed as $-L^2 \approx L^\frac 32 \jb{Q^{\eta_L}_{x_0}, \Re h  }$. In the following, by integrating out the $\jb{Q^{\eta_L}_{x_0}, \Re h  }$-component in the integral \eqref{R1}, we derive an expression for the conditional density function
as in \eqref{CONDden}.


\noi 
From Lemma \ref{LEM:decomh}, under the decomposition $h(x)=\gamma(x) \langle Q^{\eta_L}_{x_0}, \Re h\rangle +h^\perp(x)$ in \eqref{eqn: decomp}, $\langle Q^{\eta_L}_{x_0}, \Re h\rangle$ and $h^\perp$ are independent Gaussian random fields under the measure $\mu^\perp_{x_0,\Ld}$. Therefore, from \eqref{measuredecom0}, we write 
\begin{align}
\mathrm{d}\mu_{x_0, \Ld}^\perp(h)=\frac{1}{\sqrt{2\pi }\s }e^{-\frac{t^2}{2\s}}\, \mathrm{d}t \, \mathrm{d}\nu_{x_0}^\perp(h^\perp),
\label{OUperp}
\end{align}

\noi
where $\sigma^2=O(1)$, uniformly in $x_0 \in \T_{L^2}$ and $\dr \in [0,2\pi]$.  From the coordinate $h(x)=h^\perp(x)+t\g(x)$, where  $t$ is the symbol for the Gaussian random varible $\jb{Q^{\eta_L}_{x_0}, \Re h }$, we can write
\begin{align}
\| h\|_{L^2}^2&=\| h^\perp\|_{L^2}^2+2\jb{Q^{\eta_L}_{x_0}, \Re h }g+|\jb{Q^{\eta_L}_{x_0}, \Re h }|^2 \| \g\|_{L^2}^2 \notag \\
&=\| h^\perp\|_{L^2}^2+2tg+t^2 \| \g\|_{L^2}^2,
\label{R2}
\end{align}

\noi
where $g$ is a Gaussian random variable derived from  $\jb{h^\perp, 
\g}$. This also implies that the conditioning $S_{Q,x_0}$  in \eqref{SQK1} can be expressed as
\begin{align}
S_{Q,x_0}=\big\{ -\frac 12 \|h^\perp \|_{L^2}^2-L^{0+}< G_1(t)  <-\frac 12 \|h^\perp \|_{L^2}^2-L^{0+} \big\},
\label{COND3}
\end{align}

\noi
where $G_1(t)=\frac {1}{2}\| \g \|_{L^2}^2t^2 +(L^\frac 32+g)t$.

Using the coordinate $h(x)=h^\perp(x)+t\g(x)$,  along with \eqref{OUperp}, \eqref{R2}, \eqref{COND3}, and $\Im h=\Im h^\perp$, we can rewrite
\begin{align}
\eqref{R1}&=\int_{h^\perp} \int_{S_{h^\perp}}F(h^\perp, t)e^{\bar{\mathcal{E} }(h^\perp, t) }e^{\frac{\Ld}{2}\|h^\perp\|_{L^2}^2}e^{\Ld G_1(t)}e^{G_2(t,h^\perp)} e^{-\frac{t^2}{2\sigma^2}}\frac{\mathrm{d}t}{\sqrt{2\pi} \sigma} \notag \\
&\hphantom{XXXXXXXXX}\cdot e^{\frac 32  \int_{\T_{L^2}}(Q^{\eta_L}_{x_0} )^2  |\Re h^\perp|^2\, \mathrm{d}x } e^{\frac 12  \int_{\T_{L^2}}(Q^{\eta_L}_{x_0} )^2  |\Im h^\perp|^2\, \mathrm{d}x } \,\nu^\perp_{x_0}(\mathrm{d}h^\perp) \notag \\
&=Z_{x_0,L}\int_{h^\perp} \int_{S_{h^\perp}}F(h^\perp, t)e^{\bar{\mathcal{E} }(h^\perp, t) }e^{\frac{\Ld}{2}\|h^\perp\|_{L^2}^2}e^{\Ld G_1(t)}e^{G_2(t,h^\perp)} e^{-\frac{t^2}{2\sigma^2}}\frac{\mathrm{d}t}{\sqrt{2\pi} \sigma}\, \nu^\perp_{Q,x_0}(\mathrm{d}h^\perp),
\label{eqn: plug-here}
\end{align}

\noi
where 
\begin{align}
\bar{\mathcal{E} }(h^\perp, t)&=L^3 \mathcal{E}(Q^{\eta_L}_{x_0}, L^{-\frac 32}(h^\perp+t\g) ) \label{errht} \\
G_1(t)&=\frac{\|\g \|_{L^2}^2}{2} t^2+ (L^{\frac{3}{2}}+g)t \label{eqn: G-def}\\
G_2(t,h^\perp)&=\frac {3}{2}t^2 \int_{\T_{L^2} } (Q^{\eta_L}_{x_0}\g)^2\,\mathrm{d}x+3t\int_{\T_{L^2}} (Q^{\eta_L}_{x_0})^2 \g \Re h^\perp \, \mathrm{d}x,  
\label{eqn:G_2-def}\\
S_{h^\perp}:&=\Big\{t\in\mathbb{R}:-\frac{1}{2}\|h^\perp\|_{L^2}^2-L^{0+}<G_1(t)< -\frac{1}{2}\|h^\perp\|^2_{L^2}\Big\}.
\label{Shperp}
\end{align}

\noi
In \eqref{eqn: plug-here}, $\nu^\perp_{Q,x_0}$ is the Gaussian measure associated with Schr\"odinger operators, as defined in Lemma \ref{LEM:SCHOP}, and $Z_{x_0,L}$ is the partition function for the measure $\nu^\perp_{Q,x_0}$.

\begin{remark}\rm
The partition function $Z_{x_0, L}$ in \eqref{eqn: plug-here}  is independent of $x_0$ due to the translation invariance of the measure, that is, $\Law_{ \nu^\perp_{Q, x_0} }( h^\perp(\cdot+x_0))= \nu^\perp_{Q, 0}$ since
\begin{align*}
\E_{ \nu^\perp_{Q, x_0}}\big[ h^\perp(x+x_0)   h^\perp(y+x_0)  \big]=\E_{ \nu^\perp_{Q,0 }}\big[ h^\perp(x)  h^\perp(y)  \big].
\end{align*}

\noi
Therefore, when taking the average integral over both tangential components $x_0\in \T_{L^2}$ and $\dr \in [0,2\pi]$ in \eqref{REDUCT}, $Z_{x_0, L}$ cancels out with the corresponding factor in the denominator
$\cj Z_L$ in \eqref{REDUCT}.  Consequently, we will disregard it in the following discussion.
\end{remark}

In the following lemma, 
by integrating out the $t$-component, we represent the conditional expectation in \eqref{condssig} in terms of a conditional density under the set
\begin{equation}\label{eqn: DL-def}
D_L:=\big\{\|h^\perp\|^2_{L^2}\le DL^{2} \big\} \cap \{ |g|\le L^{\frac{3}{2}-\eps}\}
\end{equation}

\noi 
for some large $D \ge 1$. Under the measure $\nu^\perp_{Q,x_0}$, $D_L$  is a high-probability event due to the law of large numbers for the Ornstein–Uhlenbeck-type measure $\nu^\perp_{Q,x_0}$, and $g$ is the Gaussian random variable in \eqref{R2}. See Proposition \ref{PROP: gaussian-conc}.




\begin{lemma}\label{lem: replace}
Let $F$, $H$ be bounded functions that are Lipschitz in $t$, with the Lipschitz norm given by
\begin{equation}\label{eqn: F-lip}
\sup_{h^\perp}\|F(h^\perp,\cdot)\|_{\mathcal{L}}=\sup_{h^\perp}\sup_{t,t'\in \mathbb{R}}\frac{|F(h^\perp,t)-F(h^\perp, t')|}{|t-t'|}.
\end{equation}

\noi 
Here,  the Lipschitz norms $\sup\limits_{h^\perp}\| F\|_{\mathcal{L}}$ and $\sup\limits_{h^\perp}\|H \|_{\mathcal{L}}$ are allowed to be at most $L^{\frac 32-\eps}$. Then, on the set $D_L$, we can integrae out the $t$-component   as follows 
\begin{align*}
&\int_{S_{h^\perp}}F(h^\perp,t^{+} )\, e^{H (h^\perp, t^{+}) } e^{\frac{\Ld}{2}\|h^\perp\|_{L^2}^2}e^{\Ld G_1(t)}e^{G_2(t,h^\perp)}e^{-\frac{t^2}{2\sigma^2}}\frac{\mathrm{d}t}{\sqrt{2\pi} \sigma}\\
=&~ F(h^\perp, t^+) \frac{1}{L^{\frac{3}{2}}} \frac{ 
e^{H (h^\perp, t) }   e^{c_0(t^+)^2+b(h^\perp)t^+}}{\sqrt{2\pi}\s}(1+O(L^{-\eps}))
\end{align*}


\noi
as $L \to \infty$, where 
\begin{align}
G_1(t^+)&=-\frac 12 \| h^\perp \|_{L^2}^2\\
b:&=b(h^\perp)=3\int_{\T_{L^2}}(Q^{\eta_L}_{x_0})^2\g(x)h^\perp(x)\, \mathrm{d}x \label{bhper}\\
c_0:&=\frac 32\int_{\T_{L^2}}(Q^{\eta_L}_{x_0}\g)^2(x) \, \mathrm{d}x-\frac 1{2\s^2}.
\label{c0}
\end{align}

\end{lemma}

\begin{proof}[Proof of Lemma \ref{lem: replace}]

On the set $D_L$ in \eqref{eqn: DL-def}, the inequality from $S_{h^\perp}$ in \eqref{Shperp}
\[-\frac{1}{2}\|h^\perp\|_{L^2}^2-L^{0+} \le G_1(t)\le -\frac{1}{2}\|h^\perp\|^2_{L^2}\]
implies
\begin{align}
-DL^2\cdot (1-o_L(1)) \le G_1(t)\le 0,
\label{DL1}
\end{align}

\noi
as $L\to \infty$. Recall the definition of $G_1(t)$  in \eqref{eqn: G-def}
\begin{align}
G_1(t)=\frac{\|\g \|_{L^2}^2}{2} t^2+ (L^{\frac{3}{2}}+g)t.
\label{DL2}
\end{align}

\noi
Hence, combining \eqref{DL1} and \eqref{DL2} shows that $t<0$ with magnitude given by
\begin{equation}
t=O(L^{\frac{1}{2}}).
\label{tL12}
\end{equation}

\noi 
Note that  on the interval
\[\left[-(L^{\frac{3}{2}}+2g)/(2\| \g\|_{L^2}^2),0\right],\]

\noi 
$t\mapsto G_1(t)$ is a negative, strictly increasing function, with derivative
\begin{equation}\label{eqn: derivative-lb}
G'_1(t)=\| \g\|_{L^2}^2t+L^{\frac{3}{2}}+g\ge \frac{1}{2}L^{\frac{3}{2}},
\end{equation}

\noi
since $t=O(L^\frac 12)$ and $|g|\le L^{\frac 32-\eps}$ on the set $D_L$.
Let $t^+$ be such that
\[G_1(t^+)=-\frac{1}{2}\|h^\perp\|_{L^2}^2\]

\noi 
and $t^-$ such that
\begin{align*}
G_1(t^-)=-\frac{1}{2}\|h^\perp\|_{L^2}^2-L^{0+}.
\end{align*}

\noi 
Then, we have
\begin{equation}\label{eqn: interval-map}
\begin{split}
S_{h^\perp}&=\big\{-\frac{1}{2}\|h^\perp\|_{L^2}^2-L^{0+} \le G_1(t)\le -\frac{1}{2}\|h^\perp\|_{L^2}^2\}\\
&=\{t^-\le t\le t^+\big\}.
\end{split}
\end{equation}

\noi 
From \eqref{eqn: derivative-lb} and Taylor expansion, we have
\begin{align}
t^+-t^-=L^{0+}/L^{\frac{3}{2}}.
\label{sizetint}
\end{align}

\noi 
Solving the quadratic equation $G_1(t^{+})$ for $t^+$, we find\footnote{The other solution to the quadratic is not $O(L^{1/2})$.}
\begin{equation}\label{eqn: t-solve}
\begin{split}
t^+=&\frac{-(L^{\frac{3}{2}}+2g)+ \sqrt{(L^{\frac{3}{2}}+2g)^2-\frac{\| \g\|_{L^2}^2}{2}\|h^\perp\|^2_{L^2}}}{\|\g \|^2_{L^2}}\\
=&-\frac{1}{2}\frac{1}{L^{\frac{3}{2}}+2g}\|h^\perp\|^2_{L^2}+\frac{3a^2}{32}\frac{\|h^\perp\|^4_{L^2}}{(L^{\frac{3}{2}}+2g)^3}+O\bigg(\frac{\|h^\perp\|^6_{L^2}}{(L^{\frac{3}{2}}+2g)^5}\bigg).
\end{split}
\end{equation}

\noi 
For notational convenience, we set $t^+=0$ if 
\[\|h^\perp\|_{L^2}^2 > \frac{2}{\|\g \|^2_{L^2}}(L^{\frac{3}{2}}+2g)^2.\]

\noi
On the set $D_L$, from \eqref{eqn: t-solve}, we thus have
\begin{equation}\label{eqn: t-oot}
t^+=-\frac{L^{\frac{1}{2}}}{2}\frac{\|h^\perp\|^2_{L^2}}{L^2} (1+O(L^{-\varepsilon}))+\frac{3}{32}\frac{\|\g \|^2_{L^2} }{L^{\frac{1}{2}}}\frac{\|h^\perp\|^4_{L^2}}{L^4} (1+O(L^{-\varepsilon}))+O(L^{-\frac{3}{2}})
\end{equation}

\noi
and correspondingly, since $\|h^\perp \|_{L^2}^2=O(L^2)$ on $D_L$,
\begin{align}
t^{+}=O(L^\frac 12).
\label{toot}
\end{align}

\noi 
Using \eqref{eqn: interval-map}, we write 
\begin{align}
&\int_{S_{h^\perp}}F(h^\perp,t )\, e^{H(h^\perp, t) }e^{\frac{\Ld}{2}\|h^\perp\|_{L^2}^2}e^{\Ld G_1(t)}e^{G_2(t,h^\perp)}e^{-\frac{t^2}{2\sigma^2}}\frac{\mathrm{d}t}{\sqrt{2\pi} \sigma} \notag \\
&= \int_{t^-}^{t^+} F(h^\perp,t )\, e^{H(h^\perp, t) }e^{\frac{\Ld}{2}\|h^\perp\|_{L^2}^2}e^{\Ld G_1(t)}e^{G_2(t,h^\perp)}e^{-\frac{t^2}{2\sigma^2}}\frac{\mathrm{d}t}{\sqrt{2\pi} \sigma}.
\label{REP1}
\end{align}

\noi 
Thanks to the Lipschitz assumption on $F$ and $H$, for $t\in[t^-,t^+]$ with $t^{+}-t^{-}=O(L^{-\frac 32+ })$ from \eqref{sizetint} we have
\begin{align}
|F(h^\perp,t)-F(h^\perp,t^{+})|&\le \sup_{h^\perp}\|F\|_{\mathcal{L}}L^{-\frac{3}{2}+}\le L^{-\eps} \label{REP2}\\
|G_2(h^\perp, t) - G_2(h^\perp, t^{+}) |&\le \sup_{h^\perp }\| G_2 \|_{\mathcal{L}} L^{-\frac 32+}\le L^{-\eps}
\label{REP22}\\
|H(h^\perp, t) - H(h^\perp, t^{+}) |&\le \sup_{h^\perp}\| H\|_{\mathcal{L}} L^{-\frac 32+}\le L^{-\eps},
\notag 
\end{align}

\noi
which implies
\begin{align}
|e^{H(h^\perp, t) }-e^{H(h^\perp, t^{+}) }| &=e^{H(h^\perp, t^{+}) }|1-e^{H(h^\perp, t)  -H(h^\perp, t^{+})  }| \notag \\
&= e^{H(h^\perp, t^{+}) }(1-O(L^{-\eps})).
\label{REP3}
\end{align}

\noi
Furthermore, we have 
\begin{align}
e^{-\frac{t^2}{2\sigma^2}}&=e^{-\frac{(t^+)^2}{2\sigma^2}}e^{|t^+|O(L^{-\frac{3}{2}+})} \notag \\
&=e^{-\frac{(t^+)^2}{2\sigma^2}}(1+L^{-1+}),
\label{REP4}
\end{align}
and
\begin{align}
e^{\frac{\|\g \|_{L^2}^2}{2}t^2}=e^{\frac{\|\g \|_{L^2}^2}{2}(t^+)^2}(1+O(L^{-1+})).
\label{REP5}
\end{align}

\noi
Hence, by combining \eqref{REP1}, \eqref{REP2}, \eqref{REP22}, \eqref{REP3}, and \eqref{REP4}, we obtain 
\begin{align*}
&\int_{S_{h^\perp}}F(h^\perp,t )\, e^{H(h^\perp, t) }e^{\frac{\Ld}{2}\|h^\perp\|_{L^2}^2}e^{\Ld G_1(t)}e^{G_2(t,h^\perp)}e^{-\frac{t^2}{2\sigma^2}}\frac{\mathrm{d}t}{\sqrt{2\pi} \sigma} \\
&= F(h^\perp ,t^{+})    e^{H(h^\perp, t^{+} ) }e^{G_2(t^{+}, h^\perp ) } e^{-\frac{(t^{+})^2 }{2\s^2}} e^{\frac{|\Lambda|}{2}\|h^\perp\|_{L^2}^2} 
\cdot\int_{t^-}^{t^+}  e^{\Ld G_1(t) }\,\frac{\mathrm{d}t}{\sqrt{2\pi}\sigma}(1+o_L(1))
\end{align*}

By using \eqref{REP5} and recalling the definitions of $G_1$, $G_2$, $c_0$, and $b(h^\perp)$ from  \eqref{eqn: G-def}, \eqref{eqn:G_2-def}, \eqref{c0}, and \eqref{bhper}, we have 
\begin{align}
&=F(h^\perp,t^+)e^{H(h^\perp, t) }e^{c_0 (t^+)^2 +b(h^\perp)t^{+}} e^{\frac{|\Lambda|}{2}\|h^\perp\|_{L^2}^2} e^{\frac{|\Lambda| \| \g\|^2_{L^2} }{2} (t^+)^2  } \int_{t^-}^{t^+}e^{|\Lambda|(L^{\frac{3}{2}}+g)t}\frac{\mathrm{d}t}{\sqrt{2\pi}\sigma}(1+o_L(1)) \notag \\
&=\frac{1}{\sqrt{2\pi}\sigma}  \frac{1}{|\Lambda|(L^{\frac{3}{2}}+g)}
F(h^\perp,t^+)e^{H(h^\perp, t) }e^{c_0 (t^+)^2 +b(h^\perp)t^{+}}(1+o_L(1)),
\end{align}

\noi
where in the last line we used $G(t^{+})=-\frac 12 \|h^\perp \|_{L^2}^2$ and  $t^{+}-t^{-}=O(L^{-\frac 32+ })$ from \eqref{sizetint}.
This completes the proof of Lemma \ref{lem: replace}.

\end{proof}

To show that the term $\bar{\mathcal{E} }(h^\perp, t)$ in \eqref{errht} satisfies the Lipschitz assumption in Lemma \ref{lem: replace}, we define the set
\begin{equation}\label{eqn: BK-def}
B_K:=\{\|h^\perp\|_{L^\infty(\T_{L^2})}\le K\sqrt{\log L^2}\}.
\end{equation} 

\noi
Here, for sufficiently large $K \ge 1$, $B_K$  is a high-probability event as $L\to \infty$ under the Ornstein–Uhlenbeck-type measure $\nu^\perp_{Q,x_0}$. See Proposition \ref{prop: BK-prob}. In particular, $K=M\sqrt{ \frac{L}{\log L^2} } $ is chosen from Proposition \eqref{PROP:error1} and so $\| h^\perp\|_{L^\infty} \le M L^\frac 12 $ on the set $B_K$. See Remark \ref{REM:choiK} for an explanation of the choice of $K$.

\begin{lemma}\label{LEM:EpsLip}
On the set $D_L \cap B_K$, $\bar{\mathcal{E} }(h^\perp, t)$ is 
Lipschitz in $t$, with the Lipschitz norm   given in \eqref{eqn: F-lip}
\begin{equation*}
\sup_{h^\perp}\|\bar{\mathcal{E} }(h^\perp,\cdot)\|_{\mathcal{L}}=\sup_{h^\perp}\sup_{t,t'\in \mathbb{R}}\frac{|\bar{\mathcal{E} }(h^\perp,t)-\bar{\mathcal{E} }(h^\perp, t')|}{|t-t'|}.
\end{equation*}

\noi 
In particular, the Lipschitz norm satisfies $\sup\limits_{h^\perp}\| \bar{\mathcal{E}}\|_{\mathcal{L}}=O(L^{\frac 32+}) $. 
\end{lemma}

\begin{proof}

Recall that 
$\bar{\mathcal{E} }(h^\perp, t)=L^3 \mathcal{E}(Q^{\eta_L}_{x_0}, L^{-\frac 32}(h^\perp+t\g) )$ from \eqref{errht} and \eqref{Eerror}. Hence, it suffices to consider the cubic and quartic integrals in \eqref{Eerror}
in terms of the coordinate $(h^\perp, t)$.

Based on the decomposition $h(x)=h^\perp(x)+\g(x)t$ in  Lemma \ref{LEM:decomh}, where 
$t$ denotes the Gaussian random variable $\jb{Q^{\eta_L}_{x_0}, \Re h }$, we expand 
\begin{align}
\frac{1}{L^{\frac{3}{2}}}\int_{\T_{L^2}} Q^{\eta_L}_{x_0} |h|^2 \Re h = \sum_{u_i\in \{h^\perp, t \gamma\}} \frac{1}{L^{\frac{3}{2}}}\int_{\T_{L^2}} Q^{\eta_L}_{x_0} u_1 \bar{u}_2 \Re u_3\,\mathrm{d}x.
\label{EXPA}
\end{align}

\noi 
To check the Lipschitz assumption, it suffices to have a derivative bound.
The cubic term in $t$ in \eqref{EXPA} is
\begin{align*}
\frac{t^3}{L^{\frac{3}{2}}}\int_{\T_{L^2}} Q^{\eta_L}_{x_0}\gamma^3\,\mathrm{d}x= O\Big(\frac{t^3}{L^\frac 32}\Big).
\end{align*}

\noi 
The derivative in $t$ is $ O\Big(\frac{t^2}{L^\frac 32}\Big)=O(L^{-\frac 12 })$ since $t=O(L^\frac 12)$ on the set $D_L$ from \eqref{tL12}.  The quadratic term in $t$ in \eqref{EXPA}  is 
\begin{align*}
\frac{t^2}{L^\frac{3}{2}}\int_{\T_{L^2}} Q^{\eta_L}_{x_0}\gamma^2 |h^\perp|\,\mathrm{d}x= O\Big(\frac{t^2}{L}\Big).
\end{align*}

\noi
since  $\|  h^\perp\|_{L^\infty} \le M^2 \sqrt{L}$.
The derivative in $t$ is $ O\Big(\frac{t}{L}\Big)=O(L^{-\frac 12 })$
since $t=O(L^\frac 12)$ on the set $D_L$ from \eqref{tL12}.  The  linear term in $t$ in \eqref{EXPA} is 
\begin{align*}
\frac{t}{L^\frac 32} \int_{\T_{L^2}} Q^{\eta_L}_{x_0} \g |h^\perp|^2 \, \mathrm{d}x= O\Big(\frac{t}{L^\frac 12}\Big)
\end{align*}

\noi 
since  $\|  h^\perp\|_{L^\infty} \le M^2 \sqrt{L}$ on the set $B_K$.
The derivative in $t$ is $O(L^{-\frac 12})$. Therefore, the cubic integral in \eqref{EXPA} satisfies the Lipschitz assumption.

For the quartic integral, we expand 
\begin{align}
\frac{1}{L^3}\int_{\T_{L^2} } |h |^4 = \sum_{u_i\in \{h^\perp, t \gamma\}} \frac{1}{L^3}\int u_1 \bar{u}_2 u_3 \bar{u}_4\,\mathrm{d}x.
\label{EXPA1}
\end{align}

The quartic term in $t$ in \eqref{EXPA1} is
\begin{align*}
\frac{t^4}{L^4 }\int_{\T_{L^2}} \gamma^4\,\mathrm{d}x= O\Big(\frac{t^4}{L^3}\Big).
\end{align*}

\noi 
The derivative in $t$ is $ O\Big(\frac{t^3}{L^3}\Big)=O(L^{-\frac 32 })$ since $t=O(L^\frac 12)$ on the set $D_L$ from \eqref{tL12}. The cubic term in $t$ in \eqref{EXPA1} is
\begin{align*}
\frac{t^3}{L^3 }\int_{\T_{L^2}} h^\perp \gamma^3\,\mathrm{d}x= O\Big(\frac{t^3}{L^{\frac 52}}\Big).
\end{align*}

\noi 
since  $\|  h^\perp\|_{L^\infty} \le M^2 \sqrt{L}$ on the set $B_K$.
The derivative in $t$ is $ O\Big(\frac{t^2}{L^\frac 52}\Big)=O(L^{-\frac 32 })$ since $t=O(L^\frac 12)$ on the set $D_L$ from \eqref{tL12}. 
The quadratic term in $t$ in \eqref{EXPA1} is 
\begin{align*}
\frac{t^2}{L^3 }\int_{\T_{L^2}} |h^\perp|^2 \gamma^2\,\mathrm{d}x= O\Big(\frac{t^2}{L^{2}}\Big).
\end{align*}

\noi
since  $\|  h^\perp\|_{L^\infty} \le M^2 \sqrt{L}$ on the set $B_K$.
The derivative in $t$ is $ O\Big(\frac{t}{L^2}\Big)=O(L^{-\frac 32 })$ since $t=O(L^\frac 12)$ on the set $D_L$ from \eqref{tL12}. 
The linear term in $t$ in \eqref{EXPA1} is 
\begin{align*}
\frac{t}{L^3 }\int_{\T_{L^2}} |h^\perp|^3 \gamma\,\mathrm{d}x= O\Big(\frac{t}{L^{\frac 32}}\Big).
\end{align*}

\noi
The derivative in $t$ is $ O(L^{-\frac 32 })$.  Therefore, the quartic integral in \eqref{EXPA} satisfies the Lipschitz assumption.

\end{proof}

\section{Representation using the conditional density function}\label{SEC:REP}
In this section, we use the conditional density function in Lemma \ref{lem: replace} to study the conditional expectation  in \eqref{condssig}
\begin{align*}
\E_{\s_{x_0,\dr}^\perp }\bigg[F(h) \Big|  \jb{Q^{\eta_L}_{x_0}, \Re h }\approx -\frac 1{2L^\frac 32} \| h\|_{L^2}^2 \bigg]=\frac{\mathcal{I}(F) }{\mathcal{I}(1) }.
\end{align*}

\noi
Before we provide the precise expression using the density function (see Proposition \ref{PROP:main}), we decompose the main term and the error terms as follows
\begin{align}
\frac{\mathcal{I}(F)}{\mathcal{I}(1)}=\frac{\mathcal{I}(F,D_L\cap B_K)}{\mathcal{I}(1)} +\frac{\mathcal{I}(F,D_L\cap B_K^c)}{\mathcal{I}(1)} +\frac{\mathcal{I}(F,D_L^c)}{\mathcal{I}(1)}.
\label{sepmaer}
\end{align}

\noi
where $D_L$, $B_K$ are defined in \eqref{eqn: DL-def} and \eqref{eqn: BK-def}. Here,  $\mathcal{I}(F,D_L\cap B_K^c)$, $\mathcal{I}(F,D_L^c)$ are error terms as $L\to \infty$, while $\mathcal{I}(F,D_L\cap B_K)$ is the main contribution, yielding white noise in Theorem \ref{THM:1}.

We begin with the error estimates in the next subsection.

\subsection{Tail estimates for probability measures}
In this subsection, we provide tail estimates for the Gaussian measure
$ \nu_{Q,x_0}^\perp$ associated with Schr\"odinger operators in Lemma \ref{LEM:SCHOP}.

\begin{proposition}\label{prop: BK-prob}
There are constants $C$, $c>0$ such that, for $K$ large enough, we have 
\begin{align}
\nu^\perp_{Q, x_0}(B_K^c) \le C e^{-cK^2 \log L^2},
\label{Tail}
\end{align}

\noi
uniformly in $x_0 \in \T_{L^2}$ and $\dr \in [0,2\pi]$, where $B_K=\{\|h^\perp \|_{L^\infty(\T_{L^2})  } \le K \sqrt{\log L^2}  \}$ was introduced in \eqref{eqn: BK-def}.
\end{proposition}

\begin{proof}
Since $\nu^\perp_{Q,x_0}$ is the Gaussian measure in Lemma \ref{LEM:SCHOP},  we have 
\begin{align}
\E_{ \nu^\perp_{Q,x_0}} \big[ | h^\perp(x)- h^\perp(y)|^p \big] 
\le p^{\frac p2} \Big( \E_{\nu^\perp_{Q,x_0}} \big[ | h^\perp(x)- h^\perp(y)|^2 \big] \Big)^{\frac p2}.
\label{S0}
\end{align}

\noi
We expand
\begin{align}
\E_{\nu^\perp_{Q,x_0}} \big[ |h^\perp(x)-h^\perp(y)|^2 \big]&=\big(\E_{\nu^{\perp}_{Q, x_0}} \big[ h^\perp(x)  \cj{h^\perp}(x) \big]-\E_{\nu^{\perp}_{Q, x_0}} \big[ h^\perp(x)  \cj{h^\perp}(y) \big] \big) \notag \\
&\hphantom{X}+\big(\E_{\nu^{\perp}_{Q, x_0}} \big[ h^\perp(y)  \cj{h^\perp}(y) \big]-\E_{\nu^{\perp}_{Q, x_0}} \big[ h^\perp(y)  \cj{h^\perp}(x) \big] \big) \notag \\
&=\J_1+\J_2.
\label{CJ00}
\end{align}

\noi
Using the independence property $\nu^\perp_{Q,x_0}= \nu_{C^Q_{1,x_0} }^{\perp} \otimes  \nu_{C^Q_{2,x_0} }^{\perp}$ for the real and imaginary parts,  we have 
\begin{align}
\J_1&=\big(\E_{ \nu^{\perp}_{C^Q_{1,x_0} }} \big[ \Re  h^\perp(x)  \Re  h^\perp(x) \big]-\E_{ \nu^{\perp}_{C^Q_{1,x_0} } } \big[ \Re  h^\perp(x) \Re  h^\perp(y) \big] \big) \notag \\
&\hphantom{X}+\big(\E_{ \nu^{\perp}_{C^Q_{2,x_0} }  } \big[ \Im  h^\perp(x)  \Im  h^\perp(x) \big]-\E_{ \nu^{\perp}_{C^Q_{2,x_0} }} \big[ \Im  h^\perp(x) \Im  h^\perp(y) \big] \big) \notag \\
&= (G_{1,x_0}(x,x)-G_{1,x_0}(x,y) \big)+\big (G_{2,x_0}(x,x)-G_{2,x_0}(x,y) \big)
\label{CJ01}
\end{align}

\noi
and
\begin{align}
\J_2&=\big(\E_{ \nu^{\perp}_{C^Q_{1,x_0} }  } \big[ \Re  h^\perp(y) \Re  h^\perp(y) \big]-\E_{ \nu^{\perp}_{C^Q_{1,x_0} }} \big[ \Re  h^\perp(y) \Re  h^\perp(x) \big] \big) \notag \\
&\hphantom{X}+\big(\E_{ \nu^{\perp}_{C^Q_{2,x_0} }} \big[ \Im  h^\perp(y) \Im  h^\perp(y) \big]-\E_{ \nu^{\perp}_{C^Q_{2,x_0} }} \big[ \Im  h^\perp(y) \Im  h^\perp(x) \big] \big) \notag \\
&=(G_{1,x_0}(y,y)-G_{1,x_0}(y,x) \big)+(G_{2,x_0}(y,y)-G_{2,x_0}(y,x) \big),
\label{CJ02}
\end{align}

\noi
where $G_{i,x_0}(x,y)$ is the Green's function associated with the measure $ \nu_{C_i^Q,x_0}^{\perp}$,  defined in Lemma \ref{lem: covariance bound}. Combining \eqref{CJ00}, \eqref{CJ01}, \eqref{CJ02}, and Proposition \ref{PROP:regular} yields 
\begin{align}
\E_{ \nu^{\perp}_{Q, x_0}} \big[ | h^\perp(x)- h^\perp(y)|^2 \big] \le C|x-y|,
\label{S2}
\end{align}

\noi
uniformly in $x_0\in \T_{L^2}$ and $\dr \in [0,2\pi]$. From \eqref{S0} and \eqref{S2}, there exists a constant $C>0$ such that for any $p \ge 2$, we have 
\begin{align}
\E_{ \nu^{\perp}_{Q, x_0}} \big[ | h^\perp(x)- h^\perp(y)|^p \big] \le C p^{\frac p2}|x-y|^{\frac p2},
\label{S3}
\end{align}

\noi
uniformly in $x_0 \in \T_{L^2}$ and $\dr \in [0,2\pi]$.

Given $x\in [0,1]$ and $n\in \N$, we set $x_n:=2^{-n } \lfloor 2^n x \rfloor=\max \big\{\frac{k}{2^n}: \frac{k}{2^n} \le x  \big\} $, where $\big\{\frac{k}{2^n}: 0\le k \le 2^n\}$ represents equally $2^{-n}$-spaced points on $[0,1]$. Furthermore, we have $|x_n-x_{n-1}|\le 2^{-n}$. Then, for any $x\in [0,1]$, we have
\begin{align}
| h^\perp(x)|\le | h^\perp(0)|+\sum_{n=1}^\infty | h^\perp(x_n)- h^\perp(x_{n-1})|,
\label{S4}
\end{align}

\noi
where $x_0=0$. Note that
\begin{align}
\max_{x\in [0,1]}| h^\perp(x_n)- h^\perp(x_{n-1})| &\le \max_{1 \le \ell \le 2^n}\big| h^\perp(\ell 2^{-n})- h^\perp( (\ell-1)2^{-n} )\big|\notag \\
&\le  \bigg(\sum_{\ell=1}^{2^n}  \big| h^\perp(\ell 2^{-n})- h^\perp( (\ell-1)2^{-n} )\big|^p \bigg)^{\frac 1p}
\label{S5}
\end{align}

\noi 
Using \eqref{S4} and \eqref{S5}, along with the inequality $(a+b)^p \le 2^{p-1}(a^p+b^p)$ for $a,b>0$, we have 
\begin{align}
&\E_{ \nu^{\perp}_{Q, x_0}} \Big[ \max_{x\in [0,1]} | h^\perp(x)|^p \Big]^{\frac 1p} \notag \\
&\le 2^{1-\frac 1p} \E_{ \nu^{\perp}_{Q, x_0}} \big[ | h^\perp(0)|^p \big]^{\frac 1p}+2^{1-\frac 1p} \Bigg( \E_{ \nu^{\perp}_{Q, x_0}}\Bigg[ \bigg|\sum_{n=1}^\infty \bigg(\sum_{\ell=1}^{2^n}  \big| h^\perp(\ell 2^{-n})- h^\perp( (\ell-1)2^{-n} )\big|^p \bigg)^{\frac 1p}  \bigg|^p \Bigg] \Bigg)^{\frac 1p}\notag \\
&=\I_1+\I_2.
\label{S6}
\end{align}

\noi
From the Green function estimate in Lemma \ref{lem: covariance bound}, we have
\begin{align}
\I_1 \le  2^{1-\frac 1p}p^{\frac 12}\E_{ \nu^\perp_{Q, x_0}}\big[| h^\perp(0)|^2 \big]^\frac{1}{2}    \le  cp^{\frac 12},
\label{S7}
\end{align}

\noi
uniformly in $x_0 \in \T_{L^2}$ and $\dr \in [0,2\pi]$.
Regarding $\I_2$, from the Minkowski's inequality and \eqref{S3}, we obtain that for any $p > 2$
\begin{align}
\I_2 &\le \sum_{n=1}^\infty  \bigg( \sum_{\ell=1}^{2^n}  \E_{ \nu^{\perp}_{Q, x_0}}\Big[ \big| h^\perp(\ell 2^{-n})- h^\perp( (\ell-1)2^{-n} )\big|^p \Big]  \bigg)^{\frac 1p} \notag \\
&\le \sum_{n=1}^\infty \Big( 2^{n} C p^{\frac p2} 2^{-\frac{np}{2}} \Big)^{\frac 1p}\notag \\
&\le (Cp^{\frac p2})^{\frac 1p}\sum_{n=1}^\infty 2^{-n(\frac 12-\frac 1p)} \le (Cp^{\frac p2})^{\frac 1p},
\label{S8}
\end{align}

\noi
uniformly in $x_0 \in \T_{L^2}$ and $\dr \in [0,2\pi]$. Therefore, it follows form \eqref{S6}, \eqref{S7}, and \eqref{S8} that
\begin{align}
 \nu^{\perp}_{Q, x_0}\bigg(\Big\{ \max_{x\in [0,1]} | h^\perp(x)| \ge \rho  \Big\} \bigg) \le \rho^{-p} \E_{ \nu^{\perp}_{Q, x_0}} \Big[  \max_{x\in [0,1]} | h^\perp(x)|^p \Big] \le C \rho^{-p} p^{\frac p2},
\label{S9}
\end{align}

\noi
uniformly in $x_0 \in \T_{L^2}$ and $\dr \in [0,2\pi]$. Since the right hand side of \eqref{S9} holds for any $p > 2$, we can optimize with respect to $p$ to obtain
\begin{align}
 \nu^{\perp}_{Q, x_0}\bigg(\Big\{ \max_{x\in [0,1]} | h^\perp(x)| \ge \rho  \Big\} \bigg)  \le Ce^{-c \rho^2},
\label{S10}
\end{align}

\noi
uniformly in $x_0 \in \T_{L^2}$ and $\dr \in [0,2\pi]$. Using the union bound and \eqref{S10}, we obtain
\begin{align}
\nu^{\perp}_{Q, x_0}\bigg(\Big\{ \max_{x\in [-L^2,L^2]} | h^\perp(x)| \ge K \sqrt{\log L^2}  \Big\} \bigg) &\le 2C L^2 e^{-cK^2 \log L^2} \notag \\
&\le 2C e^{-\frac{c}{2} K^2 \log L^2 }
\end{align}

\noi
for sufficiently large $K \ge 1$. Hence, we obtain the desired result.

\end{proof}

\begin{remark}\rm
The arguments in \eqref{S2} and \eqref{S6} also show that there exists a constant $C>0$ such that 
\begin{align}
\E_{\nu^\perp_{Q, x_0}} \Big[ \| h^\perp \|^p_{L^\infty([-L^2,L^2]) } \Big]\le C p^{\frac p2} (\log L^2)^\frac{p}{2},
\label{LinfL}
\end{align}

\noi
uniformly in $x_0\in \T_{L^2}$ and $\dr \in [0,2\pi]$. To obtain the bound \eqref{LinfL}, set
\[a_n:=\sup_{x\in [n-1,n]}| h^\perp(x)|.\]

\noi 
Then, 
\[M_{L}:=\sup_{x\in [-L^2,L^2]}| h^\perp(x)|= \max_{n\in \{-L^2+1,\dots,L^2-1, L^2\} } a_n. \]

\noi
It follows from Jensen's inequality and \eqref{S9} that for $q>0$:
\begin{align*}
\mathbb{E}_{\nu^\perp_{Q, x_0}}\big[M_L^p\big]&\le \mathbb{E}_{\nu^\perp_{Q, x_0}}\Big[\big(\sum_n a_n^{pq}\big)^\frac{1}{q}\Big] \\ 
&\le \Big(\mathbb{E}_{\nu^\perp_{Q, x_0}}\Big[\sum_n a_n^{pq}\Big]\Big)^{\frac{1}{q}}\\
&\le C^p(pq)^{\frac{p}{2}}L^\frac{2}{q}.
\end{align*}

\noi 
Letting $q=\log L$ and using
\[L^{1/(\log L)}=e,\]

\noi
we have
\[\E_{ \nu^\perp_{Q, x_0}} \Big[ \|  h^\perp \|^p_{L^\infty([-L^2,L^2]) } \Big]\le C p^{\frac p2} (\log L^2)^\frac{p}{2}.\]

\noi
Hence, we obtain \eqref{LinfL}.

\end{remark}

\begin{proposition}\label{PROP: gaussian-conc}
There exists a constant $c>0$ such that for sufficiently small $\dl >0$
\begin{align}
\E_{\nu_{Q, x_0}^\perp }\Big[ e^{\dl \int_{-L^2}^{L^2} |h^\perp|^2 \,\mathrm{d}x } \Big] \les e^{c L^2},
\label{DLC}
\end{align}

\noi
uniformly in $x_0 \in \T_{L^2}$. In particular, we have the probability estimate
\begin{align*}
\nu^{\perp}_{Q, x_0}(D_L^c) \les  e^{- c D L^2},
\end{align*}

\noi
uniformly in $x_0 \in \T_{L^2}$, where $D_L=\big\{ \| h^\perp \|_{L^2(\T_{L^2})}^2 \le DL^2 \big\}$ and $D$ is chosen to be sufficiently large. 

\end{proposition}

\begin{proof}

In the following, we apply the Boué-Dupuis variational formula introduced in \eqref{DUP}. Let $W(t)$ be the cylindrical Wiener process  with respect to the underlying probability measure $\PP$, that is, 
\begin{align}
W(t)=(W_1(t), W_2(t)) = \bigg( \sum_{n \in \Z } B_{n,1} (t) e^{L}_{n,1}(x), \sum_{n \in \Z } B_{n,2} (t) e^{L}_{n,2}(x) \bigg),
\end{align}

\noi
where $\{ B_{n,1} \}_{n \in \Z}$ and $\{ B_{n,2} \}_{n \in \Z}$   are two sequences of mutually independent complex-valued Brownian motions, and $\{e^L_{n,1}\}_{n \in \Z}$ and $\{e^L_{n,2}\}_{n \in \Z}$ are sequences of normalized eigenfunctions of the operators $(C^Q_{1,x_0})^{-1}$ and $(C^Q_{2,x_0})^{-1}$, defined in \eqref{CovC}. 
We  then define a centered Gaussian process $Y(t)$
by 
\begin{align*}
Y(t)=  \big( (C_{1,x_0}^{Q})^\frac 12 W_1(t), (C_{2,x_0}^{Q})^\frac 12  W_2(t)  \big),
\end{align*}

\noi
where $C_{i,x_0}^{Q}$ is the covariance operator for
the measure $\nu_{C_i^Q, x_0}^{\perp}$, $i=1,2$, defined in \eqref{GFFSch}. Then, we have 
\begin{align}
\Law_{\PP} (Y(1)) = \nu^\perp_{Q,x_0},
\label{lawnu}
\end{align}

\noi 
as given in \eqref{GFFSch}.

We next recall the  Bou\'e-Dupuis variational formula \cite[Proposition A.2]{TW}, \cite[Theorem 7]{Ust}, and \cite[Theorem 2]{BG}. Let $F: L^2(\T_{L^2}) \to \R$ be a measurable function bounded from below. Then,
\begin{align}
\log \E\Big[e^{-F(Y(1)) }\Big]
&= \sup_{V \in \mathbb H_a^1}
\E\bigg[ F( Y(1)+V(1)) - \frac{1}{2} \int_0^1 \| \dot{V} (t) \|_{H^1_{C^Q}}^2 \, \mathrm{d}t \bigg] \notag \\
&\le\sup_{V \in \mathbb H_a^1}
\E\bigg[ F( Y(1)+V(1)) -\frac 12  \| V(1) \|_{H^1_{C^Q}}^2 \bigg],
\label{DUP}
\end{align}

\noi
where $\Ha^1$ denotes the space of $H^1_{C^Q}:=H^1_{C_1^Q}\times H^1_{C_2^Q}$-valued progressively measurable processes $V(t)=(V_1(t), V_2(t))$ such that $V(0)=0$ and
\begin{align*}
\E \bigg[\int_0^1 \| \dot{V} (t) \|_{H^1_{C^Q}}^2 \, \mathrm{d}t \bigg]<\infty.
\end{align*}

\noi
Here, the norm $\| \cdot \|_{H^1_{C^Q} }^2=\| \cdot \|_{H^1_{C^Q_1} }^2+\| \cdot \|_{H^1_{C^Q_2} }^2 $ is defined by
\begin{align*}
\| u \|_{H^1_{C^Q_i} }=\| (C^Q_{i,x_0})^{-\frac 12}  u \|_{L^2(\T_{L^2})}
\end{align*}

\noi
for $i=1,2$.

\noi
By using the Bou\'e-Dupuis formula \eqref{DUP} and Young's inequality, we have 
\begin{align}
\log \E_{ \nu_{Q, x_0}^\perp }\Big[ e^{\dl \int_{-L^2}^{L^2} | h^\perp|^2 \,\mathrm{d}x } \Big] &=\sup_{V\in \Ha^1 }\E\bigg[ \dl \int_{-L^2}^{L^2} |Y(1)+V(1)|^2\, \mathrm{d}x -\frac 12 \| V(1) \|_{H^1_{C^Q}}^2  \, \mathrm{d}t \bigg] \notag \\
&\le \sup_{V\in \Ha^1 } \E\bigg[ 2 \dl \int_{-L^2}^{L^2} |Y(1) |^2\, \mathrm{d}x \bigg]
\label{DUP1}
\end{align}

\noi
by choosing $\dl>0$ sufficiently small. Note that from \eqref{lawnu}
\begin{align}
\E\bigg[ \int_{-L^2}^{L^2} |Y(1) |^2 \, \mathrm{d}x \bigg]=\int_{-L^2}^{L^2}  \big(G_{1,x_0}(x,x)+G_{2,x_0}(x,x) \big) \,  \mathrm{d}x,
\label{DUP2}
\end{align}

\noi
where $G_{i,x_0}$ is the Green's function for the covariance operator $C^Q_{i,x_0}$, defined in Lemma \ref{lem: covariance bound}.
Then, from Lemma \ref{lem: covariance bound}, we have 
\begin{align}
|G_{i,x_0}(x,x) | \les 1,
\label{DUP3}
\end{align}

\noi
uniformly in $x_0, x \in \T_{L^2}$ for $i=1,2$. Combining \eqref{DUP1}, \eqref{DUP2} and \eqref{DUP3} yields that there exists a constant $c>0$ such that,  for sufficiently small $\dl>0$,  
\begin{align*}
\E_{ \nu_{Q, x_0}^\perp }\Big[ e^{\dl \int_{-L^2}^{L^2} | h^\perp|^2 \,\mathrm{d}x } \Big]  \les e^{ cL^2},
\end{align*}

\noi
uniformly in $x_0 \in \T_{L^2}$. This completes the proof of Proposition \ref{PROP: gaussian-conc}.

\end{proof}

\subsection{Partition function estimate}

In this subsection, we establish a uniform lower bound for the partition function $Z_{x_0,\dr}=\mathcal{I}(1)$ in \eqref{Idef}, which will be used later.
\begin{lemma}\label{LEM: CL-est}
There exists a constant $c>0$ such that 
\begin{align}
Z_{x_0,\dr}=\mathcal{I}(1) \ge c e^{-cL},
\end{align}

\noi
uniformly in $x_0\in \T_{L^2}$ and $\dr \in  [0,2\pi]$, where $\mathcal{I}(1)$ is as defined in \eqref{Idef}.

\end{lemma}

\begin{proof}
Since the integrands in $\mathcal{I}(1)$ are positive, 
\begin{align}
Z_{x_0,\dr}=\mathcal{I}(1)\ge \mathcal{I}(\ind_{D_L \cap B_K } ),
\label{partfun}
\end{align}

\noi
where $D_L$ and $B_K$ are defined in \eqref{eqn: DL-def} and \eqref{eqn: BK-def}.






Thanks to the orthogonal decomposition of the field  $h(x)=\g(x)\jb{Q^{\eta_L}_{x_0}, \Re h }+h^\perp(x)$ in \eqref{eqn: decomp}, 
\begin{align}
\|h\|_{L^\infty} &\le \| h^\perp\|_{L^\infty}+\| \g\|_{L^\infty} |\jb{Q^{\eta_L}_{x_0},  \Re h }| \notag \\
& \le  \| h^\perp \|_{L^\infty}+ c L^{\frac 12},
\label{LL1}
\end{align}

\noi
where in the last line we used the fact that on $D_L\cap S_{Q,x_0,\dr}$, $|\jb{Q^{\eta_L}_{x_0},  \Re h }|=O(L^\frac 12)$ from \eqref{tL12} and $\|\g \|_{L^\infty}=O(1)$. Therefore, \eqref{LL1} implies that on the set $B_K$ 
\begin{align}
\mathrm{Det}_L(h)&=1+O(L^{-3}\|h \|^2_{L^\infty}) \ge 1-cL^{-3} \| h^\perp\|_{L^\infty}^2-cL^{-2} \ge \frac 12
\label{LL2}
\end{align}

\noi
for sufficiently large $L \ge 1$, where we used  $\| h^\perp\|_{L^\infty}=O(L^\frac 12)$ on $B_K$. Here, $K$ is chosen as $K=M\frac{L^\frac 12}{\sqrt{\log L^2}}$ for some large $M \ge 1$, as explained in \eqref{eqn: BK-def} and Remark \ref{REM:choiK}.

Using the orthogonal decomposition of the field  $h(x)=\g(x)\jb{Q^{\eta_L}_{x_0}, \Re h }+h^\perp(x)$ in \eqref{eqn: decomp} and $L^{3}\mathcal{E}(Q^{\eta_L}_{x_0}, L^{-\frac 32}h):=\mathcal{E}_1(h)+\mathcal{E}_2(h) $ in \eqref{Eerror}, we expand
\begin{align}
\mathcal{E}_1(h)&:=\frac 1{L^{\frac 32}} \int_{\T_{L^2} }  Q^{\eta_L}_{x_0}  |h|^2 \Re h dx \notag \\
&= \sum_{\substack{u_i\in \{h^\perp, \gamma(x)\langle Q^{\eta_L}_{x_0}, \Re h\rangle\}\\i=1,2,3} } \frac{1}{L^{\frac{3}{2}}}\int_{\T_{L^2}} Q_{x_0}^{\eta_L}u_1\bar{u}_2\Re u_3\,\mathrm{d}x, \label{LL3}\\
\mathcal{E}_2(h):&=\frac{1}{L^3} \int_{\T_{L^2}}|h(x)|^4\, \mathrm{d}x \notag \\
&= \sum_{\substack{u_i\in \{h^\perp, \gamma(x)\langle Q^{\eta_L}_{x_0}, \Re h\rangle\}\\i=1,2,3,4}} \frac{1}{L^3}\int_{\T_{L^2}} u_1\bar{u}_2 u_3\bar{u}_4\,\mathrm{d}x. \label{LL4}
\end{align}

\noi 
From \eqref{LL1}, \eqref{LL3}, and \eqref{LL4}, we obtain that on $D_L$
\begin{align}
|\mathcal{E}_1(h)|&\le  C\sum_{k=0}^3 L^{-\frac{k}{2}}\int_{\mathbb{T}_{L^2}} Q^{\eta_L}_{x_0} |h^\perp|^k \gamma(x)^{3-k}\,\mathrm{d}x, \notag \\
&\le C_1L^{-\frac{3}{2}}\int_{\T_{L^2}} Q^{\eta_L}_{x_0}|h^\perp|^3\,\mathrm{d}x  +C_2, \label{LL5}\\
|\mathcal{E}_2(h)|&\le C\sum_{k=0}^3 L^{-\frac{k}{2}-1}\int_{\mathbb{T}_{L^2}} |h^\perp|^k \gamma(x)^{4-k}\,\mathrm{d}x \notag \\
&\le C_1L^{-3}\int_{\mathbb{T}_{L^2}} |h^\perp|^4\,\mathrm{d}x+L^{-1}C_2, \label{LL6}
\end{align}

\noi
where we used $|\jb{Q^{\eta_L}_{x_0}, \Re h }|=O(L^\frac 12) $ on $D_L$.
Since $\|h^\perp \|_{L^\infty}=O(L^\frac 12)$ on $B_K$, we have that on $D_L \cap B_K$
\begin{align}
\| h^\perp\|_{L^4(\T_{L^2})}^4 \le \|h^\perp \|_{L^2(\T_{L^2})}^2 \| h^\perp \|_{L^\infty(\T_{L^2})}^2 \les L^3.
\label{LLL6}
\end{align}

\noi 
Thanks to \eqref{LL5}, \eqref{LL6}, and \eqref{LLL6}, we obtain that on the set $D_L\cap B_K$ 
\begin{align}
|L^{3}\mathcal{E}(Q^{\eta_L}_{x_0}, L^{-\frac 32}h )| \le c
\label{LL7}
\end{align}

\noi
for some constant $c>0$.

Using the decomposition of the field
$h(x)=\g(x) \jb{Q^{\eta_L}_{x_0}, \Re h  }+h^\perp(x)$ in \eqref{eqn: decomp} and conditioning $D_L$, we have 
\begin{align}
\| h\|_{L^2} &\le \| h^\perp \|_{L^2}+|\jb{Q^{\eta_L}_{x_0}, \Re h }| \|\g \|_{L^2} \notag \\
&\le D^\frac 12 L+L^{\frac 12 } \| \g\|_{L^2},
\end{align}

\noi
where in the last line we used \eqref{eqn: DL-def} to obtain $|\jb{Q^{\eta_L}_{x_0}, \Re h }|=O(L^\frac  12)$. Therefore, on the set $S_{Q,x_0,\dr} \cap D_L$, we have $\| h\|_{L^2} \le \dl L^{\frac 32}$. This implies  
\begin{align}
S_{Q,x_0}\cap D_L \subset \mathcal{K}^L_\delta 
\label{LLL2}
\end{align}

\noi 
for large $L$.

Combining \eqref{partfun},  \eqref{LLL2}, \eqref{LL7}, and \eqref{LL2} yields
\begin{align}
Z_{x_0,\dr}&\ge \mathcal{I}(\ind_{D_L \cap B_K}) \notag \\  
&\ge c \int_{\mathcal{S}_{Q,x_0,\theta} }
e^{\frac \Ld2\|h \|_{L^2}^2+\Ld L^{\frac 32} \jb{Q^{\eta_L}_{x_0}, \Re h } } e^{\frac 12\jb{Q^{\eta_L}_{x_0}, 3|\Re h|^2+|\Im h|^2 } } \ind_{   D_L \cap B_K   }\, \mu^\perp_{x_0, \Ld}(\mathrm{d}h)
\label{LL8}
\end{align}

\noi
for sufficiently large $L \ge 1$ and some constant $c>0$. 
From \eqref{eqn: plug-here} and  Lemma \ref{lem: replace}
\begin{align}
\eqref{LL8}=\E_{\nu^\perp_{Q, x_0} }\bigg[ \frac{e^{c_0(t^+)^2+b(h^\perp)t^{+} } }{\sqrt{2\pi} \s L^{\frac 32}  }, \,D_L,\, B_K \bigg](1+o_L(1))
\label{LLL8}
\end{align}

\noi
as $L\to \infty$. Note that on $D_L \cap B_K$
\begin{align*}
t^{+}&=O(L^\frac 12) \\
|b(h^\perp)|&=\bigg| \int_{\T_{L^2}} (Q^{\eta_L}_{x_0})^2 \g h^\perp \,\mathrm{d}x \bigg| \les \| h^\perp \|_{L^\infty}\les L^\frac 12,
\end{align*}

\noi
where we used \eqref{toot}, $\| \g\|_{L^\infty}=O(1)$, and on the set $B_K$, $\|h^\perp \|_{L^\infty}=O(L^\frac 12) $. This implies that 
\begin{align}
\E_{\nu^\perp_{Q,x_0} }\bigg[ \frac{e^{c_0(t^+)^2+b(h^\perp)t^{+} }}{\sqrt{2\pi} \s L^{\frac 32}  }, \,D_L,\, B_K \bigg] \ge L^{-\frac 32} e^{-cL} \nu^\perp_{Q, x_0}(D_L \cap B_K) \ge e^{-2cL}
\label{LLLL8}
\end{align}

\noi
for some constant $c>0$ since $\nu^\perp_{Q, x_0}(D_L \cap B_K) \to 1$ as  $L\to \infty$ from Propositions \ref{prop: BK-prob} and \ref{PROP: gaussian-conc}, where $K$ is set to $K=M\sqrt{ \frac{L}{\log L^2}  }$, as explained above. Combining  \eqref{LL8}, \eqref{LLL8}, and \eqref{LLLL8} yields
\begin{align*}
Z_{x_0,\dr} \ge ce^{-cL}.
\end{align*}

\noi
This completes the proof of Lemma \ref{LEM: CL-est}.

\end{proof}

\subsection{Error estimate: Part 1}

In this subsection, we provide the following error estimate that appears in \eqref{sepmaer}.


\begin{proposition}\label{PROP:error1}
Let $F$ be a bounded and continuous function. Then, there exist constants $c>0$ such that 
\begin{align}
\bigg| \frac{\mathcal{I}(F \ind_{D_L\cap B_K^c})}{\mathcal{I}(1)} \bigg| \les e^{-cM^2 L},
\label{E00}
\end{align}

\noi 
uniformly in $x_0 \in \T_{L^2}$ and $\dr \in [0,2\pi]$, where the sets $D_L$ and $B_K$ are defined in  \eqref{eqn: DL-def} and \eqref{eqn: BK-def}, and $K=M\sqrt{\frac{L}{\log L^2}}$ for sufficiently large $M$.
\end{proposition}

\begin{remark}\rm \label{REM:choiK}
While the tail estimate \eqref{Tail} in Proposition \ref{prop: BK-prob} holds for sufficiently large  $K \ge 1$, we later set $K=M\sqrt{ \frac{L}{\log L^2} }$, where $M\ge 1$ is a large constant, to control the partition function $\mathcal{I}(1)$  in \eqref{E00}, with  $\mathcal{I}(1)\ge ce^{-cL}$ from Lemma \ref{LEM: CL-est}.  See \eqref{Beat} and Lemma \ref{lem: mu2-ests}.

\end{remark}


We present the proof of Proposition \ref{PROP:error1} at the end of this subsection. First, we simplify the numerator $\mathcal{I}(F \ind_{D_L\cap B_K^c})$ in \eqref{E00}
\begin{align}
\mathcal{I}(F \ind_{D_L \cap B_K^c })=\int_{S_{Q,x_0,\theta}\cap \mathcal{K}^L_\dl }  F(h)&e^{L^3\mathcal{E}(Q^{\eta_L}_{x_0}, L^{-\frac{3}{2}}h)} e^{\frac \Ld2\|h \|_{L^2}^2+\Ld L^{\frac 32} \jb{Q^{\eta_L}_{x_0}, \Re h } } \notag  \\
\cdot & e^{\frac 12\jb{Q^{\eta_L}_{x_0},\, 3|\Re h|^2+|\Im h|^2 }  }\mathrm{Det}_L(h)\ind_{  D_L\cap B_K^c   } \,\mathrm{d}\mu_{x_0,\Ld}^\perp(h).
\label{E0}
\end{align}

From the orthogonal decomposition of the field  $h(x)=\g(x)\jb{Q^{\eta_L}_{x_0}, \Re h }+h^\perp(x)$ in \eqref{eqn: decomp}, we have
\begin{align}
\|h\|_{L^\infty} &\le \| h^\perp\|_{L^\infty}+\| \g\|_{L^\infty} |\jb{Q^{\eta_L}_{x_0},  \Re h }| \notag \\
& \les  \| h^\perp \|_{L^\infty}+ L^{\frac 12},
\label{E1}
\end{align}

\noi
where in the last line we used the fact that on $D_L\cap S_{Q,x_0,\dr}$, $|\jb{Q^{\eta_L}_{x_0},  \Re h }|=O(L^\frac 12)$ from \eqref{tL12} and $\|\g \|_{L^\infty}=O(1)$. Therefore, \eqref{E1} implies that 
\begin{align}
\mathrm{Det}_L(h)=1+O(L^{-3}\|h \|^2_{L^\infty}) \les 1+L^{-3} \| h^\perp\|_{L^\infty}^2+L^{-2}.
\label{E2}
\end{align}

\noi
From \eqref{LL5} and \eqref{LL6}, we have that on the set $D_L$
\begin{align}
L^{3}\mathcal{E}(Q^{\eta_L}_{x_0}, L^{-\frac 32}h ) \les  \mathcal{E}_1(h^\perp)+\mathcal{E}_2(h^\perp)+1,
\label{E3}
\end{align}

\noi
where 
\begin{align*}
\mathcal{E}_1(h^\perp)&=\frac{1}{L^{\frac 32}} \int_{\T_{L^2}} Q^{\eta_L}_{x_0} |h^\perp|^3 \, \mathrm{d}x  \\
\mathcal{E}_2(h^\perp)&=\frac{1}{L^3} \int_{\T_{L^2}} |h^\perp|^4 \, \mathrm{d}x.
\end{align*}

\noi
It follows from \eqref{eqn: plug-here}, \eqref{E2}, \eqref{E3}, and Lemma \ref{lem: replace} that
\begin{align}
|\eqref{E0} | &\les  \E_{\nu^\perp_{Q, x_0} } \Bigg[ H(h^\perp) \bigg( \int_{S_{h^\perp}} e^{\frac {\Ld}2 \| h^\perp \|_{L^2}^2 +\Ld G_1(t)+G_2(t,h^\perp) } e^{-\frac{t^2}{2\s^2}} \frac{\mathrm{d}t}{\sqrt{2\pi}\s }  \bigg), \, D_L, \, B_k^c   \Bigg] \notag \\
&\les \E_{\nu^\perp_{Q, x_0} } \Bigg[ H(h^\perp)  \frac 1{L^\frac 32} \frac{e^{c_0(t^+)^2+b(h^\perp)t^+ } }{\sqrt{2\pi \s}}, \, D_L, \, B_K^c \Bigg](1+o_L(1)),
\label{P20} 
\end{align}

\noi
where 
\begin{align}
H(h^\perp):=e^{\frac{C}{L^\frac{3}{2}}\int_{-L^2}^{L^2} Q_{x_0}^{\eta_L}|h^\perp|^3\,\mathrm{d}x+\frac{C}{L^3}\int_{-L^2}^{L^2} |h^\perp|^4\,\mathrm{d}x  } 
(1+L^{-3}\|h^\perp\|_{L^\infty}^2).
\label{EE3}
\end{align}

\noi
for some constant $C>0$. Combining \eqref{P20}, Lemma \ref{LEM: CL-est}, and $t^{+}=O(L^{\frac 12})$ from \eqref{toot} yields
\begin{align}
\bigg|\frac{\mathcal{I}(F \ind_{D_L\cap B_K^c})}{\mathcal{I}(1)}\bigg|&\les e^{cL} \E_{\nu^\perp_{Q, x_0} } \Big[ 
H(h^\perp)  e^{c_0(t^+)^2+b(h^\perp)t^+ }, \, D_L, \, B_K^c \Big]    \notag \\
&\les e^{\wt cL}\E_{ \nu^\perp_{Q, x_0} } \Big[ H(h^\perp) e^{b(h^\perp) L^\frac 12 }, \, D_L, \,B_K^c  \Big]
\label{EE00}
\end{align}

\noi
for some constant $c, \wt c>0$. By Young's inequality, we have
\begin{align}
\E_{ \nu^\perp_{Q, x_0} } \Big[ H(h^\perp) e^{b(h^\perp) L^\frac 12 }, \, D_L, \,B_K^c  \Big] & \le \frac{1}{4}
\E_{\nu^\perp_{Q, x_0} } \Big[1+L^{-12} \|h^\perp \|_{L^\infty}^{8}, \, D_L, \, B_K^c \Big]    \notag \\
&\hphantom{X}+\frac{1}{4} \E_{\nu^\perp_{Q, x_0} } \bigg[ e^{\frac{4C}{L^3}\int_{-L^2}^{L^2} |h^\perp |^4\,\mathrm{d}x}, \, D_L, \, B_K^c \bigg]     \notag \\
&\hphantom{X}+\frac{1}{4}  \E_{\nu^\perp_{Q, x_0} } \bigg[ e^{\frac{4C}{L^{\frac{3}{2}}}\int_{-L^2}^{L^2} Q_{x_0}^{\eta_L}|h^\perp|^3\,\mathrm{d}x}, \, D_L, \, B_K^c  \bigg] \notag \\
&\hphantom{X}+\frac{1}{4}  \E_{\nu^\perp_{Q, x_0} } \bigg[ e^{4b(h^\perp)L^\frac 12}, \, D_L, \, B_K^c  \bigg].
\label{E4}
\end{align}

\noi
In the following lemma, we derive the estimates for the expectations that appear in \eqref{E4}.

\begin{lemma}\label{lem: mu2-ests}
Let $p\ge 1$ and $\zeta >0 $. Then, there are positive constants $c_1=c_1(\zeta,p)$ and  $c_2=c_2(\zeta,p)$ such that
\begin{align}
&\mathbb{E}_{\nu^\perp_{Q, x_0} }\Big[\|h^\perp\|^p_{L^\infty(\T_{L^2}) },\, B_K^c,\, D_L \Big]\le c_1 e^{-c_2 M^2 L },
\label{E11}\\
&\mathbb{E}_{\nu^\perp_{Q, x_0} }\Big[e^{\frac{\zeta}{L^3}\int_{-L^2}^{L^2} |h^\perp|^4\,\mathrm{d}x}, \,  B_K^c, \, D_L \Big]\le c_1e^{-c_2 M^2 L },
\label{E12}\\
&\mathbb{E}_{\nu^\perp_{Q, x_0} }\Big[e^{\frac{\zeta}{L^{3/2 } }\int_{-L^2}^{L^2} Q_{x_0}^{\eta_L}|h^\perp|^3\,\mathrm{d}x}, \, B_K^c, \,  D_L \Big]\le c_1 e^{-c_2 M^2 L }, \label{E13}\\
&\E_{\nu_{Q,x_0}^\perp } \bigg[ e^{\zeta b(h^\perp)L^\frac 12}, \,  B_K^c, \, D_L  \bigg] \le c_1 e^{-c_2 M^2 L},
\label{EE13}
\end{align}

\noi
uniformly in $x_0 \in \T_{L^2}$ and $\dr \in [0,2\pi]$, where $K=M\sqrt{\frac{L}{\log L^2}}$ for sufficiently large $M$.

\end{lemma}

To prove Lemma \ref{lem: mu2-ests}, we first need to prove the following moment-generating function estimate.

\begin{lemma}\label{LEM: mu-mgf}
There exist constants $c_1, c_2>0$ such that for any $0<\ld\le c_2$, we have
\begin{align}
\mathbb{E}_{\nu^\perp_{Q, x_0} }\Big[e^{\lambda\int_{-L^2}^{L^2}Q^{\eta_L}_{x_0}|h^\perp|^2\,\mathrm{d}x} \Big] &\le\frac{1}{1-c_1 \lambda}, 
\label{E5}
\end{align}

\noi
uniformly in $x_0\in \T_{L^2}$, $\dr\in [0,2\pi]$, and $L \ge 1$, where $c_2$ ensures that the condition $0<c_1 \ld<1$ holds. 
\end{lemma}

\begin{proof}
Expanding the exponential in \eqref{E5}, we have 
\begin{align}
&\mathbb{E}_{\nu^\perp_{Q, x_0 } }\Big[e^{\lambda\int_{-L^2}^{L^2} Q_{x_0}^{\eta_L} |h^\perp|^2\,\mathrm{d}x} \Big] \notag \\
=&\sum_{k=0}^\infty \frac{\lambda^k}{k!}\mathbb{E}_{\nu^\perp_{Q, x_0} }\Bigg[\bigg| \int_{-L^2}^{L^2} Q_{x_0}^{\eta_L}(x)|h^\perp(x)|^2\,\mathrm{d}x\bigg|^k\Bigg] \notag \\
=&\sum_{k=0}^\infty \frac{\lambda^k}{k!}\int_{-L^2}^{L^2}\cdots\int^{L^2}_{-L^2} Q_{x_0}^{\eta_L}(x_1)\cdots Q_{x_0}^{\eta_L}(x_k) \mathbb{E}_{\nu^\perp_{Q, x_0}}\big[|h^\perp(x_1)|^2\cdots |h^\perp(x_k)|^2 \big]\,\mathrm{d}x_1\cdots \mathrm{d}x_k.
\label{E6}
\end{align}

By using H\"older's inequality, we have 
\begin{align}
\eqref{E6} \le & \sum_{k=0}^\infty \frac{\lambda^k}{k!}\int\cdots\int_{-L^2}^{L^2} Q_{x_0}^{\eta_L}(x_1)\cdots Q_{x_0}^{\eta_L}(x_k) \prod_{i=1}^k\mathbb{E}_{\nu^\perp_{Q, x_0} }\big[|h^\perp(x_i)|^{2k}\big]^{\frac{1}{k}}\,\mathrm{d}x_1\cdots \mathrm{d}x_k. 
\label{E7}
\end{align}

\noi
By  the covariance estimate in Lemma \ref{lem: covariance bound}, there exists a constant $C>0$ such that 
\begin{align}
\mathbb{E}_{\nu^\perp_{Q, x_0} }\big[|h^\perp(x_i)|^{2k}\big]^{\frac{1}{k}} \le (2k-1) \mathbb{E}_{\nu^\perp_{Q, x_0} }\big[|h^\perp(x_i)|^{2}\big] \le 2k C,
\label{E8}
\end{align}

\noi
uniformly in $x_0\in \T_{L^2}$ and $\dr \in [0, 2\pi ]$.  By plugging \eqref{E8} into \eqref{E7}, there exists a constant $c_1, c_2>0$ such that for any $|\ld|\le c_2$, we have   
\begin{align}
\mathbb{E}_{\nu^\perp_{Q, x_0 } }\Big[e^{\lambda\int_{-L^2}^{L^2} Q_{x_0}^{\eta_L} |h^\perp|^2\,\mathrm{d}x} \Big] &\le \sum_{k=0}^\infty \frac{\lambda^k}{k!}\int_{-L^2}^{L^2}\cdots\int_{-L^2}^{L^2} Q_{x_0}^{\eta_L}(x_1)\cdots Q_{x_0}^{\eta_L}(x_k) (2C)^k k^k \,\mathrm{d}x_1\cdots \mathrm{d}x_k \notag \\
&\le \sum_{k=0}^\infty ( c_1 \ld)^k= \frac{1}{1-c_1\lambda},
\label{E55}
\end{align}

\noi
uniformly in $x_0\in \T_{L^2}$, $\dr \in [0,2\pi]$, and $L \ge 1$, where we used Stirling's formula to approximate $\frac{k^k}{k!}$ as $\frac{e^k}{\sqrt{2\pi k}}$. This completes the proof of Lemma \ref{LEM: mu-mgf}.

\end{proof}

Using Lemma \ref{LEM: mu-mgf}, we now prove Lemma \ref{lem: mu2-ests}.

\begin{proof}[Proof of Lemma \ref{lem: mu2-ests}]
We begin by proving \eqref{E11}. Thanks to \eqref{LinfL} and Proposition \ref{prop: BK-prob}, 
\begin{align*}
\mathbb{E}_{\nu^\perp_{Q, x_0}}\Big[\|h^\perp\|^p_{L^\infty}, \, B_K^c, \, D_L \Big] &\le \mathbb{E}_{\nu^\perp_{Q, x_0} }\Big[\|h^\perp\|^{2p}_{L^\infty} \Big]^{\frac{1}{2}}\nu^\perp_{Q, x_0}(B_K^c)^{\frac{1}{2}}\\
& \le C(\log L^2 )^{\frac p2} \nu^\perp_{Q, x_0}(B_K^c)^{\frac{1}{2}} \\
&\le Ce^{-c K^2\log L} \\
&\les e^{-cM^2 L},
\end{align*}

\noi 
uniformly in $x_0, \in \T_{L^2}$ and $\dr \in [0,2\pi]$, where
$K=M\sqrt{\frac{L}{\log L^2}}$ for sufficiently large $M$. This implies the estimate \eqref{E11}.

We now prove \eqref{E12}. We first perform the dyadic decomposition
\begin{align}
\mathbb{E}_{\nu^{\perp}_{Q, x_0} }\Big[e^{\frac{\zeta}{L^3}\int_{-L^2}^{L^2} |h^\perp|^4\,\mathrm{d}x}, \,  B_K^c, \, D_L \Big] &\le \sum _{\ell \ge  K} \mathbb{E}_{\nu^{\perp}_{Q, x_0} }\Big[e^{\frac{\zeta}{L^3}\int_{-L^2}^{L^2} |h^\perp|^4\,\mathrm{d}x}, \,  F_\ell, \,   \|h^\perp \|_{L^2}\le D^{\frac 12 }L \Big],
\label{E9}
\end{align}

\noi 
where the sum is taken over dyadic numbers $\ell=K, 2K, 4K,\dots $ and 
\begin{equation}\label{eqn: Fl}
F_\ell:=\big\{\ell\sqrt{\log L^2} \le  \|h^\perp\|_{L^\infty} < 2\ell \sqrt{\log L^2} \big\}.
\end{equation}

\noi
By the Cauchy-Schwarz inequality and Proposition \ref{prop: BK-prob}, we have
\begin{align}
\eqref{E9} &\le \sum_{\ell \ge K}  \mathbb{E}_{\nu^{\perp}_{Q, x_0} }\Big[e^{\frac{2\zeta}{L^3}\int_{-L^2}^{L^2} |h^\perp|^4\,\mathrm{d}x}, \,  F_\ell, \,   \|h^\perp \|_{L^2}\le D^{\frac 12 }L \Big]^\frac 12 \nu^\perp_{Q, x_0}(F_\ell)^{\frac 12 } \notag \\
& \les  \sum_{\ell \ge K} e^{\frac{4D \zeta}{L}\ell^2 \log L^2} e^{-c\ell^2 \log L^2 }   \notag \\
&\les e^{-\wt c K^2 \log L^2 } \notag \\
&\les e^{-c M^2 L} \notag ,
\end{align}

\noi
uniformly in $x_0\in \T_{L^2}$ and $\dr \in [0,2\pi]$, where
$K=M\sqrt{\frac{L}{\log L^2}}$ for sufficiently large $M$. This shows the estimate \eqref{E12}.

We now prove \eqref{E13}. As before, we take the dyadic decomposition and use Proposition \ref{prop: BK-prob}
\begin{align}
&\mathbb{E}_{\nu^\perp_{Q, x_0} }\Big[e^{\frac{\zeta}{L^{\frac{3}{2}}}\int_{-L^2}^{L^2} Q_{x_0}^{\eta_L}|h^\perp|^3\,\mathrm{d}x}, \, B_K^c, \, D_L \Big]\notag \\
&=\sum_{\ell \ge K} \mathbb{E}_{\nu^\perp_{Q, x_0} } \Big[e^{\frac{\zeta}{L^{\frac{3}{2}}}\int_{-L^2}^{L^2} Q_{x_0}^{\eta_L}|h^\perp|^3\,\mathrm{d}x}, \, F_\ell, \, D_L \Big] \notag \\
&\le \sum_{\ell \ge K} \mathbb{E}_{\nu^\perp_{Q, x_0} }\Big[e^{\frac{2\zeta}{L^{\frac{3}{2}}}\int_{-L^2}^{L^2} Q_{x_0}^{
\eta_L}|h^\perp|^3\,\mathrm{d}x},\,  F_\ell, \,  D_L \Big]^{\frac{1}{2}} \nu^\perp_{Q, x_0}(F_\ell)^{\frac{1}{2}} \notag \\
&\le   \sum_{\ell \ge K} \mathbb{E}_{\nu^\perp_{Q, x_0} }\Big[e^{\frac{2\zeta}{L^{\frac{3}{2}}}\int_{-L^2}^{L^2} Q_{x_0}^{
\eta_L}|h^\perp|^3\,\mathrm{d}x},\,  F_\ell, \,  D_L \Big]^{\frac{1}{2}} e^{-2c \ell^2 \log L^2},
\label{E90}
\end{align}

\noi
where the sum is taken over dyadic numbers $\ell=K, 2K, 4K,\dots $, and $F_\ell$ is given in \eqref{eqn: Fl}. To prove the desired estimate \eqref{E13}, it suffices to show 
\begin{align}
\mathbb{E}_{\nu^\perp_{Q, x_0} }\Big[e^{\frac{2\zeta}{L^{\frac{3}{2}}}\int_{-L^2}^{L^2} Q_{x_0}^{
\eta_L}|h^\perp|^3\,\mathrm{d}x},\,  F_\ell, \,  D_L \Big] \les 1
\label{EE0}
\end{align}

\noi
uniformly in $x_0 \in \T_{L^2}$ and $\dr \in [0,2\pi]$.

For $\ell > \ell^*$, where $\ell^*$ is to be determined below, we proceed with the following estimate. Thanks to the conditions $F_\ell$, $\| h^\perp \|_{L^2}\le DL^{\frac 12}$, and Proposition \ref{prop: BK-prob}, we have 
\begin{align}
\mathbb{E}_{\nu^\perp_{Q, x_0}}\Big[e^{\frac{2\zeta}{L^{\frac{3}{2}}}\int_{-L^2}^{L^2} Q_{x_0}^{\eta_L} |h^\perp|^3\,\mathrm{d}x},\, F_\ell,\,  \|h^\perp\|_{L^2}\le D^{\frac{1}{2}} L \Big] &\le e^{4\zeta \ell D L^{\frac{1}{2}}\sqrt{\log L^2}}\nu^\perp_{x_0, \dr}(F_\ell) \notag \\
&\les e^{4\zeta \ell D L^{\frac{1}{2}} \sqrt{\log L^2}}e^{-c\ell^2\log L^2} \notag \\ 
&\les 1,
\label{E10}
\end{align}

\noi 
uniformly in $x_0 \in \T_{L^2}$ and $\dr \in [0,2\pi]$, where 
that the bound \eqref{E10} holds provided
\[\ell>\frac{4\zeta D}{c}\frac{L^{\frac{1}{2}}}{\sqrt{\log L^2}}.\]

\noi
For $\ell \le \ell^*$, thanks to Lemma \ref{LEM: mu-mgf}, we obtain
\begin{align}
\mathbb{E}_{\nu_{Q,x_0}^\perp} \Big[e^{\frac{2\zeta}{L^{\frac{3}{2}}}\int_{-L^2}^{L^2} Q_{x_0}^{\eta_L} |h^\perp|^3\,\mathrm{d}x}, \, F_\ell,\,  \|h^\perp \|_{L^2} \le D L^{\frac 12} \Big]&\le \mathbb{E}_{\nu^\perp_{Q, x_0}}\Big[e^{\frac{4\zeta \ell\sqrt{\log L^2}}{L^{\frac{3}{2}}}\int_{-L^2}^{L^2} Q_{x_0}^{\eta_L} |h^\perp|^2\,\mathrm{d}x} \Big] \notag \\
&\le \frac{1}{1-c_1 \lambda},
\label{EE1}
\end{align}

\noi
uniformly in $x_0\in \T_{L^2}$ and $\dr \in [0,2\pi]$, with
\[\lambda = \frac{4 \zeta \ell \sqrt{\log L^2}}{L^{\frac{3}{2}}}. \]

\noi 
Note that the bound \eqref{EE1} is valid if $c_1 \lambda <1$, where the condition ensures the convergence of the geometric series in \eqref{E55}, or equivalently
\begin{align}
\ell \le \frac{1}{4\zeta  c_1} \frac{L^\frac{3}{2} }{\sqrt{\log L^2}} .
\end{align}

\noi 
By choosing $\ell^*= \frac{1}{4 \zeta c_1} \frac{L^\frac{3}{2} }{\sqrt{\log L^2}}$, we obtain \eqref{E10} and \eqref{EE1}. It follows from \eqref{E90} and \eqref{EE0} that
\begin{align*}
\mathbb{E}_{\nu^\perp_{Q, x_0} }\Big[e^{\frac{\zeta}{L^{\frac{3}{2}}}\int_{-L^2}^{L^2} Q_{x_0}^{\eta_L}|h^\perp|^3\,\mathrm{d}x}, \, B_K^c, \, D_L \Big] &\le  \sum_{\ell \ge K} \mathbb{E}_{\nu^\perp_{Q, x_0} }\Big[e^{\frac{2\zeta}{L^{\frac{3}{2}}}\int_{-L^2}^{L^2} Q_{x_0}^{
\eta_L}|h^\perp|^3\,\mathrm{d}x},\,  F_\ell, \,  D_L \Big]^{\frac{1}{2}} e^{-2c \ell^2 \log L^2}\\
&\les  e^{-\wt c K^2\log L}\\
&\les e^{-cM^2 L},
\end{align*}

\noi
uniformly in $x_0 \in \T_{L^2}$ and $\dr \in [0,2\pi]$, where
$K=M\sqrt{\frac{L}{\log L^2}}$ for sufficiently large $M$. This shows the estimate \eqref{E13}.

Finally, we prove \eqref{EE13}. By taking the dyadic decomposition and using Proposition \ref{prop: BK-prob}, 
\begin{align}
\E_{\nu^\perp_{Q, x_0} } \bigg[ e^{\zeta b(h^\perp)L^\frac 12}, \,  B_K^c, \, D_L  \bigg]&=\sum_{\ell \ge K}\E_{\nu^\perp_{Q, x_0} } \bigg[ e^{\zeta b(h^\perp)L^\frac 12}, \,  F_\ell, \, \| h^\perp \|_{L^2} \le D^{\frac 12}L  \bigg] \notag \\
&\le \sum_{\ell \ge K}\E_{\nu^\perp_{Q, x_0} } \bigg[ e^{2\zeta b(h^\perp)L^\frac 12}, \,  F_\ell, \, \| h^\perp \|_{L^2} \le D^{\frac 12}L  \bigg]^\frac 12 \nu^\perp_{Q, x_0}(F_\ell)^\frac 12 \notag \\
&\le  \sum_{\ell \ge K}\E_{\nu^\perp_{Q, x_0} } \bigg[ e^{2\zeta b(h^\perp)L^\frac 12}, \,  F_\ell, \, \| h^\perp \|_{L^2} \le D^{\frac 12}L  \bigg]^\frac 12 e^{-c\ell^2 \log L^2},
\label{EEE1}
\end{align}

\noi 
where the sum is taken over dyadic numbers $\ell=K, 2K, 4K,\dots $, and $F_\ell$ is given in \eqref{eqn: Fl}. Note that on the set $F_\ell$, we have 
\begin{align}
|\zeta b(h^\perp)|=\bigg|3\zeta \int_{\T_{L^2}}(Q^{\eta_L}_{x_0})^2\g(x)h^\perp(x)\, \mathrm{d}x\bigg|\le c\zeta \ell \sqrt{\log L^2}
\label{EEEE1}
\end{align}

\noi
for some constant $c>0$, where $b(h^\perp)$ is given in \eqref{bhper}. From \eqref{EEE1} and \eqref{EEEE1}, we obtain 
\begin{align*}
\E_{\nu^\perp_{Q, x_0} } \bigg[ e^{\zeta b(h^\perp)L^\frac 12}, \,  B_K^c, \, D_L  \bigg] &\les \sum_{\ell \ge K} e^{c\zeta L^\frac 12 \ell \sqrt{\log L^2}} e^{-c\ell^2 \log L^2}\\
&\les \sum_{\ell \ge K} e^{-c\ell(\ell \log L^2- \zeta L^\frac 12  \sqrt{\log L^2})}, 
\end{align*}

\noi 
which is summable if 
\begin{align*}
\ell>\frac{\zeta L^{\frac{1}{2}}}{\sqrt{\log L^2}}.
\end{align*}

\noi
By choosing $K=M\sqrt{\frac{L}{\log L^2}}$ for sufficiently large $M=M(\zeta)$, we have
\begin{align*}
\E_{\nu^\perp_{Q, x_0} } \bigg[ e^{\zeta b(h^\perp)L^\frac 12}, \,  B_K^c, \, D_L  \bigg] \les e^{-cK^2 \log L^2}\les e^{-cM^2 L},
\end{align*}

\noi 
uniformly in $x_0 \in \T_{L^2}$ and $\dr \in [0,2\pi]$.
This completes the proof of Lemma \ref{lem: mu2-ests}.

\end{proof}

We are now ready to prove Proposition \ref{PROP:error1}.   
\begin{proof}[Proof of Proposition \ref{PROP:error1} ]
It follows from \eqref{EE00}, \eqref{E4}, and Lemma \ref{lem: mu2-ests} that
\begin{align}
\bigg|\frac{\mathcal{I}(F \ind_{D_L\cap B_K^c})}{\mathcal{I}(1)} \bigg| \les e^{\wt cL}\E_{ \nu^\perp_{Q, x_0} } \Big[ H(h^\perp) e^{b(h^\perp) L^\frac 12 }, \, D_L, \,B_K^c  \Big] \les e^{\wt cL} e^{-cM^2 L}\les e^{-\frac{c}{2}M^2L}.
\label{Beat}
\end{align}

\noi
This completes the proof of Proposition \ref{PROP:error1}.
\end{proof}

\subsection{Error estimate: Part 2}

In this subsection, we prove the following error estimate in \eqref{sepmaer}.


\begin{proposition}\label{PROP:error2}
Let $F$ be a bounded and continuous function. Then, there exists a constant $c>0$ such that  
\begin{align}
\bigg| \frac{\mathcal{I}(F \ind_{D_L^c})}{\mathcal{I}(1)} \bigg| \les e^{-c DL^2 },
\label{Y1}
\end{align}

\noi 
uniformly in $x_0 \in \T_{L^2}$ and $\dr \in [0,2\pi]$, where the sets $D_L$ is given in  \eqref{eqn: BK-def} and $D$ is a sufficiently large constant.
\end{proposition}

The proof of Proposition \ref{PROP:error2} is provided at the end of this subsection. To begin, we simplify the numerator in
\eqref{Y1}
\begin{align}
\mathcal{I}(F \ind_{D_L^c})=\int\limits_{S_{Q,x_0}\cap \mathcal{K}^L_\dl \cap D_L^c}  F(h)e^{L^3\mathcal{E}(Q^{\eta_L}_{x_0}, L^{-\frac{3}{2}}h)} &e^{\frac \Ld2\|h \|_{L^2}^2+\Ld L^{\frac 32} \jb{Q^{\eta_L}_{x_0}, \Re h } } \notag  \\
\cdot & e^{\frac 12\jb{Q^{\eta_L}_{x_0}, 3|\Re h|^2+|\Im h|^2 }  }\mathrm{Det}_L(h) \,\mathrm{d}\mu_{x_0,\Ld}^\perp(h).
\label{YY1}
\end{align}



From Lemma \ref{LEM:Error1}
\begin{align}
&\mathcal{E}(Q^{\eta_L}_{x_0}, L^{-\frac{3}{2}} h) \notag \\
&\les  \frac 1{L^{\frac 32}} \int_{-L^2}^{L^2} Q^{\eta_L}_{x_0} |h^\perp|^3\, \mathrm{d}x+\frac 1{L^{3}} \int_{-L^2}^{L^2}  |h^\perp|^4 \, \mathrm{d}x+ \dl \int_{-L^2}^{L^2} |h^\perp| \, \mathrm{d}x.
\label{SQQQ0}
\end{align}

\noi
for some small $\dl>0$.

\noi
On the set $S_{Q,x_0}$ in \eqref{SQK}
\begin{align}
\frac \Ld 2 \| h\|_{L^2}^2 + \Ld L^{\frac 32} \jb{Q^{\eta_L}_{x_0}, \Re h} \le 0.
\label{SQ0}
\end{align}

\noi
By using Young's inequality, along with \eqref{eqn: decomp} and \eqref{hperpL21}, we have 
\begin{align}
|\Im h(x)|&=|\Im h^\perp(x)| \notag \\
|\Re h(x)|^2&\le (1+\dl)|\Re h^\perp(x)|^2+(1+c(\dl))\dl_1|\g(x)|^2 \|h^\perp \|_{L^2}^2 
\label{SQ02}
\end{align}

\noi
for some small $\dl, \dl_1>0$ and a large constant $c(\dl)>0$.

\noi 
From \eqref{Det}, \eqref{Det0}, along with \eqref{hperpLinf} and \eqref{hperpL2} in Lemma \ref{LEM:Error0},  
\begin{align}
\mathrm{Det}_L(h)=1+O(L^{-3}\|h \|^2_{L^\infty})\les 1+O(L^{-3}\|h^\perp \|^2_{L^\infty})+\dl^2.
\label{Y4}
\end{align}



\noi 
Combining \eqref{SQQQ0}, \eqref{SQ0}, \eqref{SQ02}, \eqref{Y4}, and 
$S_{Q,x_0}\cap \mathcal{K}^L_\dl \subset \| h^\perp\|_{L^2} \le \dl L^\frac 32$ in \eqref{hperpL2} yields
\begin{align}
|\eqref{YY1} | &\les \| F\|_{L^\infty}
\E_{\nu^\perp_{Q, x_0} } \Big[ e^{2\dl \int_{-L^2}^{L^2} |h^\perp|^2\,\mathrm{d}x}  \Big]^\frac 12
\E_{\nu^\perp_{Q, x_0} } \Big[ H(h^\perp)^2,\,   \|h^\perp \|_{L^2}\le  \dl L^{\frac 32}, \, D_L^c    \Big]^\frac 12,
\label{Y8}
\end{align}

\noi
where $h=h^\perp+\g \jb{Q^{\eta_L}_{x_0}, \Re h  }$ in \eqref{eqn: decomp}, $\nu^\perp_{Q, x_0}$ is the measure in \eqref{GFFSch}, and 
\begin{align}
H(h^\perp):=e^{\frac{C_1}{L^3}\int_{-L^2}^{L^2} |h^\perp|^4\,\mathrm{d}x+\frac{C_2}{L^\frac{3}{2}}\int_{-L^2}^{L^2} Q_{x_0}^{\eta_L}|h^\perp|^3\,\mathrm{d}x \,\mathrm{d}x} 
(1+L^{-3}\|h^\perp\|_{L^\infty}^2).
\label{Y7}
\end{align}

By the Cauchy-Schwarz inequality and Proposition \ref{PROP: gaussian-conc}, 
\begin{equation}\label{eqn: IC-est}
\begin{split}
\eqref{Y8}&\les \|F\|_{L^\infty}e^{c L^2}\mathbb{E}_{\nu^\perp_{Q, x_0}}\Big[H(h^\perp)^4,\,\|h^\perp\|_{L^2}\le \delta L^{\frac{3}{2}}\Big]^{\frac{1}{4}}\nu^\perp_{Q, x_0} (D^c_L)^{\frac{1}{4}}\\
&\les \|F\|_{L^\infty}\mathbb{E}_{\nu_{Q,x_0}^\perp}\Big[H(h^\perp)^4, \,\|h^\perp \|_{L^2}\le \delta L^{\frac{3}{2}}\Big]^{\frac{1}{4}} e^{-c DL^2},
\end{split}
\end{equation}

\noi 
where $D$ is chosen to be sufficiently large.  From Lemma \ref{LEM: CL-est}, the partition function is bounded by 
\begin{align}
\mathcal{I}(1)& \ges e^{-cL}.
\label{Y9}
\end{align}

\noi
By combining \eqref{eqn: IC-est} and \eqref{Y9}, we have
\begin{align}
\bigg| \frac{\mathcal{I}(F \ind_{D_L^c})}{\mathcal{I}(1)} \bigg| & \les e^{cL} \mathbb{E}_{\nu^\perp_{Q, x_0}}\Big[H(h^\perp)^4, \,\|h^\perp\|_{L^2}\le \delta L^{\frac{3}{2}}\Big]^{\frac{1}{4}} e^{-c DL^2} \notag \\
&\les \mathbb{E}_{\nu^\perp_{Q, x_0} }\Big[H(h^\perp)^4, \,\|h^\perp\|_{L^2}\le \delta L^{\frac{3}{2}}\Big]^{\frac{1}{4}}  e^{-\frac{c}{2}D L^2}.
\label{YYYY9}
\end{align}

\noi
Hence,  to prove Proposition \ref{PROP:error2}, it is sufficient to obtain the following lemma. 


\begin{lemma} \label{lem: H-est}
There exists a constant $C$ such that, for sufficiently small $\delta$,
\begin{align}
\mathbb{E}_{\nu^\perp_{Q, x_0}}\Big[H(h^\perp)^4, \,\|h^\perp\|_{L^2}\le \delta L^{\frac{3}{2}}\Big] \le C,
\end{align}

\noi 
uniformly in $x_0\in \T_{L^2}$, $\dr \in [0,2\pi]$, and $L\ge 1$.
\end{lemma}

\begin{proof}
By Young's inequality, we have
\begin{align}
\E_{\nu^\perp_{Q, x_0} } \Big[ 
H(h^\perp)^4, \, \| h^\perp \|_{L^2} \le \dl L^{\frac 32}  \Big] &\le \frac{1}{3}
\E_{\nu^\perp_{Q, x_0} } \Big[1+L^{-36} \|h^\perp \|_{L^\infty}^{24}, \, \| h^\perp \|_{L^2} \le \dl L^{\frac 32} \Big]    \notag \\
&\hphantom{X}+\frac{1}{3} \E_{\nu^\perp_{Q, x_0} } \bigg[ e^{\frac{12C_1}{L^3}\int_{-L^2}^{L^2} |h^\perp |^4\,\mathrm{d}x}, \, \| h^\perp\|_{L^2} \le \dl L^{\frac 32} \bigg]     \notag \\
&\hphantom{X}+\frac{1}{3}  \E_{\nu^\perp_{Q, x_0} } \bigg[ e^{\frac{12C_2}{L^{\frac{3}{2}}}\int_{-L^2}^{L^2} Q_{x_0}^{\eta_L}|h^\perp|^3\,\mathrm{d}x}, \, \| h^\perp\|_{L^2} \le \dl L^{\frac 32}  \bigg].
\end{align}

\noi
From \eqref{LinfL}, we have
\begin{align}
\E_{\nu^\perp_{Q, x_0} } \Big[ L^{-36} \| h\|^{24}_{L^\infty([-L^2,L^2]) } \Big]\le C L^{-36}  (\log L^2)^{12}\les 1,
\end{align}

\noi
uniformly in $x_0\in \T_{L^2}$ and $\dr \in [0,2\pi]$. Hence, it suffices to show that for every $C_1'$, $C_2'>0$
\begin{align}
\mathbb{E}_{\nu^\perp_{Q, x_0}} \bigg[e^{\frac{C_1'}{L^3}\int_{-L^2}^{L^2} |h^\perp |^4\,\mathrm{d}x},\ \|h^\perp \|_{L^2}\le \delta L^{\frac{3}{2}} \bigg]&\le M_1, \label{YY9}\\
\mathbb{E}_{\nu^\perp_{Q, x_0}} \bigg[e^{\frac{C_2'}{L^\frac{3}{2}}\int_{-L^2}^{L^2} Q_{x_0}^{\eta_L}|h^\perp|^3\,\mathrm{d}x},\ \|h^\perp\|_{L^2}\le \delta L^{\frac{3}{2}} \bigg]&\le M_2, \label{YYY9}
\end{align}

\noi 
where the constants $M_1$ and $M_2$ are independent of $x_0, \dr$, and $L$.

We begin by performing the dyadic decomposition over the dyadic numbers $\ell=1, 2,4 \dots$ 
\begin{align}
\mathbb{E}_{\nu^\perp_{Q, x_0}} \bigg[e^{\frac{C_1'}{L^3}\int_{-L^2}^{L^2} |h^\perp|^4\,\mathrm{d}x},\, \|h^\perp \|_{L^2}\le \delta L^{\frac{3}{2}} \bigg]&=\mathbb{E}_{\nu^\perp_{Q, x_0}} \bigg[e^{\frac{C_1'}{L^3}\int_{-L^2}^{L^2} |h^\perp|^4\,\mathrm{d}x},\,  F_0,  \,\|h^\perp \|_{L^2}\le \delta L^{\frac{3}{2}} \bigg] \notag \\
&\hphantom{X}+\sum_{\ell\ge 1} \mathbb{E}_{\nu^\perp_{Q, x_0}} \bigg[e^{\frac{C_1'}{L^3}\int_{-L^2}^{L^2} |h^\perp|^4\,\mathrm{d}x},\, F_\ell, \, \|h^\perp\|_{L^2}\le \delta L^{\frac{3}{2}} \bigg]
\label{Y10}
\end{align}

\noi
where $F_0:=\big\{   \|h^\perp\|_{L^\infty} \le \sqrt{\log L^2}   \big\} $ and $F_\ell:=\big\{\ell\sqrt{\log L^2} \le  \|h^\perp \|_{L^\infty} < 2\ell \sqrt{\log L^2} \big\}$. Then, thanks to Proposition \ref{prop: BK-prob}, 
\begin{align*}
\eqref{Y10} &\le e^{c L^{-1+}} + \sum_{\ell\ge 1} e^{c\ell^2\delta^2\log L^2}\nu^\perp_{Q, x_0}(F_\ell)\\
&\le e^{c L^{-1+}}+ \sum_{\ell\ge 1} e^{c'\ell^2\delta^2\log L^2}e^{-c\ell^2\log L^2}\\
&\le M_1<\infty,
\end{align*}

\noi
uniformly in $x_0 \in \T_{L^2}, \dr\in [0,2\pi]$, and $L\ge 1$, if $\delta$ is sufficiently small. This completes the proof of \eqref{YY9}.

We now prove \eqref{YYY9}. As before, we take the dyadic decomposition over the dyadic numbers $\ell=1,2,4,\dots$ and use the Cauchy–Schwarz inequality and Proposition \ref{prop: BK-prob}
\begin{align}
&\mathbb{E}_{\nu^\perp_{Q, x_0}} \bigg[e^{\frac{C_2'}{L^\frac{3}{2}}\int_{-L^2}^{L^2} Q_{x_0}^{\eta_L}|h^\perp|^3\,\mathrm{d}x},\ \|h^\perp \|_{L^2}
\le \delta L^{\frac{3}{2}} \bigg] \notag \\
&=\mathbb{E}_{\nu^\perp_{Q, x_0} } \bigg[e^{\frac{C_2'}{L^\frac{3}{2}}\int_{-L^2}^{L^2} Q_{x_0}^{\eta_L}|h^\perp|^3\,\mathrm{d}x},\ F_0, \, \|h^\perp \|_{L^2}\le \delta L^{\frac{3}{2}} \bigg] \notag \\
&\hphantom{X}+\sum_{\ell \ge 1}\mathbb{E}_{\nu^\perp_{Q, x_0}} \bigg[e^{\frac{C_2'}{L^\frac{3}{2}}\int_{-L^2}^{L^2} Q_{x_0}^{\eta_L} |h^\perp|^3\,\mathrm{d}x},\ F_\ell, \, \|h^\perp\|_{L^2}\le \delta L^{\frac{3}{2}} \bigg] \notag \\
&\le  e^{cL^{-\frac 32+} } +\sum_{\ell \ge 1}\mathbb{E}_{\nu^\perp_{Q, x_0}} \bigg[e^{\frac{2C_2'}{L^\frac{3}{2}}\int_{-L^2}^{L^2} Q_{x_0}^{\eta_L}|h^\perp|^3\,\mathrm{d}x},\ F_\ell, \, \|h^\perp \|_{L^2}\le \delta L^{\frac{3}{2}} \bigg]^{\frac 12}e^{-c\ell^2 \log L^2}.
\label{YY10}
\end{align}

\noi

For $\ell > \ell^*$, with $\ell^*$ to be determined subsequently,
we proceed with the following estimate. Based on the conditions $F_\ell$, $\| h^\perp \|_{L^2}\le \dl L^{\frac 32}$, and Proposition \ref{prop: BK-prob}, we have 
\begin{align}
\mathbb{E}_{\nu^\perp_{Q, x_0} }\Big[e^{\frac{C_2'}{L^{\frac{3}{2}}}\int_{-L^2}^{L^2} Q_{x_0}^{\eta_L} |h^\perp|^3\,\mathrm{d}x},\, F_\ell,\,  \|h^\perp\|_{L^2}\le \dl L^{\frac 32} \Big] &\le e^{2 C_2' \ell \dl^2 L^{\frac{3}{2}}\sqrt{\log L^2}}\nu_{x_0, \dr}^\perp(F_\ell) \notag \\
&\les e^{ 2 C_2' \ell \dl^2 L^{\frac{3}{2}}\sqrt{\log L^2} }e^{-c\ell^2\log L^2} \notag \\
& \les 1,
\label{Y11}
\end{align}

\noi
uniformly in $x_0 \in \T_{L^2}$ and $\dr \in [0,2\pi]$, where we need the following condition  
\begin{align}
\ell>\frac{2\sqrt{2} C_2' \dl^2}{c}\frac{L^{\frac{3}{2}}}{\sqrt{\log L^2}}.
\label{Y12}
\end{align}

\noi

\noi
For $\ell \le \ell^*$, using the moment-generating function estimate in  Lemma \ref{LEM: mu-mgf}, we obtain
\begin{align}
\mathbb{E}_{\nu_{Q , x_0}^\perp} \Big[e^{\frac{2}{L^{\frac{3}{2}}}\int_{-L^2}^{L^2} Q_{x_0}^{\eta_L} |h^\perp |^3\,\mathrm{d}x}, \, F_\ell,\,  \|h^\perp \|_{L^2} \le D L^{\frac 12} \Big]&\le \mathbb{E}_{\nu^\perp_{Q, x_0} }\Big[e^{\frac{2\ell\sqrt{\log L^2}}{L^{\frac{3}{2}}}\int_{-L^2}^{L^2} Q_{x_0}^{\eta_L} |h^\perp|^2\,\mathrm{d}x} \Big] \notag \\
&\le \frac{1}{1-c_1 \lambda},
\label{Y13}
\end{align}

\noi
uniformly in $x_0\in \T_{L^2}$ and $\dr \in [0,2\pi]$, with
\[\lambda = \frac{2\ell \sqrt{\log L^2}}{L^{\frac{3}{2}}}. \]

\noi 
Here the uniform bound \eqref{Y13} holds if $c_1 \lambda <1$, ensuring the convergence of the geometric series in \eqref{E55}, or equivalently
\begin{align}
\ell \le \frac{1}{2 c_1} \frac{L^\frac{3}{2} }{\sqrt{\log L^2}} .
\label{Y14}
\end{align}

\noi
By comparing \eqref{Y12} with \eqref{Y14} and choosing $\dl>0$ sufficiently small, we have \eqref{Y11} and \eqref{Y13}. This implies that 
\begin{align*}
\eqref{YY10} \les e^{c L^{-\frac {3}{2}+} } +\sum_{\ell \ge 1} e^{-c\ell^2 \log L^2} \les 1.
\end{align*}

\noi
This completes the proof of \eqref{YYY9}.

\end{proof}

We are now ready to prove \ref{PROP:error2}.
\begin{proof}[Proof of Proposition \ref{PROP:error2}].
It follows from \eqref{YYYY9} and Lemma \ref{lem: H-est} that
\begin{align*}
\bigg| \frac{\mathcal{I}(F \ind_{D_L^c})}{\mathcal{I}(1)} \bigg| &\les e^{-\frac{c}{2} DL^2},
\end{align*}

\noi
uniformly in $x_0 \in \T_{L^2}$ and $\dr \in [0,2\pi]$. Hence, we obtain the desired result. 
\end{proof}

\subsection{Main term}
In this subsection, we simplify the main term in \eqref{sepmaer} using the conditional density function.
\begin{proposition}\label{PROP:main}
For $F=F(h^\perp,t)$ which is bounded, Lipschitz in the sense of \eqref{eqn: F-lip}, we have 
\begin{align}
\frac{\mathcal{I}(F,D_L\cap B_K)}{\mathcal{I}(1)}=\frac{
\E_{\nu^\perp_{Q, x_0}}  \Big[  F(h^\perp,t^+)  e^{\bar {\mathcal{E}}(h^\perp, t^{+}) } e^{c_0(t^+)^2+b(h^\perp)t^+}  , \,D_L, \, B_K  \Big]}{ \E_{\nu^\perp_{Q, x_0}}  \Big[ e^{\bar {\mathcal{E}}(h^\perp, t^{+}) } e^{c_0(t^+)^2+b(h^\perp)t^+} , \,D_L, \, B_K  \Big] }(1+o_L(1))
\label{parta0}
\end{align}

\noi
as $L\to \infty$, where $t^{+}, c_0, $ and $b(h^\perp)$  are given in Lemma \ref{lem: replace}.  Here, the sets $D_L$ and $B_K$ are defined in  \eqref{eqn: DL-def} and \eqref{eqn: BK-def}.
\end{proposition}

We present the proof of Proposition \ref{PROP:main} at the end of this subsection. First, we separate the main contribution and the error term from $\mathcal{I}(F, D_L\cap B_K)$. Note that

\begin{align}
\mathcal{I}(F,D_L\cap B_K)=\mathcal{I}_1(F, D_L\cap B_K)+\mathcal{I}_2(F, D_L\cap B_K),
\label{Ma1}
\end{align}

\noi
where
\begin{align}
\mathcal{I}_1(F, D_L\cap B_K)=\int_{S_{Q,x_0,\theta}\cap \mathcal{K}^L_\dl }  F(h)&e^{L^3\mathcal{E}(Q^{\eta_L}_{x_0}, L^{-\frac{3}{2}}h)} e^{\frac \Ld2\|h \|_{L^2}^2+\Ld L^{\frac 32} \jb{Q^{\eta_L}_{x_0}, \Re h } } \notag \\
\cdot & e^{\frac 12\jb{Q^{\eta_L}_{x_0},\, 3|\Re h|^2+|\Im h|^2 }  }\ind_{  D_L\cap B_K   } \,\mathrm{d}\mu_{x_0,\Ld}^\perp(h)
\label{IM1}
\end{align}

\noi
and
\begin{align*}
\mathcal{I}_2(F, D_L\cap B_K)=\int_{S_{Q,x_0,\theta}\cap \mathcal{K}^L_\dl }  F(h)&e^{L^3\mathcal{E}(Q^{\eta_L}_{x_0}, L^{-\frac{3}{2}}h)} e^{\frac \Ld2\|h \|_{L^2}^2+\Ld L^{\frac 32} \jb{Q^{\eta_L}_{x_0}, \Re h } } \notag  \\
\cdot & (\mathrm{Det}_L(h)-1)e^{\frac 12\jb{Q^{\eta_L}_{x_0},\, 3|\Re h|^2+|\Im h|^2 }  }\ind_{  D_L\cap B_K   } \,\mathrm{d}\mu_{x_0,\Ld}^\perp(h).
\end{align*}

\noi
Here, $\mathcal{I}_1$ plays a role as the main term, while $\mathcal{I}_2$ represents the error term. In the following lemma,   we first simplify the main term $\mathcal{I}_1$.

\begin{lemma}\label{LEM:Main0}
For $F=F(h^\perp,t)$ which is bounded, Lipschitz in the sense of \eqref{eqn: F-lip}, we have 
\begin{align*}
&\mathcal{I}_1(F, D_L\cap B_K)\\
&=\frac{1}{L^{\frac 32}(1+o_L(1))}  \E_{\nu^\perp_{Q, x_0}}  \Big[  F(h^\perp,t^+)  e^{\bar {\mathcal{E}}(h^\perp, t^{+}) } e^{c_0(t^+)^2+b(h^\perp)t^+}  , \,D_L, \, B_K  \Big]
(1+o_L(1)).
\end{align*}

\noi
as $L\to \infty$, where $M$ is a large constant from Lemma \ref{lem: mu2-ests}.

\end{lemma}

\begin{proof}

In the definition of $\mathcal{I}_1(F, D_L \cap B_K)$ in \eqref{IM1}, being on the set $D_L$  implies
\begin{align*}
S_{Q,x_0}\cap D_L \subset \mathcal{K}^L_\delta 
\end{align*}

\noi 
from  \eqref{LLL2}. Hence, we can ignore the condition $\mathcal{K}^L_\dl$. It follows  from \eqref{R1}, \eqref{eqn: plug-here}, Lemma \ref{LEM:EpsLip}, and Lemma \ref{lem: replace} that  
\begin{align}
&\mathcal{I}_1(F,D_L\cap B_K)\notag \\
&=\int_{D_L\cap B_K} \int_{S_{h^\perp}}F(h^\perp, t)e^{\bar{\mathcal{E}}(h^\perp, t) }e^{\frac{\Ld}{2}\|h^\perp\|_2^2}e^{\Ld G_1(t)}e^{G_2(t,h^\perp)} e^{-\frac{t^2}{2\sigma^2}}\frac{\mathrm{d}t}{\sqrt{2\pi} \sigma} \,\nu^{\perp}_{Q, x_0}(\mathrm{d}h^\perp) \notag \\
&=\frac{1}{L^{\frac 32}(1+o_L(1))}  \E_{\nu^\perp_{Q, x_0}}  \Big[  F(h^\perp,t^+)  e^{\bar {\mathcal{E}}(h^\perp, t^{+}) } e^{c_0(t^+)^2+b(h^\perp)t^+}  , \,D_L, \,B_K  \Big]
(1+o_L(1))
\label{IM2}
\end{align}

\noi
as $L\to \infty$.

\end{proof}

In the following lemma, we present an error estimate for the term $\mathcal{I}_2(F, D_L\cap B_K)$ in \eqref{Ma1}.

\begin{lemma}\label{LEM:Main1}
For $F=F(h^\perp,t)$ which is bounded, Lipschitz in the sense of \eqref{eqn: F-lip}, we have 
\begin{equation}\label{eqn: I-denom}
\begin{split}
&|\mathcal{I}_2(F, D_L\cap B_K)|\\
&=\frac{1}{L^{\frac 32}(1+o_L(1))}   \E_{\nu^\perp_{Q, x_0}}  \Big[  F(h^\perp,t^+)  e^{\bar {\mathcal{E}}(h^\perp, t^{+}) } e^{c_0(t^+)^2+b(h^\perp)t^+}  , \,D_L, \, B_K  \Big] O(L^{-2})
\end{split}
\end{equation}

\noi
as $L\to \infty$.
\end{lemma}

\begin{proof}
From \eqref{E2}, on the set $D_L \cap B_K$, we have 
\begin{align}
\mathrm{Det}_L(h)=1+O(L^{-2}),
\label{Ma3}
\end{align}

\noi
where $K$ is chosen as $K=M\frac{L^\frac 12}{\sqrt{\log L^2}}$ for some large $M \ge 1$, based on Lemma \ref{lem: mu2-ests}.
By proceeding as in \eqref{IM2} with \eqref{Ma3}, we have
\begin{align}
&\mathcal{I}_2(F,D_L\cap B_K)\notag \\
&=\frac{1}{L^{\frac 32}(1+o_L(1))}  \E_{\nu^\perp_{Q, x_0}}  \Big[  F(h^\perp,t^+)  e^{\bar {\mathcal{E}}(h^\perp, t^{+}) } e^{c_0(t^+)^2+b(h^\perp)t^+}  , \,D_L, \,B_K  \Big]
O(L^{-2}).
\label{IM5}
\end{align}

\noi 
This completes the proof of Lemma \ref{LEM:Main1}.

\end{proof}

We are now ready to prove Proposition \ref{PROP:main}.

\begin{proof}[Proof of Proposition \ref{PROP:main}].
We first consider the denominator in \eqref{sepmaer}. Note that 
\begin{align}
Z_{x_0,\dr}=\mathcal{I}(1)=\mathcal{I}(1, D_L \cap B_K)+\mathcal{I}(1, D_L \cap B_K^c)+\mathcal{I}(1, D_L^c).
\label{MA777}
\end{align}

From \eqref{Ma1}, Lemma \ref{LEM:Main0}, and Lemma \ref{LEM:Main1},
\begin{align}
&\mathcal{I}(1, D_L\cap B_K) \notag \\
&=\frac{1}{L^{\frac 32}(1+o_L(1))}  \E_{\nu^\perp_{Q, x_0}}  \Big[   e^{\bar {\mathcal{E}}(h^\perp, t^{+}) } e^{c_0(t^+)^2+b(h^\perp)t^+}  , \,D_L, \, B_K  \Big]
(1+o_L(1))
\label{MA77}
\end{align}

\noi
From \eqref{EE00}, \eqref{E4}, Lemma \ref{lem: mu2-ests}, \eqref{LL7} and \eqref{LLLL8},
\begin{align}
&\mathcal{I}(1, D_L \cap B_K^c)=O(e^{- M^2 L})\notag \\
&=\frac{1}{L^{\frac 32}(1+o_L(1))}  \E_{\nu^\perp_{Q, x_0}}  \Big[   e^{\bar {\mathcal{E}}(h^\perp, t^{+}) } e^{c_0(t^+)^2+b(h^\perp)t^+}  , \,D_L, \, B_K  \Big]
(1+o_L(1))O(e^{- \frac{M^2}{2} L}).
\label{MA7}
\end{align}

\noi

From \eqref{YYYY9}, Lemma \ref{lem: H-est}, \eqref{LL7}, and \eqref{LLLL8},
\begin{align}
&\mathcal{I}(1, D_L^c)=O(e^{-DL^2})\notag \\
&=\frac{1}{L^{\frac 32}(1+o_L(1))}  \E_{\nu^\perp_{Q, x_0}}  \Big[   e^{\bar {\mathcal{E}}(h^\perp, t^{+}) } e^{c_0(t^+)^2+b(h^\perp)t^+}  , \,D_L, \, B_K  \Big]
(1+o_L(1))O(e^{- \frac{D}{2} L^2}),
\label{MA8}
\end{align}

\noi
\noi
where $D$ is chosen to be sufficiently large. By combining \eqref{MA777}, \eqref{MA77}, \eqref{MA7}, and \eqref{MA8},  we obtain
\begin{align}
Z_{x_0,\dr}=\mathcal{I}(1)=\frac{1}{L^{\frac 32}(1+o_L(1))}  \E_{\nu^\perp_{Q, x_0}}  \Big[   e^{\bar {\mathcal{E}}(h^\perp, t^{+}) } e^{c_0(t^+)^2+b(h^\perp)t^+}, \,D_L, \, B_K  \Big]
(1+o_L(1))
\label{MA9}
\end{align}

\noi
as $L\to \infty$.

Therefore, it follows from \eqref{Ma1}, Lemma \ref{LEM:Main0}, Lemma \ref{LEM:Main1}, and \eqref{MA9} that 
\begin{align*}
\frac{\mathcal{I}(F,D_L\cap B_K)}{\mathcal{I}(1)}
&=\frac{\E_{\nu^\perp_{Q, x_0}}  \Big[ F(h^\perp, t^{+})   e^{\bar {\mathcal{E}}(h^\perp, t^{+}) } e^{c_0(t^+)^2+b(h^\perp)t^+}  , \,D_L, \, B_K  \Big] (1+o_L(1))}{\E_{\nu^\perp_{Q, x_0}}  \Big[   e^{\bar {\mathcal{E}}(h^\perp, t^{+}) } e^{c_0(t^+)^2+b(h^\perp)t^+}  , \,D_L , \, B_K \Big] (1+o_L(1)) }
\end{align*}

\noi
as $L\to \infty$, where in the last line we used \eqref{LL7} and \eqref{LLLL8} to get $O(e^{-\frac {M^2}2 })$. This completes the proof of Proposition \ref{PROP:main}.

\end{proof}

\section{White noise limit on normal space}\label{sec: Gaussian-limit}\label{SEC:WHNOR}

In this section, we analyze characteristic functions and prove their convergence to white noise on the normal space.

From Proposition \ref{PROP:main}, we consider
\begin{align*}
\frac{
\E_{\nu^\perp_{Q, x_0}}  \Big[  F(h_L^\perp,t^+)  e^{\bar {\mathcal{E}}(h^\perp, t^{+}) } e^{c_0(t^+)^2+b(h^\perp)t^+}  , \,D_L, \, B_K  \Big]}{ \E_{\nu^\perp_{Q, x_0}}  \Big[ e^{\bar {\mathcal{E}}(h^\perp, t^{+}) } e^{c_0(t^+)^2+b(h^\perp)t^+}  , \,D_L, \, B_K  \Big] },
\end{align*}

\noi
where the scaling $h_L^\perp(x)=L^{\frac 12}h^\perp(Lx)$ comes from \
Proposition \ref{PROP:REDUC}. To obtain the Gaussian limit, we study the characteristic function by choosing $F$
\begin{align*}
F(h_L)&=e^{i\jb{\Re h_L,g}}=e^{i\jb{ \Re h^\perp_L,g  }+i t^{+}\jb{\g_L, g} }\\
F(h_L)&=e^{i\jb{\Im h_L,g}}=e^{i\jb{ \Im h^\perp_L,g  }},
\end{align*}

\noi 
where $h_L(x)=h^\perp_L(x)+t^{+} \g_L(x) $, $\g_L(x)=L^{\frac 12} \g(Lx)$, and $\Im h=\Im h^\perp$ from the orthogonal decompositions \eqref{eqn: decomp}. Here, $g$ denotes a real-valued test function.

\begin{remark}\rm
As explained in Remark \ref{REM:Lx_0}, from now on, we may assume that the approximate soliton $Q^{\eta_L}_{x_0}$ is localized at $Lx_0$, as expressed in \eqref{appground0}. Consequently, the tangential component
$x_0$-integral in \eqref{REDUCT} is performed over $\T_{L}$.
\end{remark}

In the following, the goal is to prove the convergence of the characteristic function under the condition $\text{dist}(\supp g, x_0) > 0$, where $g$ is a test function whose support lies away from the tangential direction $x_0 \in \T_L$.

\begin{proposition}\label{PROP:gglimit}
Let $g$ be a real-valued, smooth, compactly supported function with $\textup{dist}(\supp g, x_0)>0$, where $x_0\in \T_L$. Then, we have
\begin{align}
&\frac{ \E_{\nu^\perp_{Q, x_0}}  \Big[  e^{i\jb{\Re h^\perp, g_L}+it^{+}\jb{\g_L,g}  }  e^{\bar {\mathcal{E}}(h^\perp, t^{+}) } e^{c_0(t^+)^2+b(h^\perp)t^+}  , \,D_L, \, B_K  \Big]}{ \E_{\nu^\perp_{Q, x_0}}  \Big[ e^{\bar {\mathcal{E}}(h^\perp, t^{+}) } e^{c_0(t^+)^2+b(h^\perp)t^+}  , \,D_L, \, B_K  \Big] }=e^{-\frac{1}{2}\|g\|^2_{L^2}}(1+o_L(1)) \label{CHARE}\\
&\frac{ \E_{\nu^\perp_{Q, x_0}}  \Big[  e^{i\jb{\Im h^\perp, g_L} }  e^{\bar {\mathcal{E}}(h^\perp, t^{+}) } e^{c_0(t^+)^2+b(h^\perp)t^+}    , \,D_L, \, B_K  \Big]}{ \E_{\nu^\perp_{Q, x_0}}  \Big[ e^{\bar {\mathcal{E}}(h^\perp, t^{+}) } e^{c_0(t^+)^2+b(h^\perp)t^+}, \,D_L, \, B_K  \Big] }=e^{-\frac{1}{2}\|g\|^2_{L^2}}(1+o_L(1)),\label{CHAIM}
\end{align}

\noi
as $L\to \infty $, uniformly in $x_0 \in \T_{L}$ and $\dr \in [0,2\pi]$, where $g_L(x)=L^{-1/2}g(L^{-1}x)$ and $t^{+}$ is given in \eqref{eqn: t-oot}.
\end{proposition}

We present the proof of Proposition \ref{PROP:glim} in the next subsection.  In the following, we prove \eqref{CHARE} only, as the proof of \eqref{CHAIM} is identical.

\begin{remark}\rm 

In Proposition \ref{PROP:glim}, the mean-zero white noise limit holds only for test functions $g$ whose support is separated from the parameter $x_0 \in \T_{L}$, that is, $\text{dist}(\supp g,x_0)>0$. As explained in Subsection \ref{SUBSEC:strucpf} (in particular, Step 5), taking an average over both tangential directions $x_0$ and $\dr$, causes the region where the mean-zero white noise limit fails to become relatively small as the size of the circle grows. For the argument establishing the white noise limit over the entire space, see the proof of Proposition \ref{PROP:CHAa}.

\end{remark}


\subsection{Scaling limit of covariance}\label{SUBSEC:scalim}
Before proving Proposition \ref{PROP:gglimit}, we study the scaling behavior of the variance in this subsection.

By ignoring the high probability events $D_L$ and $B_K$ under $\nu^\perp_{Q,x_0}$ for the explanation, the main strategy in proving Proposition \ref{PROP:gglimit} is to show that $\jb{\Re h^\perp ,g_L}$ and $t^{+}$  are weakly correlated, after which we can apply the almost independent structure as follows
\begin{align}
&\frac{ \E_{\nu^\perp_{Q, x_0}}  \Big[  e^{i\jb{\Re h^\perp, g_L} }  e^{\bar {\mathcal{E}}(h^\perp, t^{+}) } e^{c_0(t^+)^2+b(h^\perp)t^+}     \Big]}{ \E_{\nu^\perp_{Q, x_0}}  \Big[ e^{\bar {\mathcal{E}}(h^\perp, t^{+}) } e^{c_0(t^+)^2+b(h^\perp)t^+}   \Big] } \notag \\
&\approx \frac{\E_{\nu^\perp_{Q, x_0}}\Big[ e^{i\langle \Re  h^\perp, g_L\rangle} \Big]  \E_{\nu^\perp_{Q, x_0}}\Big[
e^{\bar {\mathcal{E}}(h^\perp, t^{+}) } e^{c_0(t^+)^2+b(h^\perp)t^+}(1+o_L(1))   \Big]}{ \E_{\nu^\perp_{Q, x_0}}  \Big[ e^{\bar {\mathcal{E}}(h^\perp, t^{+}) } e^{c_0(t^+)^2+b(h^\perp)t^+} \Big] } \notag \\
&\approx \E_{\nu^\perp_{Q, x_0}}\Big[ e^{i\langle \Re h^\perp, g_L\rangle} \Big] 
\label{ALINDEP}
\end{align}

\noi
as $L\to \infty$. Hence, we can reduce the analysis of the conditional expectation to that of the expectation $\E_{\nu^\perp_{Q, x_0}}\Big[ e^{i\langle \Re h^\perp, g_L\rangle} \Big] $.

\begin{remark}\rm\label{REM:hperpRe}
For the imaginary part, we can follow the same procedure \eqref{ALINDEP} by showing that $\jb{\Im h^\perp,g_L}$ and $t^{+}$  are weakly correlated. Therefore, in the following, we may assume $h^\perp$ is real-valued, and unless otherwise specified,
the inner product $\jb{h^\perp, g_L }$ refers to $\jb{\Re h^\perp, g_L}$, that is,
\begin{align}
\jb{h^\perp, g_L }=\jb{\Re h^\perp, g_L}.
\label{hperpRe}
\end{align}
\end{remark}

\noi 
To achieve \eqref{ALINDEP}, under the measure $\nu^\perp_{Q, x_0}$, we make the following decomposition
\begin{align}
h^\perp(x)&= h^{\perp\perp}(x)+\frac{\E_{\nu^\perp_{Q, x_0}}\big[ \jb{h^\perp, g_L} h^\perp(x) \big]  }{\E_{\nu^\perp_{Q, x_0}} \big[  |\jb{h^\perp, g_L}|^2 \big]  }\jb{h^\perp, g_L} \notag \\
&=h^{\perp\perp}(x)+\al(x)\jb{h^\perp, g_L},
\label{K1}
\end{align}

\noi 
where $\al(x)$ denote the projection of $h^\perp$ onto $\jb{h^\perp, g_L}$
\begin{align}
\alpha(x):=\frac{\E_{\nu^\perp_{Q, x_0}}\big[ \jb{h^\perp, g_L} h^\perp(x) \big]  }{\E_{\nu^\perp_{Q, x_0}} \big[  |\jb{h^\perp, g_L}|^2 \big]  }.
\label{KK2}
\end{align}

\noi
This orthogonal decomposition \eqref{K1} implies that $h^{\perp \perp}$ and $\jb{h^\perp, g_L}$ are independent Gaussian. Therefore, the measure $\nu^\perp_{Q, x_0}$ can be decomposed as follows
\begin{align}
\mathrm{d}\nu^{\perp }_{Q, x_0}(h^{\perp})=\frac{1}{\sqrt{2\pi }\s_g }e^{-\frac{t^2}{2\s_g}}\, \mathrm{d}t \, \mathrm{d}\nu_{Q,x_0}^{\perp \perp}(h^{\perp \perp}),
\label{K2}
\end{align}

\noi
where
\begin{align}
\s_g^2:=\E_{\nu^\perp_{Q, x_0} }\big[ |\jb{h^\perp, g_L} |^2\big].
\label{K3}
\end{align}

In the following lemma, we first control the projection coefficient $\al(x)$ in \eqref{KK2} and the variance $\s_g^2$ in \eqref{K3}.

\begin{lemma}\label{LEM:varest}
Under the condition $\textup{dist}(\supp g, x_0)=M>0$, where $g$ is a real-valued, smooth, compactly supported function, we have 
\begin{align}
\s_g^2&=\E_{\nu^\perp_{Q, x_0}}\big[ |\jb{h^\perp, g_L}|^2 \big]=\|g \|_{L^2}^2(1+o_L(1))
\label{L80}
\end{align}

\noi 
as $L\to \infty$, uniformly in $x_0 \in \T_{L}, \dr\in [0,2 \pi]$, where $g_L(x)=L^{-\frac 12}g(L^{-1}x)$. Moreover, the projection coefficient $\al(x)$ satisfies
\begin{align}
\|  \al \|_{L^2(\T_{L^2})}&=O(1)
\label{L81}
\end{align}

\noi
as $L\to \infty$.

\end{lemma}

\begin{proof}

From the definition of the measure $\nu_{Q,x_0}^\perp$ in Lemma \ref{LEM:SCHOP} and Remark \ref{REM:hperpRe}, we can write 
\begin{align}
\E_{\nu^\perp_{Q, x_0}}\big[ |\jb{h^\perp , g_L}|^2 \big]&=\E_{\nu^\perp_{C_1^Q, x_0}}\big[ |\jb{\Re h^\perp , g_L}|^2 \big] \notag \\
&=\jb{g_L, G_{1,x_0} g_L}.
\label{VVV0}
\end{align}

\noi
Here, $G_{1,  x_0 }(x,y)$ is the Green's function  for the covariance operator $C_{1,x_0}^{Q}= \mathbf{P}_{\cj V^{L^2, \text{Re}}_{x_0,0}} (B_1^Q)^{-1}\mathbf{P}_{\cj V^{L^2, \text{Re}}_{x_0,0}} $, as defined in \eqref{CovC}, where $B_1^Q=-\dx^2-3(Q^{\eta_L}_{x_0})^2+\Ld$.

\noi 
From the resolvent identity, the Schr\"odinger operator  $C_{1, x_0}^Q$
can be viewed as a perturbation of the projected Ornstein–Uhlenbeck operator $\mathbf{P}_{\cj V^{L^2, \text{Re}}_{x_0,0}} (-\dx^2+\Ld)^{-1}\mathbf{P}_{\cj V^{L^2, \text{Re}}_{x_0,0}} $ as follows 
\begin{align}
G_{1,x_0}(x,y)=G_{\text{OU}}(x,y)-3(G_{\text{OU} }(Q_{x_0}^{\eta_L})^2G_{1,x_0})(x,y),
\label{SM00}
\end{align}

\noi
where 
\begin{align}
G_{\text{OU} }(x,y)=\big(\mathbf{P}_{\cj V^{L^2, \text{Re}}_{x_0,0}} (-\dx^2+\Ld)^{-1}\mathbf{P}_{\cj V^{L^2, \text{Re}}_{x_0,0}} \big)(x,y)
\label{GOUdef}
\end{align}

\noi 
is the Green’s function for the operator $\mathbf{P}_{\cj V^{L^2, \text{Re}}_{x_0,0}} (-\dx^2+\Ld)^{-1}\mathbf{P}_{\cj V^{L^2, \text{Re}}_{x_0,0}} $.

We first consider the first term $\jb{g_L, G_{\text{OU} } g_L }$ in \eqref{SM00}
Recall the definition of the projector $\mathbf{P}_{\cj V^{L^2, \text{Re}}_{x_0,0}} $ in \eqref{eqn: project} 
\begin{align}
\mathbf{P}_{\cj V^{L^2, \text{Re}}_{x_0,0}} =\Id-\P_{1,x_0}^{L}-\P_{2,x_0}^{L},
\label{SM0}
\end{align}

\noi
where each projection is 
\begin{align}
\P_{1,x_0}^{L}&= \langle Q_{1,x_0}^{\eta_L},\cdot\,\rangle Q_{1,x_0}^{\eta_L} \label{SM1}\\
\P_{2,x_0}^{L}&= \langle Q_{2,x_0}^{\eta_L},\cdot\,\rangle Q_{2,x_0}^{\eta_L}, \label{SM2}
\end{align}

\noi
where $Q_{1,x_0}^{\eta_L}$ and $Q_{2,x_0}^{\eta_L}$ are defined in \eqref{TQF1} and \eqref{TQF3}.  From \eqref{SM0}, we expand
\begin{align}
\mathbf{P}_{\cj V^{L^2, \text{Re}}_{x_0,0}} (-\dx^2+\Ld)^{-1}\mathbf{P}_{\cj V^{L^2, \text{Re}}_{x_0,0}} &=(-\dx^2+\Ld )^{-1}   +\sum_{1 \le i,j \le 2} \P_{i,x_0}^L(-\dx^2+\Ld )^{-1}\P_{j,x_0}^L \notag \\
&\hphantom{X}-\sum_{1\le j\le 2} (-\dx^2+\Ld )^{-1}  \P_{j,x_0}^L \notag \\
&\hphantom{X}-\sum_{1\le i\le 2}  \P_{i,x_0}^L (-\dx^2+\Ld )^{-1}. 
\label{SM4}
\end{align}

\noi
From \eqref{SM1} and \eqref{SM2},  we have 
\begin{align}
\jb {g_L, \P_{i,x_0}^L  (-\dx^2+\Ld)^{-1}   \P_{j,x_0}^L g_L} &=\jb{Q_{i,x_0}^{\eta_L}, g_L} \jb{Q_{j,x_0}^{\eta_L}, g_L} \jb{ Q_{i,x_0}^{\eta_L},  (-\dx^2+\Ld)^{-1}  Q_{j,x_0}^{\eta_L} } \notag \\
\jb {g_L,   (-\dx^2+\Ld)^{-1}   \P_{j,x_0}^L g_L}&= \jb {g_L,   (-\dx^2+\Ld)^{-1}  Q_{j,x_0}^{\eta_L} }  \jb{Q_{j,x_0}^{\eta_L}, g_L}  \notag \\
\jb {g_L, \P_{i,x_0}^L  (-\dx^2+\Ld)^{-1}  g_L}&= \jb {Q_{i,x_0}^{\eta_L},   (-\dx^2+\Ld)^{-1} g_L }  \jb{Q_{i,x_0}^{\eta_L}, g_L} .
\label{V4}
\end{align}

\noi
Recall that $Q^{\eta_L}_{i,x_0}$ for $i=1,2$ is localized around $Lx_0$ (see Remak \ref{REM:Lx_0}) and has an exponential tail.  By a slight abuse of the notation, we write $Q^{\eta_L}_{i,x_0}(\cdot-Lx_0)$. Then, for $i=1,2$, recalling $g_L(x)=L^{-\frac 12}g(L^{-1}x)$, we have 
\begin{align}
\jb{Q_{i,x_0}^{\eta_L}, g_L}&=L^\frac 12\int_{\supp g} Q_{i,x_0}^{\eta_L}(L(x-x_0))g(x)\, \mathrm{d}x \notag \\
& \les L^\frac 12 \int_{\supp g} e^{-L|x-x_0|}g(x)\, \mathrm{d}x\notag \\
&\les L^\frac 12 e^{-LM }\|g \|_{L^1},
\label{V5}
\end{align}

\noi
where $M=\text{dist}(\supp g, x_0) >0$. From \eqref{GOUdef}, \eqref{SM4}, and \eqref{V4}, we obtain 
\begin{align}
\jb{g_L, G_{\text{OU}}  g_L}=\jb{g_L,  (-\dx^2+\Ld)^{-1} g_L }+O(e^{-\frac{M}{2}L}).
\label{VVV5}
\end{align}

\noi
Due to the exponential decay of correlations for the operator $(-\dx^2+\Ld)^{-1}$ in \eqref{OUdec}, we have
\begin{align}
\jb{g_L, G_{\text{OU}} g_L }&= \int_{\supp g} \int_{\supp g} L e^{-c\text{dist}(L(x-y), 2L^2 \Z ) } g(x) g(y) dx dy+O(e^{-\frac{M}{2}L}) \notag \\
&= \| g\|_{L^2}^2(1+O(L^{-2}))+O(e^{-\frac{M}{2}L}) 
\label{VVVV5}
\end{align}

\noi
as $L\to \infty$.

\noi 
Next, we consider the second term  $\jb{g_L, G_{\text{OU} } (Q^{\eta_L}_{x_0})^2 G_{1,x_0} g_L }$ on the right-hand side of \eqref{SM00}. By using the exponential correlation decay of $G_{\text{OU}}$ and $G_{1,x_0}$ in \eqref{OUDEC0} and  Lemma \ref{lem: covariance bound}, along with the fact that $Q^{\eta_L}_{i,x_0} $ is localized at $Lx_0$ with an exponential tail, we have 
\begin{align}
&|\jb{g_L, G_{\text{OU} } (Q^{\eta_L}_{x_0})^2 G_{1,x_0} g_L }| \notag \\
&=\bigg|\int_{-L^2}^{L^2} g_L(x) \int_{-L^2}^{L^2}  G_{\text{OU}}(x,z)(Q^{\eta_L}_{x_0})^2(z) \int_{-L^2}^{L^2} G_{1,x_0}(z,y)g_L(y)\, \mathrm{d}y \, \mathrm{d}z\, \mathrm{d}x\bigg| \notag \\
&\les\bigg| \int_{-L^2}^{L^2} g_L(x) \int_{-L^2}^{L^2} \int_{-L^2}^{L^2} e^{-\text{dist}(x-z, 2L^2 \Z) } e^{-|z-Lx_0|} e^{-\text{dist}(z-y, 2L^2 \Z) } g_L(y) \, \mathrm{d}y \, \mathrm{d}z \, \mathrm{d}x \bigg| \notag \\
&\les\bigg| L^2\int_{-L^2}^{L^2} \int_{-L^2}^{L^2} e^{-\frac 14\text{dist}(x-Lx_0), 2L^2 \Z) } e^{-\frac 14\text{dist}(y-Lx_0), 2L^2 \Z) } g_L(y) g_L(x)\, \mathrm{d}y\, \mathrm{d}x \bigg| \notag \\
&\les \bigg| L^2 \int_{\supp g} \int_{\supp g} e^{-\frac 14\text{dist}(L(x-x_0), 2L^2 \Z) } e^{-\frac 14\text{dist}(L(y-x_0), 2L^2 \Z) } g(y) g(x)\, \mathrm{d}y\, \mathrm{d}x\bigg|.
\label{V6}
\end{align}

\noi 
Combining  \eqref{V6} and $\text{dist}(\supp g, x_0)=M>0$, we obtain
\begin{align}
|\jb{g_L, G_{\text{OU} } (Q^{\eta_L}_{x_0})^2 G_{1,x_0} g_L }| \les L^2e^{-cLM} \|g \|_{L^1} \|g \|_{L^1}
\label{VV6}
\end{align}

\noi
for some constant $c>0$. Therefore, it follows from   \eqref{SM00} \eqref{VVVV5}, and \eqref{VV6} that  
\begin{align}
\jb {g_L, G_{1,x_0} g_L}&=\jb{g_L, G_{\text{OU}}g_L}-3\jb{g_L, G_{\text{OU} }(Q_{x_0}^{\eta_L})^2G_{1,x_0}g_L} \notag \\
&=\|g \|_{L^2}^2(1+o_L(1))
\label{SM5}
\end{align}

\noi
as $L\to \infty$. Therefore, \eqref{VVV0} and \eqref{SM5} imply that 
\begin{align*}
\E_{\nu^\perp_{Q, x_0}}\big[ |\jb{h^\perp , g_L}|^2 \big]&=\jb{g_L, G_{1,x_0} g_L}\\
&=\|g \|_{L^2}^2(1+o_L(1))
\end{align*}

\noi
as $L\to \infty$. This completes the proof of \eqref{L80}.

It remains to prove \eqref{L81}. Recall the definition of $\al(x)$ in \eqref{KK2}
\begin{align}
\alpha(x):=\frac{\E_{\nu^\perp_{Q, x_0}}\big[ \jb{h^\perp, g_L} h^\perp(x) \big]  }{\E_{\nu^\perp_{Q, x_0}} \big[  |\jb{h^\perp, g_L}|^2 \big]  }.
\label{V10}
\end{align}

\noi
Thanks to \eqref{L80}, we can control the denominator in \eqref{V10}. Hence, it is enough to show that $\|\E_{\nu^\perp_{Q, x_0}}\big[ \jb{h^\perp, g_L} h^\perp(x) \big]\|_{L^2}=O(1)$ as $L\to \infty$. Note that
\begin{align}
\E_{\nu^\perp_{Q, x_0}}\big[ \jb{h^\perp, g_L} h^\perp(x) \big]=\int_{-L^2}^{L^2} \E_{\nu^\perp_{Q, x_0}}\big[ h^\perp(x) \overline{ h^\perp}(y) \big] g_L(y)\, \mathrm{d}y.
\label{VV11}
\end{align}

\noi 
By applying the exponential correlation decay in Lemma \ref{lem: covariance bound}, we have
\begin{align}
\int_{-L^2}^{L^2}\bigg|\int_{-L^2}^{L^2} \E_{\nu^\perp_{Q, x_0}}\big[ h^\perp(x) \overline{ h^\perp}(y) \big] g_L(y)\, \mathrm{d}y \bigg|^2 \, \mathrm{d}x  &\les \int_{-L}^{L}\bigg|\int_{\supp g}  Le^{-\text{dist}(L(x-y), 2L^2 \Z) } g(y)\, \mathrm{d}y \bigg|^2 \, \mathrm{d}x   \notag \\
&\too \int_{\R} |g(x)|^2 \, \mathrm{d}x=\| g\|_{L^2}^2
\label{V12}
\end{align}

\noi
as $L\to \infty$. Then, \eqref{VV11} and \eqref{V12} imply that $\|\al \|_{L^2_x(\T_{L^2})}=O(1)$ as $L\to \infty$. This completes the proof of Lemma \ref{LEM:varest}.

\end{proof}

\begin{remark}\rm
In Lemma \ref{LEM:varest}, the correlation analysis shows that if the test function $g$ has support at a  distance from the tangential component  $x_0 \in \T_{L}$, then we have
\begin{align*}
\E_{\nu^\perp_{C_1^Q, x_0}}\big[ |\jb{\Re h^\perp , g_L}|^2 \big]&=\jb{g_L, G_{1,x_0} g_L}\approx  \jb{g_L, G_{\text{OU}} g_L }\\
\E_{\nu^\perp_{C_2^Q, x_0}}\big[ |\jb{\Im h^\perp , g_L}|^2 \big]&=\jb{g_L, G_{2,x_0} g_L}\approx  \jb{g_L, G_{\text{OU}} g_L }
\end{align*}

\noi
as $L\to \infty$.  In other words, for the Schrödinger operators $B_1^Q$, $B_2^Q$ in \eqref{SCHOP}, the effect of the ground states
$-3(Q^{\eta_L}_{x_0})^2$ and $-(Q^{\eta_L}_{x_0})^2$ is negligible as there is almost no interaction between the ground state and the scaled test function $g_L$. In particular, the scaling $g_L(x)=L^{-\frac 12}g(L^{-1}x)$, which arises from $h_L(x)=L^\frac 12 h(Lx)$ in Proposition \ref{PROP:REDUC}, 
also ensures the convergence of the Ornstein-Uhlenbeck operator $(-\dx^2+\Ld)^{-1}$ to white noise.

\end{remark}

Following the arguments in the proof of Lemma \ref{LEM:varest}, we can establish the following lemma, which will be used later.

\begin{lemma}\label{LEM:bhperp}
Under the condition $\textup{dist}(\supp g, x_0)=c_1>0$, where $g$ is smooth with compact support and $x_0 \in \T_L$, $b(h^{\perp })$ in \eqref{bhper} can be expressed in the coordinate $(h^{\perp \perp}, s)$ in \eqref{K1}
\begin{align*}
b(h^{\perp })=3\int_{\T_{L^2}} (Q^{\eta_L}_{x_0} )^2 \g(x) h^{\perp \perp}(x)\, \mathrm{d}x+O(e^{-\frac{c_1}{2}L })s
\end{align*}

\noi
as $L\to \infty$, where $s$ is the symbol for the Gaussian random variable $\jb{h^\perp, g_L }$ and $g_L(x)=L^{-\frac 12}g(L^{-1}x)$.

\end{lemma}

\begin{proof}
Under the orthogonal decomposition $h^\perp(x)=h^{\perp \perp}(x)+\al(x)\jb{h^\perp, g_L}$ from \eqref{K1}, we write 
\begin{align}
b(h^{\perp })&=3\int_{\T_{L^2}} (Q^{\eta_L}_{x_0} )^2 \g(x) h^{\perp }(x)\, \mathrm{d}x \notag \\
&=3\int_{\T_{L^2}} (Q^{\eta_L}_{x_0} )^2 \g(x) h^{\perp \perp }(x)\, \mathrm{d}x+3\jb{h^\perp, g_L}\int_{\T_{L^2}} (Q^{\eta_L}_{x_0} )^2 \g(x) \al(x)\, \mathrm{d}x.
\label{lemhperp0}
\end{align}

\noi
Note that from the definition of \eqref{KK2} and Lemma \ref{LEM:varest},
\begin{align*}
\al(x)\sim \E_{\nu^\perp_{Q, x_0}}\big[ \jb{h^\perp, g_L} h^\perp(x) \big]=\int_{\T_{L^2}} G(x,y)g_L(y)\, \mathrm{d}y=L^{\frac 12}\int_{\T_L} G(x,Ly)g(y)\,\mathrm{d}y.
\end{align*}

\noi
Since $Q^{\eta_L}_{x_0}$ and $\g(x)$ in \eqref{ga0} are exponentially localized around $Lx_0$, we can write $(Q^{\eta_L}_{x_0} )^2\g(x)=\psi(x-Lx_0)$ for some Schwartz function $\psi$. This, together with Lemma \ref{lem: covariance bound}, implies
\begin{align}
\int_{\T_{L^2}} (Q^{\eta_L}_{x_0} )^2 \g(x) \al(x)\, \mathrm{d}x&=\int_{\T_{L^2}} \psi(x-Lx_{0}) \int_{\T_L}L^{\frac 12} G(x,Ly)g(y)\,\mathrm{d}y\, \mathrm{d}x\notag \\
&=L^{\frac 32}\int_{\T_{L}} \psi(L(x-x_{0})) \int_{\T_L}G(Lx,Ly)g(y)\,\mathrm{d}y\, \mathrm{d}x \notag \\
&\les L^{\frac 32} \int_{\T_L } e^{-L|x-x_0|} \int_{\T_L} e^{-L|x-y|} g(y)\, \mathrm{d}y\, \mathrm{d}x \notag \\
&\les L^{\frac 32} \int_{\T_L} \int_{\T_L} e^{-L|y-x_0|} g(y)\, \mathrm{d}y \, \mathrm{d}x \notag \\
&\les \| g\|_{L^2} e^{-\frac{c_1}{2}L },
\label{lemhperp1}
\end{align}

\noi
where $c_1=\text{dist}(\supp g, x_0)>0$. By combining \eqref{lemhperp0} and \eqref{lemhperp1}, we complete the proof of Lemma \ref{LEM:bhperp}.

\end{proof}

\subsection{Proof of the Gaussian limit}\label{SUBSEC:Gaussianlimit}
In this subsection, by combining the ingredients in the previous and this subsection, we present the proof of Proposition \ref{PROP:glim}.

The main ingredient in this subsection is the following Gaussian limit.
Before we start the proof, recall from Remark \ref{REM:hperpRe} that we are considering the case  $h^\perp=\Re h^\perp$.

\begin{proposition}\label{PROP:glim}
Let $g$ be a smooth, compactly supported function with $\textup{dist}(\supp g, x_0)>0$, where $x_0\in \T_{L}$. Then, we have
\begin{align}
\frac{ \E_{\nu^\perp_{Q, x_0}}  \Big[  e^{i\jb{h^\perp, g_L} }  e^{\bar {\mathcal{E}}(h^\perp, t^{+}) } e^{c_0(t^+)^2+b(h^\perp)t^+}, \, D_L, \, B_K     \Big]}{ \E_{\nu^\perp_{Q, x_0}}  \Big[ e^{\bar {\mathcal{E}}(h^\perp, t^{+}) } e^{c_0(t^+)^2+b(h^\perp)t^+}, \, D_L, \, B_K   \Big] }=e^{-\frac{1}{2}\|g\|^2_{L^2}}(1+o_L(1))
\label{Glim0}
\end{align}

\noi
as $L\to \infty $, uniformly in $x_0 \in \T_{L}$ and $\dr \in [0,2\pi]$, where $g_L(x)=L^{-1/2}g(L^{-1}x)$. 
\end{proposition}

\begin{proof}

We first separate the main term and the error term in the numerator in \eqref{Glim0}
\begin{align}
\E_{\nu^\perp_{Q, x_0}}  \Big[  e^{i\jb{h^\perp, g_L} }  e^{\bar {\mathcal{E}}(h^\perp, t^{+}) } e^{c_0(t^+)^2+b(h^\perp)t^+}, \, D_L, \, B_K     \Big].
\label{Glim1}
\end{align}

\noi
Note that 
\begin{align}
\eqref{Glim1}&=\E_{\nu^\perp_{Q, x_0}}  \Big[  e^{i\jb{h^\perp, g_L} }  e^{\bar {\mathcal{E}}(h^\perp, t^{+}) } e^{c_0(t^+)^2+b(h^\perp)t^+}, \, D_L, \, B_K, \, |\jb{h^\perp, g_L }| \le L^\eps      \Big] \notag \\
&\hphantom{X}+\E_{\nu^\perp_{Q, x_0}}  \Big[  e^{i\jb{h^\perp, g_L} }  e^{\bar {\mathcal{E}}(h^\perp, t^{+}) } e^{c_0(t^+)^2+b(h^\perp)t^+}, \, D_L, \, B_K, \,  L^\eps \le |\jb{h^\perp, g_L }| \le M L^\frac 12      \Big] \notag \\
&\hphantom{X}+\E_{\nu^\perp_{Q, x_0}}  \Big[  e^{i\jb{h^\perp, g_L} }  e^{\bar {\mathcal{E}}(h^\perp, t^{+}) } e^{c_0(t^+)^2+b(h^\perp)t^+}, \, D_L, \, B_K, \,   |\jb{h^\perp, g_L }| \ge M L^\frac 12      \Big] \notag \\
&=\I_1+\I_2+\I_3. 
\label{GGlim1}
\end{align}

\noi
Here, $\I_1$ is the main contribution, while $\I_2, \I_3$ are treated as error terms.

We first study the error term $\I_3$. From \eqref{errht}, \eqref{LL7}, \eqref{toot}, and \eqref{bhper}, we have that on the set $D_L\cap B_K$
\begin{align}    
|\bar {\mathcal{E}}(h^\perp, t^{+})|&=O(1)\label{GGGlim1}\\
|t^{+}|&=O(L^\frac 12)\label{GGGGlim1}\\
|b(h^\perp)|&=O(L^\frac 12).
\end{align}

\noi
This implies that on the set $D_L \cap B_K$
\begin{align}
e^{\bar {\mathcal{E}}(h^\perp, t^{+}) } e^{c_0(t^+)^2+b(h^\perp)t^+} \les e^{cL}
\label{Glim2}
\end{align}

\noi
for some constant $c>0$. Recall the orthogonal decomposition $h^\perp(x)=h^{\perp \perp}(x)+\al(x)\jb{h^\perp, g_L}$ from \eqref{K1}, where $h^{\perp \perp}$ and $\jb{h^\perp, g_L }$ are independent Gaussian. 
Using independence together with \eqref{Glim2}, we have
\begin{align}
|\I_3| \les e^{cL} \int  \int_{|s| \ge ML^\frac 12} e^{-\frac{s^2}{2\s_g^2}} \, \frac{\mathrm{d}s}{\sqrt{2\pi}\s_g}\,  d\nu^{\perp \perp}_{Q, x_0}(h^{\perp \perp}) \les  e^{cL} e^{-M^2 L}\les e^{-\frac{M^2}{2} L} 
\label{Glim4}
\end{align}

\noi
uniformly in $x_0 \in \T_{L}$ and $\dr \in [0,2\pi]$, where $\s_g^2:=\E_{\nu^\perp_{Q, x_0} }\big[ |\jb{h^\perp, g_L} |^2\big]$  and $M$ is chosen to be sufficiently large.

Regarding $\I_1$ and $\I_2$ in \eqref{GGlim1}, we express them in terms of the independent Gaussian fields $\jb{h^\perp, g_L}$ and $h^{\perp \perp}$. Define  
\begin{align*}
\cj D_L&=\big\{  \| h^{\perp \perp} \|_{L^2(\T_{L^2})}^2 \le 4DL^2 \big\}\\
\cj B_K&=\big\{  \| h^{\perp \perp} \|_{L^2(\T_{L^2})} \le 2K\sqrt{\log L^2} \big\},
\end{align*}

\noi
where $K=M\sqrt{\frac{L}{\log L^2}}$. Then, under $h^\perp(x)=h^{\perp \perp}(x)+\al(x)\jb{h^\perp, g_L}$ in \eqref{K1} and $|\jb{h^\perp, g_L}|\le ML^\frac 12$, we have  $D_L \subset \cj D_L$ and $B_K \subset \cj B_K$. This implies that 
$\ind_{D_L}=\ind_{\cj D_L}-\ind_{\cj D_L \setminus D_L}$ and $\ind_{B_K}=\ind_{\cj B_K}-\ind_{\cj B_K \setminus B_K}$.
Hence, we can write 
\begin{align}
\ind_{D_L}\ind_{B_K}=\ind_{\cj B_K} \ind_{\cj D_L}-\ind_{\cj D_L} \ind_{\cj B_K \setminus B_K}-\ind_{\cj B_K} \ind_{\cj D_L \setminus D_L}+\ind_{\cj B_K \setminus B_K} \ind_{\cj D_L \setminus D_L}.
\end{align}

\noi
Let $I_1(A,C)$ denote the expectation in \eqref{GGlim1} over the sets $A$ and $C$. From  \eqref{Glim2},  Proposition \ref{prop: BK-prob}, and Proposition \ref{PROP: gaussian-conc} 
\begin{align}
&|\I_1(\cj D_L, \cj B_K \setminus B_K ) +\I_1(\cj B_K, \cj D_L \setminus D_L ) +\I_1( \cj B_K \setminus B_K  ,\cj D_L \setminus D_L ) | \notag \\
&\les e^{cL}(e^{-M^2 L}+e^{-D L^2}) \notag \\
&\les e^{-\frac{M^2}{2}L},
\label{GGlim5}
\end{align}

\noi
uniformly in $x_0 \in \T_{L}$ and $\dr \in [0,2\pi]$, where $M$ is chosen to be sufficiently large.

Hence, it suffices to study the main part $\cj D_L \cap \cj B_K$
\begin{align}
\I_1(\cj D_L, \cj B_K)=\int_{\cj D_L \cap \cj B_K} \int_{|s|\le L^\eps } e^{is }  e^{\cj {\mathcal{E}}_1(h^{\perp \perp}, s) } e^{\cj {\mathcal{E}}_2(h^{\perp \perp}, s) } e^{-\frac{s^2}{2\s_g^2}} \, \frac{\mathrm{d}s}{\sqrt{2\pi} \s_g}\, \mathrm{d}\nu^{\perp \perp}_{Q, x_0}(h^{\perp \perp}),
\label{GGlim3}
\end{align}

\noi
where the coordinate $(h^{\perp \perp}, s)$ with  $h^\perp(x)=h^{\perp \perp}(x)+\al(x)s$ in \eqref{K1}, and $t^{+}=t^{+}(h^{+})$ are used to denote
\begin{align}
\cj {\mathcal{E}}_1(h^{\perp \perp}, s):&=\bar {\mathcal{E}}(h^\perp, t^{+}) \notag  \\
\cj {\mathcal{E}}_2(h^{\perp \perp}, s):&=c_0(t^+)^2+b(h^\perp)t^+.
\label{cjE2}
\end{align}

\noi
From \eqref{errht} and \eqref{Eerror}, we can write 
\begin{align}
\bar {\mathcal{E}}_1(h^{\perp \perp}, s)&=\bar {\mathcal{E}}_{3}(h^{\perp \perp}, s)+\bar {\mathcal{E}}_{4}(h^{\perp \perp}, s),
\label{GGGlim4}
\end{align}

where
\begin{align*}
\bar {\mathcal{E}}_{4}(h^{\perp \perp}, s)&=\frac 1{L^3} \int_{\T_{L^2}} |h^{\perp \perp}(x)+\al(x) s +t^{+}\g(x)|^4 \, \mathrm{d}x\\
\bar {\mathcal{E}}_{3}(h^{\perp \perp}, s)&=\frac 1{L^{\frac 32}} \int_{\T_{L^2}} |h^{\perp \perp}(x)+\al(x) s +t^{+}\g(x)|^2(h^{\perp \perp}(x)+\al(x)s +t^{+}\g(x))Q^{\eta_L}_{x_0}(x) \, \mathrm{d}x.
\end{align*}

\noi
Note that from $h^\perp(x)=h^{\perp \perp}(x)+\al(x)s$ in \eqref{K1},
\begin{align}
\|h^\perp\|^2_{L^2}= \|h^{\perp\perp}\|_{L^2}^2+2\Re \langle h^{\perp\perp},  \cj \alpha\rangle  s+|s|^2 \|\alpha\|_{L^2}^2.
\label{GGGglim4}
\end{align}

\noi 
From \eqref{eqn: t-oot} and \eqref{GGGglim4}, we obtain that on the set $\cj D_L \cap \cj B_K$ and $|s| \le M L^\frac 12$,
\begin{align}
t^{+}=- \frac{\| h^{\perp \perp} \|_{L^2}^2}{2L^\frac 32}+O(L^{-\frac 12})=O(L^\frac 12).
\label{GGlim4}
\end{align}

Under $\cj D_L \cap \cj B_K$,  $|s| \le ML^\frac 12 $, and \eqref{GGlim4}, the only term in the quartic expression $\bar {\mathcal{E}}_{4}(h^{\perp \perp}, s)$ that does not go to zero as $L \to\infty$ is 
\begin{align}
V_4(h^{\perp \perp}):=\frac{1}{L^3} \int_{\T_{L^2}} |h^{\perp \perp}|^4 \, \mathrm{d}x=O(1)
\label{Glim5}
\end{align}

\noi
since $\| h^{\perp \perp} \|_{L^2}^4 \le \|  h^{\perp \perp}\|_{L^2}^2 \|  h^{\perp \perp}\|_{L^\infty}^2=O(L^3)$, and $t^{+}=O(L^\frac 12)$ from \eqref{GGGGlim1}. 
Hence, we write 
\begin{align}
\cj {\mathcal{E}}_4(h^{\perp \perp}, s)=V_4(h^{\perp \perp })+O(L^{-\frac 12}),
\label{Glim6}
\end{align}

\noi
where the remaining terms going to zero as $L\to \infty$ are included in $O(L^{-\frac 12})$.

On the one hand, under $\cj D_L \cap \cj B_K$ and  $|s| \le L^\eps $,  the terms in the cubic expression $\bar {\mathcal{E}}_{3}(h^{\perp \perp}, s)$ that do not go to zero as $L \to\infty$ are  
\begin{align}
V_3(h^{\perp \perp}):&=\frac{1}{L^\frac 32 }  \int_{\T_{L^2}} |h^{\perp \perp}(x)|^2 h^{\perp \perp}(x) Q^{\eta_{L}}_{x_0}(x) \, \mathrm{d}x+ \frac{1}{L^\frac 32 }  \int_{\T_{L^2}} |h^{\perp \perp}(x)|^2 t^{+} \g(x) Q^{\eta_L}_{x_0}(x)\, \mathrm{d}x \notag \\
&+\hphantom{X} \frac{1}{L^\frac 32 }  \int_{\T_{L^2}} h^{\perp \perp}(x) (t^{+} \g(x))^2 Q^{\eta_L}_{x_0}(x)\, \mathrm{d}x+\frac{1}{L^\frac 32 }  \int_{\T_{L^2}}  (t^{+} \g(x))^3 Q^{\eta_L}_{x_0}(x)\, \mathrm{d}x \notag\\
&=O(1)
\label{Glim6}
\end{align}

\noi
since $\| h^{\perp \perp} \|_{L^\infty}=O( \sqrt{L})$ and $t^{+}=O(L^\frac 12)$ from \eqref{GGGGlim1}. Hence, we write 
\begin{align}
\cj {\mathcal{E}}_3(h^{\perp \perp}, s)=V_3(h^{\perp \perp })+O(L^{-\zeta})
\end{align}

\noi
for some $\zeta>0$, where
the remaining terms going to zero as $L\to \infty$ are included in $O(L^{-\zeta})$.

From Lemma \ref{LEM:bhperp}, under $\cj D_L \cap \cj B_K$,  the terms in  $\bar {\mathcal{E}}_{2}(h^{\perp \perp}, s)$ in \eqref{cjE2} that do not vanish as $L \to\infty$ are  
\begin{align}
V_1(h^{\perp \perp})=3t^{+}\int_{\T_{L^2}} (Q^{\eta_L}_{x_0})^2 \g(x) h^{\perp \perp}(x)\, \mathrm{d}x+c_0 (t^{+})^2=O(L)
\label{GGlim6}
\end{align}

\noi
since $\| h^{\perp \perp} \|_{L^\infty}=O( \sqrt{L})$ and $t^{+}=O(L^\frac 12)$ from \eqref{GGGGlim1}. Therefore, we write 
\begin{align}
\cj {\mathcal{E}}_2(h^{\perp \perp}, s) =V_1(h^{\perp \perp})+O(e^{-cL})s,
\label{GGGlim6}
\end{align}

\noi
where $O(e^{-cL})s$ comes from Lemma \ref{LEM:bhperp}.

\noi
Hence, combining \eqref{GGGlim4}, \eqref{Glim5}, \eqref{Glim6}, and \eqref{GGlim6} shows that under $\cj D_L \cap \cj B_K$ and  $|s| \le L^\eps $,
\begin{align}
\cj {\mathcal{E}}_1(h^{\perp \perp}, s)+\cj {\mathcal{E}}_2(h^{\perp \perp}, s)=V(h^{\perp \perp})+O(L^{-\zeta})
\label{Glim7}
\end{align}

\noi
for some small $\zeta>0$, where
\begin{align*}
V(h^{\perp \perp}):=V_3(h^{\perp \perp})+V_4(h^{\perp \perp})+V_1(h^{\perp \perp})
\end{align*}

\noi
with  $V_3(h^{\perp \perp})=O(1)$, $V_4(h^{\perp \perp})=O(1)$, and $V_1(h^{\perp \perp})=O(L)$.



We are now ready to study $\I_1(\cj D_L, \cj B_K)$ in \eqref{GGlim3}.
In $\I_1(\cj D_L, \cj B_K)$, we also separate the main contribution and the error terms as follows
\begin{align}
&\I_1(\cj D_L, \cj B_K)\notag \\
&=\int\limits_{\cj D_L \cap \cj B_K} \int_{|s|\le L^\eps } e^{is }  e^{-\frac{s^2}{2\s_g^2}} \, \frac{\mathrm{d}s}{\sqrt{2\pi} \s_g}\,  e^{V(h^{\perp \perp})} \mathrm{d}\nu^{\perp \perp}_{Q, x_0}(h^{\perp \perp}) \notag \\
&+ \int\limits_{\cj D_L \cap \cj B_K} \int_{|s|\le L^\eps } e^{is } ( e^{\cj {\mathcal{E}}_1(h^{\perp \perp}, s) +\cj {\mathcal{E}}_2(h^{\perp \perp}, s) }-e^{V(h^{\perp \perp}) }) e^{-\frac{s^2}{2\s_g^2}} \, \frac{\mathrm{d}s}{\sqrt{2\pi} \s_g} \, \mathrm{d}\nu^{\perp \perp}_{Q, x_0}(h^{\perp \perp})\notag \\
&=\I_{11}(\cj D_L, \cj B_K)+\I_{12}(\cj D_L, \cj B_K).
\label{Glim8}
\end{align}

\noi 
From \eqref{Glim8} and \eqref{Glim7}, we have 
\begin{align}
&|\I_{12}(\cj D_L, \cj B_K)|\notag \\
&=\bigg| \int\limits_{\cj D_L \cap \cj B_K} \int_{|s|\le L^\eps } e^{is } e^{V(h^{\perp \perp}) } ( e^{O(L^{-\zeta}) }-1) e^{-\frac{s^2}{2\s_g^2}}  \, \frac{\mathrm{d}s}{\sqrt{2\pi} \s_g} \,  \mathrm{d}\nu^{\perp \perp}_{Q, x_0}(h^{\perp \perp})\bigg| \notag \\
&=O(L^{-\zeta}) \E_{\nu^{\perp \perp}_{x_0}}\Big[ e^{V(h^{\perp \perp}) }, \, \cj D_L, \, \cj B_K  \Big].
\label{GGlim8}
\end{align}

\noi
We now study the main term $\I_{11}(\cj D_L, \cj B_K)$ in \eqref{Glim8}. Note that 
\begin{align}
\int_{|s|\le L^\eps } e^{is }  e^{-\frac{s^2}{2\s_g^2}} \, \frac{\mathrm{d}s}{\sqrt{2\pi} \s_g}=\int_{\R} e^{is }  e^{-\frac{s^2}{2\s_g^2}} \, \frac{\mathrm{d}s}{\sqrt{2\pi} \s_g}+O(e^{-L^{2\eps}}).
\label{Glim9}
\end{align}

\noi 
Thanks to Lemma \ref{LEM:varest}, we have
\begin{align}
\int_{\R} e^{is} e^{-\frac{s^2}{2\s_g^2}}\, \frac{\mathrm{d}x}{\sqrt{2\pi}\s_g}=e^{-\frac {\s_g^2}2}=e^{-\frac 12 \| g\|_{L^2}^2(1+o_L(1))}
\label{Glim10}
\end{align}

\noi
as $L\to \infty$. Hence, combining \eqref{Glim9} and \eqref{Glim10} shows 
\begin{align}
\I_{11}(\cj D_L, \cj B_K)=\big(e^{-\frac{1}{2} \|g \|_{L^2}^2(1+o_L(1)) }+O(e^{-L^{2\eps }) } \big) \E_{\nu^{\perp \perp}_{Q, x_0}}\Big[ e^{V(h^{\perp \perp}) }, \, \cj D_L, \, \cj B_K  \Big].
\label{Glim11}
\end{align}

\noi
Combining \eqref{Glim8}, \eqref{GGlim8}, and \eqref{Glim11} yields 
\begin{align}
\I_1(\cj D_L, \cj B_K)=\big(e^{-\frac{1}{2} \|g \|_{L^2}^2(1+o_L(1)) }+O(L^{-\zeta}) \big) \E_{\nu^{\perp \perp}_{Q, x_0}}\Big[ e^{V(h^{\perp \perp}) }, \, \cj D_L, \, \cj B_K  \Big].
\label{Glim17}
\end{align}

\noi
Hence, we obtain the desired result for the term $\I_1$.

We now study the term $\I_2$ in \eqref{GGlim1}.
By proceeding as in \eqref{GGlim5}, it suffices to study the part $\cj D_L \cap \cj B_K$ as follows
\begin{align}
\I_2(\cj D_L, \cj B_K)=\int_{\cj D_L \cap \cj B_K} \int_{L^\eps \le|s|\le ML^\frac 12 } e^{is }  e^{\cj {\mathcal{E}}_1(h^{\perp \perp}, s) } e^{\cj {\mathcal{E}}_2(h^{\perp \perp}, s) } e^{-\frac{s^2}{2\s_g^2}} \, \frac{\mathrm{d}s}{\sqrt{2\pi} \s_g}\, \mathrm{d}\nu^{\perp \perp}_{Q, x_0}(h^{\perp \perp}).
\end{align}

\noi

By using $\cj {\mathcal{E}}_1(h^{\perp \perp}, s)=O(1) $ from \eqref{GGGlim1},   $\cj {\mathcal{E}}_2(h^{\perp \perp}, s) =V_1(h^{\perp \perp})+O(e^{-cL})s$ from \eqref{GGGlim6}, and 
\begin{align*}
\int_{ L^\eps \le |s| \le M L^\frac 12 }  e^{-\frac{s^2}{2\s_g^2}} \, \frac{\mathrm{d}s}{\sqrt{2\pi} \s_g}=O(e^{-L^{2\eps}}),
\end{align*}

\noi 
we have 
\begin{align}
\I_{2}(\cj D_L, \cj B_K)=O(e^{-L^{2\eps}})\E_{\nu^{\perp \perp}_{Q, x_0}}\Big[ e^{V_1(h^{\perp \perp}) }, \, \cj D_L, \, \cj B_K  \Big].
\label{Glim18}
\end{align}

In summary, from \eqref{GGlim1}, \eqref{Glim4}, \eqref{GGlim5}, \eqref{Glim17} and \eqref{Glim18}, we obtain 
\begin{align}
&\E_{\nu^\perp_{Q, x_0}}  \Big[  e^{i\jb{h^\perp, g_L} }  e^{\bar {\mathcal{E}}(h^\perp, t^{+}) } e^{c_0(t^+)^2+b(h^\perp)t^+}, \, D_L, \, B_K     \Big]=\I_1+\I_2+\I_3 \notag \\
&=\big(e^{-\frac{1}{2} \|g \|_{L^2}^2 }+o_L(1) \big) \E_{\nu^{\perp \perp}_{Q, x_0}}\Big[ e^{V(h^{\perp \perp}) }, \, \cj D_L, \, \cj B_K  \Big]+o_L(1)\E_{\nu^{\perp \perp}_{Q, x_0}}\Big[ e^{V_1(h^{\perp \perp}) }, \, \cj D_L, \, \cj B_K  \Big]+O(e^{-M^2 L}).
\label{Glim14}
\end{align}

We now look into the denominator in \eqref{Glim0}. Compared to the numerator in \eqref{Glim0}, the only difference arises from \eqref{Glim9} and \eqref{Glim10}, where $e^{is}$ is replaced by $1$.
Thus, we obtain
\begin{align}
&\E_{\nu^\perp_{Q, x_0}}  \Big[    e^{\bar {\mathcal{E}}(h^\perp, t^{+}) } e^{c_0(t^+)^2+b(h^\perp)t^+}, \, D_L, \, B_K     \Big] \notag \\
&=\big(1+o_L(1) \big) \E_{\nu^{\perp \perp}_{Q, x_0}}\Big[ e^{V(h^{\perp \perp}) }, \, \cj D_L, \, \cj B_K  \Big]+o_L(1)\E_{\nu^{\perp \perp}_{Q, x_0}}\Big[ e^{V_1(h^{\perp \perp}) }, \, \cj D_L, \, \cj B_K  \Big]+O(e^{-M^2 L}).
\label{Glim15}
\end{align}

Since $V_3(h^{\perp \perp})=O(1)$, $V_4(h^{\perp \perp})=O(1)$, and $V_1(h^{\perp \perp})=O(L)$ on the set $\cj D_L \cap \cj B_K$, and $V=V_3+V_4+V_1 $ in \eqref{Glim7}, we have 
\begin{align}
\frac{e^{-M^2 L}}{\E_{\nu^{\perp \perp}_{Q, x_0}}\Big[ e^{V(h^{\perp \perp}) }, \, \cj D_L, \, \cj B_K  \Big]} \les  \frac{e^{-M^2 L}}{\E_{\nu^{\perp \perp}_{Q, x_0}}\Big[ e^{V_1(h^{\perp \perp}) }, \, \cj D_L, \, \cj B_K  \Big]} \les e^{cL} e^{-M^2 L}\les e^{-\frac{M^2}{2} L},
\label{Glim16}
\end{align}

\noi
where $M$ is chosen to be sufficiently large. By combining \eqref{Glim14}, \eqref{Glim15} and \eqref{Glim16}, we obtain 
\begin{align}
&\frac{ \E_{\nu^\perp_{Q, x_0}}  \Big[  e^{i\jb{h^\perp, g_L} }  e^{\bar {\mathcal{E}}(h^\perp, t^{+}) } e^{c_0(t^+)^2+b(h^\perp)t^+}, \, D_L, \, B_K     \Big]}{ \E_{\nu^\perp_{Q, x_0}}  \Big[ e^{\bar {\mathcal{E}}(h^\perp, t^{+}) } e^{c_0(t^+)^2+b(h^\perp)t^+}, \, D_L, \, B_K   \Big] } \notag \\
&=     \frac{  e^{-\frac{1}{2} \|g \|_{L^2}^2 }+o_L(1)  +o_L(1)\ld(\cj D_L, \cj B_K)   }{1+o_L(1) +o_L(1)\ld(\cj D_L, \cj B_K)}+ O(e^{-\frac{M^2}{2} L})
\label{Glim12}
\end{align}

\noi
as $L\to \infty$, where 
\begin{align*}
\ld(\cj D_L, \cj B_K)=\frac{  \E_{\nu^{\perp \perp}_{Q, x_0}}\Big[ e^{V_1(h^{\perp \perp}) }, \, \cj D_L, \, \cj B_K  \Big]   }{ \E_{\nu^{\perp \perp}_{Q, x_0}}\Big[ e^{V(h^{\perp \perp}) }, \, \cj D_L, \, \cj B_K  \Big]   }.
\end{align*}

\noi
Since $V_3(h^{\perp \perp})=O(1)$, $V_4(h^{\perp \perp})=O(1)$, and $V_1(h^{\perp \perp})=O(L)$ on the set $\cj D_L \cap \cj B_K$, we have 
\begin{align}
\ld(\cj D_L, \cj B_K)=O(1).
\label{Glim13}
\end{align}

\noi
Therefore, combining \eqref{Glim12} and \eqref{Glim13} shows 
\begin{align*}
\frac{ \E_{\nu^\perp_{Q, x_0}}  \Big[  e^{i\jb{h^\perp, g_L} }  e^{\bar {\mathcal{E}}(h^\perp, t^{+}) } e^{c_0(t^+)^2+b(h^\perp)t^+}, \, D_L, \, B_K     \Big]}{ \E_{\nu^\perp_{Q, x_0}}  \Big[ e^{\bar {\mathcal{E}}(h^\perp, t^{+}) } e^{c_0(t^+)^2+b(h^\perp)t^+}, \, D_L, \, B_K   \Big] }&= \frac{  e^{-\frac{1}{2} \|g \|_{L^2}^2 }+o_L(1)     }{1+o_L(1)}+ O(e^{-cM^2 L})\\
&=e^{-\frac{1}{2}\|g\|^2_{L^2}}(1+o_L(1))
\end{align*}

\noi
as $L\to \infty$. This completes the proof of Proposition \ref{PROP:glim}.

\end{proof}

In the following lemma, we remove the term $i t^{+}\jb{\g_L, g}$ appearing in Proposition \ref{PROP:gglimit}.

\begin{lemma}\label{LEM:remvit}
Let $g$ be a smooth, compactly supported function with $\textup{dist}(\supp g, x_0)=c_1>0$, where $x_0\in \T_{L}$. Then, we have
\begin{align}
&\frac{\E_{\nu^\perp_{Q, x_0}}  \Big[  e^{i\jb{h^\perp, g_L}}(e^{i t^{+}\jb{\g_L, g} }-1)    e^{\bar {\mathcal{E}}(h^\perp, t^{+}) } e^{c_0(t^+)^2+b(h^\perp)t^+}  , \,D_L, \, B_K  \Big] }{ \E_{\nu^\perp_{Q, x_0}}  \Big[ e^{\bar {\mathcal{E}}(h^\perp, t^{+}) } e^{c_0(t^+)^2+b(h^\perp)t^+}, \, D_L, \, B_K   \Big] } \label{remit} \\
&=O(e^{-\frac{c_1}{2} L} ),\notag 
\end{align}

\noi 
uniformly in $x_0\in \T_{L}$ and $\dr \in [0,2\pi]$, where $g_L(x)=L^{-\frac 12}g(L^{-1}x)$.

\end{lemma}

\begin{proof}
Recall that $\g$ in \eqref{ga0} is localized around $Lx_0$ (see Remark \ref{REM:Lx_0}), allowing us  to write $\g_L(x)=L^{\frac 12}\wt \g(L(x-x_0))$, where $\wt \g$ has an exponential tail.  By using  $t^{+}=O(L^\frac 12)$ and $|\cj {\mathcal{E} }(h^\perp, t^{+}) |=O(1)$ from \eqref{GGGGlim1} and \eqref{GGGlim1} on the set $D_L \cap B_K$,
and the condition $\text{dist}(\supp g, x_0)=c_1>0$, we have 
\begin{align*}
\eqref{remit} &\le  \frac{ \E_{\nu^\perp_{Q, x_0}}  \Big[ |\jb{\g, g_L}| |t^{+}|   e^{\bar {\mathcal{E}}(h^\perp, t^{+}) } e^{c_0(t^+)^2+b(h^\perp)t^+}  , \,D_L, \, B_K  \Big] }{\E_{\nu^\perp_{Q, x_0}}  \Big[ e^{\bar {\mathcal{E}}(h^\perp, t^{+}) } e^{c_0(t^+)^2+b(h^\perp)t^+}, \, D_L, \, B_K   \Big]} \notag \\
&\le C \int L\wt \g(L(x-x_0))g(x)\, \mathrm{d}x \notag \\
&\les  L\int_{\supp g} e^{-L|x-x_0 |} g(x) \, \mathrm{d}x  \notag \\
&=O(e^{-\frac{c_1}{2} L}) \| g\|_{L^1}
\end{align*}

\noi
as $L\to \infty$. Hence, we obtain the desired result.

\end{proof}

We are now ready to prove Proposition \ref{PROP:gglimit}.
\begin{proof}[Proof of Proposition \ref{PROP:gglimit}]
Combining Proposition \eqref{PROP:glim} and Lemma \ref{LEM:remvit} yields 
\begin{align*}
&\frac{ \E_{\nu^\perp_{Q, x_0}}  \Big[  e^{i\jb{h^\perp, g_L}+i t^{+}\jb{\g_L, g}  }  e^{\bar {\mathcal{E}}(h^\perp, t^{+}) } e^{c_0(t^+)^2+b(h^\perp)t^+}  , \,D_L, \, B_K  \Big]}{ \E_{\nu^\perp_{Q, x_0}}  \Big[ e^{\bar {\mathcal{E}}(h^\perp, t^{+}) } e^{c_0(t^+)^2+b(h^\perp)t^+} , \,D_L, \, B_K  \Big] }\\
&=O(e^{-\frac{c_1}{2}L }) +\frac{ \E_{\nu^\perp_{Q, x_0}}  \Big[  e^{i\jb{h^\perp, g_L} }  e^{\bar {\mathcal{E}}(h^\perp, t^{+}) } e^{c_0(t^+)^2+b(h^\perp)t^+}, \, D_L, \, B_K     \Big]}{ \E_{\nu^\perp_{Q, x_0}}  \Big[ e^{\bar {\mathcal{E}}(h^\perp, t^{+}) } e^{c_0(t^+)^2+b(h^\perp)t^+}, \, D_L, \, B_K   \Big] }\\
&=o_L(1)+e^{-\frac{1}{2}\|g\|^2_{L^2}}(1+o_L(1))
\end{align*}

\noi
as $L\to \infty$. This completes the proof of Proposition \ref{PROP:gglimit}.

\end{proof}

\section{Tightness in the infinite volume limit}\label{SEC:Tight}
In this section, we establish the tightness of the family $\{(T_L)_{\#}\rho_L\}_{L \ge 1}$ of measures in the infinite volume limit as $L\to \infty$
\begin{align*}
\int F(L(\phi-\pi_L(\phi) ) ) \rho_L(d\phi)=\int F(\phi ) (T_L)_{\#}\rho_L(d\phi),
\end{align*}

\noi
where $T_L(\phi)=L(\phi-\pi_L(\phi))$ and $\pi_L$ is the projection onto the soliton manifold $\M_L$, defined in \eqref{project}.


Before proving the tightness of the family $\{(T_L)_{\#}\rho_L\}_{L \ge 1}$, we present moment estimates for the Ornstein-Uhlenbeck measures $\nu^\perp_{Q,x_0}$ associated with Schrödinger operators, defined in Lemma \ref{LEM:SCHOP}.

\begin{lemma}\label{LEM:moment}
Let $1\le r, q<\infty$. Then, for any finite $p \ge 1$, there exists a constant $C=C(p)>0$, depending on $p$, such that  
\begin{align}
\sup_{L \ge 1}\E_{\nu^\perp_{Q, x_0 } }\Big[   \| h^\perp_L \|_{B^{-s,\mu}_{r,q} }^{p} \Big] &\le C \label{TTI0},
\end{align}

\noi
uniformly in $x_0 \in \T_{L^2}$ and $\dr \in [0,2\pi]$, where $h^\perp_L(x)=L^{\frac 12}h^\perp(Lx)$.
\end{lemma}

\begin{proof}
We may assume $p \ge  \max\{r,q\}$.
By using the Minkowski inequality and the fact that $2^k \varphi(2^k\cdot) * h^\perp_L (x) $ is a Gaussian random variable,
\begin{align}
\E_{\nu_{Q, x_0}^\perp }\Big[   \| h_L^\perp \|_{B^{-s, \mu}_{r,q} }^{p} \Big]&=\E_{\nu_{Q, x_0}^\perp }\bigg| \sum_{k \ge 0} 2^{-skq} \| \Dl_k h_L^\perp \|_{L^r_\mu}^q  \bigg|^p
\notag \\
&\le \Bigg(\sum_{k \ge 0} 2^{-sk q} \bigg( \int_{\R} \E_{\nu_{Q,x_0}^\perp } \Big[ \big|2^k \varphi(2^k\cdot) * h_L^\perp (x)  \big|^p \Big]^{\frac rp} w_\mu(x) dx\bigg)^{\frac qr} \Bigg)^p \notag \\
&\les \Bigg(\sum_{k \ge 0} 2^{-skq} \bigg( \int_{\R} \E_{\nu_{Q,x_0}^\perp} \Big[ \big|2^k \varphi(2^k\cdot) * h_L^\perp (x)  \big|^2 \Big]^{\frac r2}   w_\mu(x) dx\bigg)^{\frac qr} \Bigg)^p.
\label{st1}
\end{align}

\noi
From Lemma \ref{lem: covariance bound} and the change of variables,
\begin{align}
&\E_{\nu_{Q,x_0}^\perp } \Big[ \big|2^k \varphi(2^k\cdot) * h_L^\perp (x)  \big|^2 \Big] \notag \\
&=L 2^{2k} \int_{|z-x|\les 2^{-k}} \int_{|y-x|\les 2^{-k}} \varphi(2^k(x-y)) \varphi(2^k(x-z)) \E_{\nu_{Q,x_0}^\perp}\big[ h^\perp(Ly) h^\perp(Lz) \big] dy dz \notag \\
&\les L 2^{2k} \int_{|z-x|\les 2^{-k}} \int_{|y-x|\les 2^{-k}}  \varphi(2^k(x-y)) \varphi(2^k(x-z)) e^{-L|z-y|} dy dz \notag \\
&\les 2^{2k} \int_{\R} |\varphi(2^k y) |^2 \int_{\R} Le^{-L|y-z|} dzdy+2^{2k} \int_{\R} \varphi(2^ky) \int_{\R} (\varphi(2^k z)-\varphi(2^ky) )Le^{-L|y-z|}dz dy \notag \\
&\les 2^{k}+ \int_{\R} \varphi(y) \int_{\R} Le^{-L|2^{-k}z|}dz dy \notag \\
&\les 2^k+1.
\label{st2}
\end{align}

\noi
as $L\to \infty$. From \eqref{st1} and \eqref{st2}, we have that if $s >\frac 12$
\begin{align}
\sup_{L \ge 1}\E_{\nu_{Q,x_0}^\perp }\Big[   \| h_L^\perp \|_{B^{-s, \mu}_{r,q} }^{p} \Big] &\les \Big(\sum_{k \ge 0} 2^{-skq} 2^{\frac {k q}2} \Big)^p \notag \\
&\les \Big(\sum_{k \ge 0} 2^{-kq(s-\frac 12)} \Big)^p \les 1,
\label{TI3}
\end{align}

\noi
uniformly in $x_0 \in \T_{L^2}$ and $\dr \in [0,2\pi]$.

\end{proof}

We now present the proof of the tightness of the family $\{(T_L)_{\#}\rho_L\}_{L \ge 1}$.

\begin{proposition}\label{PROP:Tight}
Let $r,q \in [1,\infty]$ with $r \neq \infty$, $s>\frac 12$ and $\mu>0$. As probability measures on $B^{-s,\mu}_{r,q}$, the family  $\{(T_L)_{\#}\rho_L\}_{L \ge 1}$ is tight. 
\end{proposition}

\begin{proof}

By assuming the following moment bound 
\begin{align}
\sup_{L \ge 1}\E_{(T_L)_{\#}\rho_L}\big[\| \phi \|^p_{B^{-s, \mu}_{r,q}} \big]<\infty
\label{momrho}
\end{align}

\noi
for any $1\le p<\infty$, $s>\frac 12$, $\mu>0$, and $r,q\in [1,\infty]$ with $r\neq \infty$, we first prove the tightness of the family $\{(T_L)_{\#}\rho_L\}_{L \ge 1}$. 

Let $0<\mu' \ll \mu$. Given $M \ge 1$, which will be chosen in \eqref{chosen}, it follows from Lemma \ref{LEM:compact} and $\|f \|_{B^{s,\mu}_{r,q_1} } \le \| f \|_{B^{s,\mu}_{r,q_2} } $ with $q_1 \ge q_2$ that 
\begin{align*}
K_M:=\big\{\phi \in B^{-s,\mu}_{r,q} : \| \phi\|_{B^{-s+\eps, \mu'}_{r,q} } \le M  \big\}  
\end{align*}

\noi 
is compact in $B^{-s,\mu}_{r,q}$. Combining Chebyshev’s inequality and \eqref{momrho} shows that given $\eta>0$, there exists $M \ge 1 $ such that 
\begin{align}
(T_L)_{\#}\rho_L(K_M^c)=(T_L)_{\#}\rho_L\Big( \big\{  \| \phi\|_{B^{-s+\eps, \mu'}_{r,q} } >M  \big\}  \Big)\les M^{-p}<\eta,
\label{chosen}
\end{align}

\noi
uniformly in $L \ge 1$. Therefore, by Prokhorov’s theorem, the family
$\{(T_L)_{\#}\rho_L\}_{L \ge 1} $ is tight.

\noi 
It suffices to prove the moment bound \eqref{momrho}. From $T_L(\phi)=L(\phi-\pi_L(\phi) )$, where $\pi_L$ is the projection onto the soliton manifold $\M_L$, defined in $\eqref{project}$, we note that 
\begin{align}
\E_{(T_L)_{\#}\rho_L}\big[\| \phi \|^p_{B^{-s,\mu}_{r,q}} \big]&=\int \| L(\phi-\pi_L(\phi)) \|_{B^{-s,\mu}_{r,q} }^p \rho_L(d\phi) \notag \\
&=\int_{\{  \text{dist}(\phi,\M_L)    \ge \dl L^{\frac 12} \} } \| L(\phi-\pi_L(\phi)) \|_{B^{-s, \mu}_{r,q} }^p \rho_L(d\phi)  \notag \\
&\hphantom{X} + \int_{\{  \text{dist}(\phi,\M_L)   < \dl L^{\frac 12} \} } \| L(\phi-\pi_L(\phi)) \|_{B^{-s, \mu}_{r,q} }^p \rho_L(d\phi) \notag \\
&=\I_1+\I_2.
\label{I1I2}
\end{align}

\noi 
Thanks to Proposition \ref{PROP:con},
\begin{align}
\I_1 &\les \E_{\rho_L}\Big[ \| L(\phi-\pi_L(\phi)) \|_{B^{-s,\mu}_{r,q} }^{2p}\Big]^{\frac 12} e^{-\frac{c(\dl)}2 L^3} \notag \\
&\les \E_{\rho_L}\big[ \|\phi \|_{B^{-s,\mu}_{r,q} }^{2p} \big]^{\frac 12} e^{-\frac{c(\dl)}4 L^3}.
\label{PAS0} 
\end{align}

From Minkowski and Bernstein inequality in Lemma \ref{LEM:Ber},
\begin{align}
\E_{\rho_L}\big[ \|\phi \|_{B^{-s,\mu}_{r,q} }^{2p} \big]&\le \bigg(\sum_{k\ge 0} 2^{-skq}  \E_{\rho_L} \big[ \| \Dl_k \phi \|_{L^r_\mu}^{2p} \big]^{\frac q{2p}} \bigg)^\frac {2p}q \notag \\
&\les \bigg(\sum_{k\ge 0} 2^{-kq(s-1/2+1/r)}  \E_{\rho_L} \big[ \| \Dl_k \phi \|_{L^2_{2\mu /r }}^{2p} \big]^{\frac q{2p}} \bigg)^\frac {2p}q
\label{TTI10}
\end{align}

\noi
Thanks to the condition $\M_L(\phi)=\| \phi \|_{L^2(\T_L)}^2\le LD$ under the $2L$-periodic Gibbs measure $\rho_L$,
\begin{align}
\| \Dl_k \phi \|_{L^2_{2\mu /r }}^{2p}&=\bigg(\sum_{n\in \Z} \int_{[(2n-3)L, (2n-1)L]} |\Dl_k \phi(x)|^2 e^{-\frac{2\mu}{r}\jb{x}^\dl } dx  \bigg)^{p} \notag \\
&\les (LD)^p   \Big( \sum_{n \in \Z} e^{-c\mu  L^\dl n^\dl} \Big)^p \notag \\
&\les (LD)^p.
\label{TTI11}
\end{align}

\noi
From \eqref{TTI10} and \eqref{TTI11}, 
\begin{align}
\E_{\rho_L}\big[ \|\phi \|_{B^{-s,\mu}_{r,q} }^{2p} \big] \les (LD)^p.
\label{TTI12}
\end{align}

\noi
By combining  \eqref{PAS0} and \eqref{TTI12}, we have  
\begin{align}
\I_1 &\les \E_{\rho_L}\big[ \|\phi \|_{B^{-s,\mu}_{r,q} }^{2p} \big]^{\frac 12} e^{-\frac{c(\dl)}4 L^3} \les e^{-\frac{c(\dl)}8 L^3} \les 1,
\label{uniI10}
\end{align}

\noi
uniformly in $L \ge 1$. This completes the estimate for $\I_1$.

We next study $\I_2$ in \eqref{I1I2}. To apply Lemma \ref{LEM:chan}, we need the integrand to be both bounded and continuous. However, the integrand
\begin{align*}
\| L(\phi-\pi_L(\phi)) \|_{B^{-s, \mu}_{r,q} }^p e^{\frac 14 \int_{-L}^L |\phi|^4  dx} \ind_{ \{ M_L(\phi) \le LD \} }
\end{align*}

\noi 
is neither bounded nor continuous. The discontinuity arises from the sharp cutoff $\ind_{\{ M_L(\phi) \le LD \} }$. However, since the set of discontinuities has $\mu$-measure zero, we can initially use a smooth cutoff and then take the limit. Alternatively, we can replace the integrand with a bounded one, provided the bounds are uniform, employing a straightforward approximation argument that we omit here.

By following the arguments in \eqref{CHA00}, we have
\begin{align}
\I_2=\int_{x_0 \in \T_{L^2} } \int_{\dr \in [0,2\pi ]} \mathcal{I}(\| h_L\|_{B^{-s,\mu}_{r,q} }^p)\, \mathrm{d} \dr  \, \mathrm{d}x_0 \bigg/ \int_{x_0 \in \T_{L^2} } \int_{\dr \in [0,2\pi ]} \mathcal{I}(1)\, \mathrm{d} \dr  \, \mathrm{d}x_0, 
\label{TII0}
\end{align}

\noi
where $\mathcal{I}(F)$ is defined in \eqref{Idef}. Note that 
\begin{align}
\mathcal{I}(\| h_L\|_{B^{-s, \mu}_{r,q} }^p)=\mathcal{I}(\| h_L\|_{B^{-s,\mu}_{r,q} }^p, D_L)+\mathcal{I}(\| h_L\|_{B^{-s, \mu}_{r,q} }^p ,D_L^c).  
\label{SM7}
\end{align}

\noi
We first study $\mathcal{I}(\| h_L\|_{B^{-s, \mu}_{r,q} }^p ,D_L^c)$.
From \eqref{SQ0}, \eqref{Y4} and \eqref{Y7}, we write  
\begin{align}
\mathcal{I}(\| h_L^\perp \|_{B^{-s, \mu}_{r,q } }^p ,D_L^c) \les  \mathbb{E}_{\nu_{Q,x_0}^\perp }\Big[ \| h^\perp_L\|_{B^{-s}_{r,q,\mu} }^p H(h^\perp),\,\|h^\perp\|_{L^2}\le \delta L^{\frac{3}{2}}, \, D_L^c \Big].
\label{TI0} 
\end{align}

\noi 
By the H\"older's inequality and Proposition \ref{PROP: gaussian-conc}, 
\begin{equation}\label{TI1}
\begin{split}
|\eqref{TI0}|&\les \E_{\nu_{Q, x_0 }^\perp }\Big[   \| h_L^\perp\|_{B^{-s ,\mu}_{r,q} }^{3p} \Big]^{\frac 13} \mathbb{E}_{\nu_{Q,x_0}^\perp }\Big[H(h^\perp)^3,\,\|h^\perp \|_{L^2}\le \delta L^{\frac{3}{2}}\Big]^{\frac{1}{3}}\nu_{Q, x_0}^\perp (D^c_L)^{\frac{1}{3}}\\
&\les \E_{\nu_{Q, x_0 }^\perp }\Big[   \| h_L^\perp \|_{B^{-s,\mu}_{r,q} }^{3p} \Big]^{\frac 13}   \mathbb{E}_{\nu_{Q, x_0}^\perp}\Big[H(h^\perp)^3, \,\|h^\perp \|_{L^2}\le \delta L^{\frac{3}{2}}\Big]^{\frac{1}{3}} e^{-c DL^2}.
\end{split}
\end{equation}

By combining \eqref{TI1}, Lemma \ref{LEM: CL-est}, Lemma \ref{lem: H-est}, and Lemma \ref{LEM:moment}, we have
\begin{align}
\Bigg| \frac{\mathcal{I}(\| h_L\|_{B^{-s,\mu}_{r,q} }^p ,D_L^c)}{\mathcal{I}(1)} \Bigg| &\les  e^{cL} \E_{\nu_{Q, x_0 }^\perp }\Big[   \| h_L^\perp \|_{B^{-s, \mu}_{r,q} }^{3p} \Big]^{\frac 13}   \mathbb{E}_{\nu_{Q, x_0}^\perp }\Big[H(h^\perp)^3, \,\|h^\perp \|_{L^2}\le \delta L^{\frac{3}{2}}\Big]^{\frac{1}{3}} e^{-c DL^2}.
\notag \\
&\les   e^{-\frac{c}{2}D L^2},
\label{TI2}
\end{align}

\noi
uniformly in $x_0\in \T_{L^2}$ and $\dr \in [0,2\pi]$, where $D$ is chosen to be sufficiently large.

We next consider $\mathcal{I}(\| h_L\|_{B^{-s, \mu}_{p,q} }^p, D_L)$ in \eqref{SM7}. From Lemma \ref{lem: replace}, we have 
\begin{align}
&\mathcal{I}(\| h_L \|_{B^{-s, \mu}_{r,q} }^p, D_L)\\
&\les  \mathbb{E}_{\nu_{Q, x_0}^\perp }\bigg[ \| h_L^\perp \|_{B^{-s, \mu}_{q,r} }^p  \int_{S_{h^\perp} }e^{\bar{\mathcal{E} }(h^\perp, t) }e^{\frac{\Ld}{2}\|h^\perp\|_2^2}e^{\Ld G_1(t)}e^{G_2(t,h^\perp)} e^{-\frac{t^2}{2\sigma^2}}\frac{\mathrm{d}t}{\sqrt{2\pi} \sigma}, \, D_L \bigg], \notag \\
&\les  \E_{\nu^\perp_{Q, x_0} } \bigg[  \| h_L^\perp \|_{B^{-s,\mu}_{q,r} }^p e^{\bar{\mathcal{E} }(h^\perp, t^+) }e^{c_0(t^+)^2+b(h^\perp)t^+}, \, D_L \bigg](1+o_L(1)).
\label{TI4}
\end{align}

\noi

We may assume $p \ge \max\{q,r\}$. By using the Minkowski inequality, we have  
\begin{align}
&\E_{\nu^\perp_{Q,x_0}}\Big[ \| h^\perp_L\|_{B^{-s,\mu}_{q,r} }^p e^{\mathcal{E}(h^\perp, t^{+}) } e^{c_0 (t^+)^2 +b(h^\perp) t^{+}  }, \, D_L \Big] \notag \\
&=\E_{\nu^\perp_{Q,x_0}}\bigg[  \Big(\sum_{k \ge 0 } 2^{-skq} \| \Dl_k h^\perp_L \|_{L^r_\mu}^q \Big)^\frac pq  e^{\mathcal{E}(h^\perp, t^{+}) } e^{c_0 (t^+)^2 +b(h^\perp) t^{+}  }, \, D_L\bigg] \notag \\
&\le \Bigg( \sum_{k \ge 0} 2^{-skq} \bigg(\int_{\R} \E_{\nu^\perp_{Q,x_0}}\Big[ \big|2^k \varphi(2^k\cdot)*h^\perp_L(x) \big|^p e^{\mathcal{E}(h^\perp, t^{+}) } e^{c_0 (t^+)^2 +b(h^\perp) t^{+}  }, \, D_L
\Big]^{\frac rp} w_\mu(x)\, \mathrm{d}x \bigg)^\frac qr \Bigg)^\frac pq. 
\label{SM8}
\end{align}

\noi
By following the proof of Proposition \ref{PROP:error1}, in particular, \eqref{EE00} and Lemma \ref{lem: mu2-ests}, putting $B_K^c $ into \eqref{SM8} is an error term, and so it suffices to consider the expectation
\begin{align}
\E_{\nu^\perp_{Q,x_0}}\Big[ \big| 2^k \varphi(2^k\cdot)*h^\perp_L(x) \big|^p e^{\mathcal{E}(h^\perp, t^{+}) } e^{c_0 (t^+)^2 +b(h^\perp) t^{+}  }, \, D_L, \, B_K\Big].
\label{SM88}
\end{align}

\noi 
Also, using Lemma \ref{lem: replace} gives the lower bound for the denominator 
\begin{align}
\mathcal{I}(1)\ges \E_{\nu^\perp_{Q,x_0}}\Big[ e^{\mathcal{E}(h^\perp, t^{+}) } e^{c_0 (t^+)^2 +b(h^\perp) t^{+}  }, \, D_L, \, B_K\Big].
\label{SM9}
\end{align}

\noi
Hence, \eqref{SM8}, \eqref{SM88}, and \eqref{SM9} imply that we can reduce our focus to studying
\begin{align}
\frac{\mathcal{I}(\| h_L \|_{B^{-s, \mu}_{r,q} }^p, D_L)}{\mathcal{I}(1) } \les  \Bigg( \sum_{k \ge 0} 2^{-sk q} \bigg( \int_{\R}  \mathcal{M}(k,L,x,x_0)^\frac rp w_\mu(x)\, \mathrm{d}x \bigg)^\frac qr \Bigg)^\frac pq,
\label{SM13}
\end{align}

\noi
where 
\begin{align*}
 \mathcal{M}(k,L,x,x_0):=\frac{\E_{\nu^\perp_{Q,x_0}}\Big[ \big| 2^k \varphi(2^k\cdot)*h^\perp_L(x) \big|^p e^{\mathcal{E}(h^\perp, t^{+}) } e^{c_0 (t^+)^2 +b(h^\perp) t^{+}  }, \, D_L, \, B_K\Big] }{\E_{\nu^\perp_{Q,x_0}}\Big[ e^{\mathcal{E}(h^\perp, t^{+}) } e^{c_0 (t^+)^2 +b(h^\perp) t^{+}  }, \, D_L, \, B_K\Big]}.
\end{align*}

\noi 
Note that $2^k\varphi(2^k\cdot)*h^\perp_L(x)=\jb{2^k \varphi(2^k(x-\cdot)), h^\perp_L }$ is a Gaussian random variable with variance $\s_{k,x}^2$.  Compared to \eqref{Glim0}, the only difference is the Gaussian term $\jb{2^k \varphi(2^k(x-\cdot)), h^\perp_L }$ replacing  $e^{i\jb{h^\perp,g_L  }}$, where  $\jb{h^\perp,g_L  }$ is also Gaussin random variable. To follow the same arguments in the proof of Proposition \ref{PROP:glim}, we proceed as in \eqref{K1} by defining an orthogonal projection under the measure $\nu^\perp_{Q,x_0}$ onto the direction $\jb{2^k \varphi(2^k(x-\cdot)), h^\perp_L }$ 
\begin{align}
h^\perp(x)=\beta(x) 
\jb{2^k \varphi(2^k(x-\cdot)), h^\perp_L }+ h^{\perp \perp}(x),
\label{SM10}
\end{align}

\noi
where $\jb{2^k \varphi(2^k(x-\cdot)), h^\perp_L }$ and $h^{\perp \perp}$ are independent Gaussian random fields and 
\begin{align*}
\beta(x)=\frac{\E_{\nu^\perp_{Q,x_0}} \big[  \jb{2^k \varphi(2^k(x-\cdot)), h^\perp_L } h^\perp(x)  \big]   }{\E_{\nu^\perp_{Q,x_0}} \big[  |\jb{2^k \varphi(2^k(x-\cdot)), h^\perp_L }|^2   \big]  }.
\end{align*}

\noi
Then, by following the proof of \eqref{Glim12},  we can cancel terms in the numerator and denominator related to $h^{\perp \perp}$ to obtain
\begin{align}
&\frac{\E_{\nu^\perp_{Q,x_0}}\Big[ \big| 2^k \varphi(2^k\cdot)*h^\perp_L(x) \big|^p e^{\mathcal{E}(h^\perp, t^{+}) } e^{c_0 (t^+)^2 +b(h^\perp) t^{+}  }, \, D_L, \, B_K\Big] }{\E_{\nu^\perp_{Q,x_0}}\Big[ e^{\mathcal{E}(h^\perp, t^{+}) } e^{c_0 (t^+)^2 +b(h^\perp) t^{+}  }, \, D_L, \, B_K\Big]} \notag \\
&=\frac{(\s_{k,x}^p p^{\frac p{2}}+o_L(1)) \E_{\wt \nu^{\perp \perp}_{Q,x_0} }\Big[e^{\cj V_1(h^{\perp \perp})},\, \cj D_L, \, \cj B_K  \Big]  }{(1+o_L(1))\E_{\wt \nu^{\perp \perp}_{Q,x_0} }\Big[e^{\cj V_2(h^{\perp \perp})}, \, \cj D_L, \, \cj B_K  \Big]  } +o_L(1) \notag \\
&\approx \frac{(\s_{k,x}^p p^{\frac p{2}}+o_L(1))   }{(1+o_L(1)) } +o_L(1)
\les \s_{k,x}^p p^{\frac p2},
\label{SM11}
\end{align}

\noi
independent of $x\in \R$, where $\wt \nu^{\perp \perp}_{Q,x_0}$ is the underlying measure for the field $h^{\perp \perp}$ in the decomposition in \eqref{SM10}, $\cj V_1$, $\cj V_2$ are some functions in $h^{\perp \perp}$, and the ratio of their expected values has the size $O(1)$ as in \eqref{Glim13}. In particular,
the term $\s_{k,x}^p p^{\frac p{2}}$, which replace $e^{-\frac 12\| g\|_{L^2}^2}$ in \eqref{Glim12} and arises from \eqref{Glim10}, comes from 
\begin{align*}
\int_{\R} |y|^p e^{-\frac{s^2}{2\s_{k,x}^2}} \, \frac{\mathrm{d}s}{\sqrt{2\pi} \s_{k,x}}\les \s_{k,x}^p p^{\frac p2}.
\end{align*}

\noi
Therefore, combining \eqref{SM13}, and \eqref{SM11} yields 
\begin{align}
\frac{\mathcal{I}(\| h_L \|_{B^{-s, \mu}_{r,q} }^p, D_L)}{\mathcal{I}(1) }\les \Bigg( \sum_{k \ge 0} 2^{-sk q} \bigg( \int_{\R}  (\s_{k,x}^p)^{  \frac rp} w_\mu(x)\, \mathrm{d}x \bigg)^\frac qr \Bigg)^\frac pq,
\label{SM14}
\end{align}

\noi
uniformly in $x_0 \in \T_{L^2}$ and $\dr \in [0,2\pi]$.

\noi 
By undoing the scaling in the field $h^\perp_L(x)=L^\frac 12 h(Lx)$,
the variance of the Gaussian random variable $2^k\varphi(2^k\cdot)*h^\perp_L(x)$ is given by
\begin{align}
\s_{k,x}^2&=\E_{\nu^\perp_{Q,x_0} }\big[ |\jb{ 2^k\varphi(2^k(x-\cdot )), h^\perp_L   }|^2   \big] \notag \\
&=L^{-1}  \E_{\nu^\perp_{Q,x_0} }\big[ |\jb{ 2^k\varphi(2^k(x-\frac{\cdot}{L} )), h^\perp   }|^2   \big] \notag \\
&=L^{-1}  \jb{  C_{1, x_0}^Q  2^k\varphi(2^k(x-L^{-1}\cdot ))    ,2^k\varphi(2^k(x-L^{-1}\cdot ))} \notag \\
&\hphantom{X}+L^{-1}\jb{  C_{2, x_0}^Q  2^k\varphi(2^k(x-L^{-1}\cdot ))    ,2^k\varphi(2^k(x-L^{-1}\cdot ))} \notag \\
&=I_1(k,x)+I_2(k,x).
\label{SM15}
\end{align}

\noi
Since $I_1(k,x)$ and $I_2(k,x)$  are of the same nature, it suffices to estimate $I_1(k,x)$. Note that 
\begin{align*}
I_1(k,x)=&L^{-1}  \jb{  C_{1, x_0}^Q  2^k\varphi(2^k(x-L^{-1}\cdot ))    ,2^k\varphi(2^k(x-L^{-1}\cdot ))}\\
&=L^{-1}\int_{\T_{L^2}} \int_{\T_{L^2}} G_{1,x_0}(z,y)2^k\varphi(2^k(x-L^{-1}y ))\, \mathrm{d}y  \, 2^k\varphi(2^k(x-L^{-1}z ))\, \mathrm{d}z\\
&=\int_{\T_{L}} \int_{\T_{L}} LG_{1,x_0}(Lz,Ly)2^k\varphi(2^k(x-y ))\, \mathrm{d}y  \, 2^k\varphi(2^k(x-z ))\, \mathrm{d}z\\
&\les  \int_{\T_{L}} \int_{\T_{L}} L e^{-\text{dist}(L(z-y), 2L^2 \Z)  } 2^k\varphi(2^k(x-y ))\, \mathrm{d}y  \, 2^k\varphi(2^k(x-z ))\, \mathrm{d}z,
\end{align*}

\noi
where $G_{1,x_0}$ is the Green’s function for the covariance operator $C^Q_{1,x_0}$, as given in \eqref{eqn: G-est}.

We first consider the case $L \gg 2^k$. Due to the exponential decay of correlations, $z$ and $y$ are effectively forced to be equal, which implies
\begin{align}
I_1(k,x) \les \| 2^k \varphi(2^k(x-\cdot)  )\|_{L^2}^2 \les 2^k \|\varphi \|_{L^2}^2.
\label{SM16}
\end{align}

Next, we consider the case $L \les 2^k$. Then, from Lemma \ref{lem: covariance bound},  
\begin{align*}
G_{1,x_0}(x,y) \les 1,
\end{align*}

\noi
uniformly in $x_0,x,y \in \T_{L^2} $. This implies that
\begin{align}
\I_1(k,x)&\les  L\int_{\T_L} \int_{\T_L} 2^k\varphi(2^k(x-y ))\, \mathrm{d}y  \, 2^k\varphi(2^k(x-z ))\, \mathrm{d}z \notag \\
&\les L \| \varphi \|_{L^1} \| \varphi\|_{L^1}\les 2^k \|\varphi \|_{L^1}^2.
\label{SM17}
\end{align}

\noi
Hence, by combining \eqref{SM15}, \eqref{SM16}, and \eqref{SM17}, we have 
\begin{align}
\s_{k,x}^2 \les 2^k,
\label{SM18}
\end{align}

\noi
uniformly in  $x_0,x\in \T_{L^2}$. By plugging \eqref{SM18} into \eqref{SM14}, we obtain 
\begin{align}
\frac{\mathcal{I}(\| h_L \|_{B^{-s, \mu}_{r,q} }^p, D_L)}{\mathcal{I}(1) }&\les \Bigg( \sum_{k \ge 0} 2^{-sk q} \bigg( \int_{\R}  2^{ \frac{kr}2 } w_\mu(x)\, \mathrm{d}x \bigg)^\frac qr \Bigg)^\frac pq \notag \\
&\les \Big( \sum_{k \ge 0} 2^{-sk q} 2^{\frac {kq}2 } \Big)^\frac pq \les 1,
\label{SM19}
\end{align}

\noi
uniformly in $x_0 \in \T_{L^2}$ and $\dr \in [0,2\pi]$, where in the last line, we used the condition $s>\frac 12$.

Then, \eqref{TII0}, \eqref{TI2}, and \eqref{SM19} imply 
\begin{align}
|\I_2|&=\int_{x_0 \in \T_{L^2} } \int_{\dr \in [0,2\pi ]} \frac{\mathcal{I}(\| h_L\|_{B^{-s, \mu}_{r,q} }^p)}{\mathcal{I}(1) } \mathcal{I}(1) \, \mathrm{d} \dr  \, \mathrm{d}x_0 \bigg/ \int_{x_0 \in \T_{L^2} } \int_{\dr \in [0,2\pi ]} \mathcal{I}(1)\, \mathrm{d} \dr  \, \mathrm{d}x_0  \notag \\
&\les 1,
\label{TIITT}
\end{align}

\noi
uniformly in $L \ge 1$. By combining \eqref{I1I2}, \eqref{uniI10}, and \eqref{TIITT}, we obtain the result \eqref{momrho}.

\end{proof}

\section{Proof of Theorem \ref{THM:1}}\label{SEC:MAIN}
In this section, we present the proof of the main theorem.

\subsection{Characteristic function}
To prove Theorem \ref{THM:1}, one key ingredient is tightness, already established in Proposition \ref{PROP:Tight}. In this subsection, we provide the other main ingredient, which is the convergence of characteristic functions.

\begin{proposition}\label{PROP:CHAa}
Let $g$ be a real-valued, smooth, compactly supported function on $\R$. Then, we have
\begin{align}
\lim_{L\to \infty }\int e^{i\jb{\Re \phi,g}}  (T_L)_{\#}\rho_L(d\phi)&=e^{-\frac 12 \|g \|_{L^2}^2} \label{real0}\\
\lim_{L\to \infty }\int e^{i\jb{\Im \phi,g}}  (T_L)_{\#}\rho_L(d\phi)&=e^{-\frac 12 \|g \|_{L^2}^2}, \label{imag0}
\end{align}

\noi
where $T_L(\phi)=L(\phi-\pi_L(\phi))$ and $\pi_L$ is the projection onto the soliton manifold $\M_L$, defined in \eqref{project}.
\end{proposition}

\begin{proof}
We prove only the real part \eqref{real0}, as the imaginary part \eqref{imag0} can be proven similarly. For notational convenience, we assume that the field $\phi$ under the Gibbs measure $\rho_L$ is real-valued, that is, $\Re \phi=\phi$.

Since $\supp(\rho_L) \subset \{M_L(\phi)\le LD\} $ where $M_L(\phi)$ is in \eqref{L2mass}, we have 
\begin{align}
\int  e^{i\jb{\phi, g}} (T_L)_{\#}\rho_L(d\phi) &=\int_{A_1} e^{i\jb{T_L(\phi), g}} \rho_L(d\phi)   +\int_{A_2 }  e^{i\jb{T_L(\phi) g}} \rho_L(d\phi)+\int_{A_3 }  e^{i\jb{T_L(\phi), g}} \rho_L(d\phi) \notag \\
&=\I_1+\I_2+\I_3,
\label{MAT0}
\end{align}

\noi
where
\begin{align*}
A_1:&=\{  \text{dist}(\phi,\M_L)    \ge \dl L^{\frac 12} \} \\
A_2:&=\{  \text{dist}(\phi,\M_L)   < \dl L^{\frac 12} \}\cap  \{ M_L(\phi)\le L(D-\eps) \} \\
A_3:&=\{  \text{dist}(\phi,\M_L)   < \dl L^{\frac 12} \}\cap \{ L(D-\eps)  \le M_L(\phi)\le LD \}.
\end{align*}

\noi
Thanks to Proposition \ref{PROP:con}, \eqref{III} and \eqref{III0}, we have 
\begin{align}
\I_1&=O(e^{-c(\dl)L^3}) \label{MATT0}\\
\I_2&=O(e^{-cL^3}). \label{MATT1}
\end{align}

\noi
Hence, it suffices to consider $\I_3$. Based on the sets $A_1, A_2$, and $A_3$, note that 
\begin{align}
Z_L=Z_L[A_1]+Z_L[A_2]+Z_L[A_3],
\label{MATT2}
\end{align}

\noi
where $Z_L$ is the partition function for $\rho_L$ and 
\begin{align*}
Z_L[A_i]:=\int_{A_i} e^{\frac 14\int_{\T_{L}} |\phi|^4 \, \mathrm{d}x} \ind_{ \{ M_L(\phi) \le LD \} }\mu_L(d\phi). 
\end{align*}

\noi
From \eqref{MATT2}, Proposition \ref{PROP:con}, \eqref{III}, and \eqref{III0}, we have $1=O(e^{-cL^3})+Z_L[A_3]/ Z_L$ and so
\begin{align}
Z_L[A_3]=Z_L(1+O(e^{-cL^3})).
\label{MATT3}
\end{align}

\noi
By proceeding as in \eqref{CHA00}, along with Lemma \ref{LEM:chan}, Remark \ref{REM:Lx_0}, and using \eqref{MATT2}, \eqref{MATT3}, we have 
\begin{align}
\I_3&=\int_{ \{\text{dist}(\phi, \M_L)<\dl L^{\frac 12} \} }  e^{i\jb{T_L(\phi), g}}  \ind_{ \{ L(D-\eps)  \le M_L(\phi)\le LD \} } \rho_L(d\phi)  \notag \\
&=Z_{L}^{-1}\int_{\T_{L} } \int_0^{2\pi} \mathcal{I}(e^{i\jb{h_L,g }}) |\g_L(x_0,\dr)| \, \mathrm{d} \dr \, \mathrm{d}x_0 \notag \\
&=Z_{L}[A_3]^{-1}\int_{\T_{L} } \int_0^{2\pi} \mathcal{I}(e^{i\jb{h_L,g }}) |\g_L(x_0,\dr)| \, \mathrm{d} \dr \, \mathrm{d}x_0\notag \\
&\hphantom{XXX}\times \bigg(1+\frac{Z_L[A_1]}{Z_L(1+O(e^{-cL^3}))}+\frac{Z_L[A_2]}{Z_L(1+O(e^{-cL^3}))} \bigg)^{-1},
\label{MATTT3}
\end{align}

\noi
where $\mathcal{I}$ is defined in \eqref{Idef} and $h_L(x)=L^{\frac  12} h(Lx)$. Here, we consider only the case where $h$ is real-valued, as explained at the beginning of the proof. Since $\frac{Z_L[A_1]}{Z_L}=O(e^{-cL^3})$ and $\frac{Z_L[A_2]}{Z_L}=O(e^{-cL^3})$ from Proposition \ref{PROP:con}, \eqref{III}, and \eqref{III0}, \eqref{MATTT3} implies that
\begin{align}
\I_3= Z_{L}[A_3]^{-1}\int_{\T_{L} } \int_0^{2\pi} \mathcal{I}(e^{i\jb{h_L,g }}) |\g_L(x_0,\dr)| \, \mathrm{d} \dr \, \mathrm{d}x_0 \cdot (1+o_L(1))^{-1}.
\label{MATT4}
\end{align}


Note that
\begin{align}
\eqref{MATT4}&=Z_{L}[A_3]^{-1}\int_{\T_{L}\setminus \supp g } \int_0^{2\pi} \mathcal{I}(e^{i\jb{h_L,g }}) |\g_L(x_0,\dr)| \, \mathrm{d} \dr \, \mathrm{d}x_0 \cdot (1+o_L(1))^{-1} \label{MATTT4}\\
&\hphantom{X}+Z_{L}[A_3]^{-1}\int_{ \supp g } \int_0^{2\pi} \mathcal{I}(e^{i\jb{h_L,g }}) |\g_L(x_0,\dr)| \, \mathrm{d} \dr \, \mathrm{d}x_0 \cdot (1+o_L(1))^{-1},
\label{MAT2}
\end{align}

\noi
where
\begin{align}
Z_{L}[A_3]=\int_{\T_{L} } \int_0^{2\pi} \mathcal{I}(1) |\g_L(x_0,\dr)| \, \mathrm{d} \dr \, \mathrm{d}x_0. 
\label{MATT5}
\end{align}

\noi
In obtaining \eqref{MATT5}, we proceed as in \eqref{CHA00} with Lemma \ref{LEM:chan}. Based on the sets $D_L$ and $B_K$ in \eqref{eqn: DL-def} and \eqref{eqn: BK-def}, we write 
\begin{align*}
\mathcal{I}(e^{i\jb{h_L,g }})=\mathcal{I}(e^{i\jb{h_L,g }}, D_L \cap B_K)+\mathcal{I}(e^{i\jb{h_L,g }}, D_L \cap B_K^c)+\mathcal{I}(e^{i\jb{h_L,g }}, D_L^c).
\end{align*}

\noi 
Thanks to Proposition \ref{PROP:main} and Proposition \ref{PROP:gglimit}, under the condition $\text{dist}(\supp g, x_0)>0$ we have 
\begin{align}
\frac{\mathcal{I}(e^{i\jb{h_L,g }}, D_L \cap B_K)}{\mathcal{I}(1) }&=\frac{
\E_{\nu^\perp_{Q, x_0}}  \Big[  e^{i\jb{h^\perp, g_L}+i t^{+}\jb{\g_L, g}  }  e^{\bar {\mathcal{E}}(h^\perp, t^{+}) } e^{c_0(t^+)^2+b(h^\perp)t^+}, \,D_L  \Big]}{ \E_{\nu^\perp_{Q, x_0}}  \Big[ e^{\bar {\mathcal{E}}(h^\perp, t^{+}) } e^{c_0(t^+)^2+b(h^\perp)t^+}, \,D_L  \Big] }+o_L(1) \notag \\
&=e^{-\frac{1}{2}\|g\|^2_{L^2}}(1+o_L(1))
\label{MATT6}
\end{align}

\noi
as $L\to \infty $, uniformly in $x_0 \in \T_{L}$ and $\dr \in [0,2\pi]$. From Proposition \ref{PROP:error1} and \ref{PROP:error2}, 
\begin{align}
\bigg| \frac{\mathcal{I}(e^{i\jb{h_L,g }} \ind_{D_L\cap B_K^c})}{\mathcal{I}(1)} \bigg| &\les e^{-cM^2 L} \label{MATT7}\\
\bigg| \frac{\mathcal{I}(e^{i\jb{h_L,g }} \ind_{D_L^c})}{\mathcal{I}(1)} \bigg| &\les e^{-c DL^2 },
\label{MATT8}
\end{align}

\noi 
uniformly in $x_0 \in \T_{L}$ and $\dr \in [0,2\pi]$. It follows from  $|\g_L(x_0,\dr)|=|\g(0,0)|$ in \eqref{gasurf}, \eqref{MATT6}, \eqref{MATT7}, and \eqref{MATT8} that 
\begin{align}
\eqref{MATTT4}&=\frac{\int_{\T_{L}\setminus \supp g } \int_0^{2\pi} \mathcal{I}(e^{i\jb{h_L,g }}) |\g_L(x_0,\dr)| \, \mathrm{d} \dr \, \mathrm{d}x_0 }  {\int_{\T_{L} } \int_0^{2\pi} \mathcal{I}(1) |\g_L(x_0,\dr)| \, \mathrm{d} \dr \, \mathrm{d}x_0 } \notag \\
&=e^{-\frac 12 \|g\|_{L^2}^2}(1+o_L(1)) \frac{\int_{\T_{L}\setminus \supp g } \int_0^{2\pi} \mathcal{I}(1)  \, \mathrm{d} \dr \, \mathrm{d}x_0 }  {\int_{\T_{L} } \int_0^{2\pi} \mathcal{I}(1) \, \mathrm{d} \dr \, \mathrm{d}x_0 }+O(e^{-cM^2 L})+O(e^{-cDL^2}).
\label{MAT3}
\end{align}

\noi
Note that when we apply the change of variables $f \to f(\cdot+  Lx_0)$, $Q^{\eta_L}_{x_0}$ in \eqref{appground0}  no longer depends on
$Lx_0$. Under the measure $\mu_{Lx_0,\Ld}^\perp $, the field $h(\cdot+Lx_0)$ has the distribution $\mu_{0,\Ld}^\perp$, that is, $\Law_{\mu_{Lx_0,\Ld}^\perp }(h(\cdot+Lx_0))=\mu_{0,\Ld}^\perp$ since
\begin{align*}
\E_{\mu^\perp_{Lx_0, \Ld}}\big[h(x+Lx_0) h(y+Lx_0)  \big]=\E_{\mu^\perp_{0, \Ld}}\big[h(x) h(y)  \big].
\end{align*}

\noi 
This implies that the integral 
$\mathcal{I}(1)$ in \eqref{Idef} does not depend on $x_0$.  Therefore, by taking
the $\dr$-integral outside $x_0$-integral, we have 
\begin{align}
\frac{\int_{\T_{L}\setminus \supp g } \int_0^{2\pi} \mathcal{I}(1)  \, \mathrm{d} \dr \, \mathrm{d}x_0 }  {\int_{\T_{L} } \int_0^{2\pi} \mathcal{I}(1) \, \mathrm{d} \dr \, \mathrm{d}x_0 }=\frac{L-|\supp g|}{L}=1+O(L^{-1})
\label{MAT4}
\end{align}

\noi
as $L\to \infty$ since $g$ has a compact support. Hence, from \eqref{MAT3} and \eqref{MAT4},
\begin{align}
\eqref{MATTT4}&=\frac{\int_{\T_{L}\setminus \supp g } \int_0^{2\pi} \mathcal{I}(e^{i\jb{h_L,g }}) |\g_L(x_0,\dr)| \, \mathrm{d} \dr \, \mathrm{d}x_0 }  {\int_{\T_{L} } \int_0^{2\pi} \mathcal{I}(1) |\g_L(x_0,\dr)| \, \mathrm{d} \dr \, \mathrm{d}x_0 } \notag \\
&=e^{-\frac 12 \|g \|_{L^2}^2}(1+o_L(1))+O(e^{-cM^2 L})+O(e^{-cDL^2})
\label{MAT6}
\end{align}

\noi
as $L\to \infty$.  From $|\g_L(x_0,\dr)|=|\g(0,0)|$ in \eqref{gasurf}, we have
\begin{align}
\eqref{MAT2}=\bigg|\frac{\int_{ \supp g } \int_0^{2\pi} \mathcal{I}(e^{i\jb{h_L,g }}) |\g_L(x_0,\dr)| \, \mathrm{d} \dr \, \mathrm{d}x_0 }  {\int_{\T_{L} } \int_0^{2\pi} \mathcal{I}(1) |\g_L(x_0,\dr)| \, \mathrm{d} \dr \, \mathrm{d}x_0 } \bigg|  \le \frac{ |\supp g| }{L} \to 0
\label{MAT7}
\end{align}

\noi
as $L\to \infty$ since $g$ has a compact support. Combining \eqref{MATT4}, \eqref{MATTT4}, \eqref{MAT2}, \eqref{MAT6}, and \eqref{MAT7} yields 
\begin{align}
\I_3=e^{-\frac 12 \|g \|_{L^2}^2}(1+o_L(1))+o_L(1)
\label{MAT8}
\end{align}

\noi
as $L\to \infty$.

It follows from \eqref{MAT0}, \eqref{MATT0}, \eqref{MATT1}, and \eqref{MAT8} that 
\begin{align*}
\int  e^{i\jb{\phi, g}} (T_L)_{\#}\rho_L(d\phi)&=\I_1+\I_2+\I_3 \notag \\
&=e^{-\frac 12 \|g \|_{L^2}^2}(1+o_L(1))+o_L(1)
\end{align*}

\noi
as $L\to \infty$. This completes the proof of Proposition \ref{PROP:CHAa}.

\end{proof}

\subsection{Proof of the main theorem}

In this subsection, we present the proof of the main theorem by combining Proposition \ref{PROP:Tight} and Proposition \ref{PROP:CHAa}.  

\begin{proof}[Proof of Theorem \ref{THM:1}]

It follows from Proposition \ref{PROP:Tight} and Prohorov’s theorem that for any sequence $\{L_j\}_{j=1}^\infty $ of positive numbers $L\ge 1$  tending to $\infty$, the sequence $\{(T_{L_j})_{\#}\rho_L\}_{j \ge 1} $ is compact. Furthermore, thanks to Proposition \ref{PROP:CHAa}, $\{(T_{L_j})_{\#}\rho_L\}_{j \ge 1}$ converges weakly to the white noise measure $\nu$ on $\R$
\begin{align*}
\nu(d\phi)=Z^{-1} e^{-\frac{1}{2} \int_{\R} |\phi|^2 dx } \prod_{x\in \R} d\phi(x)
\end{align*}

\noi
satisfying $\E_{\nu}\big[e^{i\jb{\phi,g}}\big]=e^{-\frac 12\|g \|_{L^2}^2}$, where $g$ is a test function on $\R$. This implies the uniqueness of the limit point for the family $\{(T_{L})_{\#}\rho_L\}_{L \ge 1} \}$. Hence, we obtain
\begin{align*}
\int F(L(\phi-\pi_L(\phi) ) ) \rho_L(d\phi)=\int F(\phi ) (T_L)_{\#}\rho_L(d\phi) \too \int F(\phi) \nu(d\phi) 
\end{align*}

\noi
as $L\to \infty$. This completes the proof of Theorem \ref{THM:1}.

\end{proof}

\appendix

\section{Analysis of Green's function}

In this section, we study the two-point correlation function, specifically Green’s function, for the Gaussian measures associated with the Schr\"odinger operators in \eqref{SCHOP}. The following two subsections provide decay of correlation and regularity estimates for Green’s functions.

\subsection{Correlation decay}
In this subsection, we study the decay of the correlation functions for the Gaussian measure $\nu^\perp_{Q,x_0}$ given in Lemma \ref{LEM:SCHOP}.

Recall from Lemma \ref{LEM:SCHOP} the definitions of the covariance operators 
\begin{align}
C_{1,x_0}^{Q}&= \mathbf{P}_{\cj V^{L^2, \text{Re}}_{x_0,0}} (B_1^Q)^{-1}\mathbf{P}_{\cj V^{L^2, \text{Re}}_{x_0,0}} \notag \\
C_{2,x_0}^{Q}&= \mathbf{P}_{\cj V^{L^2, \text{Im}}_{x_0,0}} (B_2^Q)^{-1}\mathbf{P}_{\cj V^{L^2, \text{Im}}_{x_0,0}},
\end{align}

\noi 
projecting out the directions that cause the degeneracies.
Then, each operator $C_{i,x_0}^{Q}$ is invertible on the space $\cj V^{L^2, \text{Re}}_{x_0,0}$ and $\cj V^{L^2, \text{Im}}_{x_0,0}$, respectively. Hence, we can take its inverse to study the following Green functions.


\begin{lemma}
\label{lem: covariance bound}
For $i=1,2$, we have 
\begin{equation}\label{eqn: G-est}
G_{i,  x_0 }(x,y):=\E_{ \nu_{C_i^Q, x_0}^\perp }\Big[\phi(x)\cj \phi(y)\Big] \les e^{-\mathrm{dist}(x-y, 2L^2\mathbb{Z})},
\end{equation}


\noi
uniformly in $x_0\in \T_{L^2}$, where $\mathrm{dist}(x, 2L^2\mathbb{Z})$ denotes the distance from $x\in \R$ to the nearest point in $2L^2\mathbb{Z}$.

\end{lemma}
\begin{proof}

We prove only the case $i=1$, as the argument is exactly the same for $i=2$. For simplicity of notation, we set  $\Ld=1$ in the Schr\"odinger operator $B_1^Q$ defined in \eqref{SCHOP}.

We first note that the Green's function $(1-\dx^2)^{-1}$ on $\mathbb{R}$ is given by 
\[(1-\dx^2)^{-1}_{\R}(x,y)=\frac{1}{2\pi}\int_{\R} \frac{e^{i(x-y)\xi}}{1+\xi^2}\,\mathrm{d}\xi= \frac 12 e^{-|x-y|}.\]

\noi 
By the Poisson summation formula, the periodic Green's function on $\T_{L^2}$ satisfies
\begin{align}
(1-\dx^2)^{-1}_{\T_{L^2}}(x,y)=\frac{1}{2L^2}\sum_{m\in \mathbb{Z}}\frac{e^{i\pi \frac{m}{2 L^2}(x-y)}}{\left(\frac{m}{2 L^2}\right)^2+1}=\frac 12 \sum_{n\in \mathbb{Z}}e^{-|x-y-2L^2 n|}\lesssim e^{-\mathrm{dist}(x-y,2L^2\mathbb{Z})}.
\label{OUdec}
\end{align}

\noi
Therefore, we obtain the correlation decay for the covariance operator $(1-\dx^2)^{-1}$ on ${\T_{L^2}}$.

We now prove the correlation decay for the operator $\mathbf{P}_{\cj V^{L^2, \text{Re}}_{x_0,0}} (1-\dx^2)^{-1}\mathbf{P}_{\cj V^{L^2, \text{Re}}_{x_0,0}} $. Recall the definition of the projector $\mathbf{P}_{\cj V^{L^2, \text{Re}}_{x_0,0}} $ in \eqref{eqn: project} 
\begin{align*}
\mathbf{P}_{\cj V^{L^2, \text{Re}}_{x_0,0}} =\Id-\P_{1,x_0}^{L}-\P_{2,x_0}^{L},
\end{align*}

\noi
where each projection is 
\begin{align}
\P_{1,x_0}^{L}&= \langle Q_{1,x_0}^{\eta_L},\cdot\,\rangle Q_{1,x_0}^{\eta_L}, \label{Z6} \\
\P_{2,x_0}^{L}&= \langle Q_{2,x_0}^{\eta_L},\cdot\,\rangle Q_{2,x_0}^{\eta_L},
\label{Z7}
\end{align}

\noi
where $Q_{1,x_0}^{\eta_L}$ and $Q_{2,x_0}^{\eta_L}$ are defined in \eqref{TQF1} and \eqref{TQF3}.  From using projectors  $\P_{1,x_0}^{L}$ and $\P_{2,x_0}^{L}$ in \eqref{Z6} and \eqref{Z7}, we expand
\begin{align}
\mathbf{P}_{\cj V^{L^2, \text{Re}}_{x_0,0}} (1-\dx^2)^{-1}\mathbf{P}_{\cj V^{L^2, \text{Re}}_{x_0,0}}&=(-\dx^2+1 )^{-1}   +\sum_{1 \le i,j \le 2} \P_{i,x_0}^L(-\dx^2+1 )^{-1}\P_{j,x_0}^L \notag \\
&\hphantom{X}-\sum_{1\le j\le 2} (-\dx^2+1 )^{-1}  \P_{j,x_0}^L \notag \\
&\hphantom{X}-\sum_{1\le i\le 2}  \P_{i,x_0}^L (-\dx^2+1 )^{-1}. 
\label{Z11}
\end{align}

\noi
From \eqref{OUdec}, the first term on the right-hand side of  \eqref{Z11} shows the exponential correlation decay \eqref{eqn: G-est}. Hence, it suffices to consider the remaining terms.  Note that for $1\le i,j\le 2$,
\begin{align}
&\big\langle \dl_x , \P^L_{i,x_0} (-\dx^2+1 )^{-1} \P^L_{j,x_0}  \dl_y \big\rangle \notag \\
&=\big\langle  Q^{\eta_L}_{i, x_0},  (-\dx^2+1 )^{-1}  Q^{\eta_L}_{j,x_0}  \big\rangle Q^{\eta_L}_{j,x_0}(y)  Q^{\eta_L}_{j,x_0}(x).
\label{Z12}
\end{align}


\noi
Since $Q^{\eta_L}_{i, x_0}$ and $Q^{\eta_L}_{j,x_0}$ are exponentially localized around $Lx_0$ (see Remark \ref{REM:Lx_0}), we have 
\begin{align}
|\eqref{Z12}| \les e^{-c|x-Lx_0|   }e^{-c|y-Lx_0|} &\lesssim e^{-c|x-y|} \notag \\
&\les e^{-c\mathrm{dist}(x-y, 2L^2 \Z) }.
\label{ZZ12}
\end{align}

\noi
Next, we consider the third term on the right-hand side of \eqref{Z11}. Note that for $1\le i,j\le 2$,
\begin{align}
\big\langle \dl_x ,  (-\dx^2+1 )^{-1} \P^L_{j,x_0}  \dl_y \big\rangle 
&=\big\langle \dl_x ,  (-\dx^2+1 )^{-1}  Q^{\eta_L}_{j,x_0}  \big\rangle Q^{\eta_L}_{j,x_0}(y) \notag \\
&= \bigg(\int_{-L^2}^{L^2} G(x,z) Q^{\eta_L}_{j,x_0} (z)\, \mathrm{d}z \bigg) Q^{\eta_L}_{j,x_0}(y), 
\label{Z13}
\end{align}

\noi
where $G$ is the Green function for the operator $(-\dx^2+1)$.
From \eqref{OUdec} and the fact that $Q^{\eta_L}_{j,x_0}$ is exponentially localized around $Lx_0$, we have
\begin{align}
|\eqref{Z13} | \les e^{-c\mathrm{dist}(x-y, 2L^2 \Z) }.
\label{ZZ13}
\end{align}

\noi
Regarding the last term on the right-hand side of \eqref{Z11}, we can proceed with the above argument to obtain
\begin{align}
|\big\langle \dl_x ,  \P^L_{j,x_0}  (-\dx^2+1 )^{-1}  \dl_y \big\rangle | \les e^{-c\mathrm{dist}(x-y, 2L^2 \Z) }.
\label{Z14}
\end{align}

\noi
Hence, by combining \eqref{Z11}, \eqref{OUdec}, \eqref{Z12}, \eqref{ZZ12}, \eqref{Z13}, \eqref{ZZ13}, and \eqref{Z14}, we obtain the correlation decay
\begin{align}
G_{\text{OU} }(x,y) \les e^{-c\mathrm{dist}(x-y, 2L^2 \Z) },
\label{OUDEC0}
\end{align}

\noi
where $G_{\text{OU} }(x,y)$ is the Green function for the operator $\mathbf{P}_{\cj V^{L^2, \text{Re}}_{x_0,0}} (1-\dx^2)^{-1}\mathbf{P}_{\cj V^{L^2, \text{Re}}_{x_0,0}}$.

\noi 
We now study the Green's function $G_{1,  x_0 }(x,y)$ for the operator $C_{1,x_0}^{Q}= \mathbf{P}_{\cj V^{L^2, \text{Re}}_{x_0,0}} (B_1^Q)^{-1}\mathbf{P}_{\cj V^{L^2, \text{Re}}_{x_0,0}}$. From Lemma \ref{LEM:positive}, $C_{1,x_0}^{Q}$ is a  positive operate on the space $\cj V^{L^2, \text{Re}}_{x_0,0}$ and therefore invertible on $\cj V^{L^2, \text{Re}}_{x_0,0}$. Consequently, we can define the corresponding Green’s function
\begin{align*}
G_{1,  x_0 }(x,y):=\E_{ \nu_{C_1^Q, x_0}^\perp }\Big[\phi(x)\cj \phi(y)\Big].
\end{align*}

\noi
The resolvent identity allows us to treat the Schr\"odinger operator  $C_{1, x_0}^Q$  as a perturbation of $(-\dx^2+1)^{-1}$ as follows 
\begin{align}
G_{1,  x_0 }(x,y)&=G_{\text{OU} }(x,y)-3\int_{\T_{L^2 }} G_{\text{OU} }(x,z) (Q^{\eta_L}_{x_0})^2(z) G_{1,  x_0 }(z,y)\, \mathrm{d}z \label{Res1}\\ 
&=G_{\text{OU} }(x,y)-3\int_{\T_{L^2 }} G_{1,  x_0 }(x,z)  (Q^{\eta_L}_{x_0})^2(z) G_{\text{OU} }(z,y)\, \mathrm{d}z \label{Res2},
\end{align}

\noi
where $G_{\text{OU} }(x,y)$ is the Green’s function for the operator $\mathbf{P}_{\cj V^{L^2, \text{Re}}_{x_0,0}} (1-\dx^2)^{-1}\mathbf{P}_{\cj V^{L^2, \text{Re}}_{x_0,0}}$. Note that, from \eqref{OUDEC0} and the exponential decay of the soliton $Q^{\eta_L}_{x_0}$ localized at $Lx_0$ (see Remark \ref{REM:Lx_0}), we can select $\mu > 0$ such that
 \begin{align}
G_{\text{OU} }(x,y) &\le e^{-\mu \text{dist}(x-y, 2L^2  \Z) }, \label{OUDECMU0}\\
Q^{\eta_L}_{x_0}(z) & \le  e^{-\mu |z-Lx_0|}. \label{OUDECMU1}
\end{align}

\noi 
If $|y-Lx_0|\le \frac 12 \text{dist}(x-y, 2L^2 \Z) $, then we use \eqref{Res1}, \eqref{OUDECMU0}, and \eqref{OUDECMU1} to obtain
\begin{align}
&|G_{1,  x_0 }(x,y)| \notag \\
&\le e^{-\mu \text{dist}(x-y, 2L^2 \Z)  }+ \int_{\T_{L^2}}  e^{-\mu \text{dist}(x-z, 2L^2 \Z)  } e^{-\mu |z-Lx_0|} G_{1,  x_0 }(z,y) \, \mathrm{d}z \notag \\
&\le e^{-\mu \text{dist}(x-y, 2L^2 \Z)  }+ \int_{\T_{L^2}} 
e^{-\frac{\mu}{2} (\text{dist}(x-y, 2L^2 \Z) -|y-Lx_0|-|z-Lx_0|)  }  e^{-\mu |z-Lx_0|} G_{1,  x_0 }(z,y) \, \mathrm{d}z \notag  \\
&\le e^{-\mu \text{dist}(x-y, 2L^2 \Z)  }+e^{-\frac{\mu}{4} (\text{dist}(x-y, 2L^2 \Z)} \int_{\T_{L^2}} e^{-\frac \mu2 |z-Lx_0|}G_{1,  x_0 }(z,y) \, \mathrm{d}z 
\label{expcor00}
\end{align} 

\noi
Since $e^{-\frac \mu2 |z-Lx_0|} \in \cj V^{L^2, \text{Re}}_{x_0,0} $ defined in \eqref{Vbar}, by the elliptic regularity theorem and Sobolev embedding, we obtain
\begin{align}
 \sup_{y}  \bigg| \int_{\T_{L^2}} e^{-\frac \mu2 |z-Lx_0|}G_{1,  x_0 }(z,y) \, \mathrm{d}z  \bigg| \les 1. 
 \label{expcor000}
\end{align}

\noi
Hence, \eqref{expcor00} and \eqref{expcor000} imply that under $|y-Lx_0|\le \frac 12 \text{dist}(x-y, 2L^2 \Z) $,
\begin{align}
|G_{1,  x_0 }(x,y)| \les e^{-\frac{\mu}{4} (\text{dist}(x-y, 2L^2 \Z)}. 
\label{expcor0}
\end{align}

\noi
If $|y-Lx_0|> \frac 12 \text{dist}(x-y, 2L^2 \Z) $, then we use \eqref{Res2}, \eqref{OUDECMU0}, and \eqref{OUDECMU1} to obtain
\begin{align}
&|G_{1,  x_0 }(x,y)| \notag \\
&\le e^{-\mu \text{dist}(x-y, 2L^2 \Z)  }+ 3\int_{\T_{L^2}} G_{1,  x_0 }(x,z)  e^{-\mu |z-Lx_0|}  e^{-\mu \text{dist}(z-y, 2L^2 \Z)  }  \, \mathrm{d}z \notag \\
&\le e^{-\mu \text{dist}(x-y, 2L^2 \Z)  }+ 3\int_{\T_{L^2}} G_{1,  x_0 }(x,z)  e^{-\mu |z-Lx_0|}  e^{-\frac {\mu}2 (|y-Lx_0|- \text{dist}(z-Lx_0, 2L^2 \Z) ) }  \, \mathrm{d}z \notag \\
&\le e^{-\mu \text{dist}(x-y, 2L^2 \Z)  }+3e^{-\frac{\mu}{4} \text{dist}(x-y, 2L^2 \Z)  }\int_{\T_{L^2}} G_{1,  x_0 }(x,z)  e^{-\mu |z-Lx_0|}  e^{\frac \mu2 \text{dist}(z-Lx_0, 2L^2 \Z) )}  \, \mathrm{d}z.
\label{expcor4}
\end{align}

\noi
As before, since the integrand in \eqref{expcor4} is in the space $\cj V^{L^2, \text{Re}}_{x_0,0} $ defined in \eqref{Vbar}, by the elliptic regularity theorem and Sobolev embedding, the integral in \eqref{expcor4} is bounded uniformly in $x$.
This implies that  under $|y-Lx_0|> \frac 12 \text{dist}(x-y, 2L^2 \Z) $,
\begin{align}
|G_{1,  x_0 }(x,y)| \les 2e^{-\frac{\mu}{4} (\text{dist}(x-y, 2L^2 \Z)}. 
\label{expcor1}
\end{align}

\noi 
By combining \eqref{expcor0} and \eqref{expcor1}, we obtain the decay of correlations \eqref{eqn: G-est}.

\end{proof}

\subsection{Regularity of Green's function}

In this subsection, we present the regularity properties of the Green's function associated with the Schr\"odinger operators in \eqref{SCHOP}
\begin{align*}
G_{i,  x_0 }(x,y):=\E_{ \nu_{C_i^Q, x_0}^\perp }\Big[\phi(x)\cj \phi(y)\Big] 
\end{align*}

\noi 
for $i=1,2$. These regularity estimates for the Green's function are used to study the regularity of the Gaussian field $h^\perp$ under the measure $\nu^\perp_{Q, x_0}$ defined in \eqref{GFFSch} associated with the Schr\"odinger operators in \eqref{SCHOP}.

\begin{proposition}\label{PROP:regular}
Let $L \ge 1$. There exists a constant $C>0$, independent of $L$, such that for $i=1,2$,
\begin{align}
\sup_{x\in \T_{L^2}} |G_{i,x_0}(x,y)-G_{i,x_0}(x,z)|&\le C|y-z| \notag \\
\sup_{y\in \T_{L^2}} |G_{i,x_0}(x,y)-G_{i,x_0}(z,y)|&\le C|x-z|,
\label{reguGreen}
\end{align}

\noi
uniformly in $x_0\in\T_{L^2}$ and $\dr \in [0,2\pi]$.
\end{proposition}

\begin{proof}
We prove only the case  $i=1$. Thanks to the resolvent identity, we can  treat the Schr\"odinger operator  $C_{1, x_0}^Q$  as a perturbation of $\mathbf{P}_{\cj V^{L^2, \text{Re}}_{x_0,0}} (1-\dx^2)^{-1}\mathbf{P}_{\cj V^{L^2, \text{Re}}_{x_0,0}}$ as follows 
\begin{align*}
G_{1,  x_0 }(x,y)&=G_{\text{OU}}(x,y)-3\int_{\T_{L^2 }} G_{1,x_0 }(x,z)    (Q^{\eta_L}_{x_0})^2(z) G_{1,  x_0 }(z,y)\, \mathrm{d}z\\
&=G_{\text{OU} }(x,y)-3   (G_{1,x_0 }(Q^{\eta_L}_{x_0})^2 G_{\text{OU} })(x,y),
\end{align*}

\noi
where $G_{\text{OU} }(x,y)$ is the Green’s function for the operator $\mathbf{P}_{\cj V^{L^2, \text{Re}}_{x_0,0}} (1-\dx^2)^{-1}\mathbf{P}_{\cj V^{L^2, \text{Re}}_{x_0,0}}$. It is easy to verify that  $G_{\text{OU} }(x,y)$ satisfies the regularity estimate in \eqref{reguGreen}. Therefore, it suffices to establish the estimate for the perturbation term $(G_{1,  x_0 }(Q^{\eta_L}_{x_0})^2G_{\text{OU} } )(x,y)$.

\noi
From the duality, we have 
\begin{align}
&\sup_{x\in \T_{L^2}}\| \partial_x  G_{1,x_0 }(Q^{\eta_L}_{x_0})^2 G_{\text{OU}}(x,\cdot) \|_{L^\infty} \notag \\
&= \sup_{x\in \T_{L^2}}\sup_{ \| f\|_{L^1} \le 1 }  \bigg| \int \partial_x  G_{1,x_0 }(Q^{\eta_L}_{x_0})^2 G_{\text{OU}}(x,y) f(y) \, \mathrm{d}y \bigg| \notag \\
&\le \sup_{ \| f\|_{L^1} \le 1 } \sup_{x\in \T_{L^2}} \bigg| \int \partial_x  G_{1,x_0 }(Q^{\eta_L}_{x_0})^2 G_{\text{OU}}(x,y) f(y) \, \mathrm{d}y \bigg|.
\label{REG0}
\end{align}

\noi
Note that 
\begin{align}
& \sup_{x\in \T_{L^2}}  \bigg| \int \big(\partial_x  G_{1,x_0 }(Q^{\eta_L}_{x_0})^2 G_{\text{OU}}(x,y) \big) f(y) \, \mathrm{d}y \bigg| \notag \\
&\le  \|\partial_x G_{1,x_0 }(Q^{\eta_L}_{x_0})^2 G_{\text{OU}}  (1-\dx^2)^{\frac{1}{2}+\eps}\|_{L^2\rightarrow L^\infty}\|(1-\dx^2 )^{-\frac{1}{2}-\eps}f\|_{L^2}.
\label{REG1}
\end{align}

\noi
Here, we have 
\begin{equation}\label{eqn: Linftychain}
\begin{split}
&\|\partial_x G_{1,x_0 }(Q^{\eta_L}_{x_0})^2 G_{\text{OU}}  (1-\dx^2)^{\frac{1}{2}+\eps}\|_{L^2\rightarrow L^\infty}\\
&\le \|Q_{x_0}^{\eta_L}\|_{L^\infty} \|\partial_x G_{1,x_0}\|_{L^2\rightarrow L^\infty} \| G_{\text{OU}}  (1-\dx^2)^{\frac{1}{2}+\eps}\|_{L^2\rightarrow L^2}\\
&\les  \|\partial_x G_{1,x_0}\|_{L^2\rightarrow L^\infty}.
\end{split}
\end{equation}

\noi
Recall that the operator $B_1^Q=-\dx^2-3(Q^{\eta_L})_{x_0}^2+1$ is invertible on the space $\cj V^{L^2, \text{Re}}_{x_0,0}$.  Hence,  from the elliptic regularity theory, we have
\begin{align}
\|\partial_{x}^2 u_f\|_{L^2}^2+\|u_f\|_{L^2}^2\le C\|\P_{\cj V^{L^2, \text{Re}}_{x_0,0}} f\|_{L^2}^2,
\label{ELLIP}
\end{align}

\noi
where $B^Q_{1}u_f=\P_{\cj V^{L^2, \text{Re}}_{x_0,0}} f$. Then, from  Sobolev embedding and \eqref{ELLIP}, we have that for any $g \in L^2$
\begin{equation*}
\begin{split}
\sup_{x\in \mathbb{R}}\Big|\int \partial_x G_{1,x_0}(x,y)g(y)\,\mathrm{d}y\Big|& \le C'\|\partial_x u_g\|_{H^{\frac{1}{2}+\eps}}\\
&\le C'\|u_g\|_{H^2}\\
&\le C'\| \P_{\cj V^{L^2, \text{Re}}_{x_0,0}}  g\|_{L^2}.
\end{split}
\end{equation*}

\noi 
This implies that  
\begin{align}
\|\partial_x G_{1,x_0}\|_{L^2\rightarrow L^\infty}\le C'.
\label{REG2}
\end{align}

\noi 
Furthermore, using the Riemann sum approximation, we have
\begin{equation}\label{eqn: sobolevL2}
\begin{split}
\|(1-\dx^2 )^{-\frac{1}{2}-\eps }f\|_{L^2(\T_{L^2})}&=\left(\sum_{n\in \mathbb{Z}} \frac{|\widehat{f}(n)|^2}{(1+ (n/{L^2})^2)^{1+2\eps}}\right)^{\frac{1}{2}} \\
&\le \left(\sum_{n\in \mathbb{Z}} \frac{1}{(1+(n/L^2)^2)^{1+2\eps}} \cdot \frac 1{L^2}\right)^{\frac{1}{2}} \|f\|_{L^1} \\
&\le C\|f\|_{L^1},
\end{split}
\end{equation}

\noi
where the Fourier transform is defined as $\ft f(n)=\frac{1}{\sqrt{L^2}}\int_{\T_{L^2}}f(x) e^{-2\pi i \frac{n}{L^2} x}\, \mathrm{d}x $, with respect to the $L^2(\T_{L^2})$-normalized orthonormal basis, and $C$ is independent of $L \ge 1$.

\noi 
Combining  \eqref{REG0}, \eqref{REG1}, \eqref{eqn: Linftychain}, \eqref{REG2}, and \eqref{eqn: sobolevL2} yields 
\begin{align*}
\sup_{x\in \T_{L^2}}\| \partial_x  G_{1,x_0 }(Q^{\eta_L}_{x_0})^2 G_{\text{OU}}(x,\cdot) \|_{L^\infty} \le \sup_{ \|f \|_{L^1}} C\| f\|_{L^1}=C.
\end{align*}

\noi
This completes the proof of Proposition \ref{PROP:regular}.

\end{proof}

\noi

\begin{ackno}\rm
The work of P.S. is partially supported by NSF grants DMS-1811093 and DMS-2154090. The authors would like to thank Bjoern Bringmann for several conversations and Nikolay Tzvetkov for helpful comments on an earlier version of this manuscript.
\end{ackno}

\end{document}